\documentclass[11pt]{amsart}
\usepackage{amsthm}
\usepackage{amsmath}
\usepackage{amssymb}
\usepackage{amscd}
\usepackage{stmaryrd}
\usepackage{tikz}
\usepackage{tikz-cd}
\usepackage{lipsum}
\usepackage{adjustbox}
\usepackage{url}
\usepackage{graphicx}
\usetikzlibrary{matrix}
\usetikzlibrary{cd}
\usepackage{siunitx}
\usepackage{mathrsfs}
\usepackage{pgf,tikz}
\usepackage{mathrsfs}
\usetikzlibrary{arrows}
\usepackage{graphicx}
\usetikzlibrary{matrix}
\usetikzlibrary{cd}
\usepackage{siunitx}
\usepackage{mathrsfs}
\usepackage{pgf,tikz,pgfplots}

\usepackage[hidelinks]{hyperref}
\hypersetup{allcolors=black}

\newcommand\myurl[1]{\url{#1}}
\allowdisplaybreaks
\usepackage[leqno]{amsmath}

\makeatletter
\newcommand{\leqnomode}{\tagsleft@true}
\newcommand{\reqnomode}{\tagsleft@false}
\makeatother
\makeatletter
\def\l@subsection{\@tocline{2}{0pt}{2.5pc}{5pc}{}}
\makeatother

\DeclareRobustCommand{\SkipTocEntry}[4]{}
\theoremstyle{plain}
\title{Patching over Analytic Fibers and the Local-Global Principle}
\author{Vler\"e Mehmeti} \thanks{The author was supported by the ERC Starting Grant ``TOSSIBERG": 637027.}
\thanks{2010 Mathematics Subject Classification 14G22, 11E08}
\begin{document}
\newtheorem{thm}{Theorem}[section]
\newtheorem{lm}[thm]{Lemma}
\newtheorem{prop}[thm]{Proposition}
\newtheorem{cor}[thm]{Corollary}
\newtheorem*{cor*}{Corollary}
\newtheorem*{thm*}{Theorem}
\theoremstyle{definition}
\newtheorem{defn}[thm]{Definition}
\newtheorem*{defn*}{Definition}
\newtheorem{conv}[thm]{Convention}
\newtheorem{rem}[thm]{Remark}
\newtheorem*{rem*}{Remark}
\newtheorem{ex}{Example}
\newtheorem*{ex*}{Example}
\newtheorem{set}[thm]{Setting}
\newtheorem{nota}[thm]{Notation}
\newtheorem{fact}{Statement}
\newtheorem{hy}{Hypothesis}[section]
\newtheorem{cond}{Parameter}  
\newtheorem{summary}{Summary}
\newtheorem*{ack*}{Acknowledgements}
\maketitle
\begin{abstract}
As a starting point for higher-dimensional patching in the Berkovich setting, we show that this technique is applicable around certain fibers of a relative Berkovich analytic curve. As a consequence, we prove a local-global principle over the field of overconvergent meromorphic functions on said fibers. By showing that these germs of meromorphic functions are algebraic, we also obtain local-global principles over function fields of algebraic curves defined over a class of (not necessarily complete) ultrametric fields, thus generalizing the results of ~\cite{une}.   
\end{abstract}

\tableofcontents

\section*{Introduction}
\textit{Field patching}, introduced by Harbater and Hartmann in \cite{HH}, and extended by these authors and Krashen in \cite{HHK}, has recently seen numerous applications and is the crucial ingredient in an ongoing series of papers (see \textit{e.g.} \cite{HHK}, \cite{HHK1}, \cite{ctps}, \cite{galhhkps}, \cite{galhhkps12}). 
Amongst the main points of focus of these works are local-global principles over function fields of algebraic curves defined over complete discretely valued fields. Namely, field patching has provided a new approach to local-global principles for homogeneous varieties over certain linear algebraic groups (for example see \cite{HHK} and \cite{ctps}). In particular, in \cite{HHK}, Harbater, Hartmann, and Krashen (from now on referred to as HHK) obtained local-global principles for quadratic forms and results on
the \mbox{$u$-invariant}.

In \cite{une}, the author adapted field patching to the setting of Berkovich analytic curves (from now on also referred to as \textit{Berkovich patching}). With this point of view the technique becomes very geometric and can be interpreted as the gluing of meromorphic functions. As a consequence, local-global principles that are applicable to quadratic forms are obtained. This, combined with the nice algebraic properties of Berkovich curves, gives rise to applications on the $u$-invariant. The results obtained in ~\cite{une} generalize those of the founding paper \cite{HHK}. In particular, it is no longer required that the base field be discretely valued, but merely that it be a complete ultrametric field. 

\subsection*{Overview and main results}
The goal of this paper is twofold:
\begin{enumerate}
\item  to establish the very first steps of a strategy for higher-dimensional Berkovich patching and the corresponding applications to the local-global principle;
\item to generalize the results we obtained in \cite{une}; more precisely, to show a local-global principle for algebraic curves (\textit{i.e.} their function fields) defined over a class of ultrametric fields which are not necessarily complete.  
\end{enumerate}

In other words, in this text we show that patching is possible ``around" certain fibers of relative Berkovich analytic curves. This is then applied to obtain a local-global principle over the field of overconvergent meromorphic functions on said fibers. We also show that the latter can be interpreted as the function field of a particular algebraic curve. As in ~\cite{une}, the local-global principles obtained are applicable to quadratic forms.

Before presenting our main results, let us recall some terminology.

\begin{defn*} [HHK \cite{HHK}] Let $K$ be a field. Let $X$ be a $K$-variety and $G$ a linear algebraic group over $K.$ We say that $G$ acts \emph{strongly transitively} on $X$ if $G$ acts on $X,$ and for any field extension $L/K,$ either $X(L)=\emptyset$ or $G(L)$ acts transitively on $X(L).$
\end{defn*}

Formally, asking that $G$ act strongly transitively on $X$ is more restrictive than asking that $X$ be homogeneous over $G.$ However, it is shown in \mbox{\cite[Remark ~3.9]{HHK}} that if $G$ is a reductive linear algebraic group over $K$ and $X/K$ is a projective variety, then the two notions are equivalent.

We also recall that Berkovich spaces are constructed through building blocks, the so called \textit{affinoid domains}. Moreover, there is a good theory of dimension for Berkovich spaces (see ~\cite{Duc2}). Finally, to each point $x$ of a Berkovich space is associated a complete ultrametric field $\mathcal{H}(x)$, called the \textit{completed residue field} of $x$.

Let $k$ be a complete non-trivially valued ultrametric field. Let $S$ be a good\footnote{this means that every point has an affinoid neighborhood} normal Berkovich $k$-analytic space. 
Let $\pi: C \rightarrow S$ be a proper flat relative $S$-analytic curve. For any affinoid domain $Z$ of ~$S$, set $C_{Z}:=\pi^{-1}(Z)$ and $F_{Z}:=\mathscr{M}(C_{Z}),$ where $\mathscr{M}$ denotes the sheaf of meromorphic functions on $C.$ 
A special case of one of the main results we show is the following (see Theorem~\ref{katastrofe}(1) for the statement in all its generality):
\sloppypar
\begin{thm*}[Corollary \ref{ahhh1}, Theorem \ref{katastrofe}(1)] Suppose that $S$ is strict, regular and ${\dim{S}<\dim_{\mathbb{Q}} (\mathbb{R}_{>0}/|k^{\times}|) \otimes_{\mathbb{Z}} \mathbb{Q}}$.
Let $x \in S$ be such that $\mathcal{O}_{S,x}$ is a field and $\pi^{-1}(x) \neq \emptyset$. Let $C_x$ denote the $\mathcal{H}(x)$-analytic curve induced by the fiber of $x$ in $C$. Suppose that $C_x$ is smooth and geometrically connected. 
Then there exists a connected affinoid neighborhood ~$Z_0$ of ~$x$ such that:

\raggedright for any $F_{Z_0}$-variety $H$ 
on which a connected rational linear algebraic group $G/F_{Z_0}$ acts strongly transitively, one has the following local-global principle:
$$H(\varinjlim_{x \in Z} F_Z) \neq \emptyset \iff H(\mathscr{M}_{C,u}) \neq \emptyset \ \text{for all} \ u \in C_x,$$
where the direct limit is taken with respect to connected affinoid neighborhoods $Z \subseteq Z_0$ of~$x.$
\end{thm*}

Remark that the direct limit appearing on the left-hand side of the local-global principle above is the field of germs of meromorphic functions on the fiber of $x$ in $C.$

We work only over fibers of points for which the local ring is a field. The set of such points is dense. In fact, in the case of curves, if $x$ is any point that is not \textit{rigid} (rigid points are those that we see in rigid spaces), then $\mathcal{O}_x$ is a field. Although this might not appear explicitely in the paper, the reason behind this hypothesis is that to make one's way from ``a matrix decomposition result" (similar to \cite{HHK} and \cite{une})  to patching ``around" the fiber, we need the fiber to \textit{not} be a divisor. This means that non-zero analytic (resp. meromorphic) functions defined on a neighborhood of said fiber can not have zeros (resp. poles) everywhere on the fiber; this is a crucial property for the results of Subsection~\ref{parpl}. 
Said hypothesis is the only obstacle to proving the result around \textit{all} fibers of the relative analytic curve $C \rightarrow S.$

To show the result above, as fibers of an analytic relative curve are endowed with the structure of an analytic curve, we follow a similar line of reasoning as in \cite{une}. However, there are many additional technical difficulties that appear in this relative setting. Here is a brief outline of the proof.

We construct particular covers of a neighborhood of the fiber over which patching is possible (the so called \textit{relative nice covers}); this is a relative analogue of \textit{nice covers} as introduced in \cite[Definition 2.1]{une}. As in the one-dimensional case, \textit{type ~3} points (which are characterized by simple algebraic and topological properties; see Preliminaires for the definition) on the fiber play an important role. Their existence is guaranteed by the hypothesis on the dimension of $S$. 

We first treat the case of $\mathbb{P}_S^{1, \mathrm{an}}$ -- the relative projective analytic line over $S$. The construction of relative nice covers is easier in this particular setting: we use the notion of \textit{thickening} of an affinoid domain on the fiber $C_x$ to obtain ``well-behaved" affinoid domains on a neighborhood of $C_x$. The idea of thickenings in the case of $\mathbb{P}^{1, \mathrm{an}}$ appears in some unpublished notes of J\'er\^ome Poineau (see Remark \ref{poineau} and Lemma \ref{18aaa} for a detailed account).

We then show that the general setting of Section \ref{patching} can be applied to these relative nice covers of $\mathbb{P}^{1, \mathrm{an}}$, meaning that patching \emph{can} be applied to them, provided we restrict to a small enough neighborhood of the fiber.  

By using pullbacks of finite morphisms towards $\mathbb{P}^{1, \mathrm{an}},$ a notion of relative nice cover can be constructed more generally for the case of normal relative proper curves. By adding to this the Weil restriction of scalars, patching is shown to be possible over relative nice covers in this more general framework as well.   

Finally, once we know how to patch around the fiber $C_x$, the local-global principle of Theorem \ref{katastrofe}(1) can be obtained as a consequence, albeit not as directly as in the one-dimensional case in \cite{une}. 

\

There is a connection between the points of the fiber and the valuations that the field of its overconvergent meromorphic functions can be endowed with. We make this precise in Proposition \ref{234}. As in the one-dimensional case, combined with the Henselianity of the fields $\mathscr{M}_{C,y}, \pi(y)=x,$ this connection allows us to obtain a local-global principle with respect to completions. Before stating precisely a special case of this result, let us recall that the field $\mathcal{O}_{S,x}$ is naturally endowed with a valuation $|\cdot|_x.$ 

\begin{thm*}[Theorem \ref{katastrofe}(2)]
Using the same notation as in the statement of Theorem~\ref{katastrofe}(1) above, set $F_{\mathcal{O}_x}=\varinjlim_Z F_Z.$ Let $V(F_{\mathcal{O}_x})$ denote the set of non-trivial rank 1 valuations on $F_{\mathcal{O}_x}$ which induce either $|\cdot|_x$  or  the trivial valuation on ~$\mathcal{O}_x.$ For ${v \in V(F_{\mathcal{O}_x})},$ let $F_{\mathcal{O}_x, v}$ denote the completion of the field $F_{\mathcal{O}_x}$ with respect to $v.$

If $\mathrm{char} \ k=0$ or $H$ is smooth, then the following local-global principle holds:
$$H(F_{\mathcal{O}_x}) \neq \emptyset \iff H(F_{\mathcal{O}_x,v}) \neq \emptyset \ \text{for all} \ v \in V(F_{\mathcal{O}_x}).$$
\end{thm*}

Remark that, with the same notation as in the theorem above, $\mathcal{O}_{S,x}=\varinjlim_{Z} \mathcal{O}_S(Z),$ where the direct limit is taken with respect to affinoid neighborhoods $Z$ of $x$ in $S.$ 
Using Grothendieck's work on projective limits of schemes to construct a relative algebraic curve over $\mathcal{O}(Z)$  from an algebraic curve over $\mathcal{O}_x$, as a consequence of the theorem above, we obtain the following generalization of \cite[Corollary 3.18]{une}.

\begin{thm*}[Theorem \ref{val2}]
Let $k$ be a complete non-trivially valued ultrametric field.
Let $S$ be a good normal $k$-analytic space such that $\dim{S}<\dim_{\mathbb{Q}} \mathbb{R}_{>0}/|k^{\times}| \otimes_{\mathbb{Z}} \mathbb{Q}$. Let $x \in S$ be such that $\mathcal{O}_{x}$ is a field. Let $C_{\mathcal{O}_{x}}$ be a smooth geometrically irreducible algebraic curve over the field $\mathcal{O}_{x}.$ Let $F_{\mathcal{O}_x}$ denote the function field of $C_{\mathcal{O}_{x}}.$ 

Let $G/F_{\mathcal{O}_{x}}$ be a connected rational linear algebraic group acting strongly transitively on a variety $H/F_{\mathcal{O}_{x}}.$ Then, if $\mathrm{char} \ k=0$ or $H$ is smooth: 
$$H(F_{\mathcal{O}_{x}}) \neq \emptyset \iff H(F_{\mathcal{O}_{x}, v}) \neq \emptyset \ \text{for all} \ v \in V(F_{\mathcal{O}_x}),$$
where $V(F_{\mathcal{O}_x})$ is given as in Theorem \ref{katastrofe}(2) above. 
\end{thm*}

A crucial element for showing Theorem \ref{val2}, and more generally, to highlight the interest of this paper, is that, in the setting of Theorem \ref{katastrofe}, meromorphic functions around the fiber of $x$ are algebraic. More precisely, the field of overconvergent meromorphic functions on the fiber of $x$ is the function field of an algebraic curve over $\mathcal{O}_x$ (which is basically an ``algebraization" of a neighborhood of the fiber succeeded by a base change to $\mathcal{O}_x$; see Corollary ~\ref{212}). To show this non-trivial result, we prove the following GAGA-type theorem for the sheaf of meromorphic functions: 
\begin{thm*}[Theorem \ref{231}]
Let $k$ be a complete ultrametric field. Let $A$ be a $k$-affinoid algebra. Let $X$ be a proper scheme over $\text{Spec} \ A.$ Let $X^{\mathrm{an}}/\mathcal{M}(A)$ denote the Berkovich analytification of~$X.$ Then $\mathscr{M}_{X^{\mathrm{an}}}(X^{an})=\mathscr{M}_X(X)$, where $\mathscr{M}_{X^{\mathrm{an}}}$ (resp. $\mathscr{M}_X$) denotes the sheaf of meromorphic functions on $X^{\mathrm{an}}$ (resp. $X$).
\end{thm*}

As in \cite{HHK} and \cite{une}, seeing as the projective variety determined by a quadratic form satisfies the hypotheses of the results presented, the prime example to which the local--global principles of this text can be applied is the case of quadratic forms (under the assumption ${\mathrm{char} \ k \neq 2}$); see Corollaries \ref{ahhh'} and \ref{ahhh2}.

Here is an example of a local ring $\mathcal{O}_x$   of an analytic space which is a field and over which the results above can be applied. It corresponds to a type 3 point of the analytic affine line.

\vspace{0.1cm}

\noindent {\em Example.} Let $(k, |\cdot|)$ be a complete ultrametric field. Let $r \in \mathbb{R}_{>0} \backslash \sqrt{|k^\times|}.$ Let $x \in \mathbb{A}_k^{1, \mathrm{an}}$ be a multiplicative semi-norm on $k[T]$ such that $|T|_x=r$ 
(in fact, $x$ is the unique such point of $\mathbb{A}_k^{1, \mathrm{an}}$). 

For any $r_1, r_2 \in \mathbb{R}_{>0}$ such that $r_1< r<r_2,$ set $$A_{r_1, r_2}:=\left\lbrace \sum_{n \in \mathbb{Z}} a_n T^n: a_n \in k, \lim_{n \rightarrow +\infty} |a_n| r_2^n=0, \lim_{n \rightarrow -\infty} |a_n| r_1^n=0 \right\rbrace.$$

Then $\mathcal{O}_{\mathbb{A}_k^{1, \mathrm{an}}, x}= \varinjlim_{\substack{r_1<r<r_2}} A_{r_1, r_2}.$

\subsection*{Organization of the manuscript}
In Section \ref{patching} we develop the necessary tools for proving a ``matrix decomposition" statement (Theorem \ref{e keqe}) which is fundamental to the generalization of Berkovich patching we present here. We work over a general formal setup (Setting \ref{therealone}), which is partly why this section is of very technical nature. 

In Section \ref{4.1}, we construct the notion of \textit{relative nice covers} around a fiber of $\mathbb{P}^{1, \mathrm{an}}$, analoguous to (and a generalization of) nice covers for curves, and show that it possesses good properties, \textit{i.e.} properties that are necessary for patching. This is where the concept of thickening
of an affinoid domain appears. 

In order to be able to apply the results of Section \ref{patching} to fibers of the relative $\mathbb{P}^{1, \mathrm{an}}$, it is necessary to constantly ``shrink" to smaller neighborhoods of the fiber. Because of this, we need some uniform boundedness results and explicit norm comparisons, which is the topic of Section ~\ref{4.2}. As a consequence, this is one of the most technical sections of this paper. It also contains an explicit description of the Banach algebras of analytic functions on certain affinoid domains of the relative projective line. 
\begin{sloppypar}
In Section \ref{4.3}, we show that the results of Section \ref{patching} are indeed applicable to relative nice covers of fibers of the relative $\mathbb{P}^{1, \mathrm{an}},$ and that patching (in the sense of \mbox{\cite[Theorem 1.7]{une}}) can be obtained as a consequence thereof. The arguments used in this section are of very topological nature. 
\end{sloppypar}
In Section \ref{4.4}, we study the properties of the class of  relative analytic curves over which we know how to apply patching around certain fibers. The conditions that are required are not too restrictive: the relative proper curve needs to be normal and algebraic around the fiber, so this is satisfied for the Berkovich analytification of any normal proper relative \textit{algebraic} curve. For example, we show that smooth geometrically irreducible projective algebraic curves defined over certain fields give rise to such a situation. This makes it possible to generalize some results from \cite{une}.   

In Section \ref{4.5}, we extend the notion of relative nice cover to certain fibers of relative analytic \textit{curves} and show that patching is possible on them.

Section \ref{4.6} contains our main results: a local-global principle with respect to the stalks of the sheaf of meromorphic functions (Theorem \ref{katastrofe}) and one with respect to completions (Theorem \ref{val2}). To show the latter, we prove that the field of germs of meromorphic functions on the fiber can be realised as the function field of a certain algebraic curve. In this section, we include a subsection with a summary of the results we obtain. 

The fibers around which we apply patching are those over points for which their corresponding stalk is a field. In Section \ref{4.7}, we calculate some examples of such fields.

At the end of this paper, we provide a section of appendices. In Appendix I we introduce some basic properties of the sheaf of meromorphic functions on Berkovich analytic spaces. Among other things, we show that the meromorphic functions of the Berkovich analytification of certain schemes are \textit{algebraic} (Theorem \ref{231}). To do this, we use the ideas from a MathOverflow thread (see ~\cite{overflow}). This result is crucial for showing Theorem~\ref{val2} and connecting the results we obtain in the Berkovich setting to an algebraic one. In Appendices II and~III we show some additional results on Berkovich analytic curves which we need in this text. In Appendix III,  we deduce a particular writing for affinoid domains which allows us to contruct their thickenings in Section \ref{4.1}. 

\begin{ack*} I am most grateful to J\'er\^ome Poineau for the many invaluable discussions and remarks. I am also very thankful to him for sharing his unpublished notes which contain the idea of thickenings of affinoid domains of the projective line.  Many thanks also to Antoine Ducros for his insightful remarks and suggestions which made it possible to remove some important algebraicity hypotheses in one of the main statements. Finally, I am also grateful to the anonymous referee, whose remarks have improved the quality of this article. 
\end{ack*}

\section*{Preliminaries}
With the purpose of making the paper more self-contained, we remind here the definitions of some of the notions we use, which are originally due to Berkovich. Let $k$ denote a complete ultrametric field.

\emph{Good analytic space.} A Berkovich $k$-analytic space is said to be \emph{good} if any point has a neighborhood isomorphic to a $k$-affinoid space. Berkovich studies these spaces in \cite{Ber90}.

\emph{The completed residue field.} Let $X$ be a good $k$-analytic space. Recall that for $x \in X,$ the local ring $\mathcal{O}_{X,x}$ is endowed with a semi-norm with kernel $m_x$-the maximal ideal of~$\mathcal{O}_{X,x}$ (see \cite[Lemma 1.4.21]{doktoratura}).
The \textit{completed residue field of $x$}, denoted $\mathcal{H}(x),$ is the completion of the residue field $\kappa(x):=\mathcal{O}_{X,x}/m_x$ of ~$x$ with respect to the norm induced on $\kappa(x)$ from this semi-norm (see \cite[Definition 14.8]{coanno}). Remark that $\mathcal{H}(x)$ is a complete ultrametric field. 

\emph{Rigid points.} If $\mathcal{H}(x)/k$ is a finite field extension, then we say that $x$ is a \textit{rigid point} of ~$X$. These are the algebraic points of $X$ and those that we encounter when working with rigid spaces. 

\emph{Type 3 points.} Suppose $X$ is a curve. If $\dim_{\mathbb{Q}} |\mathcal{H}(x)^{\times}|/|k^\times| \otimes_{\mathbb{Z}} \mathbb{Q}=1$, we say that~$x$ is a \textit{type 3} point. There is a full classification of points of $X$ into 4 types (see \cite[Definition~1.8.1]{doktoratura} for more details).

\emph{Shilov boundary.} Suppose $(X, \mathcal{O}_X)$ is a $k$-affinoid space. Then there exists a finite subset $\Gamma(X)$ of $X$ such that each element $f \in \mathcal{O}_X(X)$ attains its maximum at a point of $\Gamma(X)$ (see \cite[Corollary 2.4.5]{Ber90}). We call $\Gamma(X)$ the \emph{Shilov boundary} of $X$.  

\section{Patching} \label{patching}

Following the same steps as in \cite{une}, we start by proving a ``matrix decomposition" result that generalizes \cite[Theorem 2.5]{HHK} and \cite[Lemma 1.9]{une}, and is applicable to a Berkovich framework. To do this, we follow along the lines of proof and reasoning of \cite[Section ~2.1]{HHK} making the necessary adjustements. 

We work over a general formal setup (Setting \ref{therealone}), which is partly why the content of this section is of very technical nature and may thus be skipped upon a first reading. It will be shown in the next parts of this paper that the hypotheses we adopt here are satisfied in a very natural way in Berkovich's geometry. The main statement, Theorem ~\ref{e keqe}, is fundamental to patching. 

We first show some auxiliary results.

\begin{set} \label{11aaa}
Let $k$ be a complete non-trivially valued ultrametric field. 
Let $R$ be an integral domain containing $k,$ endowed with a non-Archimedean (submultiplicative) norm ~$|\cdot|_{R}$. Suppose that for any $a \in R$ and $b \in k,$ $|ab|_R=|a|_R \cdot |b|.$ 
\end{set}

Remark that the last assumption implies the norm $|\cdot|_R$ extends $|\cdot|.$

For $p \in \mathbb{N}$ and indeterminates $X_1, \dots, X_p,$ let us use the notation $\underline{X}$ for the $p$-tuple $(X_1, \dots, X_p).$ Following \cite[Section 2]{HHK}, set $A:=R[\underline{X}]$ and $\widehat{A}:=R[[\underline{X}]].$ For any $M \geq 1,$ set $$\widehat{A_M}:=\left\{ \sum_{l \in \mathbb{N}^p} c_l \underline{X}^l \in \widehat{A} : \ \forall l \in \mathbb{N}^p, |c_l|_{R} \leqslant M^{|l|} \right\},$$
where for $l=(l_1, l_2, \dots, l_n) \in \mathbb{N}^p,$ $\underline{X}^l:=\prod_{i=1}^p X_i^{l_i}$ and $|l|:=l_1+l_2 +\cdots + l_p.$

This is a subring of $\widehat{A},$ and for any $M', M'' \geq 1,$ if $M' \leqslant M''$ then $\widehat{A_{M'}} \subseteq \widehat{A_{M''}}.$ Furthermore, $\widehat{A_M}$ is complete with respect to the $(\underline{X})$-adic topology: if $(f_n)_n$ is a Cauchy sequence in $\widehat{A_M},$ then for any $l \in \mathbb{N}^p$ and large enough $n,$ $f_{n+1}-f_n \in (\underline{X})^{|l|},$  implying that $f_n$ and $f_{n+1}$ have the same ``first few" coefficients (the larger $|l|,$ the more ``first few" coefficients that are the same).

Remark also that for any element $f=\frac{g}{h}$ of the local ring $R[\underline{X}]_{(\underline{X})}$,  where ${g, h \in R[\underline{X}]}, {h(0) \neq 0},$ if $h(0) \in R^{\times}$, then $f$ can be expanded into a formal power series over $R,$ meaning in this case $f \in \widehat{A}.$ 

The following two lemmas are generalizations of Lemmas 2.1 and 2.3 of \cite{HHK} (and their proofs follow the line of reasoning of the latter).
For any $n \in \mathbb{N},$ we keep the notation $|\cdot|_R$ for the max norm on $R^n$ induced by the norm of $R.$ For $a:=(a_1, a_2, \dots, a_n) \in R^n$ and $l:=(l_1, l_2, \dots, l_n) \in \mathbb{N}^n,$ we denote $a^l:=a_1^{l_1}\cdots a_{n}^{l_n}.$ Clearly, $a^l \in R.$

\begin{lm}\label{1}
\begin{enumerate}
\item Let $u=\sum_{l \in \mathbb{N}^p} c_l \underline{X}^l \in \widehat{A_M}.$
If $a \in R^p$ is such that ${|a|_{R}<M^{-1}},$ then the series $\sum_{l \in \mathbb{N}^p} c_l {a}^l$ is convergent in $R.$ Let us denote its sum by $u(a).$
\item For $M \geqslant 1,$ let $v, w \in \widehat{A_M}$ be such that $w$ and $vw$ are polynomials. If $a \in R^p$ is such that $|a|_R<M^{-1},$ then $vw(a)=v(a) w(a).$

\item Let $f=\frac{g}{h} \in R[\underline{X}]_{(\underline{X})}, g,h \in R[\underline{X}], h(0) \neq 0,$ be such that $g(0)=0$ and $h(0) \in R^{\times}.$ There exists $M\geq 1$ such that $f \in \widehat{A_M}$ and $h \in \widehat{A_M}^{\times}.$ 

Let $f=\sum_{l \in \mathbb{N}^p} c_l \underline{X}^l$ be the series representation of $f.$ Then for any $a \in R^p$ with $|a|_R < M^{-1}$, the series $\sum_{l \in \mathbb{N}^p} c_l a^l$ is convergent in $R$ and $f(a)=\frac{g(a)}{h(a)}.$ 
\end{enumerate}
\end{lm}

\begin{proof}
\begin{enumerate}

\item Set $m=|a|_R < M^{-1}.$ Then $|c_l a^l|_R\leqslant (Mm)^{|l|}.$ Since $Mm <1,$ $c_la^l$ tends to zero as $|l|$ tends to $+\infty,$ implying $\sum_{l \in \mathbb{N}^p}c_l a^l$ converges in $R.$

\item Let $d>\deg{vw},$ and $C:=\max_{l \in \mathbb{N}^p}(|vw_l|_R, |w_l|_R),$ where $vw_l, w_l, l \in \mathbb{N}^p,$ are the coefficients of the polynomials $vw, w,$ respectively. Let $v=\sum_{l \in \mathbb{N}^p} b_l \underline{X}^l$ be the series representation of $v.$ For any $s \in \mathbb{N},$ set $v_s=\sum_{|l|<s} b_l \underline{X}^l.$ By the first part, the sequence $(v_s(a))_{s \in \mathbb{N}}$ converges in ~$R,$ and we denote its limit by $v(a).$ For $s \geq d,$ $r_s:=v_sw-vw=(v_s-v)w$ is a polynomial whose monomials are of degree at least $s.$   
The coefficient $C_j$ corresponding to any degree $j \geqslant s$ monomial of $r_s$ is a finite sum of products of coefficients of $v_s-v $ and $w.$ Since $R$ is non-Archimedean, $M \geqslant 1,$ and $v_s-v \in \widehat{A_M},$ we obtain  
 $|C_j|_R \leqslant M^j C$ (recall the definition of $C$ at the beginning of this paragraph). 

Set $m=|a|_R.$ By the paragraph above, every degree $j$ monomial of $r_s$ evaluated at $a$ has absolute value at most $(mM)^jC.$ Since $j \geqslant s$ and $Mm<1,$  using the fact that $R$ is non-Archimedean, we obtain ${|r_s(a)|_R \leqslant (Mm)^s C}$, implying  ${r_s(a) \rightarrow 0}, s \rightarrow \infty.$ Consequently, $v_s(a)w(a) \rightarrow vw(a)$ when $s \rightarrow \infty,$ \textit{i.e.} $v(a)w(a)=vw(a).$ 

\item Set $b=h(0).$ Then $b-h \in (\underline{X}),$ and thus $1-b^{-1}h \in (\underline{X}).$ Set $e=1-b^{-1}h,$ so that $b^{-1}h=1-e$ with $e \in (\underline{X}).$ This implies $(b^{-1}h)^{-1}=b h^{-1}=\frac{1}{1-e}=\sum_{i \in \mathbb{N}}e^i \in \widehat{A},$ and so $h^{-1}=\sum_{i \in \mathbb{N}}b^{-1}e^i \in \widehat{A}.$
Consequently, $f=gh^{-1}=\sum_{i \in \mathbb{N}} b^{-1}ge^i \in \widehat{A}.$

Set $M=\max_{l \in \mathbb{N}^p}(1, |b^{-1}|_R, \sqrt[|l|]{|g_l|_R}, \sqrt[|l|]{|e_l|_R}, \sqrt[|l|]{|h_l|_R}),$ where $g_l$ (resp. $e_l,$ $h_l$), $l \in \mathbb{N}^p,$ are the coefficients of the polynomial $g$ (resp. $e, h$). Then $b^{-1}, g, e \in \widehat{A_M}$, and since $\widehat{A_M}$ is a ring, $b^{-1}e^i,$ $b^{-1}ge^i \in \widehat{A_M}$ for any $i \in \mathbb{N}.$ Finally, since $\widehat{A_M}$ is complete with respect to the $(\underline{X})-$adic norm, $h^{-1}, f \in \widehat{A_M},$ and so $h \in \widehat{A_M}^{\times}.$

The rest is a direct consequence of the first two parts of the statement.
\end{enumerate}
\end{proof}
Let $n \in \mathbb{N}$ and $S_i, T_i, i=1,2,\dots, n,$ be indeterminates. As before, we use the notation $\underline{S}$ (resp. $\underline{T}$) for the $n$-tuple $(S_1, \dots, S_n)$ (resp. $(T_1, \dots, T_n)$). For $l, m \in \mathbb{N}^n$, we denote by $|(l,m)|$ the sum $|l|+|m|,$ where $|l|$ (resp. $|m|$) is the sum of coordinates of $l$ (resp. $m$). Also, $\underline{S}^l:=\prod_{i=1}^n S_i^{l_i}$ and $\underline{T}^m:=\prod_{i=1}^n T_i^{m_i}.$ For any vector $a \in R^n,$ we denote by $a_i$ the $i$-th coordinate of $a,$ $i=1,2,\dots, p,$ meaning $a=(a_1, a_2, \dots, a_n), a_i \in R. $ As before, $a^l:=a_1^{l_1} \cdots a_n^{l_n}.$

\begin{lm} \label{2}
Let $f=\frac{h_1}{h_2} \in R[\underline{S}, \underline{T}]_{(\underline{S}, \underline{T})},$ $h_1, h_2 \in R[\underline{S}, \underline{T}], h_2(0) \neq 0,$ be such that ${h_2(0) \in R^{\times}}.$ Suppose there exists $i \in \{1,2,\dots, n\}$ such that $f(a,0)=f(0,a)=a_i$ for any $a \in R^n$ for which $f(a,0)$ and $f(0,a)$ converge in $R.$ 

Then there exists $M \geq 1$ such that $f \in \widehat{A_M}$ and its series representation  is: $$f=S_i+T_i+\sum_{|(l,m)|\geq 2}c_{l,m} \underline{S}^l \underline{T}^m.$$
\end{lm}

The proof of \cite[Lemma 2.3]{HHK} is applicable to Lemma \ref{2} with only minor changes necessary. 

\begin{rem} \label{sepse}
In the proof of \cite[Lemma 2.3]{HHK}, and hence that of Lemma \ref{2}, one needs to use that $k$ is non-trivially valued. In fact, Lemma \ref{2} is the only reason why we have adopted this hypothesis. 
\end{rem}

Here is the general setting over which we show patching results. 

\begin{set} \label{therealone}
Let $(k, |\cdot|)$ be a complete non-trivially valued ultrametric field. 
\begin{enumerate}
\item \begin{sloppypar}
Let $(R_i, |\cdot|_{R_i}), i=0,1,2,$ be an integral domain containing $k$, endowed with a non-Archimedean (submultiplicative) $k$-linear\footnote{We recall that $k$-linear here means $|ab|_{R_i}=|a|\cdot |b|_{R_i}$, $\forall a \in k, \forall b \in R_i$; it implies that ${|\cdot|_{R_i|k}=|\cdot|}$.} norm with respect to which it is complete. Furthermore, we assume that $R_1, R_2 \hookrightarrow R_0$, 
where the two embeddings are bounded.
\end{sloppypar}

\item  Let us denote $F_i=\mathrm{Frac} \ R_i, i=0,1,2$. Let $F$ be a field containing $k$ and embedded in $F_1$ and $F_2$. 

\item For $j=1,2,$ let $A_j$ be a finite $R_j$-module such that $R_j \subseteq A_j \subseteq F_j$. Let us endow~$A_j$ with the quotient semi-norm induced from a surjective morphism ${\varphi_i:R_j^{n_i} \twoheadrightarrow A_j}, j=1,2$. We assume that this semi-norm is a norm with respect to which $A_j$ is complete.
\end{enumerate}
\begin{minipage}{0.3\textwidth}
\[
\begin{tikzcd}[column sep=small]
& k  \ar[dl, hook] \ar[dr, hook] & \\
R_1 \ar[dr, hook, "bounded"] & & R_2 \ar[dl, hook] \\
& R_0 &    
\end{tikzcd}
\]
\end{minipage}
\hfill
\begin{minipage}{0.3\textwidth}
\[
\begin{tikzcd}[column sep=small]
& F \ar[dl] \ar[dr] &\\
F_1 \ar[dr]  & & F_2 \ar[dl]\\
& F_0 & 
\end{tikzcd}
\]
\end{minipage}
\hfill
\begin{minipage}{0.3\textwidth}
\[
\begin{tikzcd}[column sep=small]
R_1 \ar[swap, d, hook, "finite"] & & R_2 \ar[d, hook, "finite"]\\
A_1 \ar[d, hook] \ar[dr, hook, "bounded"] & & A_2 \ar[dl, hook] \ar[d,hook]\\
F_1 & R_0 & F_2
\end{tikzcd}
\]
\end{minipage}
\begin{enumerate}
\item[(4)] Assume there is a bounded embedding $A_j \rightarrow R_0$, $j=1,2$, which induces a surjective morphism $\psi : A_1\oplus A_2 \twoheadrightarrow R_0.$ Moreover, assume that $|\cdot|_{R_0}$ is equivalent to the quotient norm induced by $\psi$, where $A_1 \oplus A_2$ is endowed with the max norm $|\cdot|_{\mathrm{max}}$, \textit{i.e.} that the morphism $\psi$ is \textit{admissible} (see \cite[1.1]{Ber90}).
\end{enumerate}
\end{set}

Let us recall the motivation behind the interest of Theorem \ref{e keqe} to us. 

\begin{defn} \label{12aaa}
Let $K$ be a field. A \emph{rational} variety over $K$ is a $K$-variety that has a Zariski open isomorphic to an open of some $\mathbb{A}_K^n.$ 
\end{defn}

\begin{rem}
The definition above does not coincide with the standard notion of rational variety. We adopt it here because we will only use it for linear algebraic groups, in which case a connected rational linear algebraic group is rational in the traditional sense (\textit{i.e.} birationally equivalent to some $\mathbb{P}^n$). We make this distinction because there are certain statements we will show that don't require connectedness and others that do.  
\end{rem}

Using the same notation as in Setting \ref{therealone}, let $G/F$ be a rational linear algebraic group (rational here means that $G$ is a rational variety over $F$ as per Definition \ref{12aaa}). Our main goal will be to show that under certain conditions (which we will interpret geometrically in the next sections), for any $g \in G(F_0),$ there exist $g_j \in G(F_j),j=1,2,$ such that $g=g_1 \cdot g_2$ in ~$G(F_0).$
\begin{rem}\label{thingy}
Let $K/F$ be any field extension.  Since $G$ has a non-empty Zariski open subset $S'$ isomorphic to an open subset $S$ of an affine space $\mathbb{A}_K^n$, by translation we may assume that the identity element of $G$ is contained in $S',$ that $0 \in S,$ and that the identity is sent to $0.$ Let us denote the isomorphism $S' \rightarrow S$ by $\varphi.$

Let $m$ be the multiplication in $G,$ and set $\widetilde{S'}=m^{-1}(S') \cap (S' \times S'),$ which is an open of $G \times G.$ It is isomorphic to an open $\widetilde{S}$ of $\mathbb{A}_K^{2n},$ and $m_{|\widetilde{S'}}$ gives rise to a map $\widetilde{S} \rightarrow S,$ \textit{i.e.} to a rational function $f: \mathbb{A}_K^{2n}\dashrightarrow \mathbb{A}_K^n$ (see the diagram below). Note that for any $(x,0), (0,x) \in \widetilde{S},$  this function sends them both to $x.$

\begin{center}
\begin{tikzpicture}
  \matrix (m) [matrix of math nodes,row sep=3em,column sep=4em,minimum width=2em]
  {
     \widetilde{S'} & S' \\
     \widetilde{S} & S \\};
  \path[-stealth]
    (m-1-1) edge node [left] {$(\varphi \times \varphi)_{|\widetilde{S'}}$} 
    (m-2-1) edge  node [above] {$m_{|\widetilde{S'}}$} 
    (m-1-2)
    (m-2-1.east|-m-2-2) edge node [below] {$f$}
      (m-2-2)
    (m-1-2) edge node [right] {$\varphi$} 
    (m-2-2) 
    (m-2-1);
\end{tikzpicture}
\end{center}

The result we are interested in can be interpreted in terms of the map $f$. Theorem~\ref{e keqe} below shows that (under certain conditions) said result is true on some neighborhood of the origin of an affine space. 
\end{rem}

Let us start with an auxiliary lemma. Referring to Setting \ref{therealone}, let $|\cdot|_{\inf}$ be the norm on $R_0$ obtained from the admissible morphism $\psi: A_1 \oplus A_2 \twoheadrightarrow R_0.$ Since it is equivalent to $|\cdot|_{R_0},$ there exist positive real numbers $C_1, C_2$ such that $C_1  |\cdot|_{R_0} \leqslant |\cdot|_{\inf} \leqslant C_2 |\cdot|_{R_0}.$ 

Since the morphisms $A_j \hookrightarrow R_0, j=1,2,$  are bounded, there exists $C>0$ such that for any $x_j \in A_j,$ one has $|x_j|_{R_0} \leqslant C|x_j|_{A_j}.$ By changing to an equivalent norm on $A_j$ if necessary, we may assume that $C=1$.

\begin{lm} \label{3}
 There exists $d \in (0,1) \subseteq \mathbb{R}$ such that for all  $c \in R_0,$ there exist ${a \in A_1}$, ${b\in A_2},$ for which $\psi(a+b)=c$ and $d \cdot \max(|a|_{A_1}, |b|_{A_2}) \leqslant |c|_{R_0}.$
\end{lm}

\begin{proof}
This is a direct consequence of the admissibility of the map $\psi$.  
\end{proof}

From now on, instead of writing $\psi(x+y)=c$ for $x\in A_1, y\in A_2, c \in R_0,$ we will simply put $x+y=c$ when there is no risk of ambiguity. 

In what follows, for any positive integer $n,$ let us endow $R_0^n$ with the max norm induced from the norm on $R_0,$ and let us also denote it by $|\cdot|_{R_0}.$ For a normed ring $A$ and $\delta>0,$ we denote by $D_A(0,\delta)$ the open disc in $A$ centered at $0$ and of radius $\delta.$

\begin{thm} \label{e keqe}
For $n \in \mathbb{N},$ let $f: \mathbb{A}_{F_0}^n \times \mathbb{A}^n_{F_0} \dashrightarrow \mathbb{A}_{F_0}^n$ be a rational map defined on a Zariski open $\widetilde{S},$ such that $(0,0) \in \widetilde{S},$ and $f(x,0)=f(0,x)=x$ whenever $(x,0),(0,x) \in \widetilde{S}.$ 
Write $f=(f_1, f_2, \dots, f_n),$ where $f_i=\frac{g_i}{h_i}$ for some $g_i, h_i \in R_0[\underline{S}, \underline{T}],$ $i=1,2,\dots, n.$ Suppose $h_i(0) \in R_0^{\times}$ for all $i.$ 

Let $M \geqslant 1$ be such that $f_i \in \widehat{A_M}$ and $h_i \in \widehat{A_M}^\times$ for all $i$ (applying Lemma \ref{1} with $R=R_0$). Suppose there exists $\delta>0$ such  $D_{R_0^{2n}}(0, \delta) \subseteq \widetilde{S}(F_0).$ Let $d$ be as in Lemma \ref{3}. Let $\varepsilon >0$ be such that $\varepsilon<\min(\frac{d}{2M}, \frac{d^3}{M^4}, \frac{d\delta}{2})$. Then for any $a \in \mathbb{A}^n(F_0)$ with $a \in R_0^n$ and $|a|_{R_0} \leqslant \varepsilon,$ there exist $u \in A_1^n$ and $v \in A_2^n$ for which $(u,v) \in \widetilde{S}(F_0)$ and $f(u,v)=a.$
\end{thm}

\begin{proof}
Since $f_i(0,0)=0$ for all $i,$ the functions $g_i$ belong to the maximal ideal $(\underline{S}, \underline{T})$ of $R_0[\underline{S}, \underline{T}]$. From Lemmas \ref{1} and \ref{2}:
\begin{enumerate}
\item{we can see these rational functions as elements of $R_0[[\underline{S}, \underline{T}]];$}
\item{the constant $M$ is such that  
$$f_i=S_i+T_i+ \sum_{|(l,m)|\geq 2}c_{l,m}^i \underline{T}^l \underline{S}^m \in R_0[[\underline{S}, \underline{T}]],$$ with $|c_{l,m}^i|_{R_0} \leqslant M^{|(l,m)|},$ for  $i=1,2,\dots, n$ and $(l,m) \in \mathbb{N}^{2n},$ where $|(l,m)|$ is the sum of the coordinates of $(l,m).$}
\end{enumerate}

By the choice of $\delta,$ for any $(x,y) \in R_0^{2n}$ satisfying $|(x_0, y_0)|_{R_0} < \delta,$ $(x,y) \in \widetilde{S}(F_0),$ so the function $f(x,y)$ is well-defined (meaning the functions $f_i$ are well-defined for all $i$).

Set $\varepsilon'=\frac{\varepsilon}{d}.$ Then $0 <\varepsilon' < \min \{1/2M, d^2/M^4, \delta/2\}.$  Since ${\varepsilon< \varepsilon' < \min(1/M, \delta/2),}$ for any $(x,y) \in \widetilde{S}(F_0)$ satisfying $(x, y) \in R_0^{2n}$ and $|(x,y)|_{R_0} \leqslant \varepsilon',$ $f(x,y)$ is well-defined, and by Lemma \ref{1}, the series $f_i$ is convergent in $R_0$ at $(x, y),$ ${i=1,2,\dots, n}.$

Let $a =(a_1, a_2, \dots, a_n) \in \mathbb{A}^n(F_0)$ be such that $a \in R_0^n$ and $|a|_{R_0} \leqslant \varepsilon.$ Let $u_0=0 \in A_1^n,$ and $v_0=0 \in A_2^n.$ Using induction, one constructs sequences $(u_s)_s$ in $A_1^n,$ and $(v_s)_s$ in $A_2^n,$ such that the following conditions are satisfied: 
\begin{enumerate}
\item{$|u_s|_{A_1}, |v_s|_{A_2} \leqslant \varepsilon'$ for all $s \geq 0;$}
\item{$|u_s-u_{s-1}|_{A_1}, |v_s-v_{s-1}|_{A_2} \leqslant \varepsilon'^{\frac{s+1}{2}} $ for all $s \geq 1$;}
\item{$|f(u_s,v_s)-a|_{R_0} \leqslant d\varepsilon'^{\frac{s+2}{2}} $ for all $s \geq 0.$}
\end{enumerate} 
This is done as in the proof of \cite[Lemma 1.9]{une}.
\end{proof}

Using the same notation, we have proven:

\begin{prop} \label{14aaa}
Suppose $h_i(0) \in R_0^{\times}$ and there exists an open disc of $R_0^{2n}$ centered at $0$ that is contained in $\widetilde{S}.$ Then there exists $\varepsilon >0$ such that for any $g \in S'(F_0)$ with $\varphi(g) \in R_0^n$ and $|\varphi(g)|_{R_0} \leqslant \varepsilon,$ there exist $g_i \in G(F_i), i=1,2,$ satisfying $g= g_1 \cdot g_2$ in $G(F_0).$
\end{prop}  
 
\section{Nice covers for the relative projective line} \label{4.1}

As in the case of curves in \cite{une}, we construct covers around fibers of the relative projective line over which a generalized form of patching as seen in \cite[Proposition 3.3]{une} will be possible. More precisely, we construct relative analogues of \textit{nice covers}. Let us recall the latter.

\begin{defn}[{\cite[Definition 2.1]{une}}]\label{nice}
Let $k$ be a complete ultrametric field. A finite cover $\mathcal{U}$ of a $k$-analytic curve is called \textit{nice} if:
\begin{enumerate}
\item the elements of $\mathcal{U}$ are connected affinoid domains with only type 3 points in their topological boundaries;
\item for any different $U, V \in \mathcal{U},$ $U \cap V=\partial{U} \cap \partial{V},$ or equivalently, $U \cap V$ is a finite set of type 3 points; 
\item for any two different elements of $\mathcal{U},$ neither is contained in the other. 
\end{enumerate}
\end{defn}

In Appendix III, we show and recall some results on the analytic projective line which we will use extensively in this section. 

\subsection{The general setting}
Let $k$ be a complete ultrametric field. We start by recalling the important notion of dimension for $k$-analytic spaces in the Berkovich sense.

\begin{rem} \label{hajdede1}
Let $Y$ be a $k$-analytic space. We recall that the \textit{dimension} of $Y$, denoted $
\dim{Y}$, is defined to be the $\sup_{y \in Y} 
d(\mathcal{H}(y)/k),$ where $$d(\mathcal{H}(y)/k):= \deg \mathrm{tr}
_{\widetilde{k}} \widetilde{\mathcal{H}
(y)} + \dim_{\mathbb{Q}} |\mathcal{H}(y)^
\times|/|k^\times| \otimes_{\mathbb{Z}} 
\mathbb{Q},$$
and $\widetilde{k}, \widetilde{\mathcal{H}(y)}$ are the residue fields of $k, \mathcal{H}(y),$ respectively (see \cite[1.14]{Duc2}).
\end{rem}

\begin{nota} \label{17aaa}
 Let $S$ be a normal \textit{good} $k$-analytic space (\textit{i.e.}  affinoid domains form a basis of the Berkovich topology on $S$). Suppose that $\dim{S}<\dim_{\mathbb{Q}} \mathbb{R}_{>0}/|k^\times| \otimes_{\mathbb{Z}} \mathbb{Q}$. 
Let us denote by $\pi$ the structural morphism $\mathbb{P}_{S}^{1,\mathrm{an}} \rightarrow S.$ Let $x \in S$ be such that $\mathcal{O}_{S,x}$ is a field. Let $F_x$ be the fiber of $x$ on ~$\mathbb{P}_S^{1,\mathrm{an}}$, which can be endowed with the analytic structure of~$\mathbb{P}_{\mathcal{H}(x)}^{1, \mathrm{an}}$.  
\end{nota}
Remark that a connected affinoid domain of $S$ is integral. 

Let us explain the hypothesis on the dimension of $S$ in Notation \ref{17aaa}.
As in \cite{une}, type ~3 points play a very important role for obtaining patching results \textit{around} the fiber $F_x.$ Hence, their existence \textit{on} the fiber is crucial and, as will be seen in the next lemma, this is guaranteed by the condition we imposed on the dimension of $S$. Recall that for a complete ultrametric field $K$, a $K$-analytic curve contains type 3 points if and only if $\sqrt{|K^\times|} \neq \mathbb{R}_{>0}$.

\begin{lm} \label{1000}
Let $Y$ be a normal $k$-analytic space such that $\dim{Y}<\dim_{\mathbb{Q}} \mathbb{R}_{>0}/|k^\times| \otimes_{\mathbb{Z}} \mathbb{Q}$. Then for any $y \in Y,$ ${\sqrt{|\mathcal{H}(y)^\times|} \neq \mathbb{R}_{>0}}.$
\end{lm}

\begin{proof}
For any $y \in Y,$ we have $$
\dim_{\mathbb{Q}} |\mathcal{H}(y)^\times|/|k^\times| \otimes_{\mathbb{Z}} \mathbb{Q}\leqslant d(\mathcal{H}(y)/k) \leqslant \dim{Y} <\dim_{\mathbb{Q}} \mathbb{R}_{>0}/|k^\times| \otimes_{\mathbb{Z}} \mathbb{Q}.$$
Consequently, ${\sqrt{|\mathcal{H}(y)^\times|} \neq \mathbb{R}_{>0}}.$
\end{proof}
Hence, by Lemma \ref{1000}, in Notation \ref{17aaa}, $\mathbb{P}_{\mathcal{H}(x)}^{1, \mathrm{an}}$ contains type 3 points.

\enlargethispage*{0.5cm}
We recall that  if $K$ is a complete ultrametric field, for $a \in K$ and $r \in \mathbb{R}_{\geqslant 0},$  the map $$\eta_{a,r}:K[T] \rightarrow \mathbb{R}_{\geqslant 0}, $$$$ \hspace{0.5cm} \sum_{n} b_n (T-a)^n \mapsto \max_n \{|b_n|r^n\},$$ defines a multiplicative semi-norm on $K[T]$, meaning $\eta_{a,r}$ is a point of $\mathbb{P}_K^{1, \mathrm{an}}.$ See ~\mbox{\cite[1.1.2.3]{poilibri}} for more details. 
\begin{rem} \label{poineau}
The author is very thankful to J\'er\^ome Poineau for sharing his unpublished notes which contain the following idea related to thickenings in $\mathbb{P}^{1, \mathrm{an}}$: for a complete ultrametric field $K$, an affinoid domain $A$ in $\mathbb{P}_{K}^{1,\mathrm{an}}$ is determined through finitely many polynomials defined over $K.$ Moreover, this ``polynomial writing" of $A$ is not unique.  If $L \subseteq K$ is a dense subfield of $K$, then we can find a polynomial writing for $A$ via polynomials defined over the smaller field $L$. 

We apply this idea to $K=\mathcal{H}(x)$ and $L=\mathcal{O}_x$ (see Lemma \ref{18aaa}), in which case it means that $A$ is determined by polynomials that are defined over some $\mathcal{O}(Z),$ where $Z$ is an affinoid neighborhood of $x$. This way, the thickening of $A$ can be defined via the same polynomial writing in $\mathbb{P}_Z^{1, \mathrm{an}}$, and we proceed to show that it has reasonably nice properties. 
\end{rem}

\begin{lm} \label{18aaa}
Let $U$ be a connected affinoid domain of $\mathbb{P}_{\mathcal{H}(x)}^{1,\mathrm{an}}$ with only type 3 points in its boundary. Suppose $U$ is not a point. Let us fix a copy of $\mathbb{A}_{\mathcal{H}(x)}^{1, \mathrm{an}}$ in $\mathbb{P}_{\mathcal{H}(x)}^{1, \mathrm{an}}$ with coordinate $T$. Let $\partial{U}=\{\eta_{R_i, r_i}: i=1,2,\dots, n\},$ where $R_i \in \mathcal{H}(x)[T]$ are irreducible polynomials and $r_i \in \mathbb{R}_{>0} \backslash \sqrt{|\mathcal{H}(x)^{\times}|}$ for $i=1,2,\dots, n$, so that (by Proposition \ref{projectivedom}) $U=\{y \in \mathbb{P}_{\mathcal{H}(x)}^{1, \mathrm{an}}: |R_i|_y \bowtie_i r_i, i=1,2,\dots, n\}$ with $\bowtie_i \in \{\leqslant, \geqslant\}$ for all $i$. Then the polynomials $R_i$ can be chosen such that $R_i \in \mathcal{O}_x[T]$ for all $i.$
\end{lm}

\begin{proof}
Let $\eta$ be any type 3 point of $\mathbb{P}_{\mathcal{H}(x)}^{1, \mathrm{an}}.$ It suffices to show that there exist $P \in \mathcal{O}_x[T]$ irreducible over $\mathcal{H}(x)$ and $p>0$ such that $\eta=\eta_{P,p}.$

The connected components of $\mathbb{P}_{\mathcal{H}(x)}^{1, \mathrm{an}} \backslash \{\eta\}$ are virtual discs. Let us fix one that does not contain the point $\infty.$ We need to show it contains a rigid point $\eta_{R,0}$ with $R \in \mathcal{O}_x[T]$ with ~$R$ irreducible over $\mathcal{H}(x).$ This follows immediately from the density of $\mathcal{O}_x$ in $\mathcal{H}(x).$
\end{proof}

\begin{rem} Since by Lemma \ref{18aaa}, $R_i \in \mathcal{O}_x[T]$ for all $i,$ there exists some connected affinoid neighborhood $Z$ of $x$ in $S,$ such that $R_i \in \mathcal{O}(Z)[T]$ for all $i.$ Hence, the affinoid domain $U$ can be \textit{thickened} to an affinoid domain $\{u \in \mathbb{P}_{Z}^{1, \mathrm{an}} : |R_i|_u \bowtie_i r_i, i=1,2,\dots, n\}$ of $\pi^{-1}(Z)=\mathbb{P}_{Z}^{1, \mathrm{an}}.$ The role of nice covers in this relative setting will be played by covers that are constructed by thickening affinoid domains of the fiber $\mathbb{P}_{\mathcal{H}(x)}^{1, \mathrm{an}}.$ We now study some properties of such domains which make patching possible. 
\end{rem}
\subsection{Thickenings of type 3 points} Following Notation \ref{17aaa},
the goal of this part is to show Theorem \ref{ajune} below, which states that the thickenings of type 3 points exist \emph{and} behave well, meaning they are connected. 

Here is an overview of the (somewhat long) proof: we take a thickening of the type 3 point and work over a finite base change where the corresponding polynomial splits; the affinoid domain we obtain is no longer necessarily fiberwise connected, but we understand it better thanks to Lemma \ref{para}; we decompose it naturally into a union of connected sets; this allows us to conclude seeing as the thickening we took can be written as a union of connected sets containing a common point.

Here, a \emph{unitary} polynomial is a one--variable polynomial where the leading coefficient (\emph{i.e.} the coefficient corresponding to the highest degree monomial) is equal to~$1$. 

\begin{thm} \label{ajune}
Let $\eta_{R,r}$ be a type 3 point of $\mathbb{P}_{\mathcal{H}(x)}^{1, \mathrm{an}}$, where $R \in \mathcal{O}_x[T]$ is irreducible over $\mathcal{H}(x)$ and $r \in \mathbb{R}_{>0} \backslash \sqrt{|\mathcal{H}(x)^\times|}$. There exists a connected affinoid neighborhood $Z_0$ of $x$ in ~$S$ such that 
\begin{itemize}
\item $R\in \mathcal{O}(Z_0)[T],$
\item  the set $\{u \in \mathbb{P}_{Z_0}^{1, \mathrm{an}}: |R|_u=r\}$ is a connected affinoid domain of $\mathbb{P}_{Z_0}^{1, \mathrm{an}}.$
\end{itemize}
Moreover, the statement remains true when replacing $Z_0$ by any connected affinoid neighborhood $Z \subseteq Z_0$ of $x$.
\end{thm}

\begin{proof}

Without loss of generality, since $\mathcal{O}_x$ is a field, we may assume that $R(T)$ is a unitary polynomial.

Let $Z_1$ be some connected affinoid neighborhood of $x$ in $S$ such that $R \in \mathcal{O}(Z_1)[T].$ Let $E$ be a finite field extension of $\mathscr{M}(Z_1)$ over which $R(T)$ splits. Since $\mathcal{O}(Z_1)$ is Japanese (see \cite[Proposition 2.1.14]{Ber90}), its integral closure in $E$ is a finite $\mathcal{O}(Z_1)$-algebra, and in particular, an integral $k$-affinoid algebra (see \cite[0.8]{Duc2}). Let us denote by $Z'$ the corresponding integral $k$-affinoid space. 

By construction, we have a finite morphism $\varphi: Z' \rightarrow Z_1$ inducing a finite morphism $\psi: \mathbb{P}_{Z'}^{1, \mathrm{an}} \rightarrow \mathbb{P}_{Z_1}^{1, \mathrm{an}}$, and the polynomial $R(T)$ is split over $\mathcal{O}(Z').$ Set $\{x_1, x_2, \dots, x_t\}:=\varphi^{-1}(x).$ Let us study the affinoid domain $|R(T)|=\prod_{R(\alpha)=0}|T-\alpha|=r$ in $\mathbb{P}_{Z'}^{1,\mathrm{an}},$ \textit{i.e.} the affinoid $\{u \in \mathbb{P}_{Z'}^{1,\mathrm{an}}: \prod_{R(\alpha)=0} |T-\alpha|_u=r\}.$
\begin{sloppypar}
Since $\varphi$ is a finite morphism, $\sqrt{|\mathcal{H}(x)^\times|}=\sqrt{|\mathcal{H}(x_i)^\times|}$ for any $i=1,2,\dots, t,$ so ${r \not \in \sqrt{|\mathcal{H}(x_i)^\times|}}.$ By Lemma \ref{para}, for all $i \in \{1,2,\dots, t\}$ and all $\alpha$ such that $R(\alpha)=0$,
there exist positive real numbers $s_{\alpha, x_i}$  such that  $\{u \in \mathbb{P}_{\mathcal{H}(x_i)}^{1, \mathrm{an}}:|R|_u=r\}=\bigcup_{R(\alpha)=0}\{u \in \mathbb{P}_{\mathcal{H}(x_i)}^{1, \mathrm{an}}: |T-\alpha|_{u}=s_{\alpha, x_i}\}.$ Since $r \not \in \sqrt{|\mathcal{H}(x_i)^\times|},$  $\{u \in \mathbb{P}_{\mathcal{H}(x_i)}^{1, \mathrm{an}}:|R|_u=r\}$ cannot contain any type ~2 points, so ${s_{\alpha, x_i} \not \in \sqrt{|\mathcal{H}(x_i)^\times|} \cup \{0\}}$ and $\{u \in \mathbb{P}_{\mathcal{H}(x_i)}^{1, \mathrm{an}}:|R|_u=r\}=\{\eta_{\alpha, s_{\alpha, x_i}}: R(\alpha)=0\}$.  
\end{sloppypar}
\begin{lm} \label{pas}
For any $i \in \{1,2,\dots, t\},$ and any root $\alpha$ of $R(T),$ there exists a connected affinoid neighborhood $Z_i'$ of $x_i$ and a continuous function $s_{\alpha}^i: Z_i' \rightarrow \mathbb{R}_{>0}$ such that for any $y \in Z_i',$
$$\{u \in \mathbb{P}_{\mathcal{H}(y)}^{1, \mathrm{an}}:|R|_u=r\}=\bigcup_{R(\alpha)=0} \{u \in \mathbb{P}_{\mathcal{H}(y)}^{1,\mathrm{an}}: |T-\alpha|_u=s_{\alpha}^i(y)\}.$$
Furthermore, we may assume that for any $j \neq i, x_{j} \not \in Z_i'.$
\end{lm}

\begin{proof}

Let us fix an $i \in \{1,2,\dots, t\}$ and a root $\alpha$ of $R(T)$ of multiplicity $m$. Let $\alpha_1, \alpha_2, \dots, \alpha_n$ be the rest of the roots (with multiplicity) of $R(T),$ ordered in such a way that for any $j \leqslant l,$ $|\alpha-\alpha_{j}|_{x_i} \leqslant |\alpha-\alpha_l|_{x_i}.$ As remarked above, $s_{\alpha, x_i} \not \in \sqrt{|\mathcal{H}(x_i)^\times|} \cup \{0\},$ so $s_{\alpha, x_i} \neq |\alpha-\alpha_j|_{x_i}$ for all $j=1,2,\dots, n.$ Set $\alpha_0:=\alpha.$ Then there exists a unique $j_0 \in \{0,1,\dots, n\},$ such that $|\alpha-\alpha_{j}|_{x_i} < s_{\alpha, x_i} < |\alpha-\alpha_l|_{x_i}$ for all $j, l$ for which $j \leqslant j_0 < l$ (in particular, if $j_0=0,$ this means that $0<s_{\alpha, x_i} < |\alpha-\alpha_1|_{x_i},$ and if $j_0=n,$ that $|\alpha-\alpha_n|_{x_i}< s_{\alpha, x_i}$). Since in $\mathbb{P}_{\mathcal{H}(x_i)}^{1, \mathrm{an}}:$
$$r=|R|_{\eta_{\alpha, s_{\alpha, x_i}}}=|T-\alpha|_{\eta_{\alpha, s_{\alpha, x_i}}}^m \prod_{j=1}^n |T-\alpha_i|_{\eta_{\alpha, s_{\alpha, x_i}}}=s_{\alpha, x_i}^m \cdot \prod_{j=1}^n \max(s_{\alpha, x_i}, |\alpha-\alpha_j|_{x_i}),$$
we obtain that $s_{\alpha, x_i}=\sqrt[j_0+m]{\frac{r}{\prod_{j=j_0+1}^n |\alpha-\alpha_j|_{x_i}}}$ (this means that $s_{\alpha, x_i}=\sqrt[n+m]{r}$ if $j_0=n$.) Note that $|\alpha-\alpha_j|_{x_i} \neq 0$ for all $j > j_0$ seeing as $s_{\alpha, x_i}<|\alpha-\alpha_j|_{x_i}.$
 
Since the function $Z' \rightarrow \mathbb{R}_{>0}$, $y \mapsto |\alpha-\alpha_j|_{y}$ is continuous for all $j=1,2,\dots, n,$ there exists a connected affinoid neighborhood $Z_{i,1}$ of $x_i$ in $Z'$ such that $|\alpha-\alpha_j|_y \neq 0$ for all $j>j_0$ and all $y \in Z_{i,1}.$

 Let us define $s_{\alpha}^i : Z_{i,1} \rightarrow \mathbb{R}_{>0}$ by $y \mapsto \sqrt[j_0+m]{\frac{r}{\prod_{j=j_0+1}^n |\alpha-\alpha_j|_{y}}}.$ It is a continuous function and $s_{\alpha, x_i}=s_{\alpha}^i(x_i).$ Also, $|\alpha-\alpha_{j}|_{x_i} < s_{\alpha}^i(x_i) < |\alpha-\alpha_l|_{x_i}$ for all $j, l$ for which $j \leqslant j_0 < l.$ Since on all sides of these strict inequalities we have continuous functions, there exists a connected affinoid neighborhood $Z_i'$ of $x_i$ in $Z_{i,1}$ such that for all $y \in Z_i',$ $s_{\alpha}^i(y)$ is positive and $|\alpha-\alpha_{j}|_{y} < s_{\alpha}^i(y) < |\alpha-\alpha_l|_{y}$ for all $j, l$ for which $j \leqslant j_0 < l.$  

Consequently, for $y \in Z_i',$ in $\mathbb{P}_{\mathcal{H}(y)}^{1, \mathrm{an}},$  $|R(T)|_{\eta_{\alpha, s_{\alpha}^i(y)}}=s_{\alpha}^i(y)^{j_0+m} \cdot \prod_{j=j_0+1}^n |\alpha-\alpha_j|_y=r.$ We can now conclude by using Lemma \ref{para}. 

Finally, the last part of the statement is a direct consequence of the fact that $Z'$ is Hausdorff. 
\end{proof}

\begin{rem} \label{19aaa}
Lemma \ref{pas} is clearly true for any connected affinoid neighborhood of $x_i$ contained in $Z_i'.$  
\end{rem}
Let us now resume the proof of Theorem \ref{ajune}.

Let $Z_i$ be any connected affinoid neighborhood of $x_i$ such that $Z_i \subseteq Z_i'.$ 
In view of Lemma ~\ref{pas}, for any $i \in \{1,2,\dots, t\},$ $$\{u \in \mathbb{P}_{Z_i}^{1, \mathrm{an}}: |R(T)|_u=r\}=\bigcup_{R(\alpha)=0} \{u \in \mathbb{P}_{Z_i}^{1, \mathrm{an}}: |T-\alpha|_u=s_{\alpha}^i(\pi(u))\}.$$ For any root $\alpha$ of $R(T),$ set $S_{\alpha, Z_i}:=\{u \in \mathbb{P}_{Z_i}^{1, \mathrm{an}}: |T-\alpha|_u=s_{\alpha}^i(\pi(u))\}$.

\begin{lm} \label{connectedness}
For $i \in \{1,2,\dots, t\},$ the set $S_{\alpha, Z_i}$ is connected. 
\end{lm}

\begin{proof}
Seeing as $s_{\alpha}^i$ is a continuous function, $S_{\alpha, Z_i}$ is a closed and hence compact subset of $\mathbb{P}_{Z_i}^{1, \mathrm{an}}.$ Suppose that $S_{\alpha, Z_i}$ is not connected and assume it can be written as a
disjoint union of two closed subsets $S_{\alpha,Z_i}'$ and $S_{\alpha,Z_i}''.$ Since $S_{\alpha, Z_i}$ is compact in $\mathbb{P}_{Z_i}^{1,\mathrm{an}},$ so are $S_{\alpha,Z_i}'$ and $S_{\alpha,Z_i}''.$ Since the morphism $\pi$ is proper, $\pi(S_{\alpha,Z_i}')$ and $\pi(S_{\alpha,Z_i}'')$ are both compact subsets of $Z_i.$ Also, $\pi(S_{\alpha, Z_i})=Z_i,$ implying $Z_i=\pi(S_{\alpha,Z_i}') \cup \pi(S_{\alpha, Z_i}'').$ Assume that ${\pi(S_{\alpha,Z_i}') \cap \pi(S_{\alpha,Z_i}'')\neq \emptyset}.$ This means that there exists a point $y \in Z_i,$ such that both $\mathbb{P}_{\mathcal{H}(y)}^{1, \mathrm{an}}\cap S_{\alpha,Z_i}'$ and $\mathbb{P}_{\mathcal{H}(y)}^{1, \mathrm{an}}\cap S_{\alpha,Z_i}''$ are non-empty. But then the connected domain $\{u \in \mathbb{P}_{\mathcal{H}(y)}^{1, \mathrm{an}}:|T-\alpha|_u =s_{\alpha}^i(y)\}$ of $\mathbb{P}_{\mathcal{H}(y)}^{1, \mathrm{an}}$ can be written as the union of two disjoint closed subsets, which is impossible. 
Thus, $\pi(S_{\alpha,Z_i}') \cap \pi(S_{\alpha,Z_i}'')= \emptyset,$ so $Z_i$ can be written as a 
disjoint union of two closed subsets. This is impossible seeing as $Z_i$ is connected. 
Consequently, $S_{\alpha, Z_i}$ is connected.
\end{proof}

Recall that the finite morphism $Z' \rightarrow Z_1$ was denoted by $\varphi.$ Let $U_i \subseteq Z_i'$ be open neighborhoods of $x_i$ in $Z',$ $i=1,2,\dots, n.$ Then by \cite[Lemma I.1.2]{stein}, there exists a neighborhood $U$ of $x$ in $Z,$ such that $\varphi^{-1}(U) \subseteq \bigcup_{i=1}^t U_i  \subseteq \bigcup_{i=1}^t Z_i'.$ Let $Z_0 \subseteq U$ be any connected affinoid neighborhood of $x.$ Then $\varphi^{-1}(Z_0)$ (which is a subset of $\bigcup_{i=1}^t Z_i'$) is an affinoid domain of ~$Z'.$

Any connected component $C$ of $\varphi^{-1}(Z_0)$ is mapped surjectively onto $Z_0.$ To see this, remark that $\varphi$ is at the same time a closed and open morphism (see \cite[Lemma 3.2.4]{Ber90}). Consequently, $\varphi(C)$ is a closed and open subset of $Z_0.$ Since $Z_0$ is connected, $\varphi(C)=Z_0.$ Thus, for any $i,$ there exists exactly one connected component $Z_i$ of $\varphi^{-1}(Z_0)$ containing ~$x_i$ and $\varphi^{-1}(Z_0)=\bigcup_{i=1}^t Z_i.$ By construction, $Z_i \subseteq Z_i'.$ 

Let us look at the induced finite morphism $\psi: \mathbb{P}_{\varphi^{-1}(Z_0)}^{1,\mathrm{an}}= \bigsqcup_{i=1}^t \mathbb{P}_{Z_i}^{1, \mathrm{an}} \rightarrow \mathbb{P}_{Z_0}^{1, \mathrm{an}}.$ The preimage of $\{u \in \mathbb{P}_{Z_0}^{1,\mathrm{an}}: |R|_u=r\}$ by $\psi$ is the affinoid $\{u \in \mathbb{P}_{\varphi^{-1}(Z_0)}^{1,\mathrm{an}}: |R|_u=r\}$. Recall that for any $i,$ $\{u \in \mathbb{P}_{Z_i}^{1,\mathrm{an}}: |R|_u=r\}=\bigcup_{R(\alpha)=0} S_{\alpha, Z_i},$ so $$\{u \in \mathbb{P}_{\varphi^{-1}(Z_0)}^{1,\mathrm{an}}: |R|_u=r\}=\bigcup_{i=1}^t \bigcup_{R(\alpha)=0} S_{\alpha, Z_i}.$$

By Lemma \ref{connectedness}, each of the $S_{\alpha, Z_i}$ is connected, and thus so is $\psi(S_{\alpha, Z_i}).$ Since $S_{\alpha, Z_i} \cap \{u \in \mathbb{P}_{\mathcal{H}(x_i)}^{1,\mathrm{an}}: |R|_u=r\} \neq \emptyset,$ we also have $\psi(S_{\alpha, Z_i}) \cap \{u \in \mathbb{P}_{\mathcal{H}(x)}^{1,\mathrm{an}}: |R|_u=r\} \neq \emptyset$. Consequently, the type 3 point $\eta_{R, r} \in \mathbb{P}_{\mathcal{H}(x)}^{1,\mathrm{an}}$ is contained in all of the $S_{\alpha, Z_i}.$ 

Finally, seeing as $\{u \in \mathbb{P}_{\varphi^{-1}(Z_0)}^{1,\mathrm{an}}: |R|_u=r\}$ can be written as a finite union of connected sets, all of which contain a common point, it is connected. We have shown that $Z_0$ satisfies the two points of the statement. 

It is immediate from the constructions we made that the statement of Theorem \ref{ajune} remains true for any other connected affinoid neighborhood of $x$ contained in $Z_0.$
\end{proof}

\subsection{Towards relative nice covers}
We construct here a relative version of nice covers around the fiber. 

\begin{defn} \label{100}
Let $P_i \in \mathcal{O}_x[T]$ be irreducible over $\mathcal{H}(x)$ and $r_i \in \mathbb{R}_{\geqslant 0},$ ${i=1,2,\dots, n}.$ 
The set $A=\{u \in \mathbb{P}_{\mathcal{H}(x)}^{1, \mathrm{an}}: |P_i|_u \bowtie_i r_i, i=1,2,\dots, n\}$, where $\bowtie_i \in \{\leqslant, \geqslant\},$ is an affinoid domain of $\mathbb{P}_{\mathcal{H}(x)}^{1, \mathrm{an}}$. 
For any affinoid neighborhood $Z$ of $x$ for which $P_i \in \mathcal{O}(Z)[T]$ for all ${i=1,2,\dots, n},$ we will denote by $A_Z$ the affinoid domain ${\{u \in \mathbb{P}_Z^{1, \mathrm{an}}:  |P_i|_u \bowtie_i r_i, i=1,2,\dots, n\}}$ of $\mathbb{P}_Z^{1, \mathrm{an}}$ and call it the $Z$-\textit{thickening} of $A.$
\end{defn}

\begin{rem} \label{101}
The thickening of an affinoid domain of $\mathbb{P}_{\mathcal{H}(x)}^{1, \mathrm{an}}$ depends on the polynomials we choose to represent its boundary points. Hence, from now on, when speaking of the thickening of such an affinoid, we will, unless it plays a specific role (in which case we mention it explicitely), always assume that a writing of the boundary points was fixed \textit{a priori}.
\end{rem}

Recall Notation \ref{17aaa}: among other things, we take a $k$-analytic space $S$, $\pi: \mathbb{P}_S^{1, \mathrm{an}} \rightarrow S$ and $x \in S$ with fiber $F_x$ and such that $\mathcal{O}_x$ is a field.

\begin{lm} \label{nukkuptoj} Let $Z$ be a connected affinoid neighborhood of $x$ in $S$.
Let $A, B, C$ be closed subsets of $\mathbb{P}_{Z}^{1, \mathrm{an}}$ such that $A \cap B \cap F_x = C \cap F_x.$  Suppose there exists an open $W$ of $\mathbb{P}_{Z}^{1, \mathrm{an}}$ such that 
$A \cap B \cap W=C \cap W$ and ${C \cap F_x \subseteq W}.$ Then there exists a connected affinoid neighborhood $Z' \subseteq Z$ of $x$ such that for any connected affinoid neighborhood $Z'' \subseteq Z'$ of ~$x,$ $A \cap B \cap \pi^{-1}(Z'')=C \cap \pi^{-1}(Z'').$ 
\end{lm}

\begin{proof}
Set $F_1= A \cap B \cap W^c$ and $F_2=C \cap W^c,$ where $W^c$ is the complement of $W$ in $\mathbb{P}_{Z}^{1, \mathrm{an}}.$ Remark that $F_i$ is a closed, hence compact, set and that $F_i \cap F_x = \emptyset, i=1,2.$ Since $\pi$ is proper, $\pi(F_i)$ is a closed subset of $Z$, and it does not contain $x.$ Thus, there exists a connected affinoid neighborhood $Z' \subseteq Z$ of $x$ such that $Z' \cap \pi(F_i) =\emptyset, i=1,2.$ Consequently, $\pi^{-1}(Z') \cap F_i=\emptyset.$

Remark that $\pi^{-1}(Z') \cap F_1=\pi^{-1}(Z') \cap A \cap B \cap W^c =\emptyset,$ so $\pi^{-1}(Z') \cap A \cap B \subseteq W.$ Similarly, $\pi^{-1}(Z') \cap C \subseteq W.$ Finally, $A \cap B \cap \pi^{-1}(Z')=A \cap B \cap \pi^{-1}(Z') \cap W= C \cap W \cap \pi^{-1}(Z')= C \cap \pi^{-1}(Z').$  Clearly, the same remains true when replacing $Z'$ by any connected affinoid neighborhood $Z''\subseteq Z'.$
\end{proof}

Let $U$ and $V$ be connected affinoid domains of $\mathbb{P}_{\mathcal{H}(x)}^{1, \mathrm{an}}$ containing only type 3 points in their boundaries. Suppose that $U \cap V$ is a single type 3 point $\{\eta\}.$ This means that $U \cap V= \partial{U} \cap \partial{V}=\{\eta\}.$ By Lemma \ref{18aaa}, there exist $R(T) \in \mathcal{O}_x[T]$ irreducible over $\mathcal{H}(x)$ and $r \in \mathbb{R}_{>0} \backslash \sqrt{|\mathcal{H}(x)^\times|}$ such that $\eta=\eta_{R,r}$.

By Lemma ~\ref{15}, either $U \subseteq \{u \in \mathbb{P}_{\mathcal{H}(x)}^{1, \mathrm{an}}: |R|_u \leqslant r\}$ or $U \subseteq \{u \in \mathbb{P}_{\mathcal{H}(x)}^{1, \mathrm{an}}: |R|_u \geqslant r\}.$ Without loss of generality, let us assume $U \subseteq \{u \in \mathbb{P}_{\mathcal{H}(x)}^{1, \mathrm{an}}: |R|_u \leqslant r\}$. Then by Lemma ~\ref{16aaa}, $V \subseteq \{u \in \mathbb{P}_{\mathcal{H}(x)}^{1, \mathrm{an}}: |R|_u \geqslant r\}$. Set $\partial{U}=\{\eta_{R,r}, \eta_{P_i,r_i}\}_{i=1}^n$ and $\partial{V}=\{\eta_{R,r}, \eta_{P_j', r_j'}\}_{j=1}^m,$ where $P_i, P_j' \in \mathcal{O}_x[T]$ are irreducible over $\mathcal{H}(x),$ and $r_i,  r_j' \in \mathbb{R}_{>0} \backslash \sqrt{|\mathcal{H}(x)^\times|},$ for all $i$ and ~$j.$ By Proposition \ref{projectivedom}:
\begin{equation} \label{eq1}
U=\{u \in \mathbb{P}_{\mathcal{H}(x)}^{1, \mathrm{an}}: |R|_u \leqslant r, |P_i|_u \bowtie_i r_i, i=1,2,\dots,n\},
\end{equation}
\begin{equation} \label{eq2}
V=\{u \in \mathbb{P}_{\mathcal{H}(x)}^{1, \mathrm{an}}: |R|_u \geqslant r, |P_j'|_u \bowtie_j' r_j', j=1,2,\dots, m\}, 
\end{equation}
where $\bowtie_i, \bowtie_j' \in \{\leqslant, \geqslant\}$ for all $i,j.$ There exists a connected affinoid neighborhood $Z$ of $x$ in $S,$ such that $P_i, P_j', R \in \mathcal{O}(Z)[T]$ for all $i,j.$ Let us study the relationship between the $Z$-thickenings $U_Z, V_Z$ of $U$ and $V,$ respectively.

\begin{prop} \label{nukkuptoj2}
There exists a connected affinoid neighborhood $Z' \subseteq Z$ of $x$ such that:
\begin{enumerate}
\item $U_{Z'} \cap V_{Z'} = (U \cap V)_{Z'}=\{u \in \mathbb{P}_{Z'}^{1, \mathrm{an}}: |R|_u=r\};$
\item $U_{Z'} \cup V_{Z'}= (U \cup V)_{Z'}= \{u \in \mathbb{P}_{Z'}^{1, \mathrm{an}}:  |P_i|_u \bowtie_i r_i,  |P_j'|_u \bowtie_j' r_j', i,j\}$ (see Lemma~\ref{16aaa}). If $n=m=0$ (see (\ref{eq1}) and (\ref{eq2}) above), this means that $U_{Z'} \cup V_{Z'}=\mathbb{P}_{Z'}^{1, \mathrm{an}}.$
\end{enumerate}
The same is true for any connected affinoid neighborhood $Z'' \subseteq Z'$ of $x.$
\end{prop}

\begin{proof}
\begin{enumerate}
\item Set $W=\{u \in \mathbb{P}_Z^{1, \mathrm{an}}: |P_i|_u \ \underline{\bowtie}_i  \ r_i,  |P_j'|_u \ \underline{\bowtie}_j' \ r_j', i,j\},$ where $\underline{\bowtie}_i$ (resp. $\underline{\bowtie}_j'$) is the strict version of
$\bowtie_i$ (resp. $\bowtie_j'$), meaning for example if $\bowtie_i$ is $\leqslant$ then $\underline{\bowtie}_i$ is $<$. Set also
${A=U_Z,} B=V_Z,$ and $C=\{u \in \mathbb{P}_{Z}^{1, \mathrm{an}}: |R|_u=r\}.$ Remark that: $W$ is open, $A, B, C$ are closed, $A \cap B \cap W=\{u \in \mathbb{P}_Z^{1, \mathrm{an}}:|R|_u=r,  |P_i|_u \ \underline{\bowtie}_i \ r_i,  |P_j'|_u \ \underline{\bowtie}_j' \ r_j', i,j\} =C \cap W$, and $A \cap B \cap F_x=U \cap V=\{\eta_{R,r}\}=C \cap F_x.$ By Lemma \ref{nukkuptoj}, there exists a connected affinoid neighborhood $Z'$ of $x$ such that $${U_{Z'} \cap V_{Z'} = \{u \in \mathbb{P}_{Z'}^{1, \mathrm{an}}: |R|_u=r\}}=(U \cap V)_{Z'},$$ and the same remains true when replacing $Z'$ with any connected affinoid neighborhood $Z'' \subseteq Z'$ of $x.$

\item  Set $W=\{u \in \mathbb{P}_Z^{1, \mathrm{an}}: |P_j'|_u \ \underline{\bowtie}_j' \ r_j', j=1,\dots, m\},$ where $\underline{\bowtie}_j'$ is the strict version of $\bowtie_j'$. Set also $A=C=U_Z$ and ${B=\{u \in \mathbb{P}_{Z}^{1, \mathrm{an}}: |P_i|_u \bowtie_i r_i,  |P_j'|_u \bowtie_j' r_j', i,j\}}.$ Clearly, $W$ is open and $A, B, C$ are closed. Also, $$A \cap B \cap W={\{u \in \mathbb{P}_{Z'}^{1, \mathrm{an}}: |R|_u \leqslant r,} |P_i|_u \bowtie_i r_i,  |P_j'|_u \ \underline{\bowtie}_j' \ r_j', i,j\}=C \cap W.$$

Recall that by Lemma \ref{16aaa}, $B \cap F_x= U \cup V$. Hence, $A \cap B \cap F_x=U=C \cap F_x$. Thus, it only remains to show that $C \cap F_x \subseteq W$ in order for the Lemma \ref{nukkuptoj} to be applicable. Remark that $U \subseteq U \cup V= B \cap F_x$, and $B \cap F_x \subseteq \{u \in \mathbb{P}_Z^{1, \mathrm{an}}: |P_j'|_u \ {\bowtie}_j' \ r_j', j=1,\dots, m\} \cap F_x =\{u \in \mathbb{P}_{\mathcal{H}(x)}^{1, \mathrm{an}}: |P_j'|_u \ {\bowtie}_j' \ r_j', j=1,\dots, m\}$. Assume for some $y \in U,$ there exists $j_0 \in \{1,2, \dots, m\}$ such that $|P'_{j_0}|_y = r_{j_0}'.$ As there is a unique point ($\eta_{P_{j_0}', r_{j_0}'}$) in $\mathbb{P}_{\mathcal{H}(x)}^{1, \mathrm{an}}$ satisfying $|P_{j_0}'|=r'_{j_0},$ we obtain $y=\eta_{P_{j_0}', r_{j_0}'},$ implying $y \in V.$ Hence, $y \in U \cap V,$ so $y=\eta_{R,r}$, contradiction. 
Finally, $y \in \{u \in \mathbb{P}_{\mathcal{H}(x)}^{1, \mathrm{an}}: |P_j'|_u \ \underline{\bowtie}_j' \ r_j', j=1,\dots, m\} \subseteq W$, meaning $C \cap F_x= U \subseteq W$.   
\begin{sloppypar}
Lemma \ref{nukkuptoj} is applicable, so there exists a connected affinoid neighborhood ${Z_1' \subseteq Z}$ of $x$ such that $U_Z \cap B \cap \pi^{-1}(Z_1')=U_Z \cap \pi^{-1}(Z_1'),$ implying ${U_{Z_1'} \subseteq B \cap \pi^{-1}(Z_1')}.$ The same remains true for any connected affinoid neighborhood $Z_1'' \subseteq Z_1'$ of $x.$
\end{sloppypar}
Using similar arguments one shows that there exists a connected affinoid neighborhood $Z_2' \subseteq Z$ of $x$ such that $V_{Z_2'} \subseteq B \cap \pi^{-1}(Z_2'),$ and the same remains true for any connected affinoid neighborhood $Z_2'' \subseteq Z_2'$ of $x.$

Thus, there exists a connected affinoid neighborhood $Z' \subseteq Z$ of $x$ such that $U_{Z'} \cup V_{Z'} \subseteq B_{Z'}:=\{u \in \mathbb{P}_{Z'}^{1, \mathrm{an}}:  |P_i|_u \bowtie_i r_i, |P_j'|_u \bowtie_j' r_j', i,j\}$, and the same is true for any connected affinoid neighborhood $Z'' \subseteq Z'$ of $x.$  Let $u \in B_{Z''}$, where $B_{Z''}:=B_{Z'} \cap \pi^{-1}(Z'').$ If $|R|_u \leqslant r,$ then $u \in U_{Z''}.$ If $|R|_u \geqslant r,$ then $u \in V_{Z''}.$ Consequently, $u\in U_{Z''} \cup V_{Z''},$ and $U_{Z''} \cup V_{Z''}=B_{Z''}=(U\cup V)_{Z''}.$ 
\end{enumerate}
\enlargethispage*{0.5cm}
\end{proof}
Let us show that this construction of affinoid domains in $\mathbb{P}_Z^{1,\mathrm{an}},$ where $Z$ is a connected affinoid neighborhood of $x,$ gives us a family of neighborhoods of the points of the fiber $F_x$ in $\mathbb{P}_Z^{1,\mathrm{an}}$ (given we choose $Z$ small enough).

\begin{lm} \label{103}
Let $A$ be an open subset of $\mathbb{P}_S^{1,an}$ such that $A \cap F_x \neq \emptyset.$ Let $U=\{u \in \mathbb{P}_{\mathcal{H}(x)}^{1,\mathrm{an}}: |P_i|_u \bowtie_i 
r_i, i=1,2,\dots, n\},$ $\bowtie_i \in \{\leqslant, \geqslant\},$ be any affinoid domain of $\mathbb{P}_{\mathcal{H}(x)}^{1,\mathrm{an}}$ contained in $A \cap F_x,$ where $P_i \in \mathcal{O}_x[T]$ is irreducible over $\mathcal{H}(x)$ and $r_i \in \mathbb{R}_{\geqslant 0}, i=1,2,\dots, n.$
Then there exists a connected affinoid neighborhood $Z$ of $x$ such that $P_i \in \mathcal{O}(Z)[T]$ for all $i$ and $U_Z \subseteq A.$
The same is true for any connected affinoid neighborhood $Z' \subseteq Z$ of $x.$
\end{lm}
 
\begin{proof}
Let $Z_0$ be a connected affinoid neighborhood of $x$ for which the thickening $U_{Z_0}$ exists. Suppose $U_{Z_0} \not \subseteq A.$ 
Then $U_{Z_0} \backslash A$ is a non-empty compact subset of $\mathbb{P}_{S}^{1, \mathrm{an}}.$ This implies that $\pi(U_{Z_0} \backslash A)$ is a compact subset of $S.$ Furthermore, since $U \subseteq A,$  $x \not \in \pi(U_{Z_0} \backslash A),$ so there exists a connected affinoid neighborhood $Z \subseteq Z_0$ of $x$ such that $Z \cap \pi(U_{Z_0} \backslash A) = \emptyset.$ This implies that for any connected affinoid neighborhood $Z' \subseteq Z$ of $x,$  $U_{Z'} \backslash A=  \pi^{-1}(Z') \cap (U_{Z_0} \backslash A)=\emptyset,$ and finally that $U_{Z'}  \subseteq A.$
\end{proof}

Let $\mathcal{U}_x$ be a nice cover of $\mathbb{P}_{\mathcal{H}(x)}^{1, \mathrm{an}}$ (recall Definition \ref{nice}). Let $S_{\mathcal{U}_x}=\{\eta_1, \eta_2, \dots, \eta_t\}$ be the set of intersection points of the elements of $\mathcal{U}_x.$ For any $\eta_i \in S_{\mathcal{U}_x}, i=1,2,\dots, t,$ there exist $R_i \in \mathcal{O}_x[T]$ irreducible over $\mathcal{H}(x)$ and $r_i \in \mathbb{R}_{>0} \backslash \sqrt{|\mathcal{H}(x)^\times|}$ such that $\eta_i=\eta_{R_i,r_i}.$ Since ${\bigcup_{U \in \mathcal{U}_x} \partial{U}=S_{\mathcal{U}_x}},$ all pieces of $\mathcal{U}_x$ are a combination of intersections of the affinoid domains $\{|R_i| \bowtie_i r_i\}$ of $\mathbb{P}_{\mathcal{H}(x)}^{1, \mathrm{an}},$ where $\bowtie_i \in \{\leqslant, \geqslant\}, i=1,2,\dots, t.$ 

For any affinoid neighborhood $Z_a$ of $x$ such that $R_i \in \mathcal{O}(Z_a)[T]$ for all $i,$ let us denote by $\mathcal{U}_{Z_a}$ the set of $Z_a$-thickenings of the elements of $\mathcal{U}_x.$ Let $Z'$ be a fixed connected affinoid neighborhood of $x$ such that $R_i \in \mathcal{O}(Z')[T]$ for all $i=1,2,\dots, t.$
\begin{thm} \label{nicecover}
There exists a connected affinoid neighborhood $Z \subseteq Z'$ of $x$ such that  the set $\mathcal{U}_Z$ is a cover of $\mathbb{P}_Z^{1, \mathrm{an}},$ and
\begin{enumerate}
\item for any $U \in \mathcal{U}_x,$ the $Z$-thickening $U_Z$ is a connected affinoid domain of $\mathbb{P}_Z^{1, \mathrm{an}};$
\item for any different $U, V \in \mathcal{U}_x$, either $U_Z \cap V_Z=\emptyset$ or there exists a unique ${j \in \{1, \dots, t\}}$ such that $U_Z \cap V_Z=\{u \in \mathbb{P}_Z^{1, \mathrm{an}}: |R_j|_u=r_j\}=(U \cap V)_Z$ is a connected affinoid domain of $\mathbb{P}_Z^{1, \mathrm{an}}$;  in particular, $U_Z \cap V_Z \neq \emptyset$ if and only if $U \cap V \neq \emptyset;$
\item for any $U_Z, V_Z \in \mathcal{U}_Z,$ $U_Z \cup V_Z$ is either $\mathbb{P}_Z^{1, \mathrm{an}}$ or a connected affinoid domain of ~$\mathbb{P}_Z^{1, \mathrm{an}}$ that is the $Z$-thickening of $U \cup V.$
\end{enumerate}
The statement is true for any connected affinoid neighborhood $Z'' \subseteq Z$ of $x.$
\end{thm} 

\begin{proof}
By Theorem \ref{ajune}, there exists a connected affinoid neighborhood $Z$ of $x,$ such that $R_i \in \mathcal{O}(Z)[T]$ and the affinoid domains $\{u \in \mathbb{P}_Z^{1, \mathrm{an}}: |R_i|_u=r_i\}$ are all connected. We may also assume that for any two non-disjoint elements ${U=\{u \in \mathbb{P}_{\mathcal{H}(x)}^{1, \mathrm{an}}: |P_i|_u \bowtie_i r_i, |R|_u \leqslant r:}$ $i=1,\dots, n\}$ and $V=\{u \in \mathbb{P}_{\mathcal{H}(x)}^{1, \mathrm{an}}: |P_j'|_u \bowtie_j' r_j', |R|_u \geqslant r: j=1,\dots, m\}$ of $\mathcal{U}_x$, Proposition ~\ref{nukkuptoj2} holds. 

Let $\mathcal{U}_x=\{U_1, U_2, \dots, U_n\}.$ By \cite[Lemma 2.20]{une}, there exist $n-1$ elements of $\mathcal{U}_x$ whose union is connected. Without loss of generality, let us assume that $V:=\bigcup_{l=1}^{n-1} U_l$ is connected. By \cite[Th\'eor\`eme 6.1.3]{Duc}, this is a connected affinoid domain; also $V \cup U_n=\mathbb{P}_{\mathcal{H}(x)}^{1, \mathrm{an}}.$ Since $V, U_n,$ and $U_n \cup V$ are connected subsets of $\mathbb{P}_{\mathcal{H}(x)}^{1, \mathrm{an}}$, $U_n \cap V$ is a non-empty connected set, hence a single type ~3 point $\{\eta_{R_j, r_j}\}$ for some $j \in \{1, 2,\dots, t\}$. In particular, this implies that $U_n=\{u \in \mathbb{P}_{\mathcal{H}(x)}^{1, \mathrm{an}}: |R_j|_u \bowtie r_j\}$, where $\bowtie \in \{\leqslant, \geqslant\}$ (recall Proposition \ref{projectivedom}). Let us assume, without loss of generality, that $U_n=\{u \in \mathbb{P}_{\mathcal{H}(x)}^{1, \mathrm{an}}:|R_j|_u \geqslant r_j\}.$ Then $V=\{u \in \mathbb{P}_{\mathcal{H}(x)}^{1, \mathrm{an}}:|R_j|_u \leqslant r_j\}$ (see Lemma ~\ref{16aaa} to recall what the inequalities for the union of two non-disjoint elements of a nice cover look like). Consequently, $U_{n, Z}=\{u \in \mathbb{P}_Z^{1,an}: |R_j|_u \geqslant r_j\}$ and by Proposition ~\ref{nukkuptoj2}, $V_Z= \left(\bigcup_{l=1}^{n-1} U_{i} \right)_Z= \bigcup_{i=1}^{n-1} U_{i, Z}= \{u \in \mathbb{P}_Z^{1,an}: |R_j|_u \leqslant r_j\},$ so $U_{n,Z} \cup V_Z=\mathbb{P}_Z^{1,an},$ and $\mathcal{U}_Z$ is a cover of $\mathbb{P}_Z^{1, \mathrm{an}}.$

Let $U \neq  V \in \mathcal{U}_x$.  Clearly, if $U_Z \cap V_Z=\emptyset,$ then $U \cap V=\emptyset.$ Assume $U \cap V=\emptyset.$ Suppose $A:=U_Z \cap V_Z \neq \emptyset.$ Remark that ${A \cap F_x=\emptyset}.$ Since $A$ is compact and $\pi$ proper, $\pi(A)$ is a compact subset of $Z$ not containing ~$x.$ Thus, there exists a connected affinoid neighborhood $Z_1 \subseteq Z,$ such that $A \cap \pi^{-1}(Z_1)=\emptyset,$ and $U_{Z_1} \cap V_{Z_1}=\emptyset.$ Thus, we may assume that for any disjoint $U, V \in \mathcal{U}_x,$ $U_Z \cap V_Z=\emptyset,$ which, taking into account Proposition \ref{nukkuptoj2}, shows that property (2) of the statement is true. 

Property (3) is a consequence of \cite[Corollary 2.2.7(i)]{Ber90} if $U_Z \cap V_Z=\emptyset,$ and of Proposition ~\ref{nukkuptoj2} if not.

Let $Z$ be such that property (2) is satisfied. Suppose there exists $U \in \mathcal{U}_x$ such that $U_Z$ is not connected. Let $C$ be a connected component of $U_Z$ that doesn't intersect $F_x,$ and $B$ the connected component that does. For any $V \in \mathcal{U}_x$ for which $U \cap V=\emptyset,$ $C \cap V_Z \subseteq U_Z \cap V_Z=\emptyset.$ For any $V \in \mathcal{U}_x$ for which $U \cap V \neq \emptyset,$ there exists a unique $j$ such that $U_Z \cap V_Z=\{u \in \mathbb{P}_Z^{1, \mathrm{an}}: |R_j|_u=r_j\}$ is a connected affinoid domain, so $U_Z \cap V_Z=B \cap V_Z.$ Consequently, $C \cap V_Z=\emptyset.$  
This means that $C \cap \left((U_Z \backslash C) \cup \bigcup_{V \in \mathcal{U}_x, U \neq V} V_Z \right)=\emptyset,$ and $C \cup \left((U_Z \backslash C) \cup \bigcup_{V \in \mathcal{U}_x, U \neq V} V_Z \right)=\mathbb{P}_Z^{1, \mathrm{an}},$ implying $\mathbb{P}_Z^{1, \mathrm{an}}$ is not connected, contradiction. This concludes the proof of part (1). 

The last sentence of the statement is immediate from the nature of the proof.
\end{proof}

Finally:

\begin{defn} \label{104}
Let $\mathcal{U}_x$ be a nice cover of $\mathbb{P}_{\mathcal{H}(x)}^{1, \mathrm{an}},$ and $Z$ a connected affinoid neighborhood of $x$ such that the $Z$-thickening of all of the elements of $\mathcal{U}_x$ exist. Let us denote this set by $\mathcal{U}_Z.$ We will say it is a $Z$-\textit{thickening} of $\mathcal{U}_x.$
 The set $\mathcal{U}_Z$ will be said to be a $Z$-\textit{relative nice cover} of $\mathbb{P}_{Z}^{1, \mathrm{an}}$ if the statement of Theorem \ref{nicecover} is satisfied.
\end{defn}

\begin{rem} \label{105}
Whenever taking the thickening of a nice cover $\mathcal{U}_x$ of $\mathbb{P}_{\mathcal{H}(x)}^{1, \mathrm{an}}$ to obtain a $Z$-relative nice cover of $\mathbb{P}_Z^{1, \mathrm{an}}$ for a suitably chosen $Z$, we will suppose that a writing was fixed simultaneously for all of the points of $\bigcup_{U \in \mathcal{U}_x} \partial{U},$ and that constructions were made based on this ``compatible" writing of the boundary points (as we did \textit{e.g.} in Proposition ~\ref{nukkuptoj2} and Theorem \ref{nicecover}). The same principle goes for any family of affinoid domains of $\mathbb{P}_{\mathcal{H}(x)}^{1, \mathrm{an}}$ whose $Z$-thickenings we consider simultaneously.   
\end{rem}

We have shown:

\begin{thm} \label{106}
Let $\mathcal{U}_x$ be a nice cover of $\mathbb{P}_{\mathcal{H}(x)}^{1, \mathrm{an}}.$ There exists a connected affinoid neighborhood $Z$ of $x$ such that the $Z$-thickening of $\mathcal{U}_x$ exists and is a $Z$-relative nice cover of $\mathbb{P}_Z^{1, \mathrm{an}}$. The same is true for any other connected affinoid neighborhood $Z' \subseteq Z$ of $x.$
\end{thm}

\begin{cor} \label{107}

Let $U$ be a connected affinoid domain of $\mathbb{P}_{\mathcal{H}(x)}^{1, \mathrm{an}}$ containing only type ~3 points in its boundary. There exists an affinoid neighborhood $Z$ of $x$ in $S$ such that the $Z$-thickening $U_Z$ exists and is connected. The same is true for any connected affinoid neighborhood $Z' \subseteq Z$ of $x.$
\end{cor}

\begin{proof}
If $U$ is a type 3 point, then this is Theorem \ref{ajune}. Suppose this is not the case.
By \cite[Lemma 2.14]{une}, there exists a nice cover $\mathcal{U}_x$ of $\mathbb{P}_{\mathcal{H}(x)}^{1, \mathrm{an}}$ such that $U \in \mathcal{U}_x.$ Let $Z$ be a connected affinoid neighborhood of $x$ such that the $Z$-thickening $\mathcal{U}_Z$ exists and is a $Z$-relative nice cover. Then $U_Z \in \mathcal{U}_Z$ is connected. The last part of the statement is clear since the same property is true in Theorem \ref{nicecover}.
\end{proof}

\begin{rem} \label{108}
The notion of a $Z$-relative nice cover can be extended to connected affinoid domains of $\mathbb{P}_Z^{1, \mathrm{an}}$ provided the latter are $Z$-thickenings of affinoid domains of $\mathbb{P}_{\mathcal{H}(x)}^{1, \mathrm{an}}.$
\end{rem}

\section{A norm comparison} \label{4.2}

As seen in the previous section, when constructing relative nice covers we often have to restrict to smaller neighborhoods of the fiber. The same phenomenon appears when trying to apply the patching results of Section \ref{patching} to the setting introduced via Notation ~\ref{17aaa}. This is why we need some uniform-boundedness-type results.

Recall Notation \ref{17aaa}. Let $Z$ be any connected affinoid neighborhood of $x$ in $S.$ Set $A_Z=\mathcal{O}(Z).$ The $k$-algebra $A_Z$ is a $k$-affinoid algebra, and since $Z$ is connected  and reduced (recall $S$ is normal), $A_Z$ is an integral domain. By \cite[Proposition 9.13]{coanno}, the spectral norm $\rho_Z$ of $A_Z$ is equivalent to the norm of $A_Z,$ and it satisfies: for all ${f \in A_Z,} |f|_{\rho_Z}=\max_{y \in Z} |f|_y.$ In this section, for any connected affinoid neighborhood $Z$ of $x$ in $S$, we endow the corresponding affinoid algebra $A_Z$ with its spectral norm $\rho_Z.$ 

For any positive real number $r,$ we will use the notation $A_Z\{rT^{-1}\}$ (where $T$ is a fixed variable on $\mathbb{P}_Z^{1, \mathrm{an}}$) for the $A_Z$-affinoid algebra $\left\lbrace\sum_{n \geqslant 0} \frac{a_n}{T^n}: a_n \in A_Z, \lim_{n \rightarrow +\infty}|a_n|_{\rho_Z}r^{-n}=0\right\rbrace$ with corresponding submultiplicative norm $|\sum_{n \geqslant 0} \frac{a_n}{T^n}|:=\max_{n} |a_n|_{\rho_Z} r^{-n}.$

\begin{rem}
In what follows we suppose that the coefficient $r$ is not an element of $\sqrt{|k^\times|}.$ The only reason behind this assumption is to be able to guarantee the connectedness of the affinoid domains that are considered. If we assume connectedness, then the rest works the same regardless of whether $r \in \sqrt{|k^\times|}$. 
\end{rem}

\subsection{Affinoid domains defined through a polynomial}
We study here affinoid domains of the relative projective line $\mathbb{P}^{1, \mathrm{an}}$ which are defined through one (irreducible) polynomial. We will be particularly interested in finding explicit ``well-behaved" candidates for the respective (classes of equivalence of) norms they are endowed with.  

\subsubsection{The case of degree one polynomials} \label{degone} We will study here the following affinoid domains of the relative projective analytic line which are defined through a polynomial of degree one. Let $r \in \mathbb{R}_{>0} \backslash \sqrt{|\mathcal{H}(x)^\times|}$.

(1) Set $X_{|T|\leqslant r,Z}=\{u \in \mathbb{P}_Z^{1, \mathrm{an}}: |T|_u \leqslant r\}.$ It is an affinoid domain of $\mathbb{P}_Z^{1, \mathrm{an}}$, and $\mathcal{O}(X_{|T|\leqslant r,Z})=A_Z\{r^{-1}T\},$ where $$A_Z\{r^{-1}T\}=\{ \sum_{n \geqslant 0}a_n T^n, a_n \in A_Z, \lim_{n \rightarrow +\infty}|a_n|_{\rho_Z} r^n =0\}$$ and it is endowed with the norm $|\sum_{n \geqslant 0}  a_nT^n|_{|T| \leqslant r, Z}:=\max_{n \geqslant 0} |a_n|_{\rho_Z} r^n.$

(2) Set $X_{|T|\geqslant r,Z}=\{u \in \mathbb{P}_Z^{1, \mathrm{an}}: |T|_u \geqslant r\}.$ It is an affinoid domain of $\mathbb{P}_Z^{1, \mathrm{an}}$ and $\mathcal{O}(X_{|T|\geqslant r,Z})=A_Z\{rT^{-1}\},$ where $$A_Z\{rT^{-1}\}=\{\sum_{n \geqslant 0} \frac{a_n}{T^n}: a_n \in A_Z, \lim_{n \rightarrow +\infty}|a_n|_{\rho_Z} r^{-n}=0\}$$ and it is endowed with the norm $|\sum_{n \geqslant 0} a_nT^{-n}|_{|T| \geqslant r, Z}:= \max_{n \ge 0} |a_n|_{\rho_Z}r^{-n}.$

(3) Set $X_{|T|= r,Z}=\{u \in \mathbb{P}_Z^{1, \mathrm{an}}:|T|_u=r\}$. It is an affinoid domain of $\mathbb{P}_Z^{1, \mathrm{an}}$ and $\mathcal{O}(X_{|T|= r,Z})=A_Z\{r^{-1}T, rT^{-1}\},$ where $$A_Z\{r^{-1}T, rT^{-1}\}=\{\sum_{n \in \mathbb{Z}} a_n T^n : a_n \in A_Z, \lim_{n \rightarrow \pm\infty}|a_n|_{\rho_Z}r^n=0\}$$ and it is endowed with the norm $|\sum_{n \in \mathbb{Z}} a_nT^n|_{|T| = r, Z}:= \max_{n \in \mathbb{Z}} |a_n|_{\rho_Z}r^{n}.$

By Corollary \ref{107}, there exists a connected affinoid neighborhood $Z_T$ of $x$ in $S$ such that for any connected affinoid neighborhood $Z \subseteq Z_T$ of $x,$ the affinoids $X_{|T|\leqslant r,Z}, X_{|T|\geqslant r,Z}$ and $X_{|T|= r,Z}$ are connected (and hence integral). For the rest of this subsection, we suppose $Z \subseteq Z_T.$

\begin{lm} \label{109}
The norms $|\cdot|_{|T| \leqslant r,Z}, |\cdot|_{|T| \geqslant r,Z}, |\cdot|_{|T|=r,Z}$ defined above are equal to the spectral norms on $A_Z\{r^{-1}T\}, A_Z\{rT^{-1}\}, A_Z\{r^{-1}T, rT^{-1}\},$ respectively. 
\end{lm}

\begin{proof} \begin{sloppypar}
By \cite[Theorem 1.3.1]{Ber90}, for any affinoid space $X,$ its associated spectral norm $\rho_X$ has the property that $|f|_{\rho_X}=\max_{y \in X} |f|_y$ for all $f \in \mathcal{O}(X).$
Let $f=\sum_{n \geqslant 0} a_n T^n$ be any element of $A_Z\{r^{-1}T\}.$ Let $\rho_{|T| \leqslant r, Z}$ denote the spectral norm on the integral affinoid space $X_{|T| \leqslant r, Z}.$ We will show that $|f|_{|T| \leqslant r, Z}=|f|_{\rho_{|T| \leqslant r, Z}}.$ By \emph{loc.cit.}, ${|f|_{\rho_{|T| \leqslant r, Z}}=\max_{u \in X_{|T| \leqslant r, Z}} |f|_u}$. For any $y \in Z,$ the fiber of $X_{|T| \leqslant r, Z}$ over $y$ is the disc $\{u \in \mathbb{P}_{\mathcal{H}(y)}^{1, \mathrm{an}}: |T|_y \leqslant r\},$ whose Shilov boundary is the singleton $\{\eta_{0,r}^y\}$ (\textit{i.e.} the point $\eta_{0,r} \in \mathbb{P}_{\mathcal{H}(y)}^{1, \mathrm{an}}$).
 Consequently, in the fiber of $X_{|T| \leqslant r, Z}$ over $y,$ the function $f$ attains its maximum at the point $\eta_{0,r}^y,$ implying $|f|_{\rho_{|T| \leqslant r, Z}}=\max_{y \in Z} |f|_{\eta_{0,r}^y}$ (see also Lemma \ref{128}). 
\end{sloppypar}
 Since ${|f|_{\eta_{0,r}^y}=|\sum_{n \geqslant 0} a_n T^n|_{\eta_{0,r}^y}=\max_{n \geqslant 0} |a_n|_y r^n},$ we obtain that $$|f|_{\rho_{|T| \leqslant r, Z}}=\max_{y \in Z} \max_{n \geqslant 0} |a_n|_y r^n.$$ At the same time, ${|f|_{|T| \leqslant r, Z}=\max_{n \geqslant 0} |a_n|_{\rho_Z} r^n=\max_{n \geqslant 0} \max_{y \in Z} |a_n|_{y} r^n},$ implying the equality of the statement.

The result is proven in the same way for the norms $|\cdot|_{|T|\geqslant r, Z}$ and $|\cdot|_{|T|=r, Z}$.
\end{proof}

\subsubsection{The general case}
Let $P(T)$ be a unitary polynomial over $\mathcal{O}_x$ that is irreducible over~$\mathcal{H}(x)$. Then there exists an affinoid neighborhood $Z'$ of $x$ such that $P(T) \in \mathcal{O}(Z')[T].$ 

\begin{nota} \label{hajdede}
Let $r \in \mathbb{R}_{>0} \backslash \sqrt{|\mathcal{H}(x)^\times|}.$ Set $X_{|P|\leqslant r, Z}=\{u \in \mathbb{P}_Z^{1, \mathrm{an}}:|P|_u \leqslant r\},$ $X_{|P| \geqslant r, Z}=\{u \in \mathbb{P}_Z^{1, \mathrm{an}}:|P|_u \geqslant r\}$ and $X_{|P|=r, Z}=\{u \in \mathbb{P}_Z^{1, \mathrm{an}}:|P|_u = r\}.$ These are affinoid domains of $\mathbb{P}_Z^{1, \mathrm{an}}$ (furthermore, $X_{|P|\leqslant r, Z}$ and $X_{|P|=r, Z}$ are affinoid domains of~$\mathbb{A}_{Z}^{1, \mathrm{an}}$). By Corollary ~\ref{107}, there exists an affinoid neighborhood $Z_P$ of $x$ such that for any connected affinoid neighborhood $Z \subseteq Z_P,$ $X_{|P|\leqslant r, Z}, X_{|P|\geqslant r, Z}$ and $X_{|P|= r, Z}$ are connected (hence integral). For the rest of this subsection, we assume that $Z \subseteq Z' \cap Z_T \cap Z_P.$ 
\end{nota}
The rings $\mathcal{O}(X_{|P|\leqslant r, Z})$ and $\mathcal{O}(X_{|P|=r, Z})$ have been studied extensively and under more general conditions by Poineau in \cite[Chapter 5]{poilibri}. Restricted to our setting, the following is shown (see \cite[Proposition 5.3.3]{poilibri}):

\begin{lm} \label{113}
Let $Z$ be a connected affinoid neighborhood of $x$ such that ${Z \subseteq Z' \cap Z_T \cap Z_P}.$ Then $\mathcal{O}(X_{|P|\leqslant r, Z}) \cong\mathcal{O}(X_{|T| \leqslant r, Z})[X]/(P(X)-T)=A_Z\{r^{-1}T\}[X]/(P(X)-T)$ and $\mathcal{O}(X_{|P|=r,Z})=\mathcal{O}(X_{|T|=r, Z})[Y]/(P(Y)-T)=A_Z\{r^{-1}T, rT^{-1}\}[Y]/(P(Y)-T).$
\end{lm}

\begin{proof}
The statement can be seen by considering the finite morphism $\mathbb{P}_{Z}^{1, \mathrm{an}} \rightarrow \mathbb{P}_Z^{1, \mathrm{an}}$ induced by $A_Z[T] \rightarrow A_Z[T], T \mapsto P(T).$ 
\end{proof}

\begin{lm} \label{114}
Let $j_P$ denote the restriction morphism $\mathcal{O}(X_{|P| \leqslant r, Z}) \hookrightarrow \mathcal{O}(X_{|P| = r, Z})$ and similarly for $j_T.$ Then the following diagram commutes and $j_P(X)=Y.$

\[
\begin{tikzcd}
A_{Z}\{r^{-1}T\}[X]/(P(X)-T) \arrow{r}{j_P}  & A_Z\{r^{-1}T, rT^{-1}\}[Y]/(P(Y)-T)  \\
A_Z\{r^{-1}T\} \arrow{u} \arrow{r}{j_T} & A_{Z}\{r^{-1}T, rT^{-1}\} \arrow{u}
\end{tikzcd}
\]
\end{lm}

Taking this into account, we will from now on write $A_{Z}\{r^{-1}T\}[X]/(P(X)-T)$ and $A_{Z}\{r^{-1}T, rT^{-1}\}[X]/(P(X)-T)$ (\textit{i.e.} using the same variable $X$).
 
\begin{proof}
This follows again from the work of Poineau in \cite[Chapter 5]{poilibri}.  Remark that the finite morphism $A_{Z}[T] \rightarrow A_{Z}[T], T \mapsto P(T),$ induces a finite morphism ${\varphi:X_{|P| \leqslant r, Z} \rightarrow X_{|T| \leqslant r, Z}}$ and $\varphi^{-1}(X_{|T|=r, Z})=X_{|P|=r, Z}.$ The vertical maps of the diagram above are induced by $\varphi$, which implies its commutativity. Remark that ${j_T(T)=T.}$ Also, since $\varphi^{-1}(X_{|T|=r, Z})=X_{|P|=r, Z},$ we have that $$\mathcal{O}(X_{|P|=r, Z})=\mathcal{O}(X_{|P| \leqslant r, Z}) \otimes_{\mathcal{O}(X_{|T| \leqslant r, Z})} \mathcal{O}(X_{|T|=r, Z}).$$ The restriction morphism $j_P$ is given by $f \mapsto f \otimes 1,$ implying $j_P(X)=Y.$ 
\end{proof}

\begin{rem} \label{bowtie}
Recall that $\mathcal{O}(X_{|P|\leqslant r, Z}),$ $\mathcal{O}(X_{|P|\geqslant r, Z})$, and $\mathcal{O}(X_{|P|= r, Z})$ are affinoid algebras, meaning they are naturally endowed with submultiplicative norms $|\cdot|_{\leqslant}, |\cdot|_{\geqslant}$ and $|\cdot|_{=},$ respectively. (These norms are uniquely determined only up to equivalence.) 
\end{rem}

We start by giving an explicit choice for $|\cdot|_{\leqslant}$ and $|\cdot|_{=}.$ This was already done in Subsection \ref{degone} for the case $P(T)=T.$

The morphism $A_Z[T] \rightarrow A_Z[T], T \mapsto P(T),$ induces a finite morphism ${\varphi_Z:\mathbb{P}_{Z}^{1, \mathrm{an}} \rightarrow \mathbb{P}_{Z}^{1, \mathrm{an}}},$ for which $\varphi_Z^{-1}(X_{|T|\bowtie r,Z})=X_{|P| \bowtie r, Z},$ where $\bowtie \in \{ \leqslant, =, \geqslant\}.$ In particular, this gives rise to a finite morphism $X_{|P| \bowtie r, Z} \rightarrow X_{|T| \bowtie r, Z},$ hence to a finite morphism $\mathcal{O}(X_{|T|\bowtie r, Z}) \rightarrow \mathcal{O}(X_{|P| \bowtie r, Z}).$ The latter gives rise to a surjective morphism $\psi_1: \mathcal{O}(X_{|T| \bowtie r, Z})^n \twoheadrightarrow \mathcal{O}(X_{|P| \bowtie r, Z})$ for some $n \in \mathbb{N}.$ Let $|\cdot|_{\bowtie}'$ denote the norm (determined up to equivalence) on $\mathcal{O}(X_{|P| \bowtie r, Z})$ obtained by $\psi_1$, making $\psi_1$ admissible. We recall that the affinoid algebra $\mathcal{O}(X_{|T| \bowtie r, Z})$ is endowed with the norm defined on (1), (2) or (3) of Subsection ~\ref{degone}.

\begin{prop} \label{115} The norms $|\cdot|_{\bowtie}$ and  $|\cdot|_{\bowtie}'$ are equivalent for any $\bowtie \in \{\leqslant, =, \geqslant\}.$
\end{prop}

\begin{proof} We remark that the statement is trivial if $\deg  P(T)=1.$

By \cite[pg. 22]{Ber90}, there exists a complete non-trivially valued field extension $K$ of ~$k$ such that $\mathcal{O}(X_{|T|\bowtie r, Z}) \widehat{\otimes}_k K =: \mathcal{O}(X_{|T|\bowtie r, Z_K})$ and $\mathcal{O}(X_{|P|\bowtie r, Z}) \widehat{\otimes}_k K =: \mathcal{O}(X_{|P|\bowtie r, Z_K})$ are strict $K$-affinoid algebras, where $Z_K:=Z \times_k K.$ Moreover, we have the following commutative diagram
\[
\begin{tikzcd}
\mathcal{O}(X_{|T| \bowtie r, Z}) \arrow{r}{T \mapsto P(T)}\arrow{d}{\widehat{\otimes}_k K}  & \mathcal{O}(X_{|P| \bowtie r, Z})  \arrow{d}{\widehat{\otimes}_k K} \\
\mathcal{O}(X_{|T| \bowtie r, Z_K}) \arrow{r}{T \mapsto P(T)} & \mathcal{O}(X_{|P| \bowtie r, Z_K})
\end{tikzcd}
\]
which gives rise to the following commutative diagram, where $\psi_2$ is a surjective admissible morphism induced by $\psi_1$:

\[
\begin{tikzcd}
\mathcal{O}(X_{|T| \bowtie r, Z}) \arrow{r}\arrow{d}{\widehat{\otimes}_k K}  & \mathcal{O}(X_{|T| \bowtie r, Z})^n \arrow{d}{\widehat{\otimes}_k K} \arrow{r}{\psi_1} &  \mathcal{O}(X_{|P| \bowtie r, Z})  \arrow{d}{\widehat{\otimes}_k K} \\
\mathcal{O}(X_{|T| \bowtie r, Z_K}) \arrow{r} & \mathcal{O}(X_{|T| \bowtie r, Z_K})^n \arrow{r}{\psi_2}
 & \mathcal{O}(X_{|P| \bowtie r, Z_K})
\end{tikzcd}
\]
Let $|\cdot|_{\bowtie,K}'$ be the norm (determined up to equivalence) on $\mathcal{O}(X_{|P| \bowtie r, Z_K})$ induced by the morphism $\psi_2.$ Then $\mathcal{O}(X_{|P| \bowtie r, Z_K})$ is a Banach $K$-algebra with respect to $|\cdot|_{\bowtie,K}'.$

Since $\mathcal{O}(|X|_{|T| \bowtie r, Z}) \hookrightarrow \mathcal{O}(X_{|T|\bowtie r, Z_K})$ is an isometry (see \cite[Lemme 3.1]{angepoi}), the diagram above implies that $(\mathcal{O}(|X|_{|P| \bowtie r, Z}), |\cdot|'_{\bowtie}) \hookrightarrow (\mathcal{O}(X_{|P|\bowtie r, Z_K}), |\cdot|_{\bowtie,K}')$ is also an isometry.

Let $|\cdot|_{\bowtie, K}$ denote the norm that the $K$-affinoid algebra $\mathcal{O}(X_{|P| \bowtie r, Z_K})$ is naturally endowed with. Then $(\mathcal{O}(|X|_{|P| \bowtie r, Z}), |\cdot|_{\bowtie}) \hookrightarrow (\mathcal{O}(X_{|P|\bowtie r, Z_K}), |\cdot|_{\bowtie, K})$ is an isometry (again, see \cite[Lemme 3.1]{angepoi}).

Since $\mathcal{O}(X_{|P| \bowtie r, Z_K})$ is a strict $K$-affinoid algebra, by \cite[6.1.3/2]{bo}, there is a unique way to define the structure of a Banach $K$-algebra on it. Hence, $|\cdot|_{\bowtie,K}'$ is equivalent to $|\cdot|_{\bowtie, K}$, so the norms $|\cdot|'_{\bowtie},$ resp. $|\cdot|_{\bowtie},$ they induce on $\mathcal{O}(X_{|P|\bowtie r, Z})$ are equivalent.  
\end{proof}

\begin{nota} \label{star}
Set $d=\deg{P}.$ Since $P(X)$ is unitary, any $f \in A_Z\{r^{-1}T\}[X]/(P(X)-T)$ (resp. $f \in A_Z\{r^{-1}T, rT^{-1}\}[X]/(P(X)-T)$) has a unique representation of the form $\sum_{i=0}^{d-1} \alpha_i X^i,$ where $\alpha_i \in A_Z\{r^{-1}T\}$ (resp. $\alpha_i \in A_Z\{r^{-1}T, rT^{-1}\}$) for all ${i=0,1,\dots, d-1}.$ Set $|f|_{|P| \leqslant r, Z}:=\max_{i}(|\alpha_i|_{|T| \leqslant r, Z})$ (resp. $|f|_{|P|=r, Z}:=\max_{i}(|\alpha_i|_{|T|=r, Z})$). Recall Remark ~\ref{bowtie}. By Proposition ~\ref{115}, we can take $|\cdot|_{\leqslant}=|\cdot|_{|P| \leqslant r, Z}$ and $|\cdot|_{=}=|\cdot|_{|P| = r, Z}.$ (This kind of norm is called $||\cdot||_{U, \mathrm{div}}$ in \cite[5.2]{poilibri}; here $U$ is $X_{|T| \leqslant r, Z}$ or $X_{|T| = r, Z}$.) 
\end{nota}
It remains to find a good representative for $\mathcal{O}(X_{|P| \geqslant r,Z})$ and its norm. This was done in (2) of Subsection \ref{degone} for the case $P(T)=T$. 

In what follows, we identify the $k$-affinoid algebras $\mathcal{O}(X_{|P| \leqslant r, Z})$ and $\mathcal{O}(X_{|P| \geqslant r, Z})$ with $A_Z$-subalgebras of $\mathcal{O}(X_{|P| = r, Z})$ via the respective restriction morphisms. Since $H^1(X_{|P| \leqslant r, Z} \cup X_{|P|\geqslant r, Z}, \mathcal{O})=H^{1}(\mathbb{P}_Z^{1, \mathrm{an}}, \mathcal{O})=0$ (see \mbox{\cite[Th\'eor\`eme A.1 i)]{poi1})}, we have the following short exact sequence:

\reqnomode
\begin{align} \label{sequence} 0 \rightarrow A_Z \rightarrow \mathcal{O}(X_{|P| \leqslant r, Z}) \oplus \mathcal{O}(X_{|P| \geqslant r, Z}) \rightarrow \mathcal{O}(X_{|P|=r, Z}) \rightarrow 0.
\end{align} 

\

Let $f \in \mathcal{O}(X_{|P|=r, Z})=A_Z\{r^{-1}T, rT^{-1}\}[X]/(P(X)-T).$ Suppose its unique representative of degree $<d$ in the indeterminate $X$ is $f_0=\sum_{i=0}^{d-1} \sum_{n \in \mathbb{Z}} a_{n,i} T^n X^i,$ where $\sum_{n \in \mathbb{Z}} a_{n,i} T^n \in A_Z\{r^{-1}T, rT^{-1}\}$ for all $i.$ 
Then we can write the following decomposition for $f_0$: 
$$
f_0=a_{0,0} + \underbrace{\left(\sum_{n \geqslant 1}a_{n,0}T^n + \sum_{i=1}^{d-1}\sum_{n \geqslant 0} a_{n,i}T^n X^i\right)}_{\alpha_f} + \underbrace{\left(\sum_{i=0}^{d-1}\sum_{n \leqslant -1} a_{n,i}T^n X^i\right)}_{\beta_f}.$$
Remark that $\alpha_f \in A_Z\{r^{-1}T\}[X]/(P(X)-T).$

\begin{prop} \label{116}
The $A_Z$-subalgebra $\mathcal{O}(X_{|P| \geqslant r, Z})$ of $\mathcal{O}(X_{|P|=r, Z})$ is equal to
$$B:=\left\lbrace f \in A_Z\{r^{-1}T, rT^{-1}\}[X]/(P(X)-T): f=a_{0,0}+ \sum_{i=0}^{d-1}\sum_{n \geqslant 1} \frac{a_{n,i}}{T^n}X^i \right\rbrace.$$
\end{prop}

\begin{proof}
Let us first show that $B$ is closed with respect to multiplication. Let $f=a_{0,0}+\sum_{i=0}^{d-1}\sum_{n \geqslant 1}\frac{a_{n,i}}{T^n}X^i , g =b_{0,0}+\sum_{i=0}^{d-1}\sum_{n \geqslant 1}\frac{b_{n,i}}{T^n}X^i \in B.$ For any $m$ such that  $d \leqslant m <2d,$ the coefficient corresponding to $X^m$ in the product $fg$ is of the form $\sum_{n \geqslant 2}\frac{c_{n,m}}{T^n}$ where $c_{n,m} \in A_Z$ for all $n,m.$ By using Euclidean division, since $P(X)$ is unitary, we obtain $X^m=P(X)Q(X)+R(X)$ where $Q, R \in A_Z[X],$ $\deg{R}<d$ and $\deg{Q}=m-d<d.$ Hence, $\sum_{n \geqslant 2} \frac{c_{n,m}}{T^n}X^m=\sum_{n \geqslant 2}\frac{c_{n,m}}{T^n}P(X)Q(X)+ \sum_{n \geqslant 2}\frac{c_{n,m}}{T^n}R(X)=\sum_{n \geqslant 1}\frac{c_{n,m}}{T^n}Q(X)+ \sum_{n \geqslant 2}\frac{c_{n,m}}{T^n}R(X)$ in $A_Z\{r^{-1}T, rT^{-1}\}[X]/(P(X)-T),$ which is an element of $B$ seeing as $\deg{Q}, \deg{R}<d.$ Consequently, $fg \in B,$ and $B$ is an $A_Z$-algebra.  

Let us consider the restriction morphism $\psi: A_Z=\mathcal{O}(\mathbb{P}_Z^{1, \mathrm{an}}) \rightarrow \mathcal{O}(X_{|P|\geqslant r}, Z),$ a section of which is given as follows: for any $f \in \mathcal{O}(X_{|P|\geqslant r}, Z),$ let $f_{\infty}$ denote the restriction of $f$ to the Zariski closed subset $\mathcal{Z}:=\{x \in X_{|P| \geqslant r, Z}: |T^{-1}|_x=0\}.$ Remark that in the copy of $\mathbb{A}_Z^{1, \mathrm{an}}$ in $\mathbb{P}_Z^{1, \mathrm{an}}$ with coordinate $T^{-1},$ $\mathcal{Z}=\{u \in \mathbb{A}_{Z}^{1, \mathrm{an}}: |T^{-1}|_u=0\}$, so $\mathcal{O}(\mathcal{Z})=A_Z.$ 

The morphism $s: \mathcal{O}(X_{|P|\geqslant r}, Z) \rightarrow A_Z, f \mapsto f_{\infty},$ is a section of ~$\psi.$ Let $\mathcal{O}(X_{|P|\geqslant r,Z})_{\infty}$ denote the kernel of $s.$ Then $\mathcal{O}(X_{|P|\geqslant r,Z})=A_Z \oplus \mathcal{O}(X_{|P|\geqslant r}, Z)_{\infty}.$

Let us consider the following commutative diagram that is obtained from the short exact sequence (\ref{sequence}) above.
\begin{center}
\begin{tikzcd}
 & \mathcal{O}(X_{|P| \leqslant r, Z}) \oplus \mathcal{O}(X_{|P| \geqslant r, Z}) \arrow{dr}{h''} \\
\mathcal{O}(X_{|P| \leqslant r, Z}) \oplus \mathcal{O}(X_{|P| \geqslant r, Z})_{\infty} \arrow{ur}{h'} \arrow{rr}{h} && \mathcal{O}(X_{|P|=r, Z})
\end{tikzcd}
\end{center}

Let us show $h$ is bijective. Let $f \in \mathcal{O}(X_{|P|=r, Z}).$ By the surjectivity of $h''$ (from the short exact sequence (\ref{sequence})) there exist $f_1 \in \mathcal{O}(X_{|P| \leqslant r, Z})$ and $f_2 \in \mathcal{O}(X_{|P| \geqslant r, Z})$ such that $f_1+f_2=f.$ Let $f_2' \in A_Z$ and $f_2'' \in \mathcal{O}(X_{|P| \geqslant r, Z})_{\infty}$ be such that $f_2=f_2'+f_2''$ (as we saw above, such $f_2', f_2''$ are unique). Set $f_1':=f_1+f_2'$ and remark that $f_1' \in \mathcal{O}(X_{|P| \leqslant r, Z}).$ By the commutativity of the diagram, $h(f_1', f_2'')=f,$ \textit{i.e.} $h$ is surjective. Let us also show it is injective. Suppose $h(a,b)=0$ for some $a \in \mathcal{O}(X_{|P| \leqslant r, Z})$ and $b \in \mathcal{O}(X_{|P| \geqslant r, Z})_{\infty} \subseteq \mathcal{O}(X_{|P| \geqslant r, Z}).$ Then $a+b=h''(a,b)=0,$ and the exact sequence (\ref{sequence}) implies that ${a=-b \in A_Z}.$ Since $b \in A_Z$ and $b \in \mathcal{O}(X_{|P| \geqslant r, Z})_{\infty},$ we obtain that $b=0$ and $a=0,$ \textit{i.e.} $h$ is injective.  

By Lemma \ref{114}, the map $s': \mathcal{O}(X_{|P| = r, Z}) \rightarrow \mathcal{O}(X_{|P| \leqslant r, Z}),$ which to an element $f_0:= \sum_{i=0}^{d-1} \sum_{n \in \mathbb{Z}} d_{n, i} T^n X^i$ associates the element $f_{\leqslant} :=\sum_{i=0}^{d-1}\sum_{n \leqslant -1} d_{n,i}T^n X^i$, is a section of the isomorphism $\mathcal{O}(X_{|P| \leqslant r, Z}) \oplus \mathcal{O}(X_{|P| \geqslant r, Z})_{\infty}   
\rightarrow \mathcal{O}(X_{|P| = r, Z}).$ Consequently, $\mathcal{O}(X_{|P| \geqslant r, Z})_{\infty}=\left\lbrace f \in \mathcal{O}(X_{|P| =r, Z}): f=\sum_{i=0}^{d-1}\sum_{n \leqslant -1} a_{n,i}T^nX^i \right\rbrace.$

Finally, since $\mathcal{O}(X_{|P| \geqslant r, Z})=A_Z \oplus \mathcal{O}(X_{|P| \geqslant r, Z})_{\infty},$ we get: $\mathcal{O}(X_{|P| \geqslant r, Z})=B.$
\end{proof}

\begin{rem} \label{117}
Let $I$ be the ideal of $A_{Z}\{rT^{-1}\}$ generated by $T^{-1}.$ Denote by $I[X]^{d-1}$ the polynomials in $X$ with coefficients in $I$ and degree at most $d-1.$ Then the $k$-affinoid algebra $B$ can be written as $(A_Z\oplus I[X]^{d-1})/(P(X)T^{-1}-1),$ where multiplication is done using Euclidean division, just like in $B.$ 
\end{rem}

\begin{nota} \label{startstar}
The morphism $A_Z\{rT^{-1}\} \rightarrow B,$ $T^{-1} \mapsto \frac{1}{T}$ is finite (it is the one induced by $A_Z[T] \rightarrow A_Z[T], T \mapsto P(T)$), and $1,X, \dots, X^{d-1}$ is a set of generators of $B$ as an $A_Z$-module. Let $|\cdot|_{|P|\geqslant r, Z}$ be the norm on $B$ induced by the norm $|\cdot|_{|T|\geqslant r, Z}$ on $A_Z\{rT^{-1}\}.$ By~\mbox{\cite[Proposition 2.1.12]{Ber90}}, $B$ is complete with respect to this norm. As before, by Proposition ~\ref{115}, we can take $|\cdot|_{\geqslant}:=|\cdot|_{|P| \geqslant r, Z}.$ Explicitely, for any ${f:=a_{0,0} + \sum_{i=0}^{d-1}\sum_{n \geqslant 1}\frac{a_{n,i}}{T^n}X^i=\sum_{i=0}^{d-1}\alpha_i X^i \in B},$ where $\alpha_i \in A_Z\{r T^{-1}\},$
 $$|f|_{|P|\geqslant r, Z}=\max_{i} |\alpha_i|_{|T| \geqslant r, Z}.$$
\end{nota}
\begin{lm} \label{118}
\begin{sloppypar} The restriction maps  from $A_Z\{r^{-1}T\}[X]/(P(X)-T)$ and $B$ to $A_Z\{r^{-1}T, rT^{-1}\}[X]/(P(X)-T)$ are isometries with respect to the corresponding norms $|\cdot|_{|P| \leqslant r, Z},$ $|\cdot|_{|P| \geqslant r, Z}$ and $|\cdot|_{|P| = r, Z}.$
\end{sloppypar}
\end{lm}

\begin{proof}
Let us first suppose that $P(T)=T$. Let $f = \sum_{n \geqslant 0} a_n T^n \in A_Z\{r^{-1}T\}.$ Then $|f|_{|T|=r, Z}=\max_{n}|a_n|_{\rho_Z} r^n=|f|_{|T| \leqslant r, Z},$ where $|\cdot|_{\rho_Z}$ is the spectral norm on $A_Z$. The same is true for $A\{rT^{-1}\},$ so the statement is satisfied in this special case. 

For the general case, let $f=\sum_{i=0}^{d-1}\sum_{n \geqslant 0} a_{n,i} T^n X^i \in A_Z\{r^{-1}T\}[X]/(P(X)-T).$ Then $|f|_{|P|= r,Z}=\max_{i} |\sum_{n\geqslant 0} a_{n,i}T^n|_{|T| = r, Z}=\max_{i} |\sum_{n \geqslant 0} a_{n,i}T^n|_{|T| \leqslant r, Z}=|f|_{|P| \leqslant r, Z}.$ The statement for $B$ is proven in the same way.
\end{proof}

\begin{rem} \label{edhenje}
The short exact sequence (\ref{sequence}) above (page 23) gives rise to an admissible surjection $$A_{Z}\{r^{-1}T\}[X]/(P(X)-T) \oplus B \twoheadrightarrow A_{Z}\{r^{-1}T, rT^{-1}\}[X]/(P(X)-T).$$ Admissibility follows from Banach's Open Mapping Theorem if $k$ is non-trivially valued (for a proof see \cite{bou}), and by a change of basis followed by the Open Mapping Theorem if it is (see \cite[Chapter 2, Proposition ~2.1.2(ii)]{Ber90}).
\end{rem}

\begin{lm} \label{119} \begin{sloppypar}
For any $c \in A_{Z}\{r^{-1}T, rT^{-1}\}[X]/(P(X)-T),$ there exist $a \in A_{Z}\{r^{-1}T\}[X]/(P(X)-T)$ and $b \in B$ such that $a+b=c$ and $$\max(|a|_{|P| \leqslant r, Z}, |b|_{|P| \geqslant r, Z})=|c|_{|P|=r, Z}.$$
\end{sloppypar}
\end{lm}

\begin{proof}
\begin{sloppypar}
There exists a unique degree $<d$ polynomial $c_0(X)$ over $A_Z\{r^{-1}{T}, rT^{-1}\}$ such that $c=c_0$ in $A_{Z}\{r^{-1}T, rT^{-1}\}[X]/(P(X)-T).$ Let $c_0=\sum_{i=0}^{d-1} \sum_{n \in \mathbb{Z}}a_{n, i} T^n X^i.$ Let $a$ and $b$ be given as follows:
$$c_0=\underbrace{ \left(\sum_{i=0}^{d-1}\sum_{n \geqslant 0} a_{n,i}T^n X^i\right)}_{a} + \underbrace{\left(\sum_{i=0}^{d-1}\sum_{n \leqslant -1} a_{n,i}T^n X^i\right)}_{b}.$$
Clearly, $a \in A_{Z}\{r^{-1}T\}[X]/(P(X)-T)$ and $b \in B.$

Then $$|a|_{|P| \leqslant r, Z}=\max_{i} |\sum_{n \geqslant 0} a_{n,i} T^n|_{|T| \leqslant r,Z}=\max_{i} \max_{n \in \mathbb{N}} |a_{n,i}|_{\rho_Z} r^n$$$$ \leqslant \max_{i} \max_{n \in \mathbb{Z}} |a_{n,i}|_{\rho_Z} r^n=|c|_{|P|=r, Z},$$ 
and the same is true for $|b|_{|P|\geqslant r, Z}.$ Consequently, $\max(|a|_{|P| \leqslant r, Z}, |b|_{|P| \geqslant r, Z}) \leqslant |c|_{|P|=r, Z}.$

On the other hand, $c=a+b,$ so $|c|_{|P|=r, Z} \leqslant \max(|a|_{|P| = r, Z}, |b|_{|P| = r, Z}),$ which, by Lemma ~\ref{118}, is the same as $\max(|a|_{|P| \leqslant r, Z}, |b|_{|P| \geqslant r, Z}).$
\end{sloppypar}
\end{proof}

\begin{lm} \label{120}
Let $Z_1 \subseteq Z$ be a connected affinoid neighborhood of $x.$ 
 The restriction morphism $\mathcal{O}(X_{|P|=r, Z}) \hookrightarrow \mathcal{O}(X_{|P|=r, Z_1})$ is a contraction with respect to the corresponding norms $|\cdot|_{|P|=r, Z}$ and $|\cdot|_{|P|=r, Z_1}.$
\end{lm}

\begin{proof}
We remark that the statement is immediate from Lemma \ref{109} if $P(T)=T.$

Let the restriction morphism $\mathcal{O}(X_{|P|=r, Z}) \hookrightarrow \mathcal{O}(X_{|P|=r, Z_1})$ be denoted by $j_{P,1}.$ Similarly to Lemma \ref{114}, the following diagram is commutative and $j_{P,1}(X)=Y$ (remark that $j_{T,1}(T)=T, j_{T,1}(T^{-1})=T^{-1}$, and the  restriction of $j_{T,1}$ to $A_Z$ is the restriction morphism $A_Z \rightarrow A_{Z_1}$). 

\[
\begin{tikzcd}
A_{Z}\{r^{-1}T, rT^{-1}\}[X]/(P(X)-T) \arrow{r}{j_{P,1}}  & A_{Z_1}\{r^{-1}T, rT^{-1}\}[Y]/(P(Y)-T)  \\
A_Z\{r^{-1}T, rT^{-1}\} \arrow{u} \arrow{r}{j_{T,1}} & A_{Z_1}\{r^{-1}T, rT^{-1}\} \arrow{u}
\end{tikzcd}
\]

Let $f=\sum_{i=0}^{d-1} \sum_{n \in \mathbb{Z}} a_{n,i} T^n X^i \in \mathcal{O}(X_{|P|=r, Z})=A_Z\{r^{-1}T, rT^{-1}\}[X]/(P(X)-T).$ Then $|f|_{|P|=r, Z_1}=\max_{i} \max_{n} |a_{n, i}|_{\rho_{{Z_1}}} r^n.$ Since $A_{Z}$ and $A_{Z_1}$ are equipped with their respective spectral norms, $|a_{n,i}|_{\rho_{Z_1}} \leqslant |a_{n,i}|_{\rho_Z},$ implying $|f|_{|P|=r, Z_1} \leqslant \max_{i} \max_{n} |a_{n, i}|_{\rho_{Z}} r^n=|f|_{|P|=r, Z}.$
\end{proof}

\begin{rem} \label{122}
By applying the above to the case when $S$ is a point (\textit{i.e.} if everything is defined over a complete ultrametric field), it makes sense to speak of the affinoid domains $X_{|P| \bowtie r, x}$ of $\mathbb{P}_{\mathcal{H}(x)}^{1, \mathrm{an}}$  and their norms $|\cdot|_{|P| \bowtie r, x}$ for $\bowtie \in \{ \leqslant, =, \geqslant\},$ which satisfy all of the properties we have proven so far in this section.

\begin{sloppypar}
Furthermore, if $P$ is a unitary polynomial of degree $d\geqslant 1$ over $A_Z$ that is irreducible over $\mathcal{H}(x),$ then there exists a ``restriction morphism" ${(\mathcal{O}(X_{|P| \bowtie r, Z}), |\cdot|_{|P| \bowtie r, Z}) \rightarrow (\mathcal{O}(X_{|P| \bowtie r, x}), |\cdot|_{|P| \bowtie r, x})}$ on the fiber (corresponding to base change), which is a contraction. To see this, let $f = \sum_{i=0}^{d-1}\sum_{n \in \mathbb{Z}} a_{n,i} T^n X^i \in \mathcal{O}(X_{|P| \bowtie r, Z})$ (with certain $a_{n,i}$ possibly $0$ depending on what $\bowtie$ is). Then $|f|_{|P| \bowtie r, x}=\max_{i} \max_{n}|a_{n,i}|_{x}r^n \leqslant \max_{i} \max_{n}|a_{n,i}|_{\rho_Z}r^n = |f|_{|P| \bowtie r, Z}.$
\end{sloppypar}
\end{rem}

\subsection{The explicit norm comparison}
The following is mainly a special case of \cite[5.2]{poilibri} (or a rather direct consequence thereof), which we summarize here with an emphasis on the results that interest us.
\begin{sloppypar}
Let $P$ be a unitary polynomial of degree $d>0$ over $\mathcal{O}_x$ that is irreducible over $\mathcal{H}(x).$ Also, let $r \in \mathbb{R}_{>0} \backslash \sqrt{|\mathcal{H}(x)^{\times}|}.$
As in Notation \ref{hajdede}, let $Z$ be any connected affinoid neighborhood of ~$x$ contained in $Z' \cap Z_T \cap Z_P.$ 

For $t \in \{  x, Z\}$ (we understand here that $t$ can be $x$ or any connected affinoid neighborhood of $x$ contained in $Z' \cap Z_T \cap Z_P$), let $(R_t, |\cdot|_{r,t})$ be $(A_Z\{r^{-1}T, rT^{-1}\}, |\cdot|_{|T|=r, Z})$ if $t=Z$ and $(\mathcal{H}(x)\{r^{-1}T, rT^{-1}\},|\cdot|_{|T|=r, x})$ otherwise. Remark that $(R_t, |\cdot|_{r,t})$ is an affinoid algebra over $A_Z$ if $t=Z$ and over $\mathcal{H}(x)$ if $t=x.$ As mentioned in Remark ~\ref{122}, there is a contraction $R_Z \hookrightarrow R_x$ induced from the restriction $A_Z \hookrightarrow \mathcal{O}_x \hookrightarrow \mathcal{H}(x).$

For any $s \in \mathbb{R}_{>0},$ let $|\cdot|_{t,s}$ denote the norm on $R_t[X]$ induced from the $R_t$-affinoid algebra $R_t\{s^{-1}X\}$. Let $|\cdot|_{t, s, \mathrm{res}}$ denote the residue norm on $R_t[X]/(P(X)-T)$ induced by $|\cdot|_{t,s}.$
\end{sloppypar}
\begin{lm} \label{123}
For any $t \in \{x, Z\},$ there exists $v_t' >0$, such that for any $s \geqslant v_t',$ the norm $|\cdot|_{t, s, res}$ is equivalent to $|\cdot|_{|P|=r,t}$. Explicitly, for any $f \in R_t[X]/(P(X)-T),$
$$|f|_{t,s,res} \leqslant |f|_{|P|=r, t} \leqslant C_t \max_{1 \leqslant i \leqslant d-1} (s^{-i}) |f|_{t,s,res},$$
where $C_t=\max(2, 2v_t'^{-d}).$

Fix a connected affinoid neighborhood $Z_0 \subseteq Z'\cap Z_T \cap Z_P$ of $x$. There exist $v', C'>0$ such that the statement is true for any $s \geqslant v'$ and any $t \in \{x, Z: Z \subseteq Z_0\}.$
\end{lm}

\begin{proof}
For the first part of the statement, see \cite[Lemme 5.2.3]{poilibri}. The norm $|\cdot|_{t,s,res}$ is the analogue of what in \textit{loc.cit.} is denoted by $|\cdot|_{U,w,res}$ (here $U$ is $X_{|T|=r, t}$ and $s=w$).

To see the last part of the statement, let us describe $v_t'$ explicitely. Let ${\alpha_0, \dots, \alpha_{d-1} \in A_Z}$ be the coefficients of $P,$ and $\beta_0, \dots, \beta_{d-1} \in A_Z[T] \subseteq R_Z$ the coefficients of $P(X)-T$ (\textit{i.e.} $\beta_0=\alpha_0-T, \beta_i=\alpha_i$ for $1 \leqslant i \leqslant d-1$). By the proof of Th\'eor\`eme 5.2.1 of \cite{poilibri}, we only require that $v_t'>0$ satisfy  $\sum_{i=0}^{d-1}|\beta_i|_{r,t} v_t' \leqslant \frac{1}{2}$. Set $v':=v_{Z_0}'.$ Then $\sum_{i=0}^{d-1}|\beta_i|_{r,Z_0} v' \leqslant \frac{1}{2}$. By Lemma ~\ref{120} and Remark \ref{122}, $|\beta_i|_{r, t} \leqslant |\beta_i|_{r,Z_0}$ for any $t \in \{x, Z: Z \subseteq Z_0\}$, so $$\sum_{i=0}^{d-1}|\beta_i|_{r,t} v'  \leqslant \sum_{i=0}^{d-1}|\beta_i|_{r,Z_0} v' \leqslant \frac{1}{2}.$$
Set $C'=\max(2, 2v'^{-d}).$ The statement is true with this choice of $v'$ and $C'.$
\end{proof}

\begin{thm} \label{124} Let $Z_0$ be as in Lemma \ref{123}. There exist $m,s,C'>0$ such that for any $ t \in \{x, Z: Z \subseteq Z_0\}$ and any $f \in R_t[X]/(P(X)-T)$:
$$|f|_{\rho_{|P|=r,t}} \leqslant |f|_{|P|=r,t} \leqslant C' \max_{1\leqslant i \leqslant d-1} (s^{-i}) \frac{d^2 (2s)^{d^2-d}}{m}|f|_{\rho_{|P|=r, t}},$$
where $\rho_{|P|=r,t}$ is the spectral norm on $R_t[X]/(P(X)-T)=\mathcal{O}(X_{|P|=r, t}).$  
\end{thm}

\begin{proof}
The first inequality is immediate from the definition of the spectral norm. 

By the previous lemma, there exist $v'>0$ and $C'>0$ such that for any $s \geqslant v'$ and any $t \in \{x, Z: Z \subseteq Z_0\}$, $|\cdot|_{|P|=r,t} \leqslant C'\max_{1 \leqslant i \leqslant d-1} (s^{-i}) |\cdot|_{t,s,res}.$ Thus, it suffices to compare the norm $|\cdot|_{t,s,res}$ to the spectral one. For a fixed $t,$ this is done in \cite[Proposition 5.2.7]{poilibri} as follows.
\begin{sloppypar}
Let $\mathrm{Res}(\cdot, \cdot)$ denote the resultant of two polynomials (we assume the ambient ring is unambiguously determined). 
Let us show that $\mathrm{Res}(P(X)-T, P'(X)) \neq 0$ in $A_{Z_0}[T].$ Otherwise, the polynomials $P(X)-T$ and $P'(X)$ would have a common divisor of positive degree, \textit{i.e.} there would exist $Q,R,R_1 \in A_{Z_0}[T][X],$ with $\deg_{X}{Q}>0$ such that ${P(X)-T=Q(X,T)R(X,T)}$ and ${P'(X)=Q(T,X)R_1(T,X)}.$ The second expression implies that the degree in $T$ of $Q$ and $R_1$ is $0,$ meaning $Q, R_1 \in A_{Z_0}[X].$ Consequently, ${P(X)-T=Q(X)R(X,T)}$, which is impossible if ${\deg_X{Q}>0}.$ Finally, this means that ${\mathrm{Res}(P(X)-T, P'(X)) \neq 0}$ in $A_{Z_0}[T].$ 
As the resultant does not depend on the ring in which it is computed, ${\mathrm{Res}(P(X)-T, P'(X)) \neq 0}$ in $R_{t}$, so  ${|\mathrm{Res}(P(X)-T, P'(X))|_{r,t} \neq 0}$ for any $t.$ 
\end{sloppypar}
\begin{sloppypar}
Let ${\alpha_0, \beta_1, \dots, \beta_{d-1} \in A_{Z_0}}$ be the coefficients of $P(X)$, and $${\beta_0:=\alpha_0-T, \beta_1, \dots, \beta_{d-1} \in A_{Z_0}[T] \subseteq R_{Z_0}}$$ the coefficients of $P(X)-T$. Set $v''_t:=\max_{1 \leqslant i \leqslant d-1}(|\beta_i|_{r,t}^{\frac{1}{d-i}}).$ Set ${v_t=\max(v', v''_t)}.$ Let $m_t>0$ be such that $|\mathrm{Res}(P(X)-T, P'(X))|_{r,t} > m_t$ (such an $m_t$ exists by the paragraph above). 
\end{sloppypar}
Let $s>v_t.$ Then for any $f \in R_t[X]/(P(X)-T)$ (see \cite[Proposition 5.2.7]{poilibri}): 
$$|f|_{t,s,res}\leqslant \frac{d^2(2s)^{d^2-d}}{m_t} |f|_{\rho_{|P|=r,t}}.$$  
By Lemma \ref{120} and Remark \ref{122}, for any $t \in \{x, Z: Z \subseteq Z_0\},$ $v_t'' \leqslant v_{Z_0}''.$ Set ${v=\max(v', v_{Z_0}'')},$ so that for any $t,$ $v_t \leqslant v.$ 

Set $m=m_{x}.$ Note that for any $t$, $$0<m<|\mathrm{Res}(P(X)-T, P'(X))|_{r,x} \leqslant |\mathrm{Res}(P(X)-T, P'(X))|_{r,t}.$$ Consequently, for any $t \in \{x, Z: Z \subseteq Z_0\}$ and any $s \geqslant v,$ $$|f|_{t,s,res}\leqslant \frac{d^2(2s)^{d^2-d}}{m} |f|_{\rho_{|P|=r,t}}.$$
From Lemma \ref{123}, $|f|_{|P|=r, t} \leqslant C'\max_{1 \leqslant i \leqslant d-1} (s^{-i}) |f|_{t,s,res}$ for all $t,$ so finally 
$$|f|_{|P|=r,t} \leqslant C' \max_{1\leqslant i \leqslant d-1} (s^{-i}) \frac{d^2 (2s)^{d^2-d}}{m}|f|_{\rho_{|P|=r, t}},$$
for all $f \in R_t[X]/(P(X)-T)$ and all $t \in \{x, Z: Z\subseteq Z_0\}.$
\end{proof}

\begin{rem} \label{125}
The previous theorem gives an explicit comparison between the norms $|\cdot|_{|P|=r,t}$ and $\rho_{|P|=r,t}$ with a constant that is valid for all $t \in \{x, Z: Z \subseteq Z_0\}.$ By Lemma ~\ref{109}, in the case $P(T)=T$, the two coincide.
\end{rem}

Set $C=C' \max_{1\leqslant i \leqslant d-1} (s^{-i}) \frac{d^2 (2s)^{d^2-d}}{m}.$ We have shown the following:

\begin{cor} \label{126}
Let $P(T)$ be a unitary polynomial in $\mathcal{O}_x[T]$ irreducible over $\mathcal{H}(x)$ and $r \in \mathbb{R}_{>0} \backslash \sqrt{|\mathcal{H}(x)^{\times}|}.$ There exists a connected affinoid neighborhood $Z_0$ of $x$ in $S$ such that for any $t \in \{x, Z: Z \subseteq Z_0 \ \text{is a connected affinoid neighborhood of} \ x\},$
$$|\cdot|_{\rho_{|P|=r, t}} \leqslant |\cdot|_{|P|=r,t} \leqslant C |\cdot|_{\rho_{|P|=r, t}}.$$
\end{cor}

\begin{rem} \label{127}
From now on, whenever we consider spaces of the form $X_{|P| \bowtie r,t}$, ${t \in \{x, Z\}}$, $\bowtie \in \{\leqslant, =, \geqslant\},$ we will always assume its corresponding affinoid algebra to be endowed with the norm $|\cdot|_{|P| \bowtie r, t}$ defined in Notations \ref{star} and  \ref{startstar}.
\end{rem}

\subsection{A useful proposition} We show here Propositon \ref{131}, which we will use in Subsection \ref{parpl} to prove patching is possible on the relative projective line. We start with a couple of auxiliary results.
\begin{lm} \label{128}
Let $Y_1=\mathcal{M}(A)$ be a $k$-affinoid space. Let $Y_2=\mathcal{M}(B)$ be a relative affinoid space over $Y_1$ and $\phi: Y_2 \rightarrow Y_1$ the corresponding morphism. Let $y \in Y_1$ and set ${F_y:=\phi^{-1}(y)},$ which we identify with the $\mathcal{H}(y)$-analytic space $\mathcal{M}(B \widehat{\otimes}_A \mathcal{H}(y))$. For any $z \in F_y,$ $\mathcal{H}_{\mathcal{M}(B)}(z)=\mathcal{H}_{F_y}(z),$ where $\mathcal{H}_{N}(z)$ is the completed residue field of $z$ when regarded as a point of $N$, $N \in \{\mathcal{M}(B), F_y\}.$
\end{lm}

\begin{proof} Considering the bounded embedding $\mathcal{H}(y) \hookrightarrow \mathcal{H}_{\mathcal{M}(B)}(z),$ we have the following commutative diagram where all the maps are bounded:

\begin{center}
\begin{tikzcd}[column sep=small]
 {} & {}  &    \mathcal{H}_{\mathcal{M}(B)}(z)  \\ 
B  \ar[r] \ar[urr] & B \widehat{\otimes}_A \mathcal{H}(y) \arrow{ur}[swap]{\alpha} \arrow{dr}{\beta} &   \\ 
{} & {}  & \mathcal{H}_{F_y}(z) \arrow[uu, dashed]  \\
\end{tikzcd} 
\end{center}
The proof is based on the identification of $F_y$ to $\mathcal{M}(B \widehat{\otimes}_A \mathcal{H}(x)).$
Remark that the map ~$\alpha$ induces on $B \widehat{\otimes}_A \mathcal{H}(y)$ the semi-norm determined by $z$, implying there is a bounded embedding $\mathcal{H}_{F_y}(z) \hookrightarrow \mathcal{H}_{\mathcal{M}(B)}(z)$ on the diagram above. Similarly, since the map $\beta$ induces on $B$ the semi-norm determined by $z,$ we obtain that $\mathcal{H}_{F_y}(z) = \mathcal{H}_{\mathcal{M}(B)}(z).$
\end{proof}

\begin{rem}
Recall Notation \ref{17aaa}. Let $P$ be a unitary polynomial in $\mathcal{O}_x[T]$ irreducible over $\mathcal{H}(x),$ and $r \in \mathbb{R}_{>0} \backslash \sqrt{|\mathcal{H}(x)^\times|}.$ Let $\eta:=\eta_{P,r} \in \mathbb{P}_{\mathcal{H}(x)}^{1, \mathrm{an}}.$ As seen in Lemma \ref{113} (\textit{cf.} also Remark \ref{122}), $\mathcal{H}(x)\{r^{-1}T, rT^{-1}\}[X]/(P(X)-T)$ is isomorphic to $\mathcal{O}_{\mathbb{P}_{\mathcal{H}(x)}^{1, \mathrm{an}}}(\{\eta\})=\mathcal{H}(\eta)$. By Proposition ~\ref{115} (see also  Remark \ref{122}), $|\cdot|_{|P|=r, x}$  is equivalent to the norm $|\cdot|_{\eta}$ on ~$\mathcal{H}(\eta).$ 
\end{rem}

Following Notation \ref{hajdede}, let $Z_0 \subseteq Z' \cap Z_T \cap Z_P$ be a connected affinoid neighborhood of ~$x.$

\begin{lm} \label{130}
The family $\{X_{|P|=r,Z}: Z \subseteq Z_0\}$ (where $Z$ is always considered to be a connected affinoid neighborhood of $x$) forms a basis of neighborhoods of $\eta$ in $X_{|P|=r, Z_0}.$
\end{lm}
 
\begin{proof}
Let $U$ be an open neighborhood of $\eta$ in $X_{|P|=r, Z_0}$. There exists a connected affinoid neighborhood $Z \subseteq Z_0$ of $x$ such that $X_{|P|=r, Z} \subseteq U.$ To see this, remark that $X_{|P|=r,Z_0} \backslash U$ is a compact subset of $\mathbb{P}_{Z_0}^{1, \mathrm{an}},$ so $\pi(X_{|P|=r,Z_0} \backslash U)$ is a compact subset of ~$Z_0.$ Furthermore, $x \not \in \pi(X_{|P|=r,Z_0} \backslash U),$ so there exists a connected affinoid neighborhood $Z \subseteq Z_0$ of $x$ such that $Z \cap \pi(X_{|P|=r,Z_0} \backslash U) =\emptyset.$ Consequently, $X_{|P|=r, Z} \backslash U = \pi^{-1}(Z) \cap (X_{|P|=r,Z_0} \backslash U)=\emptyset,$ so $X_{|P|=r, Z} \subseteq U.$
\end{proof}

\begin{prop} \label{131}
The local ring $\mathcal{O}_{X_{|P|=r, Z_0}, \eta}$ is a field.
\end{prop}

\begin{proof}
Suppose that $\mathcal{O}_{X_{|P|=r, Z_0}, \eta}$ is not a field. Then its maximal ideal is non-zero, so there exists $f \in \mathcal{O}_{X_{|P|=r, Z_0}, \eta}$ such that $f\neq 0$ and $f(\eta)=0$ in $\mathcal{H}(\eta)$ (\textit{i.e.} $|f|_{\eta}=0$). By Lemma \ref{130}, there exists a connected affinoid neighborhood $Z \subseteq Z_0$ of $x$ such that $f \in \mathcal{O}(X_{|P|=r, Z}).$

By Lemma \ref{128}, evaluating $f \in \mathcal{O}(X_{|P|=r, Z})$ at the point $\eta \in \mathcal{O}(X_{|P|=r, Z})$ is the same as evaluating the restriction of $f$ to the fiber at the point $\eta$ on the fiber. Consequently, since the norm $|\cdot|_{\eta}$ is equivalent to $|\cdot|_{|P|=r, x}$ (see Proposition \ref{115} and Remark \ref{122}), we obtain that $|f|_{|P|=r, x} =0.$

Let $f=\sum_{i=0}^{d-1}\sum_{n \in \mathbb{Z}} a_{n,i} T^n X^i \in \mathcal{O}(X_{|P|=r, Z}).$ Then $|f|_{|P|=r, x}=\max_{i}\max_{n}|a_{n,i}|_x r^n.$ If $|f|_{|P|=r, x}=0,$ this implies that for any $n$ and any $i,$ $|a_{n,i}|_x=0,$ and since $\mathcal{O}_x$ is a field, we obtain $a_{n,i}=0$ in $A_Z.$ Consequently, $f=0$ over $X_{|P|=r,Z}.$

By Lemma \ref{130}, this means that $f=0$ in $\mathcal{O}_{X_{|P|=r,Z_0},\eta},$ contradiction. Hence, the local ring $\mathcal{O}_{X_{|P|=r, Z_0}, \eta}$ is a field.
\end{proof}

\section{Patching on the relative projective line} \label{4.3}

The goal of this section is to prove a relative analogue of \cite[Proposition 3.3]{une}. As before, let $k$ be a complete ultrametric field.

\subsection{A few preliminary results} Recall Notation \ref{17aaa}.
\begin{rem} \label{132}
By Theorem \ref{231}, for any integral $k$-affinoid space $Z$, $\mathscr{M}(\mathbb{P}_Z^{1, \mathrm{an}})=\mathscr{M}(Z)(T).$ 
\end{rem}
 
\begin{lm} \label{den}
Let $X$ be an integral $k$-affinoid space with corresponding affinoid algebra ~$R_X$. Set $F_X=\mathscr{M}(X).$ Let $z \in X$ be such that $\mathcal{O}_z$ is a field. 

The function $|\cdot|_{F_X}:=\max(|\cdot|_y:  y\in \Gamma(X) \cup \{z\})$ defines a submultiplicative norm on ~$F_X$ which when restricted to $R_X$ gives the spectral norm $|\cdot|_{\rho_X}$.

Let $X'$ be an integral $k$-affinoid space such that $X$ is a rational domain of $X'.$ Set $F_{X'}=\mathscr{M}(X').$
The field $F_{X'}$ is dense in $(F_X, |\cdot|_{F_{X}}).$
\end{lm}

\begin{proof}
Remark that $z$ (since $\mathcal{O}_z$ is a field) and all $y \in \Gamma(X)$ (because of \cite[Lemme 2.1]{Duc2}) determine multiplicative norms on $R_X,$ and hence also on $F_X.$ 

As a consequence, $|\cdot|_{F_X}$ is well-defined.  That it is a submultiplicative norm on $F_X$ extending $\rho_X$ follows from the fact that $|\cdot|_{\rho_X}=\max(|\cdot|_{y}: y \in \Gamma(X))$. Since $X$ is reduced, $\rho_X$ is equivalent to the norm on the affinoid algebra $R_X$ (\cite[Proposition 9.13]{coanno}). 

By \cite[Corollary 2.2.10]{Ber90}, for $S_X:=\{g \in \mathcal{O}(X'): |g|_x \neq 0  \ \text{for all} \ x \in X\},$ the set $S_X^{-1} \mathcal{O}(X')$ is dense in $\mathcal{O}(X)=R_X.$ As $S_X \subseteq \mathcal{O}(X')\backslash \{0\},$ by Lemma \ref{1.2}, $S_X^{-1} \mathcal{O}(X') \subseteq \mathscr{M}(X')=F_{X'},$ so $R_X \cap F_{X'} \subseteq F_X$ is a dense subset of $R_X.$ 

Let $f=\frac{u}{v} \in F_X,$ where $u, v \in R_X.$ Then by the above, $u,v$ can be approximated by some $u_0, v_0 \in R_X \cap F_{X'}.$ We will show that $\frac{u_0}{v_0}$ approximates $\frac{u}{v}$ in $F_X,$ implying (since $\frac{u_0}{v_0} \in F_{X'}$) that $F_{X'}$ is dense in $F_X$.

Since both $|u-u_0|_{\rho_X}$ and $|v-v_0|_{\rho_X}$ may be assumed to be arbitrarily small, we may suppose that $|u|_{y}=|u_0|_{y}$ and $|v_0|_y=|v|_y$ for all $y \in \Gamma(X) \cup \{z\}.$ Then $|\frac{1}{v}|_{F_X}=|\frac{1}{v_0}|_{F_{X}}.$ Finally, $|f-\frac{u_0}{v_0}|_{F_X}\leqslant |uv_0-u_0v|_{F_X} \cdot |\frac{1}{v}|^2_{F_X} =|uv_0-u_0v|_{R_X} \cdot |\frac{1}{v}|^2_{F_X} \rightarrow 0$ in $F_X$ when $u_0 \rightarrow u$ and $v_0 \rightarrow v$ in $R_X.$
\end{proof}
The following is an example of Setting \ref{therealone} which we will be working with. 

\begin{prop} \label{133}
Let $U,V$ be connected affinoid domains of $\mathbb{P}_{\mathcal{H}(x)}^{1, \mathrm{an}}$ containing only type ~3 points in their boundaries such that $U \cap V$ is a single type 3 point $\{\eta\}.$ Let $Z$ be a connected  affinoid neighborhood of $x$ in $S$ such that there exist $Z$-thickenings $U_Z, V_Z$ of $U, V,$ respectively. Assume that $Z$ is such that the statement of Proposition \ref{nukkuptoj2} is satisfied. Then the conditions of Setting ~\ref{therealone} are satisfied for: $F:=\mathscr{M}(Z)(T),$  $R_0:=\mathcal{O}(U_Z \cap V_Z),$  $R_1=A_1:=\mathcal{O}(U_Z), R_2=A_2:=\mathcal{O}(V_Z),$ and $F_i:= \mathrm{Frac} \  R_i, i=0,1,2.$ 
\end{prop}

\begin{proof} The field $F$ is clearly infinite and embeds in both $F_1$ and $F_2$.
Also, the rings $R_i,$ $i=0,1,2,$ are integral domains containing $k$ and endowed with a non-Archimedean submultiplicative norm that extends that of $k$ and is $k$-linear. The morphisms ${R_j \hookrightarrow R_0},$ $j=1,2,$ are  bounded seeing as they are restriction morphisms.

Remark that regardless of whether $U_Z \cup V_Z$ is an affinoid domain or all of~$\mathbb{P}_Z^{1, \mathrm{an}},$ ${H^{1}(U_Z \cup V_Z, \mathcal{O})=0}.$ Consequently, there exists a surjective admissible morphism ${R_1 \oplus R_2 \twoheadrightarrow R_0}$ (see e.g. Remark \ref{edhenje}).
\end{proof}

\begin{nota} \label{134}
In addition to Notation \ref{17aaa}, let $G$ be a rational linear algebraic group defined over $\mathcal{O}_x(T).$ Seeing as $\mathcal{O}_x(T)=\varinjlim_{Z} \mathscr{M}(Z)(T),$ where the direct limit is taken with respect to connected affinoid neighborhoods of $x,$ there exists such a $Z_G$ for which $G$ is a rational linear algebraic group defined over $\mathscr{M}(Z_G)(T).$ The same remains true for any connected affinoid neighborhood $Z \subseteq Z_G$ of $x.$ 
\end{nota}

\subsection{Patching over $\mathbb{P}^{1, \mathrm{an}}$} \label{parpl}
We now have all the necessary elements to show that patching is possible over $\mathbb{P}_Z^{1, \mathrm{an}}$ for a well-enough chosen affinoid neighborhood $Z$ of $x$ (both in the sense of \cite[Corollary 1.10]{une} and of \cite[Proposition 3.3]{une}). 

For the rest of this section, we assume that $k$ is a complete non-trivially valued ultrametric field. Recall Notation \ref{17aaa}. 

\begin{rem} \label{136} In order for the results of Section \ref{4.2} to be applicable, from now on, whenever taking a thickening of an affinoid domain with respect to a certain writing of its boundary points (see Definition \ref{100}), we will always assume that the corresponding polynomials were chosen to be unitary (since $\mathcal{O}_x$ is a field, this can be done without causing any restrictions to our general setting).  
\end{rem}

\begin{set}\label{137} Let $\eta$ be a type 3 point of $\mathbb{P}_{\mathcal{H}(x)}^{1, \mathrm{an}}.$ There exists a unitary polynomial ${P \in \mathcal{O}_x[T]}$ that is irreducible over $\mathcal{H}(x)$ and a real number $r \in \mathbb{R}_{>0} \backslash \sqrt{|\mathcal{H}(x)^\times|}$  such that $\eta=\eta_{P,r}.$ Let $Z_0$ be a connected affinoid neighborhood of $x$ in $S$ such that $P \in \mathcal{O}(Z_0)[T]$ and the $Z_0$-thickenings of $\{u \in \mathbb{P}_{\mathcal{H}(x)}^{1, \mathrm{an}}: |T|_u \bowtie r\}$, $\{u \in \mathbb{P}_{\mathcal{H}(x)}^{1, \mathrm{an}}: |P|_u \bowtie r\}, \bowtie \in \{\leqslant, \geqslant\}$, are connected. Let $Z \subseteq Z_0$ be any connected affinoid neighborhood of $x.$ 
\begin{sloppypar}
As before, set $X_{|P| 
\bowtie r, Z}:=\{u \in \mathbb{P}_{Z}^{1, 
\mathrm{an}}: |P|_u \bowtie r\},$ where ${\bowtie 
\in \{ \leqslant, =, \geqslant\}}$. Set ${(R_{0,Z}, 
|\cdot|_{R_{0,Z}}):=}(\mathcal{O}(X_{|P| = r, Z}), 
|\cdot|_{|P|=r, Z}),$  $(R_{1,Z}, |\cdot|
_{R_{1,Z}}):=(\mathcal{O}(X_{|P| \leqslant r, Z}), 
|\cdot|_{|P|\leqslant r, Z})$ and $(R_{2,Z}, |
\cdot|_{R_{2,Z}}):=(\mathcal{O}(X_{|P| \geqslant 
r, Z}), |\cdot|_{|P|\geqslant r, Z})$ (see Remark 
\ref{127}). Also, set ${F_{i,Z}:=\mathrm{Frac}
(R_{i,Z})}, i=0,1,2,$  and $F:=\mathscr{M}(Z)(T).$
\end{sloppypar}
 Assume that $Z_0$ is chosen so that the results (in particular,  Corollary \ref{126}) of Section~\ref{4.2} are satisfied. Moreover, assume $Z_0 \subseteq Z_G$ (see Notation \ref{134}).
\end{set}

Throughout this subsection, suppose we are in the situation of Setting \ref{137}. We now define three parameters $d, M$ and $\delta$; the goal is to have them satisfy the properties of Theorem \ref{e keqe}, which then allows us to prove a crucial patching result (Theorem \ref{145}). 

\begin{cond} \label{138} 
Since $H^{1}(\mathbb{P}_{Z}^{1, \mathrm{an}}, \mathcal{O})=0,$ there is an admissible surjection ${R_{1, Z} \oplus R_{2,Z} \twoheadrightarrow R_{0,Z}}$. Furthermore, by Lemma \ref{119}, for any $c \in R_{0,Z}$, there exist $a \in R_{1,Z}$ and $b \in R_{2, Z}$ such that $\frac{1}{2} \max(|a|_{R_{1,Z}}, |b|_{R_{2,Z}}) < |c|_{R_{0,Z}}.$ Set $d=\frac{1}{2}.$
\end{cond}

Let us set $F:=\mathscr{M}(Z_0)(T)$, so that $G$ is a rational linear algebraic group over $F$. We refer to Remark \ref{thingy} and recall that we obtain a commutative diagram (Diagram~(\ref{diagramG})): $m$ denotes the multiplication in $G$, and the Zariski open $S'$ of $G$ contains the identity element of $G$ and is isomorphic to a Zariski open $S''$ of some $\mathbb{A}_F^n$. (Here $\widetilde{S''}$ (resp. $S''$) is $\widetilde{S}$ (resp.~$S$) from Remark \ref{thingy}). 
\begin{equation} \label{diagramG}
\begin{tikzcd}[scale=1.8]
 \widetilde{S'} \arrow{r}{m_{|\widetilde{S'}}} \arrow{d}[swap]{(\varphi \times \varphi)_{|\widetilde{S'}}}& S' \arrow{d}{\varphi} \\
\widetilde{S''}  \arrow{r}[swap]{f} & S''
\end{tikzcd}
\end{equation}
\begin{cond} \label{139}
Let us look at the diagram above over the field $F_{0,Z_0}$ (recall the notation in Setting \ref{137}). We may suppose that $g_i, h_i \in R_{0, Z_0}[\underline{S}, \underline{T}]$ for all $i.$ Since $h_i(0) \neq 0$ and $\mathcal{O}_{X_{|P|=r,Z_0}, \eta}$ is a field (see Proposition \ref{131}), $|h_i(0)|_{\eta}\neq 0.$ Consequently, by Lemma \ref{130}, there exists a connected affinoid neighborhood $Z_1 \subseteq Z_0$ of ~$x$ such that $|h_i(0)|_u \neq 0$ for all $u \in X_{|P|=r, Z_1}$, $i.$ By \cite[Corollary 3.15]{coanno}, $h_i(0) \in R_{0, Z_1}^\times$ for all ~$i.$ This implies that  $h_i(0) \in R_{0, Z}^\times$ for all connected affinoid neighborhoods $Z \subseteq Z_1$ of $x.$

By Lemmas \ref{1} and \ref{2}, there exists $M \geqslant 1$ such that 
$$f_i=S_i + T_i + \sum_{|(l,m)| \geqslant 2} c_{l,m}^i \underline{S}^l\underline{T}^m \in R_{0,Z_1}[[\underline{S}, \underline{T}]], $$  and $|c_{l,m}^i|_{R_{0,Z_1}} \leqslant M^{|(l,m)|}$ for all $i,$ and all $(l,m) \in \mathbb{N}^{2n}$ such that $|(l,m)| \geqslant 2,$ where $|(l,m)|$ is the sum of the coordinates of $(l,m).$

By Lemma \ref{120},  for any connected affinoid neighborhood $Z \subseteq Z_1$ of $x,$
$f_i=S_i + T_i + \sum_{|(l,m)| \geqslant 2} c_{l,m}^i \underline{S}^l\underline{T}^m \in R_{0,Z}[[\underline{S}, \underline{T}]]$, and ${|c_{l,m}^i|_{R_{0,Z}} \leqslant |c_{l,m}^i|_{R_{0, Z_1}} \leqslant M^{|(l,m)|}}$ for all $i$ and all $(l,m) \in \mathbb{N}^{2n}$ such that $|(l,m)| \geqslant 2$.
\end{cond}

\begin{cond} \label{140}
Since $\widetilde{S''}$ is a Zariski open of $\mathbb{A}_F^{2n}$ and $F \hookrightarrow \mathcal{H}(\eta),$ we have that $\widetilde{S''}(\mathcal{H}(\eta))$ is a Zariski open of $\mathcal{H}(\eta)^{2n}$. Since the topology induced by the norm on $\mathcal{H}(\eta)$ is finer than the Zariski one and $0 \in \widetilde{S''}$, there exists $\delta>0$ such that the open disc $D_{\mathcal{H}(\eta)^{2n}}(0, \delta)$ in $\mathcal{H}(\eta)^{2n}$ (with respect to the max-norm), centered at $0$ and of radius $\delta$, is contained in $\widetilde{S''}(\mathcal{H}(\eta)) \subseteq \widetilde{S''}.$

Then for any connected affinoid neighborhood $Z \subseteq Z_0$ of $x,$ the open disc $D_{R_{0,Z}^{2n}}(0,\delta)$ in $R_{0,Z}^{2n}$ (with respect to the max-norm), centered at $0$ and of radius $\delta$, satisfies: $D_{R_{0,Z}^{2n}}(0, \delta) \subseteq D_{\mathcal{H}(\eta)^{2n}}(0, \delta) \subseteq \widetilde{S''}.$ This is clear seeing as for any $a \in R_{0, Z},$ $|a|_{\eta} \leqslant |a|_{\rho_{X_{|P|=r, Z}}} \leqslant |a|_{R_{0,Z}},$ where $\rho_{X_{|P|=r, Z}}$ is the spectral norm on $X_{|P|=r, Z}$. 
\end{cond}

\begin{rem} \label{141} 
\begin{sloppypar} Putting Parameters \ref{138}, \ref{139}, \ref{140} together, let $\varepsilon>0$ be such that ${\varepsilon < \min\left(\frac{d}{2M}, \frac{d^3}{M^4}, \frac{d\delta }{2}\right)}$. Then all of the conditions of Theorem \ref{e keqe} are satisfied for ${R_0:=R_{0,Z}}, A_1:=R_{1, Z}, A_2:=R_{2,Z}$, $F_0=\text{Frac} \ R_0,$ where $Z$ is any connected affinoid neighborhood of $x$ contained in $Z_1,$ with $Z_1$ as in Parameter \ref{139}.
\end{sloppypar}
\end{rem}

\begin{lm} \label{142}
Let $g \in G(F_{0,{Z_1}})$ (with $Z_1$ as in Parameter \ref{139}). Suppose $g \in S'$ (see Diagram~(\ref{diagramG})), and ${|\varphi(g)|_{\eta} < \frac{\varepsilon}{C}}$, where $C$ is the constant obtained in  Corollary \ref{126} corresponding to the polynomial $P$.
Then there exists a connected affinoid neighborhood $Z \subseteq Z_1$ of $x$, and $g_i \in G(F_{i,Z}), {i=1,2},$ such that $g=g_1 \cdot g_2$ in $G(F_{0,Z})$.
\end{lm}

\begin{proof}
Since $\varphi(g) \in \mathbb{A}_{F}^{n}(F_{0, Z_1})=F_{0, Z_1}^n,$ there exist $\alpha_i, \beta_i \in R_{0, Z_1}$ such that ${\varphi(g)=(\alpha_i/\beta_i)_{i=1}^n}.$ Since $\beta_i \neq 0,$ by Proposition \ref{131}, $|\beta_i|_{\eta} \neq 0$. Thus, by Lemma \ref{130}, there exists a connected affinoid neighborhood $Z' \subseteq Z_1$ of $x$ such that $|\beta_i|_u \neq 0$ for all $u \in X_{|P|=r,Z_1}$ and all $i.$ By \cite[Corollary 3.15]{coanno}, $\beta_i \in R_{0, Z'}^\times$ for all $i.$ In particular, this means that $\varphi(g) \in R_{0,Z'}^n.$ Remark that for any connected affinoid neighborhood $Z \subseteq Z'$ of~$x,$ $\varphi(g) \in R_{0,Z}^n.$

Since $|\varphi(g)|_{\eta} < \varepsilon/C,$ there exists a connected affinoid neighborhood $Z \subseteq Z'$ of $x$ such that $|\varphi(g)|_u < \varepsilon/C$ for all $u \in X_{|P|=r, Z}.$ Consequently, $|\varphi(g)|_{\rho_{X_{|P|=r, Z}}} < \varepsilon/C$, where $\rho_{X_{|P|=r, Z}}$ is the spectral norm on $X_{|P|=r, Z}.$ By Corollary \ref{126}, this means that ${|\varphi(g)|_{R_{0,Z}} < \varepsilon}.$ 

By Remark \ref{141}, the conditions of Theorem \ref{e keqe} are satisfied, meaning there exist ${g_i \in G(F_{i, Z})}, i=1,2,$ such that $g=g_1 \cdot g_2$ in $G(F_{0,Z}).$
\end{proof}

Remark that in the proposition above, we can in the same way show that there exist $g_i' \in G(F_{i,Z}), i=1,2,$ such that $g=g_2' \cdot g_1'$ in $G(F_{0,Z}).$

\begin{conv}\label{convi} Let us fix once and for all an embedding of $G$ into $\mathbb{A}_F^m$ for some $m \in \mathbb{N}.$ Let $K/F$ be a field extension, and $M \subseteq K.$
Set $G_{K}=G \times_F K.$ Let $U$ be a Zariski open subset of $G_{K}.$ We will denote by $U(M)$ the set $\mathbb{A}^m(M) \cap U,$ where $\mathbb{A}^m(M)$ is the set of vectors in $\mathbb{A}^m(K)$ with coordinates in $M$. 
\end{conv}

We now show that we can omit the hypothesis on the norm of $\varphi(g)$ in the Lemma \ref{142}. 

\begin{lm} \label{144} With the same notation as in Lemma \ref{142},
let $g \in G(F_{0,{Z_1}}).$ Suppose $g \in S'$. 
Then there exists a connected affinoid neighborhood $Z \subseteq Z_1$ of $x$, and ${g_i \in G(F_{i,Z})}, i=1,2,$ such that $g=g_1 \cdot g_2$ in $G(F_{0,Z})$.
\end{lm}

\begin{proof}
We will reduce to the first case (\textit{i.e.} Lemma \ref{142}). Recall that the fields $F_{0, Z_1}$ can be endowed with a submultiplicative norm $|\cdot|_{F_{0,Z_1}}$ as in Lemma \ref{den}, where the role of the point $z$ is played by $\eta$ here.

Let $\psi: gS'\cap S' \rightarrow \mathbb{A}^n_{F_{0,Z_1}}$ be the morphism given by $h \mapsto \varphi(g^{-1}h).$ Remark $0 \in \mathrm{Im}(\psi)$. The preimage $\psi^{-1}(D_{F_{0, Z_1}^n}(0, \varepsilon/C))$ is open in  $(gS' \cap S')(F_{0, Z_1}).$ 

Since $X_{|P|=r, Z_1}$ is a rational domain in $X_{|P| \leqslant r, Z_1},$ by Lemma \ref{den}, $F_{1, Z_1}$ is dense in $F_{0,Z_1},$  so $(gS'\cap S')(F_{1, Z_1})$ is dense in $(gS'\cap S')(F_{0, Z_1})$ (see Convention \ref{convi}). This means there exists ${h \in (gS' \cap S')(F_{1, Z_1}) \subseteq G(F_{1, Z_1})}$ such that $|\varphi(g^{-1}h)|_{F_{0, Z_1}} < \varepsilon/C,$ implying that $|\varphi(g^{-1}h)|_{\eta}<\varepsilon/C.$ 

By Lemma \ref{142}, there exists a connected affinoid neighborhood $Z \subseteq Z_1$ of $x$ and ${g_1' \in G(F_{1, Z})},$ $g_2' \in G(F_{2, Z}),$ such that $g^{-1}h=g_2' \cdot g_1'$ in $G(F_{0,Z}).$ Hence, there exist $g_1:=hg_1'^{-1}\in G(F_{1, Z})$ and $g_2:=g_2'^{-1} \in G(F_{2,Z})$ such that $g=g_1 \cdot g_2$ in $G(F_{0,Z}).$
\end{proof}

\begin{thm} \label{145}
Recall Setting \ref{137}. For any $g \in G(F_{0, Z_0}),$ there exists a connected affinoid neighborhood $Z \subseteq Z_0$ of $x$ and $g_i \in G(F_{i,Z}), i=1,2,$ such that $g=g_1 \cdot g_2$ in $G(F_{0,Z})$.
\end{thm}

\begin{proof}
Recall the construction of the connected affinoid neighborhood $Z_1 \subseteq Z_0$ of $x$ in Parameter \ref{139}. 
By \cite[Lemma 3.1]{HHK}, there exists a Zariski open $S_1'$ of $G$ isomorphic to an open $S_1''$ of $\mathbb{A}_F^n$ such that $g \in S_1'(F_{0, Z_1}).$
Since $F$ is infinite and $S_1'$ is isomorphic to an open of some $\mathbb{A}_{F}^{n}$, there exists $\alpha \in S_1'(F).$ Set $S_1:=\alpha^{-1}S_1'.$ 
Then $I \in S_1$, and $S_1$ is isomorphic to an open subset of $\mathbb{A}_F^n.$ By translation, we may assume that this isomorphism sends $I$ to $0 \in \mathbb{A}^n(F).$ Set  $g':=\alpha^{-1}g \in S_1(F_{0, Z_1}).$ Then by Lemma \ref{144}, there exists a connected affinoid neighborhood $Z \subseteq Z_1$ of $x$, and $g_1' \in G(F_{1, Z}),$ $g_2 \in G(F_{2, Z}),$ such that $g'=g_1' \cdot g_2$ in $G(F_{0,Z}).$ Consequently, for $g_1:=\alpha \cdot g_1' \in G(F_{1, Z}),$ we obtain that $g=g_1 \cdot g_2$ in $G(F_{0, Z}).$
\end{proof}
As a consequence, the following, which is the main tool for showing a local-global principle over the relative $\mathbb{P}^{1, \mathrm{an}},$ can be shown.

Recall that we are working under the hypotheses of Notation \ref{17aaa}.

\begin{prop} \label{146}
Let $U,V$ be connected affinoid domains in $\mathbb{P}_{\mathcal{H}(x)}^{1, \mathrm{an}}$  containing only type ~3 points in their boundaries such that $U \cap V$ is a single type 3 point $\{\eta_{P,r}\},$ with ${P \in \mathcal{O}_x[T]}$ irreducible over $\mathcal{H}(x)$ and $r \in \mathbb{R}_{>0} \backslash \sqrt{|\mathcal{H}(x)^\times|}$. Set ${W:= U \cap V}.$

Let $G$ be as in Notation \ref{134}, and $Z_0$ as in Setting \ref{137}.
Let $Z' \subseteq Z_0$ be a connected affinoid neighborhood of $x$ for which the $Z'$-thickenings $U_{Z'}, V_{Z'}, W_{Z'}$ exist, are connected, and Proposition \ref{nukkuptoj2} is satisfied. 
 
Then for any $g \in G(\mathscr{M}(W_{Z'})) \cup G(\mathscr{M}_{\mathbb{P}_{Z'}^{1, \mathrm{an}}, \eta})$, there exists a connected affinoid neighborhood $Z \subseteq Z'$ of $x$, and $g_U \in G(\mathscr{M}(U_Z)), g_V \in G(\mathscr{M}(V_Z)),$ such that $g=g_U \cdot g_V$ in ~$G(\mathscr{M}(W_Z))=G(\mathscr{M}(U_Z \cap V_Z)).$ 
\end{prop}

\begin{proof}
Remark that for any $g \in G(\mathscr{M}_{\mathbb{P}_{Z'}^{1, \mathrm{an}}, \eta}),$ by Lemma \ref{103}, there exists a connected affinoid neighborhood $Z \subseteq Z'$ of $x$ such that $g \in G(\mathscr{M}(W_{Z})).$ Thus, it suffices to show the result for any $g \in G(\mathscr{M}(W_{Z'})).$

By Theorem \ref{145}, there exists a connected affinoid neighborhood $Z \subseteq Z'$ of $x,$ and ${g_i \in G(F_{i, Z})}, i=1,2,$ such that $g=g_1 \cdot g_2$ in $G(\mathscr{M}(W_Z))$ (once again, recall Setting \ref{137}). Set $\partial{U}=\{\eta_{P,r}, \eta_{P_j, r_j}, j=1,2,\dots, n\},$ where $P_j \in \mathcal{O}_x[T]$ are unitary polynomials that are irreducible over $\mathcal{H}(x)$ and $r_j \in \mathbb{R}_{>0} \backslash \sqrt{|\mathcal{H}(x)^\times|}$ for all ~$j.$  

Seeing as $U=\{u \in \mathbb{P}_{\mathcal{H}(x)}^{1, \mathrm{an}}: |P|_u \bowtie r, |P_j|_u \bowtie_j r_j, j\},$ where $\bowtie, \bowtie_j \in \{\leqslant, \geqslant\}$ for all ~$j$ (Proposition  \ref{projectivedom}), $U_Z \subseteq \{u \in \mathbb{P}_{Z}^{1, \mathrm{an}}: |P|_u \bowtie r\}.$ Without loss of generality, suppose that $\bowtie$ is $\leqslant$. Then $U_Z \subseteq \{u \in \mathbb{P}_{Z}^{1, \mathrm{an}}: |P|_u \leqslant r\}$ and $V_Z \subseteq \{u \in \mathbb{P}_{Z}^{1, \mathrm{an}}: |P|_u \geqslant r\}$ (see Lemma ~\ref{16aaa}). 

Consequently, for $g_U:=g_{1|U_Z} \in G(\mathscr{M}(U_Z))$ and $g_V:=g_{2|Z} \in G(\mathscr{M}(V_Z)),$ $g=g_U \cdot g_V$ in $G(\mathscr{M}(W_Z))=G(\mathscr{M}(U_Z \cap V_Z)).$
\end{proof}

\subsection{Patching over relative nice covers}
Proposition \ref{146} is enough in itself to directly show a local-global principle over the relative projective line. However, just like in the one-dimensional case, when showing a local-global principle for relative projective curves, we use arguments that make it possible to descend to the line. The goal of this subsection is to present the necessary arguments to make this descent.

We keep using Notation \ref{17aaa} under the additional assumption that $k$ is non-trivially valued. In Definition \ref{nice}, we recalled the notion of a nice cover. We will also be using the following.

\begin{defn}[{\cite[Definition 2.18]{une}}] \label{parityfunction}
Let $C$ be a $k$-analytic curve. Let $\mathcal{U}$ be a nice cover of ~$C.$ A function $T_{\mathcal{U}} : \mathcal{U} \rightarrow \{0,1\}$ will be called \textit{a parity function for} $\mathcal{U}$ if for any different non-disjoint $U', U'' \in \mathcal{U}$, $T_{\mathcal{U}}(U') \neq T_{\mathcal{U}}(U'').$

We denote by $S_{\mathcal{U}}$ the set of points on the pairwise intersections of the elements of $\mathcal{U},$ meaning 
$S_{\mathcal{U}}:=\{x \in C: \exists U, V \in \mathcal{U}, U \neq V, x \in U \cap V\}$.

\end{defn}

\begin{thm} \label{147}
Let $\mathcal{U}_x$ be a nice cover of $\mathbb{P}_{\mathcal{H}(x)}^{1, \mathrm{an}},$ and $T_{\mathcal{U}_x}$ a parity function corresponding to $\mathcal{U}_x$.  Let $S_{\mathcal{U}_x}$ be the set of intersection points of the different elements of $\mathcal{U}_x.$ Let $Z_0$ be a connected affinoid neighborhood of $x$ such that the $Z_0$-thickening $\mathcal{U}_{Z_0}$ of $\mathcal{U}_x$ exists and is a $Z_0$-relative nice cover of $\mathbb{P}_{Z_0}^{1, \mathrm{an}}$ (see Definition \ref{104}).  

Let $G/\mathscr{M}(Z_0)(T)$ be a rational linear algebraic group.
Then for any element ${(g_s)_{s \in S_{\mathcal{U}_x}}}$ of ${\prod_{s \in S_{\mathcal{U}_x}} G\left(\mathscr{M}_{\mathbb{P}_{Z_0}^{1, \mathrm{an}}, s}\right)},$ there exists a connected affinoid neighborhood $Z \subseteq Z_0$ of $x$, and $(g_{U_Z})_{U \in \mathcal{U}_x} \in \prod_{U \in \mathcal{U}_x} G(\mathscr{M}(U_Z)),$ satisfying: for any ${s \in S_{\mathcal{U}_x}},$ there exist exactly two $U_{s}, V_{s} \in \mathcal{U}_x$ containing $s$, $g_s \in G(\mathscr{M}(U_{s, Z} \cap V_{s, Z})),$ and if ${T_{\mathcal{U}_x}(U_{s})=0},$ then $g_s=g_{U_{s,Z}} \cdot g_{V_{s,Z}}^{-1}$ in $G(\mathscr{M}(U_{s, Z} \cap V_{s, Z})).$ 
\end{thm}

\begin{proof}
Set $\mathcal{U}_x=\{U_1, U_2, \dots, U_n\}.$ Clearly, $n>1$. Using induction we will show the following statement for all $i$ such that $2 \leqslant i \leqslant n$: 

\begin{fact} \label{148} Let $I \subseteq \{1,2,\dots,n\}$ be such that $|I|=i$ and $\bigcup_{h \in I} U_h$ is connected. Let $S_I$ $(\subseteq S_{\mathcal{U}_x})$ denote the set of intersection points of the different elements of $\{U_h\}_{h \in I}.$ Let $Z' \subseteq Z_0$ be any connected affinoid neighborhood of $x.$
Then for any $(g_s)_{s \in S_I} \in \prod_{s \in S_I} G(\mathscr{M}_{\mathbb{P}_{Z'}^{1, \mathrm{an}}, s}),$ there exists a connected affinoid neighborhood $Z_I \subseteq Z'$ of $x$ and $(g_{U_h,Z_I})_{h \in I} \in \prod_{h \in I} G(\mathscr{M}(U_{h,Z_I}))$, satisfying: for any $s \in S_I$ there exist exactly two elements $U_{s}, V_s \in \{U_h\}_{h \in I}$ containing $s$, $g_s \in G(\mathscr{M}(U_{s, Z_I} \cap V_{s, Z_I})),$ and if $T_{\mathcal{U}_x}(U_{s})=0,$ then $g_s=g_{U_s,Z_I} \cdot g_{V_s, Z_I}^{-1}$ in $G(\mathscr{M}(U_{s,Z_{I}} \cap V_{s, Z_{I}})).$  The same is true for any connected affinoid neighborhood $Z'' \subseteq Z_I$ of $x.$
\end{fact}

For $i=2,$ this is Proposition \ref{146}. Suppose it is true for some $i-1, 2<i\leqslant n,$ and let us show that it is true for $i.$ Without loss of generality, we may assume that $I=\{1,2,\dots, i\},$ \textit{i.e.} that  $\bigcup_{h=1}^i U_h$ is connected. By \cite[Lemma 2.20]{une}, there exist $i-1$ elements of $\{U_h\}_{h=1}^i$ whose union is connected. Without loss of generality, let us assume that $\bigcup_{h=1}^{i-1} U_h$ is connected. Set $I'=I\backslash \{i\}.$

Let us start by making a comparison between $S_{I}$ and $S_{I'}.$ Set $V_{i-1}=\bigcup_{h=1}^{i-1} U_h.$ This is a connected affinoid domain containing only type 3 points in its boundary. Since $V_{i-1}, U_i,$ and  $V_{i-1} \cup U_i$ are connected subsets of $\mathbb{P}_{\mathcal{H}(x)}^{1, \mathrm{an}}$, $V_{i-1} \cap U_i$ is  non-empty  and connected (see \cite[Lemma 2.7]{une}). Furthermore, since $V_{i-1} \cap U_i \subseteq S_{\mathcal{U}_x}$ (\textit{i.e.} it is contained in a finite set of type 3 points), $V_{i-1} \cap U_i$ is a single type 3 point $\{\eta\}$. Hence, there exists $h_0 \in I'$ such that $U_{h_{0}} \cap U_{i} \neq \emptyset.$ By \cite[Lemma 2.12]{une}, such an $h_0$ is unique. Consequently, $S_I=S_{I'} \cup \{\eta\}.$

For some $Z' \subseteq Z_0$ as in Statement \ref{148}, let $(g_s)_{s \in S_{I}} \in \prod_{s \in S_I}G(\mathscr{M}_{\mathbb{P}_{Z'}^{1, \mathrm{an}}, s}).$ From the induction hypothesis, for $(g_s)_{s \in S_{I'}} \in \prod_{s \in S_{I'}}G(\mathscr{M}_{\mathbb{P}_{Z'}^{1, \mathrm{an}},s}),$ there exist a connected affinoid neighborhood $Z_{I'} \subseteq Z'$ of $x$ and $(g_{U_{h,Z_{I'}}})_{h \in I'} \in \prod_{h \in I'} G(\mathscr{M}(U_{h,Z_{I'}}))$, satisfying: for any $s \in S_{I'},$ there exist exactly two $U_{s}, V_s \in \{U_h\}_{h \in I'}$ containing $s$, $g_s \in G(\mathscr{M}(U_{s, Z_{I'}} \cap V_{s, Z_{I'}})),$ and if $T_{\mathcal{U}_x}(U_s)=0,$ $g_s=g_{U_{s,Z_{I'}}} \cdot g_{V_{s, Z_{I'}}}^{-1}$ in $G(\mathscr{M}(U_{s,Z_{I'}} \cap V_{s, Z_{I'}})).$ 

Remark that the affinoid domains $V_{i-1}$ and $U_i$ satisfy the properties of Proposition ~\ref{146} with $V_{i-1} \cap U_i= \{\eta\}$. As seen above, there exist exactly two elements of $\{U_h\}_{h \in I}$ containing ~$\eta$. Also, since  $g_{\eta} \in G(\mathscr{M}_{ \mathbb{P}_{Z'}^{1, \mathrm{an}},\eta}),$ by Lemma \ref{103}, we may assume that $g_{\eta} \in G(\mathscr{M}(V_{i-1, Z'} \cap U_{i,Z'})).$ Hence, we may also assume that for any connected affinoid domain $Z''' \subseteq Z'$ of $x,$ $g_{\eta} \in G(\mathscr{M}(V_{i-1, Z'''} \cap U_{i,Z'''})).$ 
\begin{itemize}
\item Suppose $T_{\mathcal{U}_x}(U_i)=0.$ By Proposition \ref{146}, there exists a connected affinoid neighborhood $Z_I \subseteq Z_{I'} \subseteq Z'$ of $x$, and $a \in G(\mathscr{M}(U_{i, Z_I})), b \in G(\mathscr{M}(V_{i-1, Z_I})),$ such that $g_{\eta} \cdot g_{U_{i-1}, Z_I} =a\cdot b$ in $G(\mathscr{M}(U_{i, Z_I} \cap V_{i-1, Z_I})).$ For any $h \in I',$ set $g'_{U_h, Z_I}:=g_{U_h, Z_I} \cdot b^{-1}$ in $G(\mathscr{M}(U_{h, Z_I})).$ Also, set $g'_{U_i, Z_I}:=a$ in $G(\mathscr{M}(U_{i, Z_I})).$
\item Suppose $T_{\mathcal{U}_x}(U_i)=1.$ By Proposition \ref{146}, there exists a connected affinoid neighborhood $Z_I \subseteq Z_{I'} \subseteq Z'$ of $x$ and $c \in G(\mathscr{M}(V_{i-1, Z_I})), d \in G(\mathscr{M}(U_{i, Z_I}))$, such that $g_{U_{i-1}, Z_I}^{-1} \cdot g_{\eta}=c \cdot d$ in $G(\mathscr{M}(V_{i-1, Z_I} \cap U_{i, Z_I})).$ For any $h \in I',$ set $g'_{U_h, Z_I}:=g_{U_h, Z_I} \cdot c$ in $G(\mathscr{M}(U_{h, Z_I})).$ Also, set $g'_{U_i, Z_I}:=d^{-1}$ in $G(\mathscr{M}(U_{i, Z_I})).$
\end{itemize}
The family $(g'_{U_h, Z_I})_{h \in I} \in \prod_{h \in I} G(\mathscr{M}(U_{h, Z_I}))$ satisfies the conditions of Statement \ref{148} for the given $(g_s)_{s \in S_I}.$ The last part of Statement \ref{148} is obtained directly by taking restrictions of $g'_{U_{h, Z_I}}$ to $G(\mathscr{M}(U_{h,Z''})), h \in I.$

In particular, for $i=n,$ we obtain the result that was announced. 
\end{proof}

\section{Relative proper curves} \label{4.4}

Throughout this section, let $k$ denote a complete ultrametric field. 
We recall that a \textit{good} Berkovich analytic space is one where the affinoid domains form a basis of the Berkovich topology. Let us fix and study the following framework.

\begin{set} \label{200}
Let $S, C$ be $k$-analytic spaces such that $S$ is good and normal.  Suppose there exists a morphism $\pi: C \rightarrow S$ that makes $C$ a proper flat relative analytic curve (\textit{i.e.} all the non-empty fibers are curves) over $S.$  Let $x \in S$ be such that the stalk $\mathcal{O}_x$ is a field and ~${\pi^{-1}(x) \neq \emptyset}.$ 

Assume there exists a connected affinoid neighborhood $Z_0$ of $x$ such that:

\begin{enumerate}
\item for any $y \in Z_0,$ the fiber $
\pi^{-1}(y)$ is a normal irreducible projective $
\mathcal{H}(y)$-analytic curve $C_y$;
\item there exists a finite type scheme $C_{\mathcal{O}(Z_0)}$ over $
\text{Spec} \ \mathcal{O}(Z_0)$ such that the 
analytification of the structural morphism $\pi_{\mathcal{O}(Z_0)}:C_{\mathcal{O}(Z_0)} 
\rightarrow \text{Spec} \ \mathcal{O}(Z_0)$ (in 
the sense of \mbox{\cite[2.6]{ber93}}) is the projection
$C_{Z_0}:=C \times_S Z_0 \rightarrow Z_0$. (We say that $C_{Z_0} \rightarrow Z_0$ is \emph{algebraic}.) 
\end{enumerate}
\end{set}

\begin{rem}\label{200bis}
(1) Since $\pi$ is proper, it is boundaryless. As $S$ is good, by \cite[Definition~1.5.4]{ber93}, $C$ is good as well. (Beware that in \cite{ber93}, a boundaryless morphism is called \textit{closed}.) 

(2) Since it is boundaryless, by \cite[Theorem 9.2.3]{famduc}, $\pi$ is open. Hence, from now on we will assume, without loss of generality, that $\pi$ is surjective.
\end{rem}

Before exploring in more depth the properties of Setting \ref{200}, let us present two particular situations which lead to this setup, and which allow us to generalize (in different ways) some of the results of \cite{une}.

\subsection{Example: smooth geometrically connected fibers} \label{lame'} We show here that assuming the fiber smooth and geometrically connected guarantees that the latter has a neighborhood satisfying the conditions of Setting \ref{200}. Let us start by recalling the following auxiliary result.

\begin{lm} \label{duh}
Let $K$ be a complete ultrametric field. Let $X/K$ be a smooth geometrically connected proper $K$-analytic curve. Then $\mathcal{O}_X(X)=K.$
\end{lm}

\begin{proof}
Since $X$ is a proper analytic curve over $K,$ it is the analytification of a projective $K$-algebraic curve $X^{\mathrm{alg}}$ (\cite[Th\'eor\`eme 3.7.2]{Duc}); moreover, since $X$ is smooth and geometrically connected, so is $X^{\mathrm{alg}}$ (see Proposition 3.4.6 (4) and Theorem 3.4.8 (iii) of \cite{Ber90}; see also Theorem 3.2.1 of \emph{loc.cit.}). In particular, $X^{\mathrm{alg}}$ is geometrically integral. By \cite[Corollary~3.4.10]{Ber90}, $\mathcal{O}_{X}(X)=\mathcal{O}_{X^{\mathrm{alg}}}(X^{\mathrm{alg}}).$ Finally, by \cite[Corollary~3.3.21]{liulibri}, $\mathcal{O}_{X}(X)=K$.
\end{proof}

We are very thankful to Antoine Ducros for bringing to our attention the following result.

\begin{thm} \label{addition1}
Suppose $k$ is non-trivially valued. Let $S$ be a regular strict $k$-affinoid space. Let $\pi: C \rightarrow S$ be a flat proper relative analytic curve. Let $x \in S$ be such that $C_x:=\pi^{-1}(x)$ is a smooth and geometrically connected analytic curve over $\mathcal{H}(x)$ (hence non-empty). 

Then we can restrict to a connected affinoid neighborhood $Z$ of $x$ in $S$ such that: 
\begin{enumerate}
\item for any $z \in Z,$ $C_z:=\pi^{-1}(z)$ is a proper smooth geometrically connected $\mathcal{H}(z)$-analytic curve;
\item $\pi_{|Z} : C_Z \rightarrow Z$ is algebraic, where $C_Z:=\pi^{-1}(Z)$.
\end{enumerate}

\end{thm}

\begin{rem}
If, in addition to the hypotheses of Theorem \ref{addition1}, we assume that $\mathcal{O}_x$ is a field, then all the conditions of Setting \ref{200} are satisfied. 
\end{rem}

\begin{proof}[Proof of Theorem \ref{addition1}] We may assume that $S$ is connected. By Remark \ref{200bis}(1), $C$ is a good $k$-analytic space. By the same argument as in Remark \ref{200bis}(2), we may assume that ~$\pi$ is surjective. 
By \cite[Theorem 10.7.5(3)]{famduc}, we may also assume that all of the fibers of ~$\pi$ are geometrically normal, \textit{i.e.} smooth since they are curves. As a consequence, since $S$ is regular, by \cite[Theorem 11.3.3(1c)]{famduc}, $C$ is also regular. Let us also remark that by \cite[Theorem 5.3.4(1)]{famduc}, $\pi$ is a quasi-smooth morphism; since it is also boundaryless, it is smooth. 

It remains to show geometric connectedness of the fibers and algebraicity ``around" the fiber of $x$. For an easier reading, we discuss these two steps separately. 

\
 
\emph{Fiberwise geometric connectedness.} For any $y \in S,$ let us denote by $C_y$ its fiber in $C$.
By \cite[Corollary 3.3.11(i)]{Ber90}, the set $A:=\{y \in S: \dim_{\mathcal{H}(y)}\mathcal{O}_{C_y}(C_y) \leqslant 1\}$ is an open analytic domain of $S$. By Lemma \ref{duh}, $x \in A$.  As ~$\pi$ is surjective, $\forall y \in A,$ ${\dim_{\mathcal{H}(y)}\mathcal{O}_{C_y}(C_y)=1}$. If $C_{y}^i, i=1,2,\dots, n_y,$ are the connected components of $C_y$, meaning they are proper smooth connected $\mathcal{H}(y)$-analytic curves, then  $\dim_{\mathcal{H}(y)}\mathcal{O}_{C_y}(C_y)=\sum_{i=1}^n \dim_{\mathcal{H}(y)}\mathcal{O}_{C_y}(C_y^i)$, implying $n=1,$ so that $C_y$ is connected. Thus, by restricting to a smaller $S$ if necessary, we may assume that $A=S$ and hence that $\pi$ has connected fibers.

By \cite[Proposition 3.3.7]{Ber90}, there is a unique Stein factorization $C \xrightarrow{f} S' \xrightarrow{g}  S$ for $\pi$, where 
\begin{itemize}
\item $f_{\star}\mathcal{O}_C=\mathcal{O}_{S'}$,
\item $f$ is a proper surjective morphism of $k$-analytic spaces with connected fibers,
\item $g$ is a surjective (by the surjectivity of $\pi$) and finite morphism, implying $S'$ is a $k$-affinoid space (\cite[Lemma 1.3.7]{ber93}).  
\end{itemize}

By \cite[Corollary 3.3.12]{Ber90}, $\pi_{\star}\mathcal{O}_C$ is a locally free sheaf on $S$ of rank $1.$ As $\pi_{\star}\mathcal{O}_C=g_{\star}f_{\star}\mathcal{O}_C=g_{\star}\mathcal{O}_{S'}$, we obtain that $g_{\star}\mathcal{O}_{S'}$ is a locally free sheaf on $S$ of rank $1$. This means that by reducing to a smaller $S$ if necessary, we may assume $g_{\star}\mathcal{O}_{S'}\cong\mathcal{O}_S$. Consequently, $\pi_{\star} \mathcal{O}_C=g_{\star}\mathcal{O}_{S'}=\mathcal{O}_S$. 

To summarize, we may assume that $\pi: C \rightarrow S$ has the following properties:
\begin{itemize}
\item $\pi_{\star}\mathcal{O}_C=\mathcal{O}_S,$
\item all the fibers of $\pi$ are connected.
\end{itemize}
(In fact, it is shown in the proof of \cite[Proposition 3.3.7]{Ber90} that the second point is a consequence of the first one, which we will now use to prove that the fibers are \textit{geometrically} connected.) 

Suppose there exists $s \in S$ such that the fiber $C_s$ is not geometrically connected. Then there exists a finite separable extension $L$ of $\mathcal{H}(s)$ such that $C_s \times_{\mathcal{H}(s)} L$ is not connected. By \cite[Theorem 3.4.1]{ber93}, there exists a finite \'etale morphism $h : U \rightarrow S$ of $k$-affinoid spaces and a point $u \in U$ in the preimage of $s$ such that the induced embedding $\mathcal{H}(s) \subseteq \mathcal{H}(u)$ is none other than $\mathcal{H}(s) \subseteq L$.    

Let us consider the commutative diagram (\ref{njeshi}) below, where $C_U:=C \times_S U.$ Seeing as~$h$ is flat, by using \cite[Proposition 2.3.1]{Ber90}, we obtain just as in the case of schemes (see \cite[Tag 02KH]{stacks-project}) that $(\pi_{U})_{\star}h_U^{\star} = h^{\star}\pi_{\star}$ (see also \cite[Corollary 5.3.6]{ber93}), thus implying that $(\pi_U)_{\star} \mathcal{O}_{C_U} = \mathcal{O}_{U}.$ By the proof of \cite[Proposition 3.3.7]{Ber90}, this implies that the fibers of $\pi_U$ are all connected.

\begin{minipage}{0.45\textwidth}
\begin{equation} \label{njeshi}
\begin{tikzcd}
C_U \arrow{r}{\pi_U} \arrow{d}[swap]{h_U}& U \arrow{d}{h} \\
C  \arrow{r}[swap]{\pi} & S
\end{tikzcd}
\end{equation}
\end{minipage}
\hfill
\begin{minipage}{0.45\textwidth}
\begin{equation} \label{dyshi}
\begin{tikzcd}[column sep=small] 
S'' \arrow{dr}[swap]{\psi} \arrow{rr}{\varphi} && C \arrow{dl}{\pi} \\
& S 
\end{tikzcd}
\end{equation}
\end{minipage}

On the other hand, the fiber $\pi_U^{-1}(u)$ of $u$ is the curve $C_s \times_{\mathcal{H}(s)} L$ and is thus by assumption not connected, contradiction. Hence, the fibers of $\pi$ are all geometrically connected.

\

\emph{Algebraicity.} By \cite[Corollary 6.2.7]{famduc}, as $\pi$ is a smooth morphism, there exists an \'etale morphism of $k$-analytic spaces $\psi: S'' \rightarrow S$ whose image contains $x$ and such that there is an $S$-morphism ${\varphi: S'' \rightarrow C}$ (see diagram (\ref{dyshi}) above). By \cite[5.2.16]{famduc}, by restricting to a smaller~$S$ if necessary, we may assume that $\psi : S'' \rightarrow S$ is a \textit{finite} (hence closed) and \'etale (hence open by  \mbox{\cite[Corollary 6.2.5]{famduc}}) morphism of $k$-affinoid spaces.

As $S''$ is a $k$-affinoid space, its connected components $(S_i)_{i\in I}$ are finite in number (by compactness) and are also $k$-affinoid spaces (\cite[Corollary 2.2.7(i)]{Ber90}). By \cite[Theorem ~11.3.3(3b)]{famduc}, $S''$ is normal, so by \cite[Proposition 5.14]{dex}, $(S_i)_i$ are the irreducible components of~$S''$. As $\psi$ is  both an open and closed morphism, for any $i \in I,$ $\psi(S_i)$ is an open and closed subset of $S$. Since the latter is connected, $\psi(S_i)=S$ and $\psi$ is surjective. By \mbox{\cite[1.5.10]{famduc}}, $\dim S_i=\dim \psi(S_i)=\dim S$ for all $i$. As $S_i$ is irreducible, it is pure-dimensional by \mbox{\cite[1.5.1]{famduc}}. Consequently, $S''$ is pure-dimensional with ${\dim S=\dim S''}.$

Set $C''=C \times_S S''$, and let us name the projections as follows: $\pi'': C'' \rightarrow S''$ and $p: C'' \rightarrow C.$ By construction, $\pi''$ has a section $s: S'' \rightarrow C''$. As $\pi''$ is separated, $s$ is a closed immersion (the proof of \cite[Proposition 3.3.9(f)]{liulibri} can be applied \emph{mutatis mutandis} to show this). Also, $p$ is a finite \'etale morphism (as a base change of $\psi$). As $\varphi=p \circ s$, we obtain that $\varphi$ is a finite morphism.

By \cite[1.5.10]{famduc}, $D:=\varphi(S'')$ is a pure-dimensional Zariski closed subset of $C$ such that $\dim D=\dim S''=\dim S.$ As $C$ is irreducible (see the proof of Lemma \ref{206}), it is also pure-dimensional, and by \cite[1.4.14(3)]{famduc}, $\dim D=\dim S=\dim C-1.$ Consequently, $D$ is a Zariski-closed subset of codimension one in $C.$

 By \cite[Corollary 3.2.9]{famduc}, if $X$ is a good analytic space and $a \in X$ a point, then there exists $\mathrm{centdim}(X,a) \in \mathbb{N} \cup \{0\}$ such that $\mathrm{centdim}(X,a) + \dim \mathcal{O}_{X,a}=\dim_a X$, where $\dim_a X$ is the \emph{local analytic} dimension of $X$ at $a$ (see \cite[1.4.9]{famduc}). 
In our case, we have that for any $z \in D,$
$$\mathrm{centdim}(D,z) + \dim \mathcal{O}_{D,z}=\dim_z D$$
$$\mathrm{centdim}(C,z) + \dim \mathcal{O}_{C,z}=\dim_z C.$$
As both $D$ and $C$ are pure-dimensional, $\dim_z D=\dim D$ and $\dim_z C=\dim C,$ so $\dim_z D=\dim_z C-1$. By \cite[3.2.4]{famduc}, $\mathrm{centdim}(D,z)=\mathrm{centdim}(C,z),$ implying that for any $z \in D,$ $\dim \mathcal{O}_{D,z}=\dim \mathcal{O}_{C,z}-1.$ 

Let us denote by $\mathcal{I}$ the ideal sheaf of $\mathcal{O}_C$ defining $D$. Seeing as $\mathcal{O}_{D,z}=\mathcal{O}_{C,z}/\mathcal{I}_z$ and $C$ is regular, by \cite[Tag 00NA]{stacks-project}, $\mathcal{I}_z$ is a height one prime ideal of $\mathcal{O}_{C,z}$; it is hence principal with a non-zero divisor (see \cite[Proposition 4.2.11]{liulibri}) as a generator. Consequently, $D$ is an \textit{effective Cartier} divisor on ~$C$ and it induces a line bundle $\mathcal{L}$ there. 

Let us remark that for any $y \in S$, by the surjectivity of $\psi,$ $D_y:=D \cap C_y \neq \emptyset.$ Moreover, seeing as $\psi$ is finite, $\psi^{-1}(y)=\sqcup_{i=1}^n \mathcal{M}(L_i),$ where $L_i$ is a finite field extension of $\mathcal{H}(y)$ for all $i$. This implies that $D_y$ is a finite set of Zariski closed points of $C_y,$ and hence that it is an effective Cartier divisor of \textit{positive} degree on $C_y$. By \cite[Proposition 7.5.5]{liulibri}, $\mathcal{L}_y$ is an \textit{ample} line bundle on $C_y$. 

By \cite[Proposition 3.1.2]{Ber90}, as $S$ is strict, so is $C$. Finally, thanks to \cite[Theorem 1.6.1]{ber93} (which is where the strictness hypothesis is used), we conclude by \cite[Theorem 3.2.7]{conrad} that $\pi: C \rightarrow S$ is algebraic (as $\mathcal{L}$ is an $S$-\textit{ample line bundle} on $C$).   

This concludes the proof of Theorem \ref{addition1}.
\end{proof}

If we know that the algebraicity condition of Theorem \ref{addition1} is satisfied, then some of its hypotheses can be slightly relaxed.

\begin{prop} \label{example lame'}
Let $S$ be a normal $k$-affinoid space. Let $\pi : C \rightarrow S$ be a proper flat relative analytic curve. Let $x \in S$ be such that $C_x:=\pi^{-1}(x) \neq \emptyset$ and $\mathcal{O}_x$ is a field.

If the fiber $C_x$ over $x$ is a smooth 
geometrically irreducible $\mathcal{H}(x)$-analytic curve and there exists an affinoid 
neighborhood $Z'$ of $x$ in $S$ such that ${\pi_{Z'}: C 
\times_S Z' \rightarrow Z'}$ is algebraic, then the 
conditions of Setting \ref{200} are satisfied. 
\end{prop}

\begin{proof} By Remark \ref{200bis}(1), $C$ is a good $k$-analytic space. We remark that since $\pi$ is a proper morphism, $\pi_{Z'}$ is also proper. Similarly, $C_x$ is a \textit{proper} $\mathcal{H}(x)$-analytic curve. Set ${C_{Z'}=C \times_S Z'}.$ As assumed in the 
statement, there exists a proper scheme $
{\pi_{\mathcal{O}(Z')}: C_{\mathcal{O}(Z')} 
\rightarrow \text{Spec} \ \mathcal{O}(Z')}$ such 
that its analytification is $\pi_{Z'}.$ Since $\pi$ is flat, so is $\pi_{Z'},$ and hence $\pi_{\mathcal{O}(Z')}$ (see \cite[Lemma 4.2.1]{famduc}).

Let $\psi: Z' \rightarrow \text{Spec} \ \mathcal{O}(Z')$ denote the canonical analytification morphism. For any $y \in Z',$ we denote $y'=\psi(y).$ For any $y' \in \text{Spec} \ \mathcal{O}(Z'),$ let $\kappa(y')$ denote the corresponding residue field and $C_{\kappa(y')}$ its fiber with respect to $\pi_{\mathcal{O}(Z')}.$  By the proof of \cite[Proposition ~2.6.2]{ber93}, $C_x$ is isomorphic to the analytification of ${C_{\kappa(x')} \times_{\kappa(x')} \mathcal{H}(x)}=:C_{x}^{\mathrm{alg}}.$ Hence, $C_x^{\mathrm{alg}}$ is a projective $\mathcal{H}(x)$-algebraic curve that is smooth (\textit{i.e.} geometrically normal) and geometrically irreducible, implying the $\kappa(x')$-algebraic curve $C_{\kappa(x')}$ has the same properties. 

Since $\text{Spec} \ \mathcal{O}(Z')$ is Noetherian, the proper morphism $\pi_{\mathcal{O}(Z')}$ is of finite presentation. By \cite[Th\'eor\`eme 12.2.4]{ega43}, the set of points $A$ of $\text{Spec} \ \mathcal{O}(Z')$ such that for any $y' \in A,$ $C_{\kappa(y')}$ is smooth and geometrically integral is a Zariski open. Remark that $x' \in A,$ so $A \neq \emptyset.$ Consequently, $\psi^{-1}(A)$ is a non-empty Zariski open subset of $Z'$ containing $x.$

Let $Z_0 \subseteq Z'$ be a connected affinoid neighborhood of $x$ in $S$ such that $Z_0 \subseteq \psi^{-1}(A).$ As remarked before, for any $y \in Z_0,$  $C_y=(C_{\kappa(y')} \times_{\kappa(y')} \mathcal{H}(y))^{\mathrm{an}},$ where $y'=\psi(y) \in A.$ Since then $C_{\kappa(y')}$ is smooth and geometrically irreducible, the same is true for $C_y.$ Set $C_{\mathcal{O}(Z_0)}=C_{\mathcal{O}(Z')} \times_{\mathcal{O}(Z')} 
\mathcal{O}(Z_0).$ By \cite[Proposition ~2.6.1]{ber93},
$$(C_{\mathcal{O}(Z_0)})^{\mathrm{an}}
=(C_{\mathcal{O}(Z')} \times_{\mathcal{O}(Z')} 
\mathcal{O}(Z_0))^{\mathrm{an}}=C_{Z'} \times_{Z'} 
Z=:C_{Z_0},$$
implying $\pi_{Z_0}: C_{Z_0} \rightarrow Z_0$ is algebraic, and that the conditions of Setting \ref{200} are satisfied. 
\end{proof}

\subsection{Example: realization of an algebraic curve over $\mathcal{O}_x$ as the thickening of a curve over $\mathcal{H}(x)$} \label{lame}
Let $S'$ be a normal good $k$-analytic space. Let $x \in S'$ be such that $\mathcal{O}_x$ is a field. Let $C_{\mathcal{O}_x}$ be a smooth geometrically irreducible projective algebraic curve over ~$\mathcal{O}_x.$ Let us denote by $\pi_x$ the structural morphism $C_{\mathcal{O}_x} \rightarrow \text{Spec} \ \mathcal{O}_x.$ 

\begin{prop}
There exists a connected affinoid neighborhood $S$ of $x$ in $S'$ and an $S$-relative analytic curve $\pi : C \rightarrow S$ satisfying the properties of Setting \ref{200}.
\end{prop}

\begin{proof}

 Remark that $\mathcal{O}_x= \varinjlim_Z \mathcal{O}(Z),$ where the limit is taken over connected affinoid neighborhoods $Z$ of $x$ in ~$S'$, implying $\text{Spec} \ \mathcal{O}_x=\varprojlim_Z \text{Spec} \ \mathcal{O}(Z)$.  
By \cite[Th\'eor\`eme ~8.8.2]{ega43}, there exists a connected affinoid neighborhood $Z_0$ of $x,$ such that
for any connected affinoid neighborhood $Z \subseteq Z_0$ of $x,$ there exists a finitely presented scheme $C_{\mathcal{O}(Z)}$ over $\text{Spec} \ \mathcal{O}(Z)$ satisfying
${C_{\mathcal{O}(Z)} \times_{\text{Spec} \ \mathcal{O}(Z)} \text{Spec} \ \mathcal{O}_x=C_{\mathcal{O}_x}}.$ Let us denote by $\pi_{\mathcal{O}(Z)}$ the structural morphism $C_{\mathcal{O}(Z)} \rightarrow \text{Spec} \ \mathcal{O}(Z).$

Remark that $\pi_x$ is a proper flat morphism. The affinoid domain $Z_0$ can be chosen so that for any connected affinoid neighborhood $Z \subseteq Z_0$ of $x,$ the morphism $\pi_{\mathcal{O}(Z)}: {C_{\mathcal{O}(Z)} \rightarrow \text{Spec} \ \mathcal{O}(Z)}$ remains proper (by \cite[Th\'eor\`eme 8.10.5]{ega43}) and flat (by \cite[Tag ~04AI]{stacks-project}). Furthermore, by 
\cite[Tag 0EY2]{stacks-project}, we may assume that $C_{\mathcal{O}(Z)}$ is a relative curve over $\mathcal{O}(Z).$ Let $C_Z$ (defined over $Z$) denote the Berkovich analytification of the proper scheme $C_{\mathcal{O}(Z)}$
 over $\text{Spec} \ \mathcal{O}(Z)$ (in the sense of \cite[2.6]{ber93}). We denote by $\pi_Z: C_Z \rightarrow Z$ the analytification of $\pi_{\mathcal{O}(Z)}.$ We remark that by \cite[Proposition ~2.6.1]{ber93},
$$C_Z=(C_{\mathcal{O}(Z)})^{\mathrm{an}}
=(C_{\mathcal{O}(Z_0)} \times_{\mathcal{O}(Z_0)} 
\mathcal{O}(Z))^{\mathrm{an}}=C_{Z_0} \times_{Z_0} 
Z=\pi_{Z_0}^{-1}(Z).$$
Let $C_x$ denote the fiber of $x$ via $\pi_Z$ (which clearly does not depend on $Z$). It is a {\em proper} $\mathcal{H}(x)$-analytic curve.

Let $x'$ denote the image of $x$ via the analytification $Z_0 \rightarrow \text{Spec} \ \mathcal{O}(Z_0)$, $\kappa(x')$ its residue field and $C_{\kappa(x')}$ its fiber via $\pi_{\mathcal{O}(Z_0)}.$ Since $\mathcal{O}_x$ is a field, there exist natural embeddings $ \kappa(x') \hookrightarrow \mathcal{O}_x \hookrightarrow \mathcal{H}(x),$ from where we obtain that ${C_{\kappa(x')} \times_{\kappa(x')}  \mathcal{O}_x=C_{\mathcal{O}_x}}.$ Since $C_{\mathcal{O}_x}$ is smooth (\textit{i.e.} geometrically normal) and geometrically irreducible, so is $C_{\kappa(x')}.$ 
By the proof of \cite[Proposition ~2.6.2]{ber93}, $C_x$ is isomorphic to the analytification of ${C_{\kappa(x')} \times_{\kappa(x')} \mathcal{H}(x)},$ implying $C_x$ is a smooth geometrically irreducible $\mathcal{H}(x)$-analytic curve.

To summarize, we have a proper flat relative analytic curve $\pi_{Z_0}:C_{Z_0} \rightarrow Z_0$, which is algebraic. As $Z_0$ is an affinoid domain of the normal space $S$, by \cite[Th\'eor\`eme 3.4]{dex}, it is also normal.  Finally, $x \in Z_0$, and the corresponding fiber $C_x$ is a smooth geometrically irreducible $\mathcal{H}(x)$-analytic curve. We can now conclude by applying Proposition \ref{example lame'}.
\end{proof}

\subsection{Consequences of Setting \ref{200}}
Recall Setting \ref{200} and Remark~\ref{200bis}. 
\begin{nota} \label{202}
Let $Z \subseteq Z_0$ be any connected affinoid neighborhood  of $x$ in $S$. 
\begin{itemize}
\item  We denote by $\pi_Z$ the structural morphism ${C_Z:=C \times_S Z \rightarrow Z},$ and by $\pi_{\mathcal{O}(Z)}$  the one ${C_{\mathcal{O}(Z)}:=C_{\mathcal{O}(Z_0)} \times_{\text{Spec} \ \mathcal{O}(Z_0)}  \text{Spec} \ \mathcal{O}(Z)\rightarrow  \text{Spec} \ \mathcal{O}(Z)}.$

\item For any $y \in Z,$ the fiber $\pi_Z^{-1}(y)$ can be endowed with the structure of an $\mathcal{H}(y)$-analytic curve $C_y:=C_Z \times_Z \mathcal{H}(y)$. Remark that $C_y$ does not depend on $Z$. 
\item For any $y' \in \text{Spec} \ \mathcal{O}(Z),$ the fiber $\pi_{Z}^{-1}(y')$ can be endowed with the structure of a $\kappa(y')$-algebraic curve $C_{\mathcal{O}(Z),\kappa(y')}:=C_{\mathcal{O}(Z)} \times_{\mathcal{O}(Z)} \kappa(y'),$ where $\kappa(y')$ denotes the residue field of $y'$ in $\text{Spec} \ \mathcal{O}(Z).$ We will use the notation $C_{\kappa(y')}$ whenever there is no risk of ambiguity.
\end{itemize} 
\end{nota}

\begin{rem}
By \cite[Proposition 2.6.1]{ber93}, $C_Z$ is the analytification of $C_{\mathcal{O}(Z)}$:
$$(C_{\mathcal{O}(Z)})^{\mathrm{an}}=(C_{\mathcal{O}(Z_0)} \times_{\mathcal{O}(Z_0)} \mathcal{O}(Z))^{\mathrm{an}}=C_{Z_0} \times_{Z_0} Z=C_Z.$$
\end{rem}

\begin{prop} \label{204} Let $Z \subseteq Z_0$ be a connected affinoid neighborhood of $x.$ 
\begin{enumerate}
\item
The space $C_Z$ is a normal proper flat relative analytic curve over $Z.$ Furthermore, $\pi_Z$ is surjective. The same properties are true for $C_{\mathcal{O}(Z)}$ and $\pi_{\mathcal{O}(Z)}.$
\item Any connected affinoid domain of $C_Z$ is normal and irreducible. 
\end{enumerate}
\end{prop}

\begin{proof}
Since $\pi_Z$ is a base change of $\pi:C \rightarrow S,$ we immediately obtain that $\pi_Z$ is proper, surjective, flat, and of relative dimension 1. Since $S$ is normal, by \cite[Th\'eor\`eme 3.4]{dex}, $Z$ is normal. As the fibers of $\pi_Z$ are assumed normal, by \cite[Theorem 11.3.3(3)]{famduc}, $C_Z$ is normal. Then $C_{\mathcal{O}(Z)}$ is also normal by \cite[Th\'eor\`eme 3.4]{dex}. By \cite[Proposition 2.6.9]{ber93}, $\pi_{\mathcal{O}(Z)}$ is proper. Its flatness is a consequence of \cite[Lemma 4.2.1]{famduc}.  Surjectivity can be obtained from the surjectivity of $\pi_Z$ as in Proposition ~3.4.6(7) of \cite{Ber90}. The relative dimension of $\pi_{\mathcal{O}(Z)}$ is the same as that of $\pi_Z$ by \cite[Proposition 2.7.7]{famduc}.

Any connected affinoid domain of $C_Z$ is normal by \cite[Th\'eor\`eme 3.4]{dex} and irreducible by \cite[Th\'eor\`eme 5.17]{dex}.  
\end{proof}

The object the following lemma deals with will be central for the rest of this paper: set $C_{\mathcal{O}_x}:=C_{\mathcal{O}(Z_0)} \times_{\mathcal{O}(Z_0)} \mathcal{O}_x.$

\begin{lm} \label{jeeeej}
The scheme $C_{\mathcal{O}_x}$ is an irreducible normal projective $\mathcal{O}_x$-algebraic curve.
\end{lm}

\begin{proof}
Let $C_x$ denote the fiber of $\pi_{Z_0}: C_{Z_0} \rightarrow Z_0.$ It is a normal irreducible projective analytic $\mathcal{H}(x)$-curve by definition. 
Let $x'$ denote the image of $x$ via the analytification morphism $\psi: Z_0 \rightarrow \text{Spec} \ \mathcal{O}(Z_0).$ By the proof of \cite[Proposition 2.6.2]{ber93}, $C_x\cong (C_{\kappa(x')} \times_{\kappa(x')} \mathcal{H}(x))^{\mathrm{an}},$ where $\kappa(x')$ denotes the residue field of $x$ in $\mathrm{Spec} \ \mathcal{O}(Z_0),$ and $C_{\kappa(x')}:=C_{\mathcal{O}(Z_0)} \times_{\mathcal{O}(Z_0)} \kappa(x')$ - the algebraic fiber of $x'$ with respect to $C_{\mathcal{O}(Z_0)} \rightarrow \text{Spec} \ \mathcal{O}(Z_0).$

 Set $C_x^{\mathrm{alg}}:=C_{\kappa(x')} \times_{\kappa(x')} \mathcal{H}(x).$ 
Seeing as $\psi(x)=x'$ and $\mathcal{O}_x$ is a field, there exist canonical embeddings $\kappa(x')\hookrightarrow \mathcal{O}_x \hookrightarrow \mathcal{H}(x).$ Consequently, $C_{\mathcal{O}_x}=C_{\kappa(x')} \times_{\kappa(x')} \mathcal{O}_x$ and $C_x=(C_{\mathcal{O}_x} \times_{\mathcal{O}_x} \mathcal{H}(x))^{\mathrm{an}}.$ As $C_x$ is a normal irreducible analytic $\mathcal{H}(x)$-curve, $C_{\mathcal{O}_x} \times_{\mathcal{O}_x} \mathcal{H}(x)$ is  a connected (\cite[Thm. ~3.5.8(iii)]{Ber90})  normal algebraic curve (\cite[Prop. 3.4.3]{Ber90}) over $\mathcal{H}(x).$ Consequently, $C_{\mathcal{O}_x}$ is connected, and by \cite[Corollaire 6.5.4]{groth}, it is normal.
Properness is immediate seeing as $C_{\mathcal{O}_x} \rightarrow \text{Spec} \ \mathcal{O}_x$ is a base change of a proper morphism. 
\end{proof}

Recall Notation \ref{202}. A very important property for the constructions we make is the following:

\begin{lm} \label{208}
For any non-rigid point $\eta$ of $C_x,$ the local ring $\mathcal{O}_{C, \eta}$ is a field. If $\eta \in C_x$ is rigid, then $\mathcal{O}_{C, \eta}$ is a discrete valuation ring.  
\end{lm}

In particular, this implies that for any type 3 point $\eta \in C_x,$ the local ring $\mathcal{O}_{C,\eta}$ is a field. 

\begin{proof}
Seeing as $x \in \text{Int} \ {Z_0},$ for any $\eta \in C_x,$ $\eta \in \text{Int} \ C_{Z_0},$ so $\mathcal{O}_{C,\eta}=\mathcal{O}_{C_{Z_0}, \eta},$ and we can use the two interchangeably. The morphism $\pi_{Z_0}:C_{Z_0} \rightarrow Z_0$ is proper, so boundaryless. As $\pi_{Z_0}$ is flat, by the proof of \cite[Lemma 4.5.11]{famduc}, $\dim{\mathcal{O}_{C, \eta}}=\dim{\mathcal{O}_{C_x, \eta}} + \dim{\mathcal{O}_{x}}.$  Since $\mathcal{O}_x$ is a field, we obtain $\dim{\mathcal{O}_{C, \eta}}=\dim{\mathcal{O}_{C_x, \eta}}.$

By \cite[Lemma 4.4.5]{famduc}, if $\eta \in C_x$ is not rigid, then $\mathcal{O}_{C_x,\eta}$ is a field, implying $\dim{\mathcal{O}_{C, \eta}}=0,$ so $\mathcal{O}_{C, \eta}$ is a field (recall $C_{Z_0}$ is normal). If $\eta \in C_x$ is rigid, by \textit{loc.cit.} $\mathcal{O}_{C_x, \eta}$ is a discrete valuation ring, implying  $\dim{\mathcal{O}_{C, \eta}}=1.$ Hence, $\mathcal{O}_{C, \eta}$
 is a Noetherian normal local ring with Krull dimension 1, meaning a discrete valuation ring.
\end{proof}

\begin{lm} \label{205} Let $Z \subseteq Z_0$ be a connected affinoid neighborhood of $x.$
For any pair of different points $u_1, u_2 \in C_{Z}$, there exist neighborhoods $B_1$ of $u_1$ and $B_2$ of $u_2$ in $C_{Z},$ such that $B_1 \cap B_2=\emptyset.$
\end{lm}

\begin{proof}
Seeing as $\pi_Z$ is proper, it is separated. Seeing as $Z$ is Hausdorff, by \cite[Proposition ~3.1.5]{Ber90}, $Z \rightarrow \mathcal{M}(k)$ is separated. Consequently, the canonical morphism $C_Z \rightarrow \mathcal{M}(k)$ is separated, and we can conclude by \textit{loc.cit.} 
\end{proof}

\begin{lm} \label{206}
Let $Z \subseteq Z_0$ be a connected affinoid neighborhood of $x.$ The spaces $C_Z, C_{\mathcal{O}(Z)}$ are irreducible.  
\end{lm}

\begin{proof}
Since all the fibers of $C_Z \rightarrow Z$ are connected, $C_Z$ is connected: if, by contradiction, $C_Z$ can be written as the disjoint union of two closed (hence compact) subsets $U$ and $V$, then $Z=\pi_Z(U) \cup \pi_Z(V).$ Since $\pi_Z(U)$ and $\pi_Z(V)$ are compact, and $Z$ is connected, their intersection is non-empty. Consequently, there exists $y \in Z,$ such that $C_y \cap U \neq \emptyset$ and $C_y \cap V \neq \emptyset.$ Since $C_y$ is connected and covered by the compacts $C_y \cap U$, $C_y \cap V,$ this is a contradiction. Thus, $C_Z$ is a connected normal analytic space. By \cite[Proposition 5.14]{dex}, it is irreducible. Then by \cite[Proposition 2.7.16]{famduc}, $C_{\mathcal{O}(Z)}$ is also irreducible.      
\end{proof}

\begin{prop} \label{207}
There exists a connected affinoid neighborhood $Z_1 \subseteq Z_0$ of $x$ such that for any connected affinoid neighborhood $Z \subseteq Z_1$ of $x,$ there exists a finite surjective morphism $f_Z: C_Z \rightarrow \mathbb{P}_Z^{1, \mathrm{an}},$ satisfying:
\begin{enumerate}
\item $f_Z$ is the analytification of a finite surjective morphism $f_{\mathcal{O}(Z)}: C_{\mathcal{O}(Z)} \rightarrow \mathbb{P}_{\mathcal{O}(Z)}^{1, \mathrm{an}};$
\item for any connected affinoid neighborhood $Z'  \subseteq Z$ of $x$, ${f_{Z} \times_Z Z'=f_{Z'}},$ \textit{i.e.} Diagram (\ref{diagramiipare}) (where the horizontal arrows correspond to the base change ${Z' \hookrightarrow Z}$) is commutative.
\begin{equation} \label{diagramiipare}
\begin{tikzcd}
C_{Z'} \arrow{r} \arrow{d}[swap]{f_{Z'}}& C_{Z} \arrow{d}{f_{Z}} \\
 \mathbb{P}_{Z'}^{1, \mathrm{an}} \arrow{r} &\mathbb{P}_{Z}^{1, \mathrm{an}}
\end{tikzcd}
\end{equation}
\end{enumerate}
\end{prop}

\begin{proof}
Remark that $\mathcal{O}_x=\varinjlim_Z \mathcal{O}(Z),$ where the limit is taken with respect to connected affinoid neighborhoods $Z \subseteq Z_0$ of $x.$ Consequently, $\text{Spec} \ \mathcal{O}_x= \varprojlim_Z \text{Spec} \ \mathcal{O}(Z)$, and $C_{\mathcal{O}_x}=C_{\mathcal{O}(Z_0)} \times_{\mathcal{O}(Z_0)} \mathcal{O}_x =C_{\mathcal{O}(Z_0)} \times_{\mathcal{O}(Z_0)} \varprojlim_Z \mathcal{O}(Z) =\varprojlim_Z C_{\mathcal{O}(Z)}.$ Recall that $C_{\mathcal{O}_x}$ is an irreducible normal projective curve (Lemma \ref{jeeeej}).

Let $f_{\mathcal{O}_x}: C_{\mathcal{O}_x} \rightarrow \mathbb{P}_{\mathcal{O}_x}^{1}$ be any  finite non-constant (hence surjective) morphism. By  \cite[Th\'eor\`eme 8.8.2]{ega43}, we may assume that $Z_0$ is such that for any connected affinoid neighborhood $Z \subseteq Z_0$ of $x$, there exists a morphism $f_{\mathcal{O}(Z)} : C_{\mathcal{O}(Z)} \rightarrow \mathbb{P}_{\mathcal{O}(Z)}^1$ such that Diagram (\ref{doublediag}) (where the horizontal arrows are the corresponding base changes) is commutative for any connected affinoid neighborhood $Z' \subseteq Z$ of $x.$

\begin{equation} \label{doublediag}
\begin{tikzcd} 
C_{\mathcal{O}_x} \arrow{r} \arrow{d}{f_{\mathcal{O}_x}} & C_{\mathcal{O}(Z')} \arrow{r} \arrow{d}{f_{\mathcal{O}(Z')}}& C_{\mathcal{O}(Z)} \arrow{d}{f_{\mathcal{O}(Z)}} \\
\mathbb{P}_{\mathcal{O}_x}^1 \arrow{r} & \mathbb{P}_{\mathcal{O}(Z')}^1 \arrow{r} &\mathbb{P}_{\mathcal{O}(Z)}^1
\end{tikzcd}
\end{equation}
Furthermore, by \cite[Th\'eor\`eme 8.10.5]{ega43}, $Z_0$ can be chosen so that for any connected affinoid neighborhood $Z \subseteq Z_0$ of $x$, the morphism $f_{\mathcal{O}(Z)}$ is finite and surjective.

 Let ${f_Z : C_Z \rightarrow \mathbb{P}_{Z}^{1, \mathrm{an}}}$ denote the Berkovich analytification of $f_{\mathcal{O}(Z)}$ in the sense of \mbox{\cite[2.6]{ber93}}. Then as in \cite[Proposition 3.4.6(7)]{Ber90}, $f_Z$ is surjective; by \cite[Proposition ~2.6.9]{ber93}, it is finite.  

Part (2) is a direct consequence of the commutativity of Diagram (\ref{doublediag}) above.  
\end{proof}

Remark that the finite surjective morphism $f_Z: C_Z \rightarrow \mathbb{P}_Z^{1, \mathrm{an}}$ induces a finite surjective morphism $f_z: C_z \rightarrow \mathbb{P}_{\mathcal{H}(z)}^{1, \mathrm{an}}$ between the fibers of $z \in Z$ in $C_Z$ and $\mathbb{P}_Z^{1, \mathrm{an}},$ respectively (recall Notation \ref{202}). 

\begin{prop} \label{209}
Let $Z \subseteq Z_0$ be a connected affinoid neighborhood of $x.$ Let $y$ be a type ~3 point 
on the fiber $\mathbb{P}_{\mathcal{H}(x)}^{1, \mathrm{an}}$ of $x$ in $\mathbb{P}_Z^{1, \mathrm{an}}.$ Let $\{z_1, z_2, \dots, z_n\}:=f_Z^{-1}(y).$ Then $$\mathscr{M}_{\mathbb{P}_Z^{1, \mathrm{an}}, y} \otimes_{\mathscr{M}(Z)(T)} \mathscr{M}(C_Z) = \prod_{i=1}^n \mathscr{M}_{C_Z, z_i}.$$
\end{prop}

\begin{proof}
Let us look at the finite surjective morphism $f_{\mathcal{O}(Z)}: C_{\mathcal{O}(Z)} \rightarrow \mathbb{P}_{\mathcal{O}(Z)}^{1}$ of $\mathcal{O}(Z)$-schemes. 
Let $y'$ be the image of $y$ via the analytification $\psi: \mathbb{P}_Z^{1, \mathrm{an}} \rightarrow \mathbb{P}_{\mathcal{O}(Z)}^1.$ Let ${\mathcal{A}:=\text{Spec} \ A}$ be an open affine neighborhood of $y'$ in $\mathbb{P}_{\mathcal{O}(Z)}^1.$ Its preimage by $\psi$ is a Zariski open $\mathcal{A}'$ of $\mathbb{P}_Z^{1, \mathrm{an}}$ containing $y.$ 

Let $\mathcal{B}:=\text{Spec} \ B$ be the pre-image of $\mathcal{A}$ by $f_{\mathcal{O}(Z)}.$ It is an affine open subset of $C_{\mathcal{O}(Z)},$ and $f_{\mathcal{O}(Z)}$ induces a finite surjective morphism $\mathcal{B} \rightarrow \mathcal{A}.$ By construction, $\mathcal{B}$ contains $f_{\mathcal{O}(Z)}^{-1}(y').$ By the proof of \cite[Proposition 2.6.10]{ber93},  there is an isomorphism ${\prod_{i=1}^n \mathcal{O}_{C_Z, z_i}\cong  \mathcal{O}_{\mathbb{P}_Z^{1, \mathrm{an}}, y}\otimes_A B}$. Since $C_{\mathcal{O}(Z)}$ and $\mathbb{P}_{\mathcal{O}(Z)}^1$ are irreducible, the function field of $C_{\mathcal{O}(Z)}$ is $\text{Frac} \ B,$ and the function field of $\mathbb{P}_{\mathcal{O}(Z)}^1$ is $\text{Frac} \ A.$

By Theorem \ref{231}, we obtain that $\mathscr{M}(C_Z)= \text{Frac} \ B,$ and ${\mathscr{M}(\mathbb{P}_Z^{1, \mathrm{an}})=\text{Frac} \ A}.$ Since $B$ is a finite $A$-module, by the last paragraph of the proof of \cite[Lemma 3.4]{une}, ${\prod_{i=1}^n \mathcal{O}_{C_Z, z_i}=  \mathcal{O}_{\mathbb{P}_Z^{1, \mathrm{an}}, y}\otimes_{\text{Frac} \ A} \text{Frac} \ B},$ so $\prod_{i=1}^n \mathcal{O}_{C_Z, z_i}=  \mathcal{O}_{\mathbb{P}_Z^{1, \mathrm{an}}, y}\otimes_{\mathscr{M}(Z)(T)} \mathscr{M}(C_Z).$
Finally, since $y$ and ~$z_i, i=1,2,\dots, n,$ 
are type ~3 points in $\mathbb{P}_{\mathcal{H}
(x)}^{1, \mathrm{an}}$ and $C_x$, respectively, $\mathcal{O}
_{\mathbb{P}_Z^{1, \mathrm{an}}, y}=
\mathscr{M}_{\mathbb{P}_Z^{1, \mathrm{an}}, 
y}$ and $\mathcal{O}_{C_Z, z_i}= \mathscr{M}
_{C_Z, z_i}$ for all $i,$ concluding the proof of the statement. 
\end{proof}

\begin{prop} \label{210}
For any connected affinoid neighborhoods $Z, Z' \subseteq Z_0$ of $x$ such that $Z' \subseteq Z,$ the base change morphism $\iota_{Z, Z'} : C_{\mathcal{O}(Z')} \rightarrow C_{\mathcal{O}(Z)}$ is dominant. If $\eta_{Z}$ (resp. $\eta_{Z'}$) is the generic point of $C_{\mathcal{O}(Z)}$ (resp. $C_{\mathcal{O}(Z')}$), then $\iota_{Z, Z'}(\eta_{Z'})=\eta_{Z}.$ 
\end{prop}

\begin{proof}
By Lemma \ref{206}, $C_{\mathcal{O}(Z)}, C_{\mathcal{O}(Z')}$ are irreducible, so it makes sense to speak of their generic points $\eta_Z, \eta_{Z'},$ respectively. It suffices to show that $\eta_Z$ is in the image of $\iota_{Z, Z'}.$ Let ~$\alpha$ be any point of $C_Z.$ Let $\alpha'$ be its image in $C_{\mathcal{O}(Z)}$ via the analytification $\phi: C_Z \rightarrow C_{\mathcal{O}(Z)}.$ Let $U$ be an open affine neighborhood of $\alpha'$ in $C_{\mathcal{O}(Z)}.$ Then $\eta_Z \in U,$ and the closure of $\{\eta_Z\}$ in $U$ is $U,$ meaning it is the generic point of $U$. 

By \cite[Proposition 2.6.8]{ber93}, $\phi^{-1}(U)=U^{\mathrm{an}}$-the analytification of $U$. Remark that  $U^{\mathrm{an}}$ is an open subspace of $C_Z.$ Let $B_{\alpha}$ be any open neighborhood of $\alpha$ in $C_Z.$ Then since $\alpha \in U^{\mathrm{an}},$ $B_{\alpha} \cap U^{\mathrm{an}}$ is an open neighborhood of $\alpha$ in $U^{\mathrm{an}},$ so by \cite[Lemma 2.6.5]{ber93}, there exists a point $\beta \in B_{\alpha} \cap U^{\mathrm{an}} \subseteq B_{\alpha},$ such that $\phi(\beta)=\eta_{Z}.$ Thus, for any point $\alpha \in C_Z$ and any open neighborhood $B_{\alpha}$ of $\alpha$ in $C_Z,$ there exists $\beta \in B_{\alpha},$ such that $\phi(\beta)=\eta_Z.$ In other words, $\overline{\phi^{-1}(\{\eta_Z\})}=C_Z.$
\[
\begin{tikzcd}
C_{Z'} \arrow{r}{\phi'} \arrow{d}[swap]{\theta_{Z, Z'}} & C_{\mathcal{O}(Z')} \arrow{d}{\iota_{Z, Z'}} \\
 C_Z \arrow{r}{\phi} &C_{\mathcal{O}(Z)}
\end{tikzcd}
\]

Let us now look at the commutative diagram above, where the horizontal maps correspond to analytification, and the vertical ones to base change. In particular, remark that since $C_Z=\pi^{-1}(Z)$ and $C_{Z'}=\pi^{-1}(Z'),$ we have $C_{Z'} \subseteq C_Z,$ so $\theta_{Z', Z}$ is an inclusion. Let $\gamma \in \pi^{-1}(\text{Int}(Z'))$ (which is non-empty considering $x \in \text{Int}(Z')$). Let $B_{\gamma}$ be an open neighborhood of $\gamma$ in the open $\pi^{-1}(\text{Int}(Z')).$ Then $B_{\gamma}$ is open in both $C_{Z'}$ and $C_Z.$ By the paragraph above, there exists $\gamma' \in B_{\gamma}$ such that $\phi(\theta_{Z, Z'}(\gamma'))=\phi(\gamma')=\eta_{Z}.$ By the commutativity of the diagram, $\eta_Z$ is in the image of $\iota_{Z, Z'}$, so $\iota_{Z, Z'}$ is dominant. Since $C_{\mathcal{O}(Z)}, C_{\mathcal{O}(Z')}$ are integral schemes, this means $\iota_{Z, Z'}(\eta_{Z'})=\eta_{Z}.$
\end{proof}
Recall that $C_{\mathcal{O}_x}=C_{\mathcal{O}(Z_0)} \times_{\mathcal{O}(Z_0)} \mathcal{O}_x=\varprojlim_{Z} C_{\mathcal{O}(Z)},$ where the limit is taken with respect to the connected affinoid neighborhoods $Z \subseteq Z_0$ of $x.$ By the lemma above, the generic points $\eta_Z$ of $C_{\mathcal{O}(Z)}$ determine a unique point $\eta \in C_{\mathcal{O}_x}.$

\begin{prop} \label{211}
The curve $C_{\mathcal{O}_x}$ is integral with generic point $\eta$. 
\end{prop}

\begin{proof}
Note that $C_{\mathcal{O}_x}$ was already shown to be integral in Lemma \ref{jeeeej}.

For any connected affinoid neighborhoods $Z, Z' \subseteq Z_0$ of $x$ such that $Z' \subseteq Z,$ the base change $\iota_{Z, Z'}: C_{\mathcal{O}(Z')}=C_{\mathcal{O}(Z)} \times_{\mathcal{O}(Z)} \mathcal{O}(Z') \rightarrow C_{\mathcal{O}(Z)}$ is an affine morphism. Since $C_{\mathcal{O}(Z)}$ is normal, it is reduced. 

By \cite[Tag 0CUG]{stacks-project}, $\varprojlim_Z \overline{\{\eta_Z\}}_{\mathrm{red}}=\overline{\{\eta\}}_{\mathrm{red}}.$ Seeing as $\overline{\{\eta_Z\}}_{\mathrm{red}}=C_{\mathcal{O}(Z)},$ we obtain that $\overline{\{\eta\}}_{\mathrm{red}}=\varprojlim_Z C_{\mathcal{O}(Z)}=C_{\mathcal{O}_x},$ so $C_{\mathcal{O}_x}$ is reduced and irreducible, \textit{i.e.} integral, with generic point $\eta.$
\end{proof}
Let $F_{N}$ denote the function field of the integral scheme $C_{N},$ where ${N \in \{\mathcal{O}_x, \mathcal{O}(Z): Z\subseteq Z_0\}}$ ($Z$ is as usual considered to be a connected affinoid neighborhood of $x$). 

\begin{cor} \label{212} The fields $F_{N}, N \in \{\mathcal{O}_x, \mathcal{O}(Z) : Z \subseteq Z_0\},$ satisfy
$F_{\mathcal{O}_x}=\varinjlim_{Z} F_{\mathcal{O}(Z)},$ where the limit is taken with respect to connected affinoid neighborhoods $Z \subseteq Z_0$ of $x.$
\end{cor}

\begin{proof} 
The projective system of integral schemes $\{C_{\mathcal{O}(Z)}\}_Z$ gives rise to a direct system of fields $\{F_{\mathcal{O}(Z)}\}_Z.$ For connected affinoid neighborhoods $Z, Z' \subseteq Z_0$ of $x$ such that $Z' \subseteq Z,$ let us denote the corresponding transition morphism $F_{\mathcal{O}(Z)} \rightarrow F_{\mathcal{O}(Z')}$ by $\chi_{Z', Z}.$ Let us denote by $F'$ the field $\varinjlim_Z F_{\mathcal{O}(Z)}.$ Recall the notation $\iota_{Z, Z'}$ from Proposition \ref{210}.

The projections $\iota_Z: C_{\mathcal{O}_x} \rightarrow C_{\mathcal{O}(Z)}$ give rise to maps $\chi_{Z}': F_{\mathcal{O}(Z)} \rightarrow F_{\mathcal{O}_x}.$ Since for any $Z' \subseteq Z,$ $\iota_Z=\iota_{Z,Z'} \circ \iota_{Z'},$ we have that $\chi'_Z=\chi'_{Z'} \circ \chi_{Z', Z}.$ Consequently, there is a map $F' \rightarrow F_{\mathcal{O}_x}.$ To show that this is an equality it suffices to show that for any field $K$ and morphisms           
$\lambda_{Z}: F_{\mathcal{O}(Z)} \rightarrow K$ such that for any $Z' \subseteq Z,$ $\lambda_Z=\lambda_{Z'} \circ \chi_{Z', Z},$ there is a map $\lambda:  F_{\mathcal{O}_x}  \rightarrow K,$ satisfying $\lambda_Z=\lambda \circ \chi'_Z.$ 

The maps $\lambda_Z: F_{\mathcal{O}(Z)} \rightarrow K$ give rise to maps $\lambda_Z': \text{Spec}  \ K \rightarrow \text{Spec} \ F_{\mathcal{O}(Z)} \rightarrow C_{\mathcal{O}(Z)},$ where the image of $\lambda_Z'$ is the generic point $\{\eta_Z\}$ of $C_{\mathcal{O}(Z)}.$ Consequently, by Proposition \ref{210}, for any $Z' \subseteq Z,$ we have $\lambda_Z'=\iota_{Z, Z'} \circ \lambda_{Z'}'$, implying there is a morphism $\lambda':\text{Spec} \ K \rightarrow C_{\mathcal{O}_x}$ that satisfies $\lambda_Z'=\iota_Z \circ \lambda'$ for all $Z.$ In turn, this gives rise to a morphism $\lambda: F_{\mathcal{O}_x} \rightarrow K,$ which satisfies $\lambda_Z=\lambda \circ \chi_Z'.$
\end{proof}

\begin{cor} \label{213} 
\begin{sloppypar}
Let $\mathscr{M}$ denote the sheaf of meromorphic functions on $C.$ Then
${F_{\mathcal{O}_x}=\varinjlim_Z \mathscr{M}(C_Z)},$ where the limit is taken over connected affinoid neighborhoods $Z \subseteq Z_0$ of $x.$
\end{sloppypar}
\end{cor}

\begin{proof}
This is a direct consequence of Corollary \ref{212} and Theorem \ref{231}. 
\end{proof}

\section{Nice covers of a relative proper curve and patching} \label{4.5}
Throughout this section we work under the hypotheses of Setting \ref{200} and the notations we have introduced along the way. Here is a summary:

\begin{nota} \label{nota}
In addition to Setting \ref{200}, for any connected affinoid neighborhood $Z \subseteq Z_0$ of $x$, let $C_x:=C_Z \times_Z \mathcal{H}(x),$ $C_Z:=C \times_S Z,$ $C_{\mathcal{O}(Z)}:=C_{\mathcal{O}(Z_0)} \times_{\mathcal{O}(Z_0)} \mathcal{O}(Z)$, and $C_{\mathcal{O}_x}:=C_{\mathcal{O}(Z_0)} \times_{\mathcal{O}(Z_0)} \mathcal{O}_x.$ Moreover, we denote by $\pi_Z,$ resp. $\pi_{\mathcal{O}(Z)}$, the structural morphisms $C_Z \rightarrow Z,$ resp. $C_{\mathcal{O}(Z)} \rightarrow \text{Spec} \ \mathcal{O}(Z).$

Finally, let $f_Z: C_Z \rightarrow \mathbb{P}_Z^{1, \mathrm{an}}$, $f_{\mathcal{O}_Z}: C_{\mathcal{O}(Z)} \rightarrow \mathbb{P}_{\mathcal{O}(Z)}^{1, \mathrm{an}}$ be finite surjective morphisms such that $f_{\mathcal{O}(Z)}^{\mathrm{an}}=f_Z,$ and for any connected affinoid neighborhood $Z' \subseteq Z$ of $x,$ $f_Z \times_Z Z'=f_{Z'}.$
\end{nota}

\subsection{Nice covers of a relative proper curve} \label{affreux}
As in the case of $\mathbb{P}^{1, \mathrm{an}},$ in addition to Setting \ref{200}, we assume that $\dim{S} < \dim_{\mathbb{Q}} \mathbb{R}_{>0}/|k^\times| \otimes_{\mathbb{Z}} \mathbb{Q}$.
The reason behind this hypothesis is the same as before: it is sufficient for the existence of type 3 points on the fiber $C_x$ (see Lemma \ref{1000}).

 \
 
\noindent \textbf{Goal:} Let $\mathcal{V}$ be an open cover of $C_{x}$ in $C.$ We construct a refinement of $\mathcal{V}$ and show that it satisfies certain properties which are necessary for patching. 

\

\textbf{(1) The construction of a nice refinement of $\mathcal{V}$.} 
Remark that the finite surjective morphism $f_{Z_0}: C_{Z_0} \rightarrow \mathbb{P}_{Z_0}^{1, \mathrm{an}}$ induces a finite surjective morphism $f_x:C_x \rightarrow \mathbb{P}_{\mathcal{H}(x)}^{1, \mathrm{an}}$ on the corresponding fibers of $x.$

Without loss of generality, we may assume that $\mathcal{V}$ is an affinoid cover of $C_x$ in $C$ such that $\{\text{Int} \ V: V \in \mathcal{V}\}$ is an open cover of $C_x$ in $C.$ Since $C_x$ is compact, we may assume $\mathcal{V}$ is finite. Let $\mathcal{V}_x$ denote the finite affinoid cover that $\mathcal{V}$ induces on $C_x.$ Remark that $\mathcal{V}_x':=\{\text{Int}_{C_x} V: V \in \mathcal{V}_x\}$ remains an open cover of $C_x.$
Since $\mathcal{V}_x$ is an affinoid cover, for any $V \in \mathcal{V}_x,$ the topological boundary $\partial_{C_x}{V}$ of $V$  in $C_x$ is finite. Consequently, for any $V \in \mathcal{V}_x',$ $\partial_{C_x}{V}$ is finite. Set $S'=\bigcup_{V \in \mathcal{V}_x'} \partial_{C_x}{V}.$ This is a finite set of points on ~$C_x.$

Seeing as $C_x$ is a connected curve, for any two points $u,v$ of $S'$, there exist  finitely many arcs $[u,v]_i,$ $i=1,2,\dots, l,$ in $C_x$ connecting them (Proposition \ref{bababa}). Let us take a type 3 point on each $[u,v]_i,$ for any two points $u, v \in S'.$ We denote this set by $S_1$. By construction of $S_1,$ since type 3 points are dense in $C_x$ (\cite[Theorem 2.6]{une}) and $f_x^{-1}(f_x(S'))$ is a finite set, we may assume that $S_1 \cap f_x^{-1}(f_x(S')) =\emptyset.$

Since $S_1$ is a finite set of type 3 points in $C_x,$ $f_{x}(S_1)$ is a finite set of type 3 points in the fiber $\mathbb{P}_{\mathcal{H}(x)}^{1, \mathrm{an}}$ of $x$ in $\mathbb{P}_{Z_0}^{1, \mathrm{an}}.$ By \cite[Lemma 2.14]{une}, there exists a nice cover $\mathcal{D}_x$ of $\mathbb{P}_{\mathcal{H}(x)}^{1, \mathrm{an}}$ such that $f_x(S_1)=S_{\mathcal{D}_x}$ (recall this notation in Definition \ref{parityfunction}).  Let $T_{\mathcal{D}_x}$ be a parity function for $\mathcal{D}_x$ (it exists by \cite[Lemma 2.19]{une}).  

\begin{lm} \label{214}
The connected components of $f_x^{-1}(D), D \in \mathcal{D}_x$, form a cover $\mathcal{U}_x$ of $C_x$ which is nice (see Definition \ref{nice}) and refines $\mathcal{V}_x.$ Furthermore, $S_{\mathcal{U}_x}=f_x^{-1}(S_{\mathcal{D}_x}),$ and the map $T_{\mathcal{U}_x}: \mathcal{U}_x \rightarrow \{0,1\},$ $U \mapsto f_{\mathcal{D}_x}(f_x(U)),$ is a parity function for $\mathcal{U}_x.$
\end{lm}

\begin{proof}
That $\mathcal{U}_x$ is a nice cover of $C_x,$ $S_{\mathcal{U}_x}=f_x^{-1}(S_{\mathcal{D}_x})$, and $T_{\mathcal{U}_x}$ is a parity function for $\mathcal{U}_x$ has been shown in \cite[Proposition 2.21]{une}. It remains to show that $\mathcal{U}_x$ refines $\mathcal{V}_x.$ For that, it suffices to show that $\mathcal{U}_x$ refines the open cover $\mathcal{V}_x'$ of $C_x.$ 

Let us start by proving that $S_{\mathcal{U}_x} \cap S'= \emptyset.$ Suppose, by contradiction, that there exists ${a \in S_{\mathcal{U}_x} \cap S'=f_x^{-1}(f_x(S_1)) \cap S'}.$ Then $f_x(a) \in f_x(S_1) \cap f_x(S'),$ so there exists $b \in S_1$ such that $f_x(a)=f_x(b) \in f_x(S_1) \cap f_x(S').$ Consequently, $b \in f_x^{-1}(f_x(S')) \cap S_1=\emptyset,$ which is impossible, so $S_{\mathcal{U}_x} \cap S'=\emptyset.$ Considering $S_{\mathcal{U}_x}=\bigcup_{U \in \mathcal{U}_x} \partial{U}$ and $S'=\bigcup_{V \in \mathcal{V}_x'} \partial{V},$ for any $U \in \mathcal{U}_x$ and any $V \in \mathcal{V}_x'$, $\partial{U} \cap \partial{V}=\emptyset.$

Let us now show that $\mathcal{U}_x$ refines $\mathcal{V}_x'.$ Suppose, by contradiction, that there exists ${U \in \mathcal{U}_x},$ such that for any $V \in \mathcal{V}_x',$ $U \not \subseteq V.$  Let $V_j, j=1,2,\dots, m,$ be the elements of $\mathcal{V}_x'$ intersecting ~$U$ ($m \neq 0$ seeing as $\mathcal{V}_x'$ is a cover of $C_x$).  Then $U \subseteq \bigcup_{j=1}^m V_j.$
Considering $U \not \subseteq V_j$ and $U$ is connected, $U \cap \partial{V_j}\neq \emptyset$ for all $j$. If $\bigcup_{j=1}^m U \cap \partial{V_j}$ is a single point $\{w\}$, then $w \in U \backslash \bigcup_{j=1}^m V_j$ (because the $V_j$ are open), which is impossible seeing as $U \subseteq \bigcup_{j=1}^m V_j.$ 
Let $x_1,x_2$ be two different points of $\bigcup_{j=1}^m U \cap \partial{V_j}.$ Since $\partial{U} \cap \partial{V}_j=\emptyset$ for all ~$j$ (this was shown in the paragraph above), $x_i \in \text{Int} (U), i=1,2.$

Since $U$ is connected, by Lemma \ref{62}, $\text{Int} \ U$ is connected, so there exists an arc $[x_1, x_2]$ connecting $x_1$ and $x_2$, which is contained entirely in $\text{Int} \ U.$ But then, by the construction of $S_1,$ since $x_1, x_2 \in S',$ there exists $y \in S_1$ such that $y \in [x_1, x_2] \subseteq \text{Int} \ U.$ Considering $y \in S_1 \subseteq f_x^{-1}(S_{\mathcal{D}_x})=S_{\mathcal{U}_x},$ there exists $U' \in \mathcal{U}_x$, such that $y \in \partial{U'}.$ But then, $\partial{U} \cap \partial{U'} \neq U \cap U'$ which is in contradiction with the fact that $\mathcal{U}_x$ is a nice cover of $C_{x}.$

Thus, there must exist $V_U \in \mathcal{V}_x'$ such that $U \subseteq V_U,$ implying $\mathcal{U}_x$ refines the cover $\mathcal{V}_x'.$
\end{proof}
The following result will be used several times in what is to come. 
\begin{lm} \label{215}
Let $Z \subseteq Z_{0}$ be a connected affinoid neighborhood of $x.$ Let $D'$ be a connected affinoid domain of $\mathbb{P}_Z^{1, \mathrm{an}}$, such that $D' \cap F_x$ is non-empty and connected, where $F_x$ is the fiber of $x$ with respect to the morphism $\mathbb{P}_Z^{1, \mathrm{an}}\rightarrow Z.$ Then the connected components of $f_Z^{-1}(D')$ are connected affinoid domains of $C_Z$ that intersect the fiber $C_x$ of $x$. Moreover, if $U$ is a connected component of $f_Z^{-1}(D'),$ then $f_Z(U)=D'.$
\end{lm}

\begin{proof}
Seeing as $f_Z$ is a finite morphism, $f_Z^{-1}(D')$ is an affinoid domain in $C_Z,$ and thus so are its connected components.

Seeing as $C_Z$ and $\mathbb{P}_Z^{1, \mathrm{an}}$ are irreducible, they are pure-dimensional (see \cite[Corollaire ~4.14]{dex}). Seeing as $f_Z$ is finite, its relative dimension is pure and equal to 0 (\textit{i.e.} all its fibers are of dimension 0). By \cite[1.4.14(3)]{famduc}, the dimension of $C_Z$ is the same as the dimension of $\mathbb{P}_Z^{1, \mathrm{an}}.$ Consequently, by \cite[Lemma 3.2.4]{Ber90}, $f_Z$ is open. 

Let $U$ be any connected component of $f_Z^{-1}(D').$ It is an open and a closed subset of $f_Z^{-1}(D').$ Seeing as $f_Z$ is open and closed, $f_Z(U)$ is an open and closed subset of $D'.$ Considering $D'$ is connected, this implies $D'=f_Z(U).$ Since $D' \cap F_x \neq \emptyset,$ we obtain $U \cap C_x \neq \emptyset.$
\end{proof}

\begin{nota} \label{edhenji}
Let $Z_{\mathcal{D}} \subseteq Z_0$ be a connected affinoid neighborhood of $x,$ such that the $Z_{\mathcal{D}}$-thickening $\mathcal{D}_{Z_{\mathcal{D}}}$ of $\mathcal{D}_x$ exists and is a $Z_{\mathcal{D}}$-relative nice cover for $\mathbb{P}_{Z_{\mathcal{D}}}^{1, \mathrm{an}}$ (see Theorem ~\ref{106}).
 
Let $Z \subseteq Z_{\mathcal{D}}$  be any connected affinoid neighborhood of $x.$ We denote by \textbf{$\mathcal{U}_Z$} the set of connected components of $f_Z^{-1}(D_Z), D \in \mathcal{D}_x.$ By Lemma \ref{215}, $\mathcal{U}_Z$ is a finite affinoid cover of $C_Z.$ Furthermore, for any $U \in \mathcal{U}_Z,$ $U \cap C_x \neq \emptyset$ and $f_Z(U) \in \mathcal{D}_Z.$
Remark that the nice cover $\mathcal{U}_x$ of Lemma \ref{214} is obtained by taking the connected components of $U \cap C_x, U \in 
\mathcal{U}_Z.$
\end{nota}
We now study the properties of the covers $\mathcal{U}_Z.$

\

\textbf{(2) The elements of $\mathcal{U}_Z$ intersect the fiber nicely.}
We show that the connected affinoid neighborhood $Z \subseteq Z_{\mathcal{D}}$ of $x$ can be chosen such that $U \cap C_x$ is connected for any $U \in \mathcal{U}_Z,$ and the same remains true when replacing $Z$ with any connected affinoid neighborhood $Z' \subseteq Z$ of $x.$ Let us start with a couple of auxiliary results. 

\begin{lm} \label{216} Let $Z \subseteq Z_0$ be a connected affinoid neighborhood of $x.$
Let $A_1, A_2$ be two disjoint compact subsets of $C_x.$ Then there exist two open subsets $B_1, B_2$ of $C_{Z}$ such that $A_i \subseteq B_i, i=1,2,$ and $B_1 \cap B_2=\emptyset.$
\end{lm}

\begin{proof}
Let $a \in A_1.$ By Lemma \ref{205}, for any $b \in A_2,$ there exist an open neighborhood $N_{a, b}$ of $a$ in $C_{Z}$, and an open neighborhood $B_{a, b}$ of $b$ in $C_{Z},$ such that $N_{a, b} \cap B_{a, b}=\emptyset.$  The family $\{B_{a, b}\}_{b \in A_2}$ forms an open cover of $A_2.$ Considering $A_2$ is a compact subset of $C_x,$ it is compact in $C_Z,$ so there exists a finite subcover $\{B_{a, b_i}\}_{i=1}^m$ of $\{B_{a, b}\}_{b \in A_2}.$ Set $N_a=\bigcap_{i=1}^m N_{a, b_i}$ and $B_a=\bigcup_{i=1}^m B_{a, b_i}.$ Then $N_a, B_a$ are open subsets of $C_Z,$ $A_2 \subseteq B_a,$ and $N_a \cap B_a=\emptyset.$

The family $\{N_a\}_{a \in A_1}$ is an open cover of $A_1.$ Since $A_1$ is compact, there exists an open subcover $\{N_{a_j}\}_{j=1}^l.$ Set $B_1=\bigcup_{j=1}^l N_{a_j}$ and $B_2=\bigcap_{j=1}^l B_{a_j}.$ Then $B_1$ and $B_2$ satisfy the statement. 
\end{proof}

\begin{lm} \label{217}
Let $D$ be a connected affinoid domain of $\mathbb{P}_{\mathcal{H}(x)}^{1, \mathrm{an}}$ containing only type 3 points in its boundary. Let $Z \subseteq Z_0$ be a connected affinoid neighborhood of $x$ such that the $Z$-thickening $D_Z$ exists, and for any connected affinoid neighborhood $Z' \subseteq Z$ of $x,$ the $Z'$-thickening $D_{Z'}$ of $D$ is connected. Let $U_{1, Z}, U_{2, Z}, \dots, U_{n,Z}$ be the connected components of $f_{Z}^{-1}(D_Z).$

Then the connected components of $f_{Z'}^{-1}(D_{Z'})$ are the connected components of ${U_{i, Z} \cap C_{Z'}},$ $i=1,2,\dots, n.$ 
\end{lm}

\begin{proof}
By commutativity of the diagram below, $f_Z^{-1}(D_Z) \cap C_{Z'} = f_{Z'}^{-1}(D_Z \cap \mathbb{P}_{Z'}^{1, \mathrm{an}})=f_{Z'}^{-1}(D_{Z'}),$ so $f_{Z'}^{-1}(D_{Z'})=\bigsqcup_{i=1}^n U_{i,Z} \cap C_{Z'}$ for any $i.$ The statement follows immediately. 
\[
\begin{tikzcd}
C_{Z'} \arrow{r} \arrow[swap]{d}{f_{Z'}} & C_{Z} \arrow{d}{f_{Z}} \\
\mathbb{P}_{Z'}^{1, \mathrm{an}} \arrow{r} & \mathbb{P}_{Z}^{1, \mathrm{an}}
\end{tikzcd}
\]
\end{proof}
We can now show the following:
\begin{prop} \label{218}
Let $D$ be a connected affinoid domain of $\mathbb{P}_{\mathcal{H}(x)}^{1, \mathrm{an}}$ containing only type ~3 points in its boundary. Let $Z \subseteq Z_0$ be a connected affinoid neighborhood of $x$ such that the $Z$-thickening $D_Z$ exists, and for any connected affinoid neighborhood $Z' \subseteq Z$ of $x,$ the $Z'$-thickening $D_{Z'}$ of $D$ is connected. 

Let $U_{1, Z}, U_{2, Z}, \dots, U_{n,Z}$ be the connected components of $f_Z^{-1}(D_Z).$ The affinoid neighborhood $Z$ of $x$ can be chosen such that:
\begin{itemize}
\item $U_{i, Z} \cap C_x$ is a non-empty connected affinoid domain of $C_x$ for all $i;$
\item there is a bijection between the connected components of $f_Z^{-1}(D_Z)$ and the connected components of $f_x^{-1}(D)$ given by $U_{i,Z} \mapsto U_{i,Z} \cap C_x;$
\item for any connected affinoid neighborhood $Z' \subseteq Z$ of ~$x,$ 
the connected components of $f_{Z'}^{-1}(Z')$ are $U_{i,Z'}:=U_{i,Z} \cap C_{Z'}, i=1,2,\dots, n.$
\end{itemize}
\end{prop}

\begin{proof}
Recall that the finite morphism $f_Z:C_Z \rightarrow \mathbb{P}_Z^{1, \mathrm{an}}$ induces a finite morphism $f_x: C_x \rightarrow \mathbb{P}_{\mathcal{H}(x)}^{1, \mathrm{an}}$ on the corresponding fibers of $x.$
Let $L_1, L_2, \dots, L_s$ be the connected components of $f_x^{-1}(D).$ They are connected affinoid domains of $C_x.$ 

Seeing as (follow the diagram below) $${\bigsqcup_{t=1}^s L_t =f_x^{-1}(D) = f_Z^{-1}(D_Z) \cap C_x= \bigsqcup_{i=1}^n U_{i,Z} \cap C_x},$$ for any $t,$ $L_t \subseteq \bigsqcup_{i=1}^n U_{i,Z}.$ Since $L_t$ is connected, there exists a unique $i_t$ such that $L_t \subseteq U_{i_t,Z} \cap C_x.$ 
\[
\begin{tikzcd}
C_{x} \arrow{r} \arrow{d}[swap]{f_{x}}& C_{Z} \arrow{d}{f_{Z}} \\
 \mathbb{P}_{\mathcal{H}(x)}^{1, \mathrm{an}} \arrow{r} &\mathbb{P}_{Z}^{1, \mathrm{an}}
\end{tikzcd}
\]

Suppose there exists $i_0$ such that $ U_{i_0, Z} \cap C_x$ is not connected. Suppose, without loss of generality, that $L_1, L_2, \dots, L_r$ are the connected components of $C_x \cap U_{i_0, Z}.$ By Lemma ~\ref{216}, there exist mutually disjoint open subsets $B_t$ of $C_Z$ such that ${L_t \subseteq B_t},$ ${t=1,2,\dots, r}.$ 
The set $U_{i_0,Z} \backslash \bigsqcup_{t=1}^r B_t$ is a compact subset of $C_Z$ that doesn't intersect the fiber $C_x.$ It is a non-empty set: otherwise, $U_{i_0, Z} \subseteq \bigsqcup_{t=1}^r B_t,$ and seeing as ${U_{i_0,Z} \cap B_t \supseteq U_{i_0,Z} \cap L_t \neq \emptyset}$ for all ~$t=1,2,\dots, r$, we obtain that $U_{i_0,Z}$ is not connected, contradiction.

 Since $\pi_Z$ is proper, $\pi_Z(U_{i_0,Z} \backslash \bigsqcup_{t=1}^r B_t)$ is a non-empty compact subset of $Z$ that does not contain $x.$ Thus, there exists a connected affinoid neighborhood $Z_1 \subseteq Z$ of $x$ such that $\pi_{Z}^{-1}(Z_1) \cap (U_{i_0,Z} \backslash \bigsqcup_{t=1}^r B_t) =\emptyset,$ implying $U_{i_0,Z} \cap C_{Z_1} \subseteq \bigsqcup_{t=1}^r B_t.$ 
\begin{sloppypar}
Let $V_{1, Z_1}, V_{2, Z_1}, \dots, V_{e, Z_1}$ be the connected components of $U_{i_0, Z} \cap C_{Z_1}.$ By Lemma \ref{217}, $V_{j, Z_1}$, $j=1,2,\dots, e,$ are connected components of $f_{Z_1}^{-1}(D_{Z_1}),$ so by Lemma \ref{215}, they all intersect the fiber $C_x.$ Moreover, ${\bigsqcup_{j=1}^e V_{j, Z_1} \cap C_x=U_{i_0, Z} \cap C_{x}=\bigsqcup_{t=1}^r L_t}.$ Hence, for any $t$, there exists a unique $e_t$ such that $L_t \subseteq V_{e_t, Z_1} \cap C_x.$ By the paragraph above, for any $j$, there exists a unique $t_j,$ such that $V_{j, Z_1} \subseteq B_{t_j},$ hence a unique $L_{t_j}$ contained in $V_{j, Z_1}.$ Consequently, $r=e$ and ${\{V_{j, Z_1} \cap C_x: j=1,2,\dots, r\}=\{L_t: t=1,2,\dots, r\}}.$ We may assume, without loss of generality, that $V_{j, Z_1} \cap C_x= L_j, j=1,2,\dots, r.$ Clearly, this induces a bijection between the connected components of $U_{i_0, Z} \cap C_{Z_1}$ and the connected components of $U_{i_0,Z} \cap C_x,$ given by $V_{j, Z_1} \mapsto V_{j,Z_1} \cap C_x=L_j, j=1,2,\dots, r.$ 
\end{sloppypar} 
Let us show that for any connected affinoid neighborhood $Z_2 \subseteq Z_1$ of $x,$ $V_{j, Z_1} \cap C_{Z_2}$ remains connected for all $j=1,2,\dots, r.$ By Lemma \ref{217}, the connected components of $V_{j, Z_1} \cap C_{Z_2}$ are connected components of $f_{Z_2}^{-1}(D_{Z_2}),$ so by Lemma \ref{215}, they all intersect the fiber $C_x.$ Seeing as $L_j=V_{j, Z_1} \cap C_x=V_{j, Z_1} \cap C_{Z_2} \cap C_x$ is connected, $V_{j, Z_1} \cap C_{Z_2}$ has to be connected for all $j$. In particular, the bijective correspondence obtained above remains true when replacing $Z_1$ by $Z_2.$

We have shown that for any $i=1,2,\dots, n,$ there exists a connected affinoid neighborhood $Z^i \subseteq Z_0$ of $x,$ such that the connected components $V_{j,i, Z^i}, j=1,2,\dots, r_i,$ of $U_{i,Z} \cap C_{Z^i}$ satisfy: (a) $V_{j,i, Z^i} \cap C_x$ is non-empty and connected for all $j$; (b) there is a bijection between the connected components of $U_{i,Z} \cap C_{Z^i}$ and the connected components of $U_{i,Z} \cap C_x,$ given by $V_{j,i, Z^i} \mapsto V_{j,i,Z^i} \cap C_x;$ (c) for any connected affinoid neighborhood $Z' \subseteq Z^i,$ $V_{j,i, Z^i} \cap C_{Z'}$ remains connected, implying the connected components of $U_{i, Z} \cap C_{Z'}$ are $V_{j,i, Z^i} \cap C_{Z'}, j=1,2,\dots, r_i.$

Let $Z' \subseteq \bigcap_{i=1}^n Z^i$ be a connected affinoid neighborhood of $x.$ Since $Z' \subseteq Z,$ by Lemma ~\ref{217}, the connected components of $f_{Z'}^{-1}(D_{Z'})$ are the connected components of $U_{i,Z} \cap C_{Z'}, i=1,2,\dots, n.$ By the paragraph above, these are $V_{j,i, Z^i} \cap C_{Z'}, j=1,2,\dots, r_i,$ ${i=1,2,\dots, n},$ and they satisfy: \emph{(a')} $V_{j,i,Z^i} \cap C_{Z'} \cap C_x$ is non-empty and connected for all ~$j,i$; \emph{(b')} for any $i,$ there is a bijection between the connected components of $U_{i,Z} \cap C_{Z'}$ and the connected components of $U_{i,Z} \cap C_x,$ given by $V_{i,j,Z^i} \cap C_{Z'} \mapsto V_{i,j,Z^i} \cap C_x,$ implying there is a bijection between the connected components of $U_{i,Z} \cap C_{Z'}, i=1,2,\dots, n$ (\textit{i.e.} of $f_{Z'}^{-1}(D_{Z'})$) and the connected components of $U_{i,Z} \cap C_x, i=1,2,\dots, n$ (\textit{i.e.} of $f_x^{-1}(D)$), given by $V_{j,i,Z^i} \cap C_{Z'} \mapsto V_{j,i,Z^i} \cap C_x, j,i;$ \emph{(c')} for any connected affinoid neighborhood $Z'' \subseteq Z'$ of $x,$ by the paragraph above, the connected components of $f_{Z''}^{-1}(D_{Z''})$ are $V_{j,i,Z^i} \cap C_{Z'} \cap C_{Z''}=V_{j,i, Z^i} \cap C_{Z''},$ ${j=1,2,\dots, r_i,}$ $i=1,2,\dots, n.$ 
\end{proof}
 
We have shown:

\begin{cor} \label{219}
There exists a connected affinoid neighborhood $Z_f \subseteq Z_{\mathcal{D}}$ of $x$ such that for any $U \in \mathcal{U}_{Z_f},$ $U \cap C_x$ is connected, and $\mathcal{U}_x=\{U \cap C_x: U \in \mathcal{U}_{Z_f}\},$ where $\mathcal{U}_x$ is the nice cover of $C_x$ obtained in the statement of Lemma \ref{214}.  Moreover, for any connected affinoid neighborhood $Z' \subseteq Z_f$ of $x,$ $\mathcal{U}_{Z'}=\{U \cap C_{Z'}: U \in \mathcal{U}_{Z_f}\}.$ 
\end{cor} 
 
\begin{rem} \label{220}
By Corollary \ref{219}, for any connected affinoid neighborhood ${Z \subseteq Z_f}$ of $x,$ there is a bijective correspondence between $\mathcal{U}_{Z}$ and $\mathcal{U}_x$ given by $V \mapsto V \cap C_x.$

Consequently, we will from now on sometimes write $U_Z$ for the unique element of $\mathcal{U}_Z$ corresponding to the element $U$ of $\mathcal{U}_x.$ In particular, $\mathcal{U}_Z=\{U_Z: U \in \mathcal{U}_x\}.$
\end{rem}

\textbf{(3) $\mathcal{U}_Z$ refines $\mathcal{V}$.} We now prove that the covers we have just constructed refine the starting open cover $\mathcal{V}.$

\begin{prop} \label{221}
There exists a connected affinoid neighborhood $Z_r \subseteq Z_f$ ($Z_f$ as in Corollary \ref{219}) of $x$ such that for any connected affinoid neighborhood $Z \subseteq Z_r,$ the cover~$\mathcal{U}_Z$ refines $\mathcal{V}.$
\end{prop}

\begin{proof}
Let $Z \subseteq Z_{f}$ be a connected affinoid neighborhood of ~$x.$ Let ${U_Z \in \mathcal{U}_Z}$. Then ${U:=U_Z \cap C_x}$ is a connected affinoid domain of $C_x$ and an element of $\mathcal{U}_x$  (recall Remark \ref{220}). By Lemma \ref{214}, there exists $V \in \mathcal{V},$ such that $U \subseteq V_x,$ where $V_x$ denotes the intersection of $V$  with the fiber $C_x$. Assume $U_Z \not \subseteq V.$ Then $U_Z \backslash V$ is a non-empty compact subset of $C_{Z}$ not intersecting the fiber $C_x.$ Seeing as $\pi_{Z}$ is proper, $\pi_{Z}(U_Z \backslash V)$ is a compact subset of $Z$ not containing $x.$ Thus, there exists a connected affinoid neighborhood $Z_1 \subseteq Z$ of ~$x,$ such that $\pi_{Z}^{-1}(Z_1) \cap (U_Z \backslash V) =\emptyset,$ \textit{i.e.} $C_{Z_1} \cap (U_Z \backslash V) = \emptyset,$ implying $C_{Z_1} \cap U_Z \subseteq V.$ Clearly, the same remains true when replacing $Z_1$ by any connected affinoid neighborhood $Z_2 \subseteq Z_1$ of $x.$ Considering $\mathcal{U}_Z$ is a finite cover, by repeating the same argument for all of its elements, we obtain that there exists a connected affinoid neighborhood $Z' \subseteq Z_f$ such that  $\{U_Z \cap C_{Z'}: U \in \mathcal{U}_{x}\}$ refines $\mathcal{V},$ and the same remains true when replacing $Z'$ with any connected affinoid neighborhood $Z'' \subseteq Z'.$ By Corollary ~\ref{219}, $\mathcal{U}_{Z'}=\{U_Z \cap C_{Z'}: U \in \mathcal{U}_x\},$ implying $\mathcal{U}_{Z'}$ is a refinement of $\mathcal{V}.$ The same remains true for any $Z'' \subseteq Z'$ as above. 
\end{proof}

\textbf{(4) Nature of the pairwise intersections of elements of $\mathcal{U}_Z$.} We now study the nature of the pairwise intersections of the elements of $\mathcal{U}_Z,$ which, as in the one-dimensional case, play a very important role for ``patching purposes". 
\begin{prop} \label{222}
There exists a connected affinoid neighborhood $Z_t \subseteq Z_r$ (with $Z_r$ as in Proposition ~\ref{221}) of $x$ such that for any connected affinoid neighborhood $Z \subseteq Z_t$, for any two non-disjoint elements $D_1, D_2$ of $\mathcal{D}_x$ with $D_1 \cap D_2=:\{y\},$ 
$$f_Z^{-1}(D_{1, Z} \cap D_{2,Z})=\bigsqcup_{s\in f_x^{-1}(y)} W_{s, Z},$$
where $W_{s,Z}$ is a connected affinoid domain of $C_Z,$ and ${W_{s, Z} \cap C_x=\{s\}}$  for any $s \in f_x^{-1}(y)$. Moreover, for any connected affinoid neighborhood $Z' \subseteq Z,$ the connected components of $f_{Z'}^{-1}(D_{1, Z'} \cap D_{2, Z'})$ are $W_{s,Z'}:=W_{s,Z} \cap C_{Z'}, s \in f_x^{-1}(y).$
\end{prop}

\begin{proof}
Let ${Z \subseteq Z_r}$  be a connected affinoid neighborhood of $x.$ Let ${D_{1}, D_{2} \in \mathcal{D}_x}$ such that  $D_1 \cap D_2 \neq \emptyset.$ Set $D_1 \cap D_2=\{y\}.$ Then $f_x^{-1}(y):=\{s_1, s_2, \dots, s_m\}$ is a subset of ~$S_{\mathcal{U}_x}.$ Set $D=D_1 \cap D_2.$ As $Z \subseteq Z_{\mathcal{D}}$ (with $Z_{\mathcal{D}}$ as in part (1)), the $Z$-thickening $D_Z$ of ~$D$ is a connected affinoid domain of $\mathbb{P}_Z^{1, \mathrm{an}}$ intersecting the fiber $\mathbb{P}_{\mathcal{H}(x)}^{1, \mathrm{an}}$ at the single type 3 point ~$y.$ 

Let $W_{i, Z}, i=1,2,\dots, n,$ be the connected components of $f_Z^{-1}(D_Z)$. By Proposition ~\ref{218}, we may assume that: (a) $W_{i,Z} \cap C_x$ is connected for all $i$; (b) there is a bijective correspondence between the connected components of $f_{Z}^{-1}(D_Z)$ and the points of $f_x^{-1}(y),$ given by $W_{i,Z} \mapsto W_{i,Z} \cap C_x, i=1,2,\dots, n;$ (c)
for any connected affinoid neighborhood $Z' \subseteq Z,$ the connected components of $f_{Z'}^{-1}(D_{Z'})$ are $W_{i,Z'}=W_{i, Z} \cap C_{Z'},$ $i=1,2,\dots, n.$

For any $s \in f_x^{-1}(y),$ let us denote by $W_{s,Z}$ the (unique) connected component of $f_Z^{-1}(D_Z)$ containing $s,$ (\textit{i.e.} $W_{s, Z} \cap C_x=\{s\}$), so the connected components of $f_{Z}^{-1}(D_Z)$ are $W_{s,Z}, s \in f_x^{-1}(y).$

Let $U_{j,Z}, j=1,2,\dots, p$ (resp. $V_{l,Z}, l=1,2,\dots, q$),  be the connected components of $f_Z^{-1}(D_{1,Z})$ (resp. $f_Z^{-1}(D_{2,Z})).$ Then $$ \bigsqcup_{j=1}^p \bigsqcup_{l=1}^q U_{j,Z} \cap V_{l,Z} = f_Z^{-1}(D_{1,Z}) \cap f_Z^{-1}(D_{2,Z}) = f_Z^{-1}(D_Z)= \bigsqcup_{s \in f_{x}^{-1}(y)} W_{s,Z}.$$
For some $j,l,$ let $s_{j,l} \in U_j \cap V_l.$ Since $s_{j,l} \in W_{s_{j,l}, Z},$ we obtain that ${W_{s_{j,l}, Z} \subseteq U_{j,Z} \cap V_{l,Z}}.$  Consequently, for any $j,l,$ $U_{j,Z} \cap V_{l,Z}=\bigsqcup_{s\in U_j \cap V_l} W_{s, Z}.$ 

Let $Z' \subseteq Z$ be any connected affinoid neighborhood of $x.$
Considering that the connected components of $f_{Z'}^{-1}(D_{1,Z'})$ (resp. $f_{Z'}^{-1}(D_{2, Z'})$) are ${U_{j,Z} \cap C_{Z'}},$ $j=1,2,\dots, p$ (resp. ${V_{l,Z} \cap C_{Z'},} l=1,2,\dots, q$), the same properties remain true when replacing $Z$ by $Z'.$

The same argument can be repeated for any two non-disjoint elements of the finite cover ~$\mathcal{D}_x$, allowing us to conclude this proof. 
\end{proof}

\begin{cor} \label{223}
Let $Z \subseteq Z_t$ be a connected affinoid neighborhood of $x.$ For any ${U, V \in \mathcal{U}_x},$ $U \cap V \neq \emptyset$ if and only if $U_Z \cap V_Z \neq \emptyset.$
\end{cor}

\begin{proof} Set $D_1=f_x(U)$ and $D_2=f_x(V).$ Then $D_1, D_2 \in \mathcal{D}_x,$ and $U_Z,$ resp. $V_Z$, are connected components of $f_Z^{-1}(D_{1,Z})$, resp. $f_Z^{-1}(D_{2,Z}).$
If $U_Z \cap V_Z \neq \emptyset,$ then ${D_{1, Z} \cap D_{2,Z} \neq \emptyset}$, which is equivalent to $D_1 \cap D_2 \neq \emptyset$ (see Theorem \ref{nicecover}(2)). By Proposition ~\ref{222}, $U_Z \cap V_Z \cap C_x \neq \emptyset,$ \textit{i.e.} $U \cap V \neq \emptyset.$ The other direction is immediate.
\end{proof}

In order to invoke more easily the properties we have just shown for $\mathcal{U}_Z,$ we introduce the following:

\begin{defn} \label{224}
Let $\mathcal{D}_x$ be a nice cover of $\mathbb{P}_{\mathcal{H}(x)}^{1, \mathrm{an}}.$ For a connected affinoid neighborhood ~$Z$ of $x,$ a cover $\mathcal{U}_Z$ of $C_Z$ constructed as in (1) and satisfying properties (2) (\textit{i.e.} Corollary ~\ref{219}) and (4) (\textit{i.e.} Proposition \ref{222}), will be called a $Z$-\textit{relative nice cover of $C_Z$ induced by ~$\mathcal{D}_x$.} 
\begin{sloppypar}
Remark that $\mathcal{U}_x:=\{U \cap C_x: U \in \mathcal{U}_Z\}$ is a nice cover of $C_x$ induced by $\mathcal{D}_x$ as in Lemma ~\ref{214}.  Also, for any connected affinoid neighborhood $Z' \subseteq Z$ of $x,$ ${\mathcal{U}_{Z'}=\{U \cap C_{Z'}: U \in \mathcal{U}_Z\}}$ is a $Z'$-relative nice cover of $C_{Z'}$ induced by $\mathcal{D}_x.$
\end{sloppypar}
\end{defn}

\begin{rem} \label{remark}
We have shown that (see Proposition \ref{221}) for any open cover $\mathcal{V}$ of $C_x$ in $C,$ there exists a nice cover $\mathcal{D}_x$ of $\mathbb{P}_{\mathcal{H}(x)}^{1, \mathrm{an}}$ and a connected affinoid neighborhood $Z_t$ of ~$x$ such that the $Z_t$-relative nice cover $\mathcal{U}_{Z_t}$ of $C_{Z_t}$ induced by $\mathcal{D}_x$ refines $\mathcal{V}$. This remains true when replacing $Z_t$ by any connected affinoid neighborhood $Z \subseteq Z_t$ of $x.$   
\end{rem}
\subsection{Patching over relative proper curves}

We now generalize the results of Section ~\ref{4.3} and obtain an application of patching on relative proper curves. 

Throughout this part, let $k$ be a non-trivially valued complete ultrametric field.
We continue working with Setting \ref{200} and Notation \ref{nota}. Moreover, we assume that $\dim{S} < \dim_{\mathbb{Q}} \mathbb{R}_{>0}/|k^\times| \otimes_{\mathbb{Z}} \mathbb{Q},$ so type 3 points exist in $C_x$. 

As in the case of $\mathbb{P}^{1, \mathrm{an}}:$

\begin{nota} \label{225}
Let $G$ be a \textit{connected} rational linear algebraic group defined over $F_{\mathcal{O}_x}.$ Since $F_{\mathcal{O}_x}=\varinjlim_Z \mathscr{M}(C_Z)$ (Corollary \ref{213}), there exists a connected affinoid neighborhood $Z_G \subseteq Z_0$ of $x$ such that $G$ is a connected rational linear algebraic group over $\mathscr{M}(C_{Z_G}).$ 
\end{nota}

\begin{rem}
We will treat the general case of a proper relative curve by using the Weil restriction of scalars to descend to the relative projective line. In order for the Weil restriction of $G$ to satisfy the necessary properties for patching, we need to assume connectedness. Because of \cite[Corollary 6.5]{HHK1} (see also \cite[4.3]{une}), the author conjectures it is a necessary hypothesis when treating the general case (but could be omitted when treating special, well-behaved, ones). 
\end{rem}

The following is an analogue of \cite[Proposition 3.3]{une}. Recall Definition \ref{224}.
\begin{thm} \label{226}
For any open cover $\mathcal{V}$ of $C_x$ in $C,$ there exists a connected affinoid neighborhood $Z \subseteq Z_G$ of $x$ and a nice cover $\mathcal{D}_x$ of $\mathbb{P}_{\mathcal{H}(x)}^{1, \mathrm{an}}$ such that: 
\begin{itemize}
\item the $Z$-relative nice cover $\mathcal{U}_Z$ of $C_Z$ induced by $\mathcal{D}_x$ refines $\mathcal{V}$;
\item for any $(g_s)_{s \in S_{\mathcal{U}_x}} \in \prod_{s \in S_{\mathcal{U}_x}} G(\mathscr{M}_{C,s}) ,$ there exists $(g_U)_{U \in \mathcal{U}_x} \in \prod_{U \in \mathcal{U}_x} G(\mathscr{M}(U_Z)),$ satisfying: for any $s \in S_{\mathcal{U}_x},$ if $U_s, V_s $  are the elements of $\mathcal{U}_x$ containing $s,$ if $W_{s,Z}$ is the connected component of $U_{s,Z} \cap V_{s, Z}$ containing $s$, and $T_{\mathcal{U}_x}(U_s)=0,$ then $g_s \in G(\mathscr{M}(W_{s,Z})),$ and $g_s=g_U \cdot g_V^{-1}$ in $G(\mathscr{M}(W_{s,Z})).$
\end{itemize}
The same remains true when replacing $Z$ by any connected affinoid neighborhood $Z' \subseteq Z$ of $x.$
\end{thm}

\begin{proof}
Seeing as for any connected affinoid neighborhood $Z$ of $x$, $x \in \text{Int}(Z),$ for any $u \in C_x,$ $u \in \text{Int}(C_Z),$ so $\mathscr{M}_{C_{Z}, u}=\mathscr{M}_{C,u}.$

By Remark \ref{remark}, there exists a connected affinoid neighborhood $Z \subseteq Z_G$ of $x$ and a nice cover $\mathcal{D}_x$ of $\mathbb{P}_{\mathcal{H}(x)}^{1, \mathrm{an}}$ which induce a refinement $\mathcal{U}_Z$ of $\mathcal{V}$ obtained as in construction (1) and satisfying properties (2) and (4) of Subsection \ref{affreux}. Let $\mathcal{U}_x$ denote the corresponding nice cover of $C_x$, $T_{\mathcal{U}_x}$ its associated parity function, and $S_{\mathcal{U}_x}$ the intersection points of the different elements of $\mathcal{U}_x.$ 

The proof is organized in three parts: in (a) we explore some properties of the neighborhoods of $s \in S_{\mathcal{U}_x}$; in (b) we make the descent to $\mathbb{P}^{1, \mathrm{an}}$ where the statement has already been proven; in (c) we conclude by using pull-backs.

\

\textit{(a) The neighborhoods of $s \in S_{\mathcal{U}_x}$.}  We will need the following:
\begin{lm} \label{227}
For $s \in S_{\mathcal{U}_x},$ let $B_s$ be a neighborhood of $s$ in $C.$ There exists a connected affinoid neighborhood $Z_1$ of $x$ such that for any $s \in S_{\mathcal{U}_x},$ if $U_s, V_s$ are the elements of $\mathcal{U}_x$ containing $s$, and $W_{s,Z_1}$ is the connected component of $U_{s,Z_1} \cap V_{s,Z_1}$ containing $s,$ then $W_{s, Z_1} \subseteq B_s.$ The neighborhood $Z_1$ can be chosen such that the statement remains true when replacing $Z_1$ by any connected affinoid neighborhood $Z_2 \subseteq Z_1$ of $x.$
\end{lm}

\begin{proof}
Let $Z \subseteq Z_t$ be a connected affinoid neighborhood of $x,$ where $Z_t$ is as in Proposition ~\ref{222}. By Lemma \ref{205},  we may suppose that $B_s \cap S_{\mathcal{U}_x}=\{s\}$ for any $s \in S_{\mathcal{U}_x}.$

Let  $y\in S_{\mathcal{D}_x}.$ By Lemma \ref{103}, there exists an open neighborhood $A_y$ of $y$ in $\mathbb{P}_{Z}^{1, \mathrm{an}},$ such that $f_{Z}^{-1}(A_y) \subseteq \bigsqcup_{s \in f_x^{-1}(y)} B_s.$  
Let $D_1, D_2$ be the elements of $\mathcal{D}_x$ containing $y.$
By \cite[Lemma I.1.2]{stein}, there exists a connected affinoid neighborhood $Z_1 \subseteq Z_{t}$ of $x$, such that $D_{1, Z_1} \cap D_{2, Z_1}=(D_1 \cap D_2)_{Z_1} \subseteq A_y.$ Then $$f_{Z_1}^{-1}(D_{1, Z_1} \cap D_{2, Z_1}) \subseteq f_{Z_1}^{-1}(A_y) =f_{Z}^{-1}(A_y) \cap C_{Z_1} \subseteq \bigsqcup_{s \in f_x^{-1}(y)} B_s.$$
Let $W_{s, Z_1}, s \in f_x^{-1}(y),$ be the connected components of $f_{Z_1}^{-1}(D_{1, Z_1} \cap D_{2, Z_1})$, where for any $s \in f_x^{-1}(y)$, $s \in W_{s, Z_1}$  (see Proposition \ref{222}). Seeing as $\bigsqcup_{s \in f_x^{-1}(y)} W_{s, Z_1} \subseteq \bigsqcup_{s \in f_x^{-1}(y)} B_s$  and $B_s \cap S_{\mathcal{U}_x}=\{s\}$ for any $s \in f_x^{-1}(y),$ we obtain that $W_{s, Z_1} \subseteq B_s$. 

Let $Z_2 \subseteq Z_1$ be any connected affinoid neighborhood of $x.$ Seeing as the connected components of $f_{Z_2}^{-1}(D_{1, Z_2} \cap D_{2, Z_2})$ are $W_{s, Z_2}=W_{s, Z_1} \cap C_{Z_2}, s\in f_x^{-1}(y)$ (Proposition \ref{222}), all of the above remains true 
when replacing $Z_1$ by $Z_2.$ 

We obtain the statement by applying the above to all points of $S_{\mathcal{D}_x}.$
\end{proof}

\begin{summary} \label{suma} Let $(g_s)_{s \in S_{\mathcal{U}_x}} \in \prod_{s\in S_{\mathcal{U}_x}} G(\mathscr{M}_{C,s}).$ For any $s \in S_{\mathcal{U}_x},$ there exists a neighborhood $B_s$ of ~$s$ in $C,$ such that $g_s \in G(\mathscr{M}(B_s)).$ By Lemma \ref{227}, there exists an affinoid neighborhood $Z \subseteq Z_t$ (with $Z_t$ as in Proposition \ref{222})  of $x$ such that for any $s \in S_{\mathcal{U}_x},$ if $U_s, V_s$ are the elements of $\mathcal{U}_x$ containing $s,$ then $W_{s, Z} \subseteq B_s,$ where $W_{s,Z}$ is the connected component of $U_{s,Z} \cap V_{s,Z}$ containing $s.$ Consequently, $g_s \in G(\mathscr{M}(W_{s,Z})).$ Seeing as for any connected affinoid neighborhood $Z' \subseteq Z$, $W_{s, Z'}=W_{s, Z} \cap C_{Z'},$ the same remains true when replacing ~$Z$ by $Z'$. 	
\end{summary}
\

\begin{sloppypar}
\textit{(b) The descent to $\mathbb{P}^{1, \mathrm{an}}$.} Let $Z$ be as in Summary \ref{suma} above. The finite surjective morphism $f_Z: C_Z \rightarrow \mathbb{P}_Z^{1, \mathrm{an}}$ induces a finite field extension ${\mathscr{M}(C_Z)/\mathscr{M}(\mathbb{P}_Z^{1, \mathrm{an}}).}$ Set $G'=\mathcal{R}_{\mathscr{M}(C_Z)/\mathscr{M}(\mathbb{P}_Z^{1, \mathrm{an}})}(G)$ - the Weil restriction of scalars from $\mathscr{M}(C_Z)$ to $\mathscr{M}(\mathbb{P}_Z^{1, \mathrm{an}})$ of $G.$ This is still a connected rational linear algebraic group (see \cite[7.6]{neronmodels} or \cite[Section 1]{Mil}). For any ${y \in S_{\mathcal{D}_x}},$ by the universal property of $\mathcal{R},$ ${G'(\mathscr{M}_{\mathbb{P}_Z^{1, \mathrm{an}}, y})=G(\mathscr{M}_{\mathbb{P}_Z^{1, \mathrm{an}}, y} \otimes_{\mathscr{M}(\mathbb{P}_Z^{1, \mathrm{an}})} \mathscr{M}(C_Z)).}$ By Proposition \ref{209}, ${G'(\mathscr{M}_{\mathbb{P}_Z^{1, \mathrm{an}}, y})=\prod_{s \in f_x^{-1}(y)} G(\mathscr{M}_{C_Z, s})}$. Let $(g_s)_{s \in S_{\mathcal{U}_x}} \in \prod_{s \in S_{\mathcal{U}_x}} G(\mathscr{M}_{C_Z, s})$. This  uniquely determines an element $(h_y)_{y \in S_{\mathcal{D}_x}} \in \prod_{y \in S_{\mathcal{D}_x}} G'(\mathscr{M}_{\mathbb{P}_{Z}^{1, \mathrm{an}}, y}).$
\end{sloppypar} 
\begin{sloppypar} By Theorem \ref{147}, there exists a connected affinoid neighborhood $Z' \subseteq Z$ of $x,$ and $(h_D)_{D \in \mathcal{D}_x} \in \prod_{D \in \mathcal{D}_x} G'(\mathscr{M}(D_{Z'}))$, satisfying: for any $y \in S_{\mathcal{D}_x},$ there exist exactly two $D_y, D_y' \in \mathcal{D}_{x}$ containing ~$y$, ${h_y \in G'(\mathscr{M}(D_{y,Z'} \cap D'_{y, Z'}))},$ and if $T_{\mathcal{D}_x}(D_y)=0,$ then ${h_y=h_{D_y} \cdot h_{D'_y}^{-1}}$ in ${G'(\mathscr{M}(D_{y, Z'} \cap D'_{y, Z'}))}$. The same expression remains true for any connected affinoid neighborhood $Z'' \subseteq Z'$ of $x.$ 
\end{sloppypar}
\begin{sloppypar}
For $D \in \mathcal{D}_x,$ let $U_{1, Z'}, U_{2, Z'}, \dots, U_{n, Z'},$ be the connected components of $f_{Z'}^{-1}(D_{Z'}).$ The natural map  ${\mathscr{M}(D_{Z'}) \otimes_{\mathscr{M}(\mathbb{P}_{Z}^{1, \mathrm{an}})} \mathscr{M}(C_Z) \rightarrow \prod_{i=1}^n \mathscr{M}(U_{i, Z'})}$ (obtained by pull-backs and multiplication),
induces a map $${G'(\mathscr{M}(D_{Z'}))=G(\mathscr{M}(D_{Z'}) \otimes_{\mathscr{M}(\mathbb{P}_{Z}^{1, \mathrm{an}})} \mathscr{M}(C_Z)) \rightarrow \prod_{i=1}^n G(\mathscr{M}(U_{i,Z'}))}.$$
Let the image of $h_{D} \in G'(\mathscr{M}(D_{Z'}))$ by this map be the element $(g_{U_1}, g_{U_2},\dots, g_{U_n})$ of $\prod_{i=1}^n G(\mathscr{M}(U_{i, Z'})).$ Thus, for any $U_{Z'} \in \mathcal{U}_{Z'},$ we have an element $g_U \in G(\mathscr{M}(U_{Z'})).$ 
\end{sloppypar}

\ 

\textit{(c) The decomposition.} Finally, it remains to show that for any $U_0, U_1 \in \mathcal{U}_x$ such that $T_{\mathcal{U}_x}(U_0)=0,$ and ${s \in U_0 \cap U_1},$ if $W_{s, Z'}$ is the connected component of $U_{0,Z'} \cap U_{1,Z'}$ containing $s,$ then
${g_{s}=g_{U_0} \cdot g_{U_1}^{-1}}$ in $G(\mathscr{M}(W_{s,Z'}))$, and that the same expression remains true when replacing $Z'$ by any connected affinoid neighborhood $Z'' \subseteq Z'$ of $x.$ 

Let $y \in S_{\mathcal{D}_x}.$ Let $D_1, D_2$ be the elements of $\mathcal{D}_x$ containing $y$ such that $T_{\mathcal{D}_x}(D_1)=0.$ For any $s \in f_{x}^{-1}(y),$ let $W_{s, Z'}$ denote the connected component of $f_{Z'}^{-1}(D_{1,Z'} \cap D_{2,Z'})$ containing ~$s.$  
There is a  natural bilinear map $$\mathscr{M}(D_{1,Z'} \cap D_{2,Z'}) \times \mathscr{M}(C_Z) \rightarrow \prod_{s \in f_x^{-1}(y)} \mathscr{M}(W_{s, Z'}), \ (a, b) \mapsto ab,$$ which induces a map $\mathscr{M}(D_{1,Z'} \cap D_{2,Z'}) \otimes_{\mathscr{M}(\mathbb{P}_Z^{1, \mathrm{an}})} \mathscr{M}(C_Z) \rightarrow \prod_{s \in f_x^{-1}(y)} \mathscr{M}(W_{s, Z'})$ (this is ``compatible" with the isomorphism $\mathscr{M}_{\mathbb{P}_Z^{1, \mathrm{an}},y} \otimes_{\mathscr{M}(\mathbb{P}_Z^{1, \mathrm{an}})}  \mathscr{M}(C_Z)\rightarrow \prod_{s\in f_x^{-1}(y)} \mathscr{M}_{C_Z,s}$, \textit{i.e.} they are both induced by multiplication). 
Finally, this gives rise to a morphism $G'(\mathscr{M}(D_{1,Z'} \cap D_{2,Z'}))=G(\mathscr{M}(D_{1,Z'} \cap D_{2,Z'}) \otimes_{\mathscr{M}(\mathbb{P}_Z^{1, \mathrm{an}})} \mathscr{M}(C_Z)) \rightarrow \prod_{s \in f_x^{-1}(y)} G(\mathscr{M}(W_{s, Z'})),$ which sends (the restriction of) $h_y$ to (the restriction of) $(g_s)_{s \in f_x^{-1}(y)}.$ 
\begin{sloppypar}
Let $U_{i}, i=1,2,\dots, n,$ (resp. $V_{j},$ $j=1,2,\dots, m$) be the connected components of $f_{x}^{-1}(D_{1})$ (resp. $f_{x}^{-1}(D_{2})$). For any $i, j,$ set ${U_i \cap V_j=\{s_{\alpha}^{i,j}: \alpha=1,2,\dots, l_{i,j}\}}$ (if $U_i \cap V_j=\emptyset$ for some $i,j,$ then we take $l_{i,j}=0$). Remark that ${f_x^{-1}(y)=\{s_{\alpha}^{i,j}: \alpha=1,\dots, l_{i,j}, i=1,\dots, n, j=1,\dots,m\}}.$ For any $i,j, \alpha,$ let $W_{s_{\alpha}^{i,j}, Z'}$ be the connected component of $U_{i,Z'} \cap V_{j,Z'}$ containing $s_{\alpha}^{i,j}.$
\end{sloppypar}
For any $i$ (resp. $j$), there is a restriction map $\mathscr{M}(U_{i, Z'}) \rightarrow \prod_{j=1}^m \prod_{\alpha=1}^{l_{i,j}} \mathscr{M}(W_{s_{\alpha}^{i,j}, Z'})$ (resp. $\mathscr{M}(V_{j, Z'}) \rightarrow \prod_{i=1}^n \prod_{\alpha=1}^{l_{i,j}} \mathscr{M}(W_{s_{\alpha}^{i,j}, Z'})$). This induces a restriction map $$\prod_{i=1}^n \mathscr{M}(U_{i,Z'}) \rightarrow \prod_{i,j, \alpha} \mathscr{M}(W_{s_{\alpha}^{i,j}, Z'}) \ \left(\text{resp.} \ \prod_{j=1}^m \mathscr{M}(V_{j,Z'}) \rightarrow  \prod_{i,j,\alpha} \mathscr{M}(W_{s_{\alpha}^{i,j}, Z'})\right).$$
The commutative diagram

{\centering
\begin{tikzpicture}[baseline= (a).base]
\node[scale=.82] (a) at (0,0){
\begin{tikzcd}
\mathscr{M}(D_{1, Z'}) \otimes_{\mathscr{M}(\mathbb{P}_{Z}^{1, \mathrm{an}})}  \mathscr{M}(C_Z) \arrow{r} \arrow{d} & \mathscr{M}(D_{1,Z'} \cap D_{2,Z'}) \otimes_{\mathscr{M}(\mathbb{P}_{Z}^{1, \mathrm{an}})}  \arrow{d} \mathscr{M}(C_Z) \arrow{d}  &  \mathscr{M}(D_{2, Z'}) \otimes_{\mathscr{M}(\mathbb{P}_{Z}^{1, \mathrm{an}})}  \mathscr{M}(C_Z) \arrow{l} \arrow{d}  \\
\prod_{i=1}^n \mathscr{M}(U_{i,Z'}) \arrow{r} & \prod_{i,j, \alpha} \mathscr{M}(W_{s_{\alpha}^{ij},Z'})  & \prod_{j=1}^m \mathscr{M}(V_{j,Z'}) \arrow{l}
\end{tikzcd}
};
\end{tikzpicture}
}
gives rise to the following commutative diagram (where $\lambda_1, \lambda_2, \lambda_3$ are isomorphisms):

{\centering
\begin{tikzpicture}[baseline= (a).base]
\node[scale=.78] (a) at (-1,0){
\begin{tikzcd}
G'(\mathscr{M}(D_{1, Z'})) \arrow{r} \arrow{d}{\lambda_1} & G'(\mathscr{M}(D_{1,Z'} \cap D_{2,Z'})) \arrow{d}{\lambda_2} & G'(\mathscr{M}(D_{2, Z'})) \arrow{l} \arrow{d}{\lambda_3}  \\
G(\mathscr{M}(D_{1, Z'}) \otimes_{\mathscr{M}(\mathbb{P}_{Z}^{1, \mathrm{an}})}  \mathscr{M}(C_Z)) \arrow{r} \arrow{d} & G(\mathscr{M}(D_{1,Z'} \cap D_{2,Z'}) \otimes_{\mathscr{M}(\mathbb{P}_{Z}^{1, \mathrm{an}})} \mathscr{M}(C_Z)) \arrow{d}  &  G(\mathscr{M}(D_{2, Z'}) \otimes_{\mathscr{M}(\mathbb{P}_{Z}^{1, \mathrm{an}})}  \mathscr{M}(C_Z)) \arrow{l} \arrow{d}  \\
\prod_{i=1}^n G(\mathscr{M}(U_{i,Z'})) \arrow{r} & \prod_{i,j, \alpha} G(\mathscr{M}(W_{s_{\alpha}^{ij},Z'}))  & \prod_{j=1}^m G(\mathscr{M}(V_{j,Z'})) \arrow{l}
\end{tikzcd}
};
\end{tikzpicture}
}
The factorization result is now a consequence of the analoguous result 
for $(h_y)_{y \in S_{\mathcal{D}_x}}$ and $
(h_D)_{D \in \mathcal{U}_x},$ the relationship 
between $T_{\mathcal{D}_x}$ and 
$T_{\mathcal{U}_x},$ and the commutativity of 
the diagram above. More precisely, 
$h_y=h_{D_1} \cdot h_{D_2}^{-1}$ in 
$G'(\mathscr{M}(D_{1,Z'} \cap D_{2,Z'})),$ and $h_y$ is sent to $(g_s)_{s \in f_{x}^{-1}(y)},$ so for any $s_{\alpha}
^{i,j} \in f_x^{-1}(y),$  ${g_{s_{\alpha}
^{i,j}}=g_{U_i} \cdot g_{V_j}^{-1}}$ in 
$G(\mathscr{M}(W_{s_{\alpha}^{i,j},Z'}))$. 
 
Considering for any connected affinoid neighborhood $Z'' \subseteq Z'$ of $x,$ ${W_{s,Z''} = W_{s, Z'} \cap C_{Z''}}$ for any  $s \in S_{\mathcal{U}_x},$ and $U_{Z''}=U_{Z'} \cap C_{Z''}$ for all $U \in \mathcal{U}_x$, the same expressions remain true when replacing $Z'$ by $Z''.$
\end{proof}

\section{The local-global principles} \label{4.6}

Let $k$ be a complete non-trivially valued ultrametric field. Throughout this entire section,
we keep working with the hypotheses of Setting \ref{200}, and the related notations we have introduced (see Notation \ref{nota}). As before, we also suppose that ${\dim{S}< \dim_{\mathbb{Q}} \mathbb{R}_{>0}/|k^\times| \otimes_{\mathbb{Z}} \mathbb{Q}}.$

\begin{rem} Recall in particular that for $C_{\mathcal{O}_x}=C_{\mathcal{O}(Z_0)} \times_{\mathcal{O}(Z_0)} \mathcal{O}_x,$ its function field was denoted by $F_{\mathcal{O}_x}.$ It was shown in Corollary \ref{213} that $F_{\mathcal{O}_x}=\varinjlim_Z \mathscr{M}(C_Z),$ where $\mathscr{M}$ denotes the sheaf of meromorphic functions on $C$, and the direct limit is taken with respect to connected affinoid neighborhoods $Z$ of $x$ in $S.$
\end{rem}

\begin{defn}
Let $K$ be a field. A linear algebraic group $G$ over $K$ acts \emph{strongly transitively} on a $K$-variety $X$ if $G$ acts on $X$ and for any field extension $E/K,$ either $X(E)=\emptyset$ or the action of $G(E)$ on $X(E)$ is transitive. 
\end{defn}

\subsection{With respect to germs of meromorphic functions}

We show here the relative analogue of \cite[Theorem 3.11]{une}.

Recall that $C_x$ denotes the fiber at $x$ of the relative proper curve $C \rightarrow S$, and that it is a normal irreducible projective $\mathcal{H}(x)$-analytic curve. 

\begin{thm} \label{228}
Let $H/F_{\mathcal{O}_x}$ be a variety and $G/F_{\mathcal{O}_x}$ a connected rational linear algebraic group acting strongly transitively on $H.$ Then
$$H(F_{\mathcal{O}_x}) \neq \emptyset \iff H(\mathscr{M}_{C, u}) \neq \emptyset \ \text{for all} \ u \in C_x.$$
\end{thm}

\begin{proof}
 By Corollary \ref{213}, $F_{\mathcal{O}_x}=\varinjlim_Z \mathscr{M}(C_Z),$ where the limit is taken over connected affinoid neighborhoods $Z \subseteq Z_0$ of $x.$ Seeing as $x \in \text{Int}(Z),$ we obtain that for any $u \in C_x,$ $u \in \text{Int}(C_Z),$ so $\mathscr{M}_{C_{Z}, u}=\mathscr{M}_{C,u}.$

 As $H$ is defined over $F_{\mathcal{O}_x},$ there exists a connected affinoid neighborhood $Z_H \subseteq Z_0$ such that $H$ is defined over the field $\mathscr{M}(C_{Z_H}).$

($\Rightarrow$):  Let us start by remarking that for any $u \in C_x$, there is a restriction morphism $\mathscr{M}(C_Z) \rightarrow \mathscr{M}_{C,u},$ thus inducing an embedding $F_{\mathcal{O}_x} \rightarrow \mathscr{M}_{C,u},$ and making this implication of the statement immediate. 

($\Leftarrow$): Let us assume $H(\mathscr{M}_{C,u}) \neq \emptyset$ for all $u \in C_x.$ This implies that for any $u \in C_x$, there exists an open neighborhood $N'_u$ of $u$ in $C$ such that $H(\mathscr{M}(N'_u)) \neq \emptyset.$ We may, without loss of generality, assume that $N'_u \subseteq C_{Z_H}$ for all $u \in C_x.$ Let $\mathcal{V}$ denote the open cover $(N'_u)_{u \in C_x}$ of $C_x$ in $C$.

By Remark \ref{remark}, there exists a connected affinoid neighborhood $Z \subseteq Z_G \cap Z_H$ of $x$ ($Z_G$ as in Notation \ref{225}), and a nice cover $\mathcal{D}_x$ of $\mathbb{P}_{\mathcal{H}(x)}^{1, \mathrm{an}}$ such that they induce a refinement $\mathcal{U}_Z$ of $\mathcal{V}$ obtained as in construction (1) and satisfying properties (2) and (4) of Subsection ~\ref{affreux}. Let $\mathcal{U}_x$ denote the corresponding nice cover of $C_x$, $T_{\mathcal{U}_x}$ its associated parity function, and $S_{\mathcal{U}_x}$ the intersection points of the different elements of $\mathcal{U}_x.$ As $\mathcal{U}_Z$ refines $\mathcal{V},$ for any $U \in \mathcal{U}_x$ and any connected affinoid neighborhood $Z' \subseteq Z$ of $x,$ $H(\mathscr{M}(U_{Z'}))\neq \emptyset.$

For any $U \in \mathcal{U}_x,$ let us fix an element $U' \in \mathcal{V}$ for which $U_Z \subseteq U'$ for any connected affinoid neighborhood $Z \subseteq Z_t \cap Z_G \cap Z_H$  of $x$ with $Z_t$ as in Remark \ref{remark} (it exists seeing as $\mathcal{U}_Z$ refines $\mathcal{V}$, and for any ${Z' \subseteq Z''  \subseteq Z_t}$ that are connected affinoid neighborhoods of ~$x$, $U_{Z'}=U_{Z''} \cap C_{Z'}$). 

\

\textit{(a) Finding good neighborhoods of $s \in S_{\mathcal{U}_x}.$} Let $s \in S_{\mathcal{U}_x}.$ Let $U_s, V_s$ be the elements of ~$\mathcal{U}_x$ containing $s.$ Then $s \in U_s \cap V_s \subseteq U_s' \cap V_s'.$
Let $N_s \subseteq U_s' \cap V_s'$ be a neighborhood of ~$s$ in $C_{Z_H}$ such that $N_s \cap S_{\mathcal{U}_x}=\{s\}$ (this is possible considering Lemma \ref{205}). 

Let us fix a connected affinoid neighborhood $Z \subseteq Z_t \cap Z_G \cap Z_H$ of $x.$ Remark that for any $y \in S_{\mathcal{D}_x},$  $\bigsqcup_{s \in f_{x}^{-1}(y)} N_s$ is an open neighborhood of $f_x^{-1}(y)$ in $C_{Z_H},$ hence in $C_Z.$ By \cite[Lemma I.1.2]{stein}, there exists a connected neighborhood $A_y$ of $y$ in $\mathbb{P}_Z^{1, \mathrm{an}}$ such that $f_Z^{-1}(A_y) \subseteq \bigsqcup_{s \in f_{Z}^{-1}(y)} N_s.$ By Lemma \ref{103} (and restricting to a smaller $Z$ if necessary), we may assume that $A_y$ is the $Z$-thickening $A_Z$ of a connected affinoid domain $A$ of $\mathbb{P}_{\mathcal{H}(x)}^{1, \mathrm{an}}$ containing only type 3 points in its boundary.  By Corollary \ref{107}, we may assume that for any connected affinoid neighborhood $Z' \subseteq Z$ of $x$, the $Z'$-thickening $A_{Z'}$ of $A$ is connected. 

Let $B_{i,Z}, i=1,2,\dots, m,$ be the connected components of $f_Z^{-1}(A_Z).$ By Lemma \ref{215}, for any $i,$ $B_{i,Z} \cap C_x \neq \emptyset$ and $f_Z(B_{i,Z})=A_Z$, implying $B_{i,Z} \cap f_x^{-1}(y) \neq \emptyset$ for all $i.$ By Proposition \ref{218}, we may assume that ${B_{i,Z} \cap C_x}$ is connected for all $i$, and for any connected affinoid neighborhood $Z' \subseteq Z$ of $x,$ the connected components of $f_{Z'}^{-1}(A_{Z'})$ are $B_{i,Z'}=B_{i,Z} \cap C_{Z'}, i=1,2,\dots, n.$ 

Seeing as ${\bigsqcup_{i=1}^n B_{i,Z} \subseteq \bigsqcup_{s \in f_x^{-1}(y)} N_s},$ for any $i$, there exists exactly one $s_i \in f_x^{-1}(y)$ such that $B_{i, Z} \subseteq N_{s_i},$ which implies that $B_{i,Z} \cap f_x^{-1}(y)=\{s_i\}.$ As $f_x^{-1}(y) \subseteq \bigsqcup_{i=1}^n B_{i,Z}$ and $B_{i,Z} \cap f_x^{-1}(y) \neq \emptyset,$ there exists a bijective correspondence between the points of $f_x^{-1}(y)$ and the connected components of $f_Z^{-1}(A_Z).$ For $s \in f_x^{-1}(y),$ let $B_{s, Z}$ be the corresponding connected component of $f_Z^{-1}(A_Z)$ containing $s,$ so that $B_{s, Z} \subseteq N_s.$ Since the connected components of $f_{Z'}^{-1}(A_{Z'})$ are $B_{s, Z} \cap C_{Z'}, s\in f_x^{-1}(y),$ the same remains true when replacing $Z$ by $Z',$ meaning there is a bijective correspondence between the points of $f_x^{-1}(y)$ and the connected components of $f_{Z'}^{-1}(A_{Z'})$.

\

\textit{(b) The transitivity of the action.} For $s \in S_{\mathcal{U}_x},$ we denote by $U_s, V_s$ the elements of ~$\mathcal{U}_x$ containing $s$ and suppose $T_{\mathcal{U}_x}(U_s)=0.$ Then $s \in B_{s,Z} \subseteq U_s' \cap V_s'$, with $B_{s,Z}$ constructed as in part (a). Let $h_{U_s} \in H(\mathscr{M}(U_s'))$ and $h_{V_s} \in H(\mathscr{M}(V_s')).$ The restrictions of $h_{U_s}, h_{V_s}$ (which we keep denoting by $h_{U_s}, h_{V_s}$) to $\mathscr{M}(B_{s,Z})$ induce elements of $H(\mathscr{M}(B_{s,Z})),$ and the same remains true for any connected affinoid neighborhood $Z' \subseteq Z.$

\begin{lm} \label{229}
There exists a connected affinoid neighborhood $Z_s \subseteq Z$ of $x$ such that there exists $g_s \in G(\mathscr{M}(B_{s, Z_s}))$ satisfying $h_{U_s}=g_s \cdot h_{V_s}$ in $H(\mathscr{M}(B_{s, Z_s})).$ For any connected affinoid neighborhood $Z' \subseteq Z_s$ of $x,$ $h_{U_s}=g_s \cdot h_{V_s}$ in $H(\mathscr{M}(B_{s, Z'})).$
\end{lm}

\begin{proof}
Set $L=\varinjlim_Z \mathscr{M}(B_{s, Z}),$ where the limit is taken with respect to the connected affinoid neighborhoods $Z \subseteq Z_0$ of $x.$ As shown in Proposition \ref{218}, we may assume that $B_{s, Z}$ is connected for all such $Z \subseteq Z_0,$ so that $\mathscr{M}(B_{s, Z})$ are fields. Consequently, $L$ is a field. 
The restriction morphisms $\mathscr{M}(C_Z) \hookrightarrow \mathscr{M}(B_{s,Z})$ induce an embedding ${F_{\mathcal{O}_x} =\varinjlim_Z \mathscr{M}(C_Z) \hookrightarrow L.}$ Hence, $G(L)$ acts transitively on $H(L).$

As $h_{U_s}, h_{V_s} \in H(L),$ there exists $g_s\in G(L),$ for which $h_{U_s}=g_s \cdot h_{V_s}$ in $H(L).$ Consequently, there exists a connected affinoid neighborhood $Z_s$ of $x$ such that $g_s \in G(\mathscr{M}(B_{s,Z_s}))$ and $h_{U_s}=g_s \cdot h_{V_s}$ in $H(\mathscr{M}(B_{s,Z_s})).$ The same remains true for any connected affinoid neighborhood $Z' \subseteq Z_s$ of $x$ seeing as $B_{s, Z'}=B_{s, Z_s} \cap C_{Z'}$ by Proposition \ref{218}.
\end{proof}

By Lemma \ref{227}, there exists a connected affinoid neighborhood $Z_1 \subseteq Z$ of $x,$ such that for any $s \in S_{\mathcal{U}_x}$, if $W_{s, Z_1}$ is the connected component of $U_{s, Z_1} \cap V_{s, Z_1}$ containing ~$s,$ then $W_{s, Z_1} \subseteq B_{s, Z},$ so $W_{s, Z_1} \subseteq B_{s, Z} \cap C_{Z_1}=B_{s, Z_1}.$ Similarly, for any connected affinoid neighborhood $Z' \subseteq Z_1$, $W_{s, Z'} \subseteq B_{s,Z'}.$ Consequently, for any $s \in S_{\mathcal{U}_x}$, the equality ${h_{U_s}=g_s \cdot h_{V_s}}$ of Lemma \ref{229} is well defined in $H(\mathscr{M}(W_{s, Z'}))$ for any connected affinoid neighborhood ${Z' \subseteq \bigcap_{s \in S_{\mathcal{U}_x}} Z_s \cap Z_1}$ of $x$.

\

\textit{(c) The patching.} Let us fix a connected affinoid neighborhood $Z \subseteq Z_t \cap Z_G \cap Z_H$ of ~$x$. Then  $\mathcal{U}_Z$ is a cover of $C_Z,$ so $\{U' \in \mathcal{V}: U \in \mathcal{U}_x\}$ is an open cover of $C_Z$ in $C.$ For any $U' \in \mathcal{V},$ let us fix an element $h_U \in H(\mathscr{M}(U')).$ This gives rise to an element of $H(\mathscr{M}(U_{Z'}))$ for any connected affinoid neighborhood $Z' \subseteq Z$ of $x,$ which we will keep denoting by $h_U.$ 

By part (b), there exists $(g_s)_{s \in \mathcal{U}_x} \in \prod_{s \in S_{\mathcal{U}_x}} G(\mathscr{M}_{C,s})$ and a connected affinoid neighborhood $Z_2 \subseteq Z$ of $x$,  such that for any $s \in \mathcal{U}_x,$ if $U_s, V_s$ are the elements of $\mathcal{U}_x$ containing ~$s$, and $T_{\mathcal{U}_x}(U_s)=0$, then $g_s \in G(\mathscr{M}(W_{s,Z_2}))$, and $h_{U_s}=g_s \cdot h_{V_s}$ in $H(\mathscr{M}(W_{s,Z_2})),$ where $W_{s, Z_2}$ is the connected component of $U_{s,Z_2} \cap V_{s,Z_2}$ containing $s.$ Moreover, the same remains true when replacing $Z_2$ by any connected affinoid neighborhood $Z' \subseteq Z_2$ of $x.$

By Theorem \ref{226}, we may assume that $Z_2$ is such that there exists an element $(g_U)_{U \in \mathcal{U}_x}$ of $\prod_{U \in \mathcal{U}_x} G(\mathscr{M}(U_{Z_2})),$ such that for any non-disjoint $U,V \in \mathcal{U}_x$ with $T_{\mathcal{U}_x}(U)=0$, and any $s \in U \cap V,$ $g_s=g_U \cdot g_V^{-1}$ in $G(\mathscr{M}(W_{s,Z_2})),$ where $W_{s, Z_2}$ is the connected component of $U_{s,Z_2} \cap V_{s, Z_2}$ containing $s.$ Moreover, the same remains true when replacing $Z_2$ with any connected affinoid neighborhood $Z' \subseteq Z_2$ of $x.$

For any $U \in \mathcal{U}_x$, set $h_U'=g_U^{-1} \cdot h_U \in H(\mathscr{M}(U_{Z_2})).$ If $U, V$ are two non-disjoint elements of $\mathcal{U}_x,$ and $T_{\mathcal{U}_x}(U)=0,$ for any $s \in U \cap V,$ one obtains $h_V'=g_V^{-1} h_V=g_U^{-1} (g_U g_V^{-1}) h_V=g_U^{-1} g_s h_V=g_U^{-1}h_U=h_U'$ in $H(\mathscr{M}(W_{s, Z_2})),$ where $W_{s,Z_2}$ is the connected component of $U_{Z_2} \cap V_{Z_2}$ containing $s.$ Thus, $h'_{U|U_{Z_2} \cap V_{Z_2}}=h'_{V|U_{Z_2} \cap V_{Z_2}}$ in $H(\mathscr{M}(U_{Z_2} \cap V_{Z_2})).$

\

To summarize, we have an affinoid cover $\mathcal{U}_{Z_2}$ of $C_{Z_2},$ and for any $U_{Z_2} \in \mathcal{U}_{Z_2}$, an element $h_U' \in H(\mathscr{M}(U_{Z_2})).$ Moreover, for any $U_{Z_2}, V_{Z_2} \in \mathcal{U}_{Z_2},$ $h'_{U|U_{Z_2} \cap V_{Z_2}}=h'_{V|U_{Z_2} \cap V_{Z_2}}.$ Consequently, there exists $h \in H(\mathscr{M}(C_{Z_2}))$ such that $h_{|U_{Z_2}}=h'_{U}$ for any $U_{Z_2} \in \mathcal{U}_{Z_2}.$ Seeing as there is an embedding $\mathscr{M}(C_{Z_2}) \hookrightarrow F_{\mathcal{O}_x},$ we obtain that $H(F_{\mathcal{O}_x}) \neq \emptyset.$
\end{proof}

\begin{rem}
In fact, what we have shown is that, using the notation of Theorem \ref{228}, there exists a connected affinoid neighborhood $Z \subseteq Z_0$ such that $H$ is defined over $\mathscr{M}(C_Z)$ and $$H(\mathscr{M}(C_Z)) \neq \emptyset \iff H(\mathscr{M}_{C,u}) \neq \emptyset \ \text{for all} \ u \in C_x.$$
\end{rem}

\subsection{With respect to valuations}

Recall the notations mentioned at the beginning of this section. 

Since $\mathcal{O}_x$ is a field, there is an embedding $\mathcal{O}_x \hookrightarrow \mathcal{H}(x)$ which induces a valuation on~$\mathcal{O}_x.$ We will say that this is the valuation induced by $x$ on $\mathcal{O}_x.$

\begin{defn} \label{232}
We denote by $V(F_{\mathcal{O}_x})$ the set of non-trivial rank one valuations $v$ on ~$F_{\mathcal{O}_x}$ such that either $v_{|\mathcal{O}_x}$ is the valuation induced by $x$ on $\mathcal{O}_x$ or $v_{|\mathcal{O}_x}$ is trivial. Set $V'(F_{\mathcal{O}_x})=\{v \in V(F_{\mathcal{O}_x}): v_{|\mathcal{O}_x} \ \text{is the valuation induced by} \ x \ \text{on} \ \mathcal{O}_x\}.$
For any $v \in V(F_{\mathcal{O}_x}),$ we denote by $F_{\mathcal{O}_x, v}$ the completion of $F_{\mathcal{O}_x}$ with respect to $v.$ 
\end{defn}

\begin{rem} \label{233}
For any non-rigid point $y \in C_x,$ $\mathcal{O}_{C_x,y}$ is a field, so by Lemma \ref{208}, $\mathcal{O}_{C,y}$ is a field, and there is an embedding $\mathcal{O}_{C,y}=\mathscr{M}_{C,y} \hookrightarrow \mathcal{H}(y).$ We endow $\mathscr{M}_{C,y}$ with the valuation induced from $\mathcal{H}(y).$

For any rigid point $y \in C_x$, $\mathcal{O}_{C_x, y}$ is a discrete valuation ring, so by Lemma \ref{208}, $\mathcal{O}_{C, y}$ is a discrete valuation ring. We endow $\mathscr{M}_{C,y}$ with the corresponding discrete valuation. 
\end{rem}

\begin{prop}\label{234}
There exists a surjective map $\mathrm{val}: C_x \rightarrow V(F_{\mathcal{O}_x}),$ $y \mapsto v_y,$ such that: if $y \in C_x$ is not rigid, then $v_{y|\mathcal{O}_x}$ induces the norm determined by $x$ on $\mathcal{O}_x$, and $F_{\mathcal{O}_x, v_y}=\widehat{\mathscr{M}_{C,y}};$ if $y \in C_x$ is rigid, then $v_y$ is discrete, $v_{y|\mathcal{O}_x}$ is trivial, and $F_{\mathcal{O}_x, v_y} \hookrightarrow \widehat{\mathscr{M}_{C,y}}.$ 
\begin{sloppypar}
Let $C_{x, \mathrm{nrig}}$ denote the set of non-rigid points on $C_x.$ The restriction ${\mathrm{val}_{|C_{x, \mathrm{nrig}}}: C_{x, \mathrm{nrig}} \rightarrow V'(F_{\mathcal{O}_x})}$ is a bijection.
\end{sloppypar} 
\end{prop}

\begin{proof}
\noindent \textbf{The construction of the map $\mathrm{val}$:}
Let $y \in C_x$ be a non-rigid point. Then $\mathcal{O}_{C_x, y}$ is a field, and so is $\mathcal{O}_{C, y}.$ Consequently, $\widehat{\mathscr{M}_{C,y}}=\mathcal{H}(y),$ so for any connected affinoid neighborhood $Z$ of $x$, $\widehat{\mathscr{M}(C_Z)}=\widehat{\mathscr{M}_{C,y}}$, where the completion of $\mathscr{M}(C_Z)$ is taken with respect to the norm induced by the embedding $\mathscr{M}(C_Z) \hookrightarrow \mathcal{H}(y).$ Considering $F_{\mathcal{O}_x}=\varinjlim_Z \mathscr{M}(C_Z) \hookrightarrow \mathscr{M}_{C,y}$, and as  $\widehat{\mathscr{M}(C_Z)}=\widehat{\mathscr{M}_{C,y}}$ for any connected affinoid neighborhood $Z \subseteq Z_0$ of $x,$ we obtain that $F_{\mathcal{O}_x, v_y}=\widehat{\mathscr{M}_{C,y}}.$ The fact that $v_{y|\mathcal{O}_x}$ is the norm determined by $x$ on $\mathcal{O}_x$ is a direct consequence of the fact that $y \in C_x$ is non-rigid. 

Let $y \in C_x$ be a rigid point. Then $\mathcal{O}_{C_x, y}$ is a discrete valuation ring, and by Lemma \ref{208}, so is $\mathcal{O}_{C,y}.$ As $\pi(y)=x,$ this induces a morphism of local rings $\mathcal{O}_{x} \rightarrow \mathcal{O}_{C,y}$. Furthermore, since $\mathcal{O}_x$ is a field, $\mathcal{O}_x \hookrightarrow \mathcal{O}_{C,y}^{\times}.$ As seen above, there is an embedding $F_{\mathcal{O}_x} \hookrightarrow \mathscr{M}_{C,y}$. Let us endow $\mathscr{M}_{C,y}$ with the discrete valuation arising from the discrete valuation ring $\mathcal{O}_{C,y}.$ This induces a discrete valuation $v_y$ on $F_{\mathcal{O}_x}.$ That $v_{y|\mathcal{O}_x}$ is trivial is immediate from the embedding $\mathcal{O}_x \hookrightarrow \mathcal{O}_{C,y}^{\times}.$ Clearly, this gives rise to an embedding $F_{\mathcal{O}_x, v_y} \hookrightarrow \widehat{\mathscr{M}_{C,y}}.$
 
\

\noindent \textbf{The map $\mathrm{val}_{|C_{x, \mathrm{nrig}}}$:} It remains to show that the restriction $\mathrm{val}_{|C_{x,\mathrm{nrig}}}: C_{x, \mathrm{nrig}} \rightarrow V'(F_{\mathcal{O}_x})$ is bijective. Let $v \in V'(F_{\mathcal{O}_x})$. Then since $\mathcal{O}_x \hookrightarrow F_{\mathcal{O}_x}$, there is an embedding ${\mathcal{H}(x) \hookrightarrow F_{\mathcal{O}_x, v}}.$ This implies that there is a morphism ${F_{\mathcal{O}_x} \otimes_{\mathcal{O}_x} \mathcal{H}(x) \rightarrow F_{\mathcal{O}_x, v}}.$ Let $C_{x}^{\mathrm{alg}}$ denote the normal irreducible projective algebraic curve over $\mathcal{H}(x)$ whose Berkovich analytification is $C_x.$ Its function field is $\mathscr{M}(C_x)$ by \cite[Proposition 3.6.2]{Ber90}.  
\begin{sloppypar}
Let $x'$ denote the image of $x$ via the morphism $Z_0 \rightarrow \text{Spec} \ \mathcal{O}(Z_0)$, where $Z_0$ is as in Setting \ref{200}. Using Notation \ref{202}, by the proof of \cite[Proposition 2.6.2]{ber93}, ${C_x=(C_{\mathcal{O}(Z_0),\kappa(x')} \times_{\kappa(x')} \mathcal{H}(x))^{\mathrm{an}}},$ so $C_x^{\mathrm{alg}}=C_{\mathcal{O}(Z_0),\kappa(x')} \times_{\kappa(x')} \mathcal{H}(x).$ Seeing as $\mathcal{O}_x$  is a field, we have an embedding $\kappa(x') \hookrightarrow \mathcal{O}_x,$ so $C_x^{\mathrm{alg}}=C_{\mathcal{O}_x} \times_{\mathcal{O}_x} \mathcal{H}(x)$. This means that its function field is $\mathscr{M}(C_x)=F_{\mathcal{O}_x} \otimes_{\mathcal{O}_x} \mathcal{H}(x).$ 

Consequently, there are embeddings $F_{\mathcal{O}_x} \hookrightarrow \mathscr{M}(C_x) \hookrightarrow F_{\mathcal{O}_x, v},$ implying $\widehat{\mathscr{M}(C_x)^v}=F_{\mathcal{O}_x, v},$ where $\widehat{\mathscr{M}(C_x)^v}$ is the completion of $\mathscr{M}(C_x)$ with respect to $v.$ By \cite[Proposition ~3.15]{une}, there exists a unique (implying both injectivity and surjectivity of $\mathrm{val}_{|C_{x, \mathrm{nrig}}}$) non-rigid point $y \in C_x$ such that $\widehat{\mathscr{M}_{C,y}}=\mathcal{H}(y)=\widehat{\mathscr{M}_{C_x, y}}=F_{\mathcal{O}_x, v}.$ Clearly, ~$v=\mathrm{val}(y).$
\end{sloppypar}

\end{proof}

Taking this result into account, as a consequence of Theorem \ref{228}, we obtain:
\begin{cor} \label{235}
 With the notation of Theorem \ref{228}, if $\mathrm{char} \ k=0$ or $H$ is smooth, then: $$H(F_{\mathcal{O}_x}) \neq \emptyset \iff H(F_{\mathcal{O}_x, v}) \neq \emptyset \ \text{for all} \ v \in V(F_{\mathcal{O}_x}).$$
\end{cor}

\begin{proof}
($\Rightarrow$): Seeing as $F_{\mathcal{O}_x}$ embeds in $F_{\mathcal{O}_x, v}$ for all $v \in V(F_{\mathcal{O}_x}),$ this direction is immediate.

($\Leftarrow$): Remark that $F_{\mathcal{O}_x}$ is perfect if and only if $\mathrm{char} \ k=0.$
Suppose $H(F_{\mathcal{O}_x, v}) \neq \emptyset$ for all $v \in V(F_{\mathcal{O}_x}).$ By Proposition \ref{234},  for any $y \in C_x,$ there exists $v \in V(F_{\mathcal{O}_x}),$ such that $F_{\mathcal{O}_x,v} \subseteq \widehat{\mathscr{M}_{C,y}}.$ Hence, ${H(\widehat{\mathscr{M}_{C,y}}) \neq \emptyset}$ for all $y \in C_x.$ If $y$ is a non-rigid point of $C_x$, then $\mathcal{O}_{C,y}=\mathscr{M}_{C,y}$ is a Henselian field by \cite[Theorem ~2.3.3]{ber93}. If $y$ is rigid point, then $\mathcal{O}_{C,y}$ is a discrete valuation ring that is Henselian, so by \cite[Proposition 2.4.3]{ber93}, ${\mathscr{M}_{C,y}=\text{Frac} \ \mathcal{O}_{C,y}}$ is Henselian. By \cite[Lemma 3.16]{une},  $H(\mathscr{M}_{C,y}) \neq \emptyset$ for all $y \in C_x.$ Finally, by Theorem \ref{228}, this implies that $H(F_{\mathcal{O}_x}) \neq \emptyset.$
\end{proof}

\subsection{Summary of results}
Recall that $(k, |\cdot|)$ denotes a complete non-trivially valued ultrametric field. As usual, we denote by $\mathscr{M}$ the sheaf of meromorphic functions.

We also recall that a morphism $f$ of $k$-analytic spaces is said to be \textit{algebraic} if there exists a morphism of schemes $g$ such that $f=g^{\mathrm{an}}.$

Let us summarize the main results we have shown in this section:

\begin{thm} \label{katastrofe}
Let $S,C$ be $k$-analytic spaces such that $S$ is good and normal. Suppose $\dim{S}<\dim_{\mathbb{Q}} \mathbb{R}_{>0}/|k^{\times}| \otimes_{\mathbb{Z}} \mathbb{Q}.$ Suppose there exists a morphism $\pi: C \rightarrow S$ that makes $C$ a proper flat relative $S$-analytic curve. Let $x \in \mathrm{Im}(\pi)$ be such that $\mathcal{O}_x$ is a field. Set $C_x=\pi^{-1}(x).$ 
\sloppypar
Assume there exists a connected affinoid neighborhood $Z_0$ of $x$ such that all the fibers of $\pi$ on $Z_0$ are normal irreducible projective analytic curves. Set $C_{Z_0}:=\pi^{-1}(Z_0)$. Suppose that $C_{Z_0} \rightarrow Z_0$ is algebraic, \textit{i.e.} the analytification of a scheme morphism ${C_{\mathcal{O}(Z_0)} \rightarrow \mathrm{Spec} \ \mathcal{O}(Z_0)}$. Set $C_{\mathcal{O}_x}=C_{\mathcal{O}(Z_0)} \times_{\mathcal{O}(Z_0)} \mathcal{O}_x.$ Let $F_{\mathcal{O}_x}$ be the  function field of ~$C_{\mathcal{O}_x}$.

For any connected affinoid neighborhood $Z \subseteq Z_0$ of $x,$ let us denote by $C_Z$ the analytic space $C \times_S Z.$ Then $F_{\mathcal{O}_x}=\varinjlim_Z \mathscr{M}(C_Z).$

Let $G/F_{\mathcal{O}_x}$ be a connected rational linear algebraic group acting strongly transitively on a variety $H/F_{\mathcal{O}_x}.$ The following local-global principles hold:
\begin{enumerate}
\item $H(F_{\mathcal{O}_x}) \neq \emptyset \iff H(\mathscr{M}_{C,u}) \neq \emptyset \ \text{for all} \ u \in C_x;$
\item if $\mathrm{char} \ k=0$ or $H$ is smooth,
$$H(F_{\mathcal{O}_x}) \neq \emptyset \iff H(F_{\mathcal{O}_x,v}) \neq \emptyset \ \text{for all} \ v \in V(F_{\mathcal{O}_x}),$$
where $V(F_{\mathcal{O}_x})$ is given as in Definition \ref{232}. 
\end{enumerate}\end{thm}

For a proof of Theorem \ref{katastrofe}(1), see Theorem \ref{228}. For a proof of Theorem \ref{katastrofe}(2), see Corollary \ref{235}.

The theorem above tells us that there is a local-global principle in the neighborhood of certain fibers of  relative proper analytic curves. More generally, we have shown that patching is possible in the neighborhood of said fibers.  Note that the statement of Theorem~\ref{katastrofe} is a local-global principle over the germs of meromorphic functions on a fixed fiber. 

As a consequence of Theorem \ref{addition1}, we obtain:

\begin{cor} \label{ahhh1}
Let $S, C$ be $k$-analytic spaces such that $S$ is strict, good and regular. Suppose $\dim{S}<\dim_{\mathbb{Q}}\mathbb{R}_{>0}/|k^{\times}| \otimes_{\mathbb{Z}} \mathbb{Q}.$ Let $\pi: C \rightarrow S$ be a morphism that makes $C$ a proper flat relative $S$-analytic curve. Let $x \in S$ be such that $\mathcal{O}_x$ is a field and $C_x:=\pi^{-1}(x) \neq \emptyset.$ Suppose that $C_x$ is a smooth geometrically irreducible $\mathcal{H}(x)$-analytic curve. Then the statement of Theorem~\ref{katastrofe} is satisfied. 
\end{cor}

As a consequence of Proposition \ref{example lame'}, we also have:

\begin{cor} \label{yll}
Let $S, C$ be $k$-analytic spaces such that $S$ is good and normal. Suppose $\dim{S}<\dim_{\mathbb{Q}} \mathbb{R}_{>0}/|k^{\times}| \otimes_{\mathbb{Z}} \mathbb{Q}.$ Let $\pi: C \rightarrow S$ be a morphism that makes $C$ a proper flat relative $S$-analytic curve. Let $x \in S$ be such that $\mathcal{O}_x$ is a field and ${C_x:=\pi^{-1}(x) \neq \emptyset}.$ Suppose that $C_x$ is a smooth geometrically irreducible $\mathcal{H}(x)$-analytic curve. Suppose also that there exists an affinoid neighborhood $Z_0$ of $x$ in $S$ such that the base change ${C \times_S Z_0 \rightarrow Z_0}$ is algebraic. Then the statement of Theorem~\ref{katastrofe} is satisfied. 
\end{cor}

Considering the example of Setting \ref{200} given in Subsection \ref{lame}, we also obtain the following theorem, which is a generalization of \cite[Corollary 3.18]{une} and of the main result of \cite{HHKP}. 

\begin{thm} \label{val2}
Let $S$ be a  $k$-analytic space that is good and normal. Suppose ${\dim{S}<\dim_{\mathbb{Q}} \mathbb{R}_{>0}/|k^{\times}| \otimes_{\mathbb{Z}} \mathbb{Q}}$. Let $x \in S$ be such that $\mathcal{O}_x$ is a field. Let $C_{\mathcal{O}_x}$ be a smooth geometrically irreducible projective algebraic curve over $\mathcal{O}_x.$ Let $F_{\mathcal{O}_x}$ denote the function field of $C_{\mathcal{O}_x}.$ 

Let $G/F_{\mathcal{O}_x}$ be a connected rational linear algebraic group acting strongly transitively on a variety $H/F_{\mathcal{O}_x}.$ Then if $\mathrm{char} \ k=0$ or $H$ is smooth: 
$$H(F_{\mathcal{O}_x}) \neq \emptyset \iff H(F_{\mathcal{O}_x, v}) \neq \emptyset \ \text{for all} \ v \in V(F_{\mathcal{O}_x}),$$
where $V(F_{\mathcal{O}_x})$ is given in Definition \ref{232}. 
\end{thm}

As in \cite{une}, the local-global principles we have obtained can be applied to quadratic forms. The following is a consequence of Theorem \ref{katastrofe}.

\begin{cor} \label{ahhh'} Suppose that $\mathrm{char} \ k \neq 2.$
Let $S,C$ be $k$-analytic spaces such that $S$ is good and normal. Suppose $\dim{S}<\dim_{\mathbb{Q}} \mathbb{R}_{>0}/|k^{\times}| \otimes_{\mathbb{Z}} \mathbb{Q}.$ Suppose there exists a morphism $\pi: C \rightarrow S$ that makes $C$ a proper flat relative $S$-analytic curve. Let $x \in \mathrm{Im}(\pi)$ be such that $\mathcal{O}_x$ is a field. Set $C_x=\pi^{-1}(x).$ 

Assume there exists a connected affinoid neighborhood $Z_0$ of $x$ such that all the fibers of $\pi$ on $Z_0$ are normal irreducible projective analytic curves. Set $C_{Z_0}:=\pi^{-1}(Z_0)$ and suppose that $C_{Z_0} \rightarrow Z_0$ is algebraic, \textit{i.e.} the analytification of an algebraic morphism $C_{\mathcal{O}(Z_0)} \rightarrow \text{Spec} \ \mathcal{O}(Z_0).$ Set $C_{\mathcal{O}_x}=C_{\mathcal{O}(Z_0)} \times_{\mathcal{O}(Z_0)} \mathcal{O}_x.$ Let $F_{\mathcal{O}_x}$ be the  function field of $C_{\mathcal{O}_x}$.

For any connected affinoid neighborhood $Z \subseteq Z_0$ of $x,$ let us denote by $C_Z$ the analytic space $C \times_S Z.$ Then $F_{\mathcal{O}_x}=\varinjlim_Z \mathscr{M}(C_Z).$

Let $q/F_{\mathcal{O}_x}$ be a quadratic form of dimension $\neq 2$. The following local-global principles hold:
\begin{itemize}
\item $q$ is isotropic over $F_{\mathcal{O}_x}$ if and only if it is isotropic over $\mathscr{M}_{C,u}$ for all $u \in C_x;$
\item $q$ is isotropic over $F_{\mathcal{O}_x}$ if and only if it is isotropic over $F_{\mathcal{O}_x,v}$ for all $v \in V(F_{\mathcal{O}_x}),$
where $V(F_{\mathcal{O}_x})$ is given in Definition \ref{232}. 
\end{itemize}
\end{cor}

Corollaries \ref{ahhh1} and \ref{yll} are also applicable to Corollary \ref{ahhh'}, meaning their statements remain true when replacing ``Theorem \ref{katastrofe}" with ``Corollary \ref{ahhh'}".

The next result is an application of Theorem \ref{val2}.

\begin{cor} \label{ahhh2} Suppose $\mathrm{char} \ k \neq 2.$
Let $S$ be a good normal $k$-analytic space such that $\dim{S}<\dim_{\mathbb{Q}}\mathbb{R}_{>0}/|k^{\times}| \otimes_{\mathbb{Z}} \mathbb{Q}$. Let $x \in S$ be such that $\mathcal{O}_x$ is a field. Let $C_{\mathcal{O}_x}$ be a smooth geometrically irreducible projective algebraic curve over $\mathcal{O}_x.$ Let $F_{\mathcal{O}_x}$ denote the function field of $C_{\mathcal{O}_x}.$ 

Let $q/F_{\mathcal{O}_x}$ be a quadratic form of dimension $\neq 2.$ Then $q$ is isotropic over $F_{\mathcal{O}_x}$ if and only if it is isotropic over the completions $F_{\mathcal{O}_x, v}$ for all $v \in V(F_{\mathcal{O}_x}),$
where $V(F_{\mathcal{O}_x})$ is given as in Definition \ref{232}. 
\end{cor}

\begin{rem}
The field $\mathcal{O}_x$ appearing in the statements of this subsection is Henselian, but in general not complete. 
Taking this into account, Theorem \ref{val2} and Corollary \ref{ahhh2} generalize \mbox{\cite[Corollary 3.18]{une}} and 
\cite[Corollary 3.19]{une}, respectively.
\end{rem}

\section{Examples of fields $\mathcal{O}_x$} \label{4.7} To illustrate on which types of fields our local-global principles can be applied, we calculate a few examples of local rings $\mathcal{O}_x$ that are fields. To do this, the key is to find a ``good" basis of neighborhoods of the point $x.$ 
\begin{sloppypar} 
We denote by $(k, |\cdot|)$ a complete ultrametric field such that ${\dim_{\mathbb{Q}} \mathbb{R}_{>0}/|k^\times| \otimes_{\mathbb{Z}} \mathbb{Q} =\infty}$ (this condition is sufficient to guarantee the existence of type 3 points on the fiber of $x$). In all of the following examples, $x$ is chosen such that $\mathcal{O}_x$ is a field. 
\end{sloppypar}

\begin{ex} Suppose $S=\mathcal{M}(k),$ where $\mathcal{M}( \ \cdot \ )$ denotes the Berkovich spectrum. Then, if $S=\{x\},$ we obtain that $\mathcal{O}_x=k,$ so a special case of Theorem~\ref{228} is \cite[Theorem 3.10]{une}. 
\end{ex}

\begin{ex} \label{exi2}
Let $\eta_{T,r} \in \mathbb{A}_k^{1, \mathrm{an}}$ be a type 3 point, meaning $r \not \in \sqrt{|k|}.$ We can deduce from \cite[3.4.19.3]{Duc}, that the family of sets  $L_{r_1, r_2}:=\{y \in \mathbb{A}_k^{1, \mathrm{an}}: r_1 \leqslant |T|_y \leqslant r_2\}, {0<r_1<r<r_2,}$ forms a basis of neighborhoods of $\eta_{T,r}$ in $\mathbb{A}_k^{1, \mathrm{an}}.$ Considering $\mathcal{O}(L_{r_1, r_2})=\{\sum_{n \in \mathbb{Z}} a_n T^n: a_n \in k,  \lim_{n \rightarrow +\infty}|a_n|r_2^n=0, \lim_{n \rightarrow -\infty} |a_n|r_1^n=0\}$, we obtain that $$\mathcal{O}_x= \left\lbrace\sum_{n \in \mathbb{Z}} a_nT^n : a_n \in k, \exists r_1, r_2 \in \mathbb{R}_{>0},\ \text{s.t.} \ r_1<r<r_2, \lim_{n \rightarrow +\infty} |a_n|r_2^n=0, \lim_{n \rightarrow -\infty} |a_n|r_1^n=0 \right\rbrace$$
The norm that $x$ induces on $\mathcal{O}_x$ is the following: $|\sum_{n \in \mathbb{Z}} a_nT^n|_x=\max_{n \in \mathbb{Z}} |a_n|r^n.$
\end{ex}

\begin{nota}
For $\alpha \in k$ and $r \in \mathbb{R}_{>0} 0,$ let us denote by $B_k(\alpha,r)$ the closed disc in $k$ centered at $a$ and of radius $r.$ Also, for $P \in k[T]$ irreducible, we denote $\mathbb{D}_k(P,r):=\{y \in \mathbb{A}_k^{1, \mathrm{an}}: |P|_{y} \leqslant r\}$ (resp. $\mathbb{D}_k^\circ (P,r):=\{y \in \mathbb{A}_k^{1, \mathrm{an}}: |P|_{y} < r\}$) the closed (resp. open) virtual disc centered at $\eta_{P,0}$ and of radius $r.$ In particular, if there exists $\alpha \in k$ such that  $P(T)=T-\alpha,$ we will simply write $\mathbb{D}_k(\alpha, r)$ (resp. $\mathbb{D}_k^\circ(\alpha,r)$). When there is no risk of ambiguity, we will forget the index $k.$
\end{nota}

\begin{ex} Suppose $k$ is algebraically closed.
Let $x=\eta_{T-\alpha,r} \in \mathbb{A}_{k}^{1, \mathrm{an}}$ be a type 2 point, meaning $r \in |k^\times|.$ By \cite[3.4.19.2]{Duc}, $x$ has a basis of neighborhoods of the form $A_{R,\alpha_i, r_i, I} :=\mathbb{D}(\alpha, R) \backslash \bigsqcup_{i \in I} \mathbb{D}^\circ(\alpha_i,r_i),$ where $I$ is a finite set, $0<r_i<r$ for all $i \in I,$ $R>r,$ $\alpha_i \in B(\alpha,r),$ and for any $i, j \in I, i\neq j,$ we have $|\alpha_i-\alpha_j|=r.$ The subset $A_{R,\alpha_i, r_i, I}$ is an affinoid domain in $\mathbb{A}_{k}^{1, \mathrm{an}}$. By \cite[Proposition 2.2.6]{frevan}, 
\begin{align*}
\mathcal{O}(A_{R,\alpha_i, r_i, I})=&\big\lbrace\sum_{n>0}\sum_{i \in I} \frac{a_{n,i}}{(T-\alpha_i)^n}+\sum_{n \geqslant 0} a_n (T-\alpha)^n: \\& a_{n,i},a_n \in k, \lim_{n \rightarrow +\infty} |a_{n,i}|r_i^{-n}=0, i \in I, \lim_{n \rightarrow +\infty}|a_n|R^n=0\big\rbrace. \end{align*} Consequently, $f \in \mathcal{O}_x$ if and only if there exist a finite set $I \subseteq \mathbb{N},$ positive real numbers $R, r_i, i\in I,$ such that $r_i < r<R$, and elements $\alpha_i \in B(\alpha,r),$ such that $|\alpha_i-\alpha_j|=r$ for any $i,j \in I, i \neq j,$ satisfying $f \in \mathcal{O}(A_{R, \alpha_i, r_i, I}).$ The norm induced by $x$ is $$\left|\sum_{n>0}\sum_{i \in I} \frac{a_{n,i}}{(T-\alpha_i)^n}+\sum_{n \geqslant 0}a_n(T-\alpha)^n\right|_x=\max_{n> 0, i \in I}(|a_0|, |a_{n,i}|r^{-n}, |a_n|r^n).$$
\end{ex}

\begin{ex} Suppose $k$ is algebraically closed.
Let $x \in \mathbb{A}_{k}^{1, \mathrm{an}}$ be a type 4 point, meaning it is determined by a strictly decreasing family of closed discs $\mathscr{D}:=(B(a_i, r_i))_{i \in \mathbb{N}}$
in $k$ such that $\bigcap_{i \in \mathbb{N}} B(a_i,r_i)=\emptyset.$ Then for any $Q(T) \in k[T],$ $|Q|_x=\inf_{i} |Q|_{\eta_{a_i, r_i}}.$
Let us remark that for any $i \in \mathbb{N},$ $x \in \mathbb{D}(a_i, r_i).$ Moreover, $x \in \mathbb{D}^\circ(a_i, r_i).$ To see the last part, assume, by contradiction, that there exists $j \in \mathbb{N}$ such that $|T-a_j|_x=r_j$. Then for any $i > j,$ $\max(|a_i-a_j|, r_i) =|T-a_j|_{\eta_{a_i, r_i}} \geqslant r_j,$ which is impossible seeing as $\mathscr{D}$ is strictly decreasing.

By \cite[3.4.19.1]{Duc}, the elements of $\mathscr{D}':=(\mathbb{D}(a_i, r_i))_{i \in \mathbb{N}}$ form a basis of neighborhoods of ~$x$. Finally, for any $f \in \mathcal{O}_x,$ there exists $i' \in \mathbb{N}$ such that $f \in \mathcal{O}(\mathbb{D}(a_{i'}, r_{i'}))$, meaning ${f=\sum_{n \in \mathbb{N}} b_n (T-a_{i'})^n},$ where $b_n \in k$ for all $n,$ and $\lim_{n \rightarrow +\infty} |b_n|r_{i'}^n=0.$
Then for any ${i \geqslant i'},$  $f \in \mathcal{O}(\mathbb{D}(a_i, r_i))$. Finally, the norm induced by $x$ is ${|f|_x=\inf_{i \geqslant i'} |f|_{\eta_{a_i, r_i}}}.$
\end{ex}

\begin{ex} Let us fix an algebraic closure $\overline{k}$ of $k.$
Let $x \in \mathbb{A}_{k}^{1, \mathrm{an}}$ be a non-rigid type 1 point. 
This means that there exists an element $\alpha \in \widehat{\overline{k}} \backslash 
\overline{k}$, such that the image of $\eta_{\alpha,0}$ with 
respect to the open surjective morphism $\varphi: \mathbb{A}
_{\widehat{\overline{k}}}^{1, \mathrm{an}} \rightarrow \mathbb{A}_{k}
^{1, \mathrm{an}}$ is $x.$ There exists a sequence $(\alpha_i)_{i 
\in \mathbb{N}}$ in $\overline{k}$ 
such that $\lim_{i \rightarrow +\infty} \alpha_i =\alpha.$ Set $r_i=|\alpha-\alpha_i|.$ Then, in $\widehat{\overline{k}},$ the point $\eta_{\alpha,0}$ is determined by the strictly 
decreasing family of closed discs $(B_{\widehat{\overline{k}}}(\alpha_i, r_i))_{i \in \mathbb{N}},$ meaning for any $Q \in \widehat{\overline{k}}[T],$ $|Q|_{\eta_{\alpha,0}}=\inf_{i} |Q|_{\eta_{\alpha_i, r_i}}.$ As in Example $4$, by \cite[3.4.19.1]{Duc}, the family $(\mathbb{D}_{\widehat{\overline{k}}}(\alpha_i, r_i))_{i \in \mathbb{N}}$ forms a family of neighborhoods of $\eta_{\alpha,0}$ in $\mathbb{A}_{\widehat{\overline{k}}}^{1, \mathrm{an}}.$ 

Seeing as $\varphi$ is an open morphism, $(\varphi(\mathbb{D}_{\widehat{\overline{k}}}(\alpha_i, r_i)))_{i \in \mathbb{N}}$ forms a basis of neighborhoods of the point $x$ in $\mathbb{A}_{k}^{1, \mathrm{an}}.$ For any $i$, let $P_i \in \mathbb{Q}_p[T]$ denote the minimal polynomial of $\alpha_i$ over ~$k$. Then $\varphi(\mathbb{D}_{\widehat{\overline{k}}}(\alpha_i, r_i))=\mathbb{D}_{k}(P_i, s_i)$, where $s_i= \prod_{P_i(\beta)=0} \max(|\alpha_i-\beta|, r_i)$.

Finally, for any $f \in \mathcal{O}_x,$ there exists $i_f \in \mathbb{N},$ such that $f \in \mathcal{O}(\mathbb{D}_{k}(P_{i_f}, s_{i_f})).$ As seen in Lemma \ref{113}, $\mathcal{O}(\mathbb{D}_{k}(P_{i_f}, s_{i_f}))$ is isomorphic to $\mathcal{O}(\mathbb{D}_{k}(0, s_{i_f}))[S]/(P_{i_f}(S)-T),$ where $\mathcal{O}(\mathbb{D}_{k}(0, s_{i_f}))=\{\sum_{n \in \mathbb{N}} b_nT^n : b_n \in k, \lim_{n \rightarrow +\infty} |b_n|s_{i_f}^n=0\}.$

Remark that for any $i \geqslant i_f,$ $f \in \mathcal{O}(\mathbb{D}_{k}(P_i, s_i)).$ The norm induced by $x$ on $\mathcal{O}_x$ is given as follows: $|f|_x=\inf_{i \geqslant i_f}|f|_{\eta_{P_i, s_i}}.$
\end{ex}

\begin{ex} Let $S,T$ denote the coordinates of $\mathbb{A}_k^{2, \mathrm{an}},$ and $\varphi: \mathbb{A}_k^{2, \mathrm{an}} \rightarrow \mathbb{A}_k^{1, \mathrm{an}}$ the projection to $\mathbb{A}_k^{1, \mathrm{an}}$ with coordinate $T.$ Let $s, t \in \mathbb{R}_{>0}$ be such that $t \not \in \sqrt{|k^\times|}$ and $s \not \in \sqrt{|\mathcal{H}(\eta_{T,t})^\times|}.$  Let $x \in \mathbb{A}_k^{2, \mathrm{an}}$ denote a point such that $|T|_x=t, |S|_x=s.$ Then $x \in \varphi^{-1}(\eta_{T,t}),$ and considering the condition on $s,$ $x$ is a type 3 point on the fiber of $\eta_{T,t}$. In particular, $x$ is the only point of $\mathbb{A}_{k}^{2, \mathrm{an}}$ that satisfies $|T|_x=t, |S|_x=s.$ 

By Lemma \ref{103} and Example 2, a basis of neighborhoods of $x$ is given by $\{y \in \mathbb{A}_k^{1, \mathrm{an}}: t_1 \leqslant |T|_y \leqslant t_2, s_1 \leqslant |S|_y \leqslant s_2\}$, where $0<t_1<t<t_2,$ $0<s_1<s<s_2.$ Consequently,
\begin{align*}
\mathcal{O}_x=\big\lbrace\sum_{m,n \in \mathbb{Z}}a_{m,n}T^mS^n : \ & a_{m,n} \in k, \exists t_1, t_2, s_1, s_2 \in \mathbb{R}_{>0}, \ \text{s.t.} \ t_1<t<t_2, s_1<s<s_2, \\ &
\lim_{m+n \rightarrow +\infty} |a_{m,n}|t_2^m s_2^n=0, \lim_{m+n \rightarrow -\infty} |a_{m,n}|t_1^ms_1^n=0\big\rbrace.
\end{align*} 
The norm on $\mathcal{O}_x$ is given by: $|\sum_{m,n \in \mathbb{Z}} a_{m,n} T^m S^n|_x=\max_{m,n \in \mathbb{Z}} |a_{m,n}|t^ms^n.$ 

By iterating the above, we can calculate the local ring of any point $x \in \mathbb{A}_k^{l, \mathrm{an}}, l \in \mathbb{N},$ satisfying similar properties.
\end{ex}

\section*{Appendices}   

The author believes that most of the  results in the Appendices are known to the mathematical community working with Berkovich spaces, but she did not manage to find references for them. 

\subsection*{Appendix I: The sheaf of meromorphic functions}\label{1.7} 
As in the complex setting, a sheaf of meromorphic functions can be defined satisfying similar properties. Moreover, its definition resembles heavily that of the sheaf of meromorphic functions for schemes (including the subtleties of the latter, see \cite{meromis}). See \mbox{\cite[7.1.1]{liulibri}} for a treatment of meromorphic functions in the algebraic setting.

Let $k$ denote a complete ultrametric field.

\begin{defn} \label{0134}
Let $X$ be a good $k$-analytic space. Let $\mathcal{S}_X$ be the presheaf of functions on $X,$ which associates to any analytic domain $U$ the set of analytic functions on $U$ whose restriction to any affinoid domain in it is not a zero-divisor.  
Let $\mathscr{M}_{-}$ be the presheaf on $X$ that associates to any analytic domain $U$ the ring $\mathcal{S}_X(U)^{-1}\mathcal{O}_X(U).$ The sheafification $\mathscr{M}_X$ of the presheaf $\mathscr{M}_{-}$ is said to be the \textit{sheaf of meromorphic functions on~$X.$}
\end{defn}

It is immediate from the definition that for any analytic domain $U$ of $X$, $\mathcal{S}_X(U)$ contains no zero-divisors of $\mathcal{O}_X(U).$

\begin{prop}
Let $X$ be a good $k$-analytic space. Let $U$ be an analytic domain of $X.$ 

\noindent (1) $\mathcal{S}_X(U)=\{f \in \mathcal{O}_X(U): f \ \text{is a non-zero-divisor in} \ \mathcal{O}_{U,x} \ \text{for all} \ x\in U\}.$

\noindent (2) $\mathcal{S}_X(U)=\{f \in \mathcal{O}_X(U): f \ \text{is a non-zero-divisor in} \ \mathcal{O}_U(G) \ \text{for any open subset} \ G \ \text{of} \ U\}.$
\end{prop}

\begin{proof}
(1) By a direct application of the definition, the elements of $\mathcal{S}_X(U)$ are non-zero-divisors on $\mathcal{O}_{U,x}$ for all $x \in U.$

Let $f \in \mathcal{O}_X(U)$ be such that $f$ is a non-zero-divisor in $\mathcal{O}_{U,x}$ for all $x\in U.$ This means that $\mathcal{O}_{U,x} \rightarrow \mathcal{O}_{U,x}, a \mapsto f \cdot a$, is an injective map for $x \in U.$ 

Let $V$ be any affinoid domain in $U.$ By \cite[4.1.11]{famduc}, for any $x \in V,$ the morphism ${\mathcal{O}_{U,x} \rightarrow \mathcal{O}_{V,x}}$ is flat. Consequently, the map $\mathcal{O}_{V,x} \rightarrow \mathcal{O}_{V,x}, b \mapsto f\cdot b,$ is injective, or equivalently, $f$ is a non-zero-divisor in $\mathcal{O}_{V,x}.$ Suppose there exists $c \in \mathcal{O}_U(V)$ such that $f \cdot c=0.$ Then $c=0$ in $\mathcal{O}_{V,x}$ for all $x \in V,$ implying $c=0$ in $\mathcal{O}_U(V).$ As a consequence, $f$ is a non-zero-divisor in $\mathcal{O}_U(V).$ We have shown that $f \in \mathcal{S}_X(U),$ concluding the proof of the first part of the statement. 

Finally, (2) is a direct consequence of (1). 
\end{proof}

\begin{lm} \label{meroaff}
Let $X$ be a good $k$-analytic space. Let $U$ be an affinoid domain in~$X.$ Then $\mathcal{S}_X(U)$ is the set of non-zero divisors of $\mathcal{O}_X(U).$ 
\end{lm}

\begin{proof}
By definition, the elements of $\mathcal{S}_X(U)$ are not zero-divisors in $\mathcal{O}_X(U).$

Let $f$ be an element of  $A_U:=\mathcal{O}_X(U)$ that is a non-zero-divisor, \textit{i.e.} such that the map $A_U \rightarrow A_U,$ $a \mapsto f \cdot a$, is injective. Let $V \subseteq U$ be any affinoid domain. Set $A_V:=\mathcal{O}_X(V).$ Then, by \cite[Proposition 2.2.4(ii)]{Ber90}, the restriction map $A_U \rightarrow A_V$ is flat. Consequently, the map $A_V \rightarrow A_V, b \mapsto f \cdot b,$ remains injective, meaning $f$ is not a zero divisor in $A_V.$ This implies that $f \in \mathcal{S}_X(U),$ proving the statement.   
\end{proof}

The proof of the following statement resembles the proof of its algebraic analogue. 

\begin{cor} \label{meroaff1}
Let $X$ be a good $k$-analytic space. 
Then for any $x \in X,$ $\mathcal{S}_{X,x}$ is the set of elements of $\mathcal{O}_{X,x}$ that are non-zero-divisors. 
\end{cor}

\begin{proof}
Let $x \in X.$ Clearly, the elements of $\mathcal{S}_{X,x}$ are not zero divisors in $\mathcal{O}_{X,x}.$

Let $f \in \mathcal{O}_{X,x}$ be a non-zero-divisor.
By restricting to an affinoid neighborhood of $x$ if necessary, we may assume, without loss of generality, that $X$ is an affinoid space and $f \in \mathcal{O}_X(X).$ Set $A=\mathcal{O}_X(X).$ Set $I=\{a \in A: f \cdot a=0\}.$ This is an ideal of $A$, and 
gives rise to the following short exact sequence $$0 \rightarrow I \rightarrow A \rightarrow A,$$ where $A \rightarrow A$ is given by $a \mapsto f \cdot a.$
Seeing as $f$ is a non-zero-divisor in $\mathcal{O}_{X,x}$, we obtain that $I\mathcal{O}_{X,x}=0.$ 

The ring $A$ is an affinoid algebra, and hence Noetherian (\textit{cf.} \cite[Proposition 2.1.3]{Ber90}). Consequently, $I$ is finitely generated. Let $a_1, a_2, \dots, a_n \in A$ be such that $I=(a_1, a_2, \dots, a_n).$ By the above, the germs $a_{i,x} \in \mathcal{O}_{X,x}$ of $a_i$ at $x$  are zero for all $i \in \{1,2,\dots, n\}.$ Consequently, there exists an affinoid neighborhood $V$ of $x$ in $X$ such that $a_{i|V}=0$ for all $i,$ implying $I\mathcal{O}_X(V)=0.$

Set $A_V:=\mathcal{O}_X(V).$ By \cite[Proposition 2.2.4(ii)]{Ber90}, the restriction morphism $A \rightarrow A_V$ is flat, so the short exact sequence above induces the following short exact sequence:
$$0 \rightarrow I \otimes_A A_V \rightarrow A_V \rightarrow A_V,$$ 
where $A_V \rightarrow A_V$ is given by $b \mapsto f_{|V} \cdot b.$ Seeing as $A_V$ is a flat $A$-module, $I \otimes_A A_V$ is isomorphic to $IA_V=0.$ Consequently, multiplication by $f_{|V}$ is injective in $A_V,$ or equivalently, $f_{|V}$ is a non-zero-divisor in $A_V.$ By Lemma \ref{meroaff}, this implies that ${f_{|V} \in \mathcal{S}_X(V)},$ and finally that $f \in \mathcal{S}_{X,x}.$
\end{proof}

By Corollary \ref{meroaff1}, if $X$ is a good $k$-analytic space, then for any $x \in X,$ $\mathscr{M}_{X,x}$ is the total ring of fractions of $\mathcal{O}_{X,x}.$ In particular, if $\mathcal{O}_{X,x}$ is a domain, then $\mathscr{M}_{X,x}=\text{Frac} \ \mathcal{O}_{X,x}.$ When there is no risk of confusion, we will simply denote $\mathcal{O}$, resp.~$\mathscr{M}$, for the sheaf of analytic, resp. meromorphic functions on $X.$ We recall (see \cite[Lemma 1.2]{une} for a proof):
\begin{lm} \label{1.2}
Let $X$ be an \textit{integral} $k$-affinoid space. Then $\mathscr{M}(X)=\mathrm{Frac} \ \mathcal{O}(X).$  
\end{lm}

We now show that the meromorphic functions of the analytification of a proper scheme defined over an affinoid algebra are algebraic. It is a non-trivial result for which GAGA-type theorems (\textit{cf.} \cite{kopf}, \cite[Annexe A]{poi1}) are crucial.
The arguments to prove the following result were given in a Mathoverflow thread (see \cite{overflow}). In the case of curves, this is shown in {\cite[Prop. 3.6.2]{Ber90}}.

Let us first mention some brief reminders on the notion of depth. Let $R$ be a ring, $I$ an ideal of $R,$ and $M$ a finitely generated $R$-module. An $M$-\textit{regular sequence of length $d$ over} ~$I$ is a sequence $r_1, r_2, \dots, r_d \in I$ such that $r_i$ is not a zero divisor in $M/(r_1, \dots, r_{i-1})M$ for $i=1,2,\dots, d.$ The \textit{depth} of $M$ over $I$, denoted $\text{depth}_R(I,M)$ in \cite[Section 1]{10bour}, is 
\begin{itemize}
\item $\infty$ if $IM=M,$
\item the supremum of the length of $M$-regular sequences over $I,$ otherwise.
\end{itemize}
In what follows, when $M=R,$ we will denote $\text{depth}_R(I,R)$ by $\text{depth}_{I} R.$ Remark that $\text{depth}_{I} R>0$ if and only if $I$ contains a non-zero divisor of $R.$

\begin{thm} \label{231}
Let $k$ be a complete ultrametric field. Let $A$ be a $k$-affinoid algebra. Let $X$ be a proper scheme over $\text{Spec} \ A.$ Let $X^{\mathrm{an}}/\mathcal{M}(A)$ denote the Berkovich analytification of~$X.$ Then $\mathscr{M}_{X^{\mathrm{an}}}(X^{an})=\mathscr{M}_X(X)$, where $\mathscr{M}_{X^{\mathrm{an}}}$ (resp. $\mathscr{M}_X$) denotes the sheaf of meromorphic functions on $X^{\mathrm{an}}$ (resp. $X$).
\end{thm}

When there is no risk of ambiguity and the ambient space is clear from context, we will simply write $\mathscr{M}$ for the sheaf of meromorphic functions.
 
\begin{proof}
As in Definition \ref{0134}, let $\mathcal{S}_{X^{\mathrm{an}}}$ denote the presheaf of analytic functions on~$X^{\mathrm{an}}$, which associates to any analytic domain $U$ the set of analytic functions on $U$ whose restriction to any affinoid domain in it is not a zero divisor. By Corollary \ref{meroaff1}, for any $x \in X^{\mathrm{an}},$ $\mathcal{S}_{X^{\mathrm{an}},x}$ is the set of non-zero-divisors of $\mathcal{O}_{X^{\mathrm{an}},x}.$

Let $\mathcal{I}$ be a coherent ideal sheaf on 
$X^{\mathrm{an}}$ that locally on $X^{\mathrm{an}}$ contains a section of $\mathcal{S}
_{X^{\mathrm{an}}}.$ This means that for any $x 
\in X^{\mathrm{an}},$ $\mathcal{S}
_{X^{\mathrm{an}},x} \cap \mathcal{I}_{x} \neq 
\emptyset.$ Let $s \in \mathcal{S}
_{X^{\mathrm{an}},x} \cap \mathcal{I}_{x}.$ Then~$s$ is a non-zero divisor in $\mathcal{O}
_{X^{\mathrm{an}},x},$ which implies $
\mathrm{depth}_{\mathcal{I}_{x}} \mathcal{O}
_{X^{\mathrm{an}},x}>0$. Suppose, on the other 
hand, that $\mathcal{I}$ is a coherent ideal sheaf 
on $X^{\mathrm{an}}$ such that $\mathrm{depth}
_{\mathcal{I}_{x}} \mathcal{O}_{X^{\mathrm{an}},x}
>0$ for all $x \in X^{\mathrm{an}}.$ Then there 
exists at least one element $s \in \mathcal{I}_{x}
$ which is a non-zero-divisor in $\mathcal{O}
_{X^{\mathrm{an}},x},$ implying $s \in \mathcal{S}
_{X^{\mathrm{an}},x}.$ To summarize, a coherent 
ideal sheaf $\mathcal{I}$ on $X^{\mathrm{an}}$ 
contains locally on $X^{\mathrm{an}}$ a section of~$\mathcal{S}_{X^{\mathrm{an}}}$ if and only if $
\mathrm{depth}_{\mathcal{I}_{x}}(\mathcal{O}
_{X^{\mathrm{an}},x})>0$ for all $x \in 
X^{\mathrm{an}}.$

Let us show that for any coherent ideal sheaf $\mathcal{I}$ on $X^{\mathrm{an}}$ containing locally on $X^{\mathrm{an}}$ a section of $\mathcal{S}_{X^{\mathrm{an}}},$ there is an embedding $\text{Hom}_{X^{\mathrm{an}}}(\mathcal{I}, \mathcal{O}_{X^{\mathrm{an}}}) \subseteq \mathscr{M}_{X^{\mathrm{an}}}(X^{\mathrm{an}}),$ where $\text{Hom}_{X^{\mathrm{an}}}(\mathcal{I}, \mathcal{O}_{X^{\mathrm{an}}})$ denotes the global sections on $X^{\mathrm{an}}$ of the hom sheaf $\mathscr{H}om(\mathcal{I}, \mathcal{O}_{X^{\mathrm{an}}}).$ Let $\varphi \in \text{Hom}_{X^{\mathrm{an}}}(\mathcal{I}, \mathcal{O}_{X^{\mathrm{an}}}).$ For any $x \in X^{\mathrm{an}},$ $\varphi$ induces a morphism $\varphi_x: \mathcal{I}_{x} \rightarrow \mathcal{O}_{X^{\mathrm{an}},x}.$ Let $s_x \in \mathcal{S}_{X^{\mathrm{an}},x} \cap \mathcal{I}_{x},$ and set ${a_x=\varphi_x(s_x)}.$  There exists a neighborhood $U_x$ of $x,$ such that $s_x \in \mathcal{I}(U_x) \cap \mathcal{S}_{X^{\mathrm{an}}}(U_x), {a_x \in \mathcal{O}_{X^{\mathrm{an}}}(U_x)}$, and $\varphi(U_x)(s_x)=a_x.$ Set $f_x=\frac{a_x}{s_x} \in \mathcal{S}_{X^{\mathrm{an}}}(U_x)^{-1} \mathcal{O}_{X^{\mathrm{an}}}(U_x) \subseteq \mathscr{M}_{X^{\mathrm{an}}}(U_x)$ (the presheaf $\mathcal{S}_{X^{\mathrm{an}}}^{-1} \mathcal{O}_{X^{\mathrm{an}}}$ is separated, so $\mathcal{S}_{X^{\mathrm{an}}}^{-1}\mathcal{O}_{X^{\mathrm{an}}} \subseteq \mathscr{M}_{X^{\mathrm{an}}}$). 

Let $U_y, U_z$ be any non-disjoint elements of the cover $(U_x)_{x \in X^{\mathrm{an}}}$ of $X^{\mathrm{an}}.$ Then considering $\varphi$ is a morphism of sheaves of $\mathcal{O}_{X^{\mathrm{an}}}$-modules, $\varphi(U_y \cap U_z)(s_y \cdot s_z)=s_y \cdot a_z=a_y \cdot s_z$ in $\mathcal{O}_{X^{\mathrm{an}}}(U_y \cap U_z)$. Consequently, $f_{y|U_y \cap U_z}=f_{z|U_y \cap U_z}$ in $\mathscr{M}_{X^{\mathrm{an}}}(U_y \cap U_z),$ implying there exists $f \in \mathscr{M}_{X^{\mathrm{an}}}(X^{\mathrm{an}})$ such that $f_{|U_x}=f_x$ in $\mathscr{M}_{X^{\mathrm{an}}}(U_x)$ for all $x \in X^{\mathrm{an}}.$ 

We associate to $\varphi$ the meromorphic function $f.$ Remark that if $f=0,$ then $a_x=0$ for all ~$x.$ This implies that for any $\alpha \in \mathcal{I}_x,$ $\varphi_x(s_x \cdot\alpha)=s_x \cdot \varphi_x(\alpha)=a_x \cdot \varphi_x(\alpha)=0,$ which, taking into account $s_x \in \mathcal{S}_{X^{\mathrm{an}},x}$, means that $\varphi_x(\alpha)=0.$ Consequently, $\varphi_x=0$ for all $x \in X^{\mathrm{an}},$ so $\varphi=0.$ Thus, the map $\psi_{\mathcal{I}}:\text{Hom}_{X^{\mathrm{an}}}(\mathcal{I}, \mathcal{O}_{X^{\mathrm{an}}}) \rightarrow \mathscr{M}_{X^{\mathrm{an}}}(X^{\mathrm{an}})$ we have constructed is an embedding.  

Remark that the set of coherent ideal sheaves on $X^{\mathrm{an}}$ containing locally on $X^{\mathrm{an}}$ a section of $\mathcal{S}_{X^{\mathrm{an}}}$ forms a directed set with respect to reverse inclusion (\textit{i.e.} if $\mathcal{I}, \mathcal{J}$ satisfy these properties, then so does $\mathcal{I} \cdot \mathcal{J} \subseteq \mathcal{I}, \mathcal{J}$). Thus, by the paragraph above, there is an embedding $\varinjlim_{\mathcal{I}} \text{Hom}_{X^{\mathrm{an}}}(\mathcal{I}, \mathcal{O}_{X^{\mathrm{an}}}) \hookrightarrow \mathscr{M}_{X^{\mathrm{an}}}(X^{\mathrm{an}}),$ where the direct limit is taken with respect to the same kind of coherent ideal sheaves $\mathcal{I}$ as above. Let us show that this embedding is an isomorphism. 

For any $f \in \mathscr{M}_{X^{\mathrm{an}}}(X^{\mathrm{an}}),$ define the ideal sheaf $D_f$ as follows: for any analytic domain $U$ of $X^{\mathrm{an}},$ set $D_f(U)=\{s \in \mathcal{O}(U): s \cdot f \in \mathcal{O}_{X^{\mathrm{an}}}(U) \subseteq \mathscr{M}_{X^{\mathrm{an}}}(U)\}.$ 
This is a coherent ideal sheaf on $X^{\mathrm{an}}$. Since $\mathscr{M}_{X^{\mathrm{an}},x}=\mathcal{S}_{X^{\mathrm{an}},x}^{-1} \mathcal{O}_{X^{\mathrm{an}},x}$ for any $x \in X^{\mathrm{an}},$ there exist $s_x \in \mathcal{S}_{X^{\mathrm{an}},x}$ and $a_x \in \mathcal{O}_{X^{\mathrm{an}},x}$ such that $f_x=\frac{a_x}{s_x}$ in $\mathscr{M}_{X^{\mathrm{an}},x}.$ Considering $D_{f,x}=\{s \in \mathcal{O}_{X^{\mathrm{an}},x}: s \cdot f_x \in \mathcal{O}_{X^{\mathrm{an}},x}\},$ we obtain that $s_x \in D_{f,x},$ so $D_f$ contains locally on $X^{\mathrm{an}}$ a section of $\mathcal{S}_{X^{\mathrm{an}}}.$ To $f \in \mathscr{M}_{X^{\mathrm{an}}}(X^{\mathrm{an}})$ we associate the morphism $\varphi_f: D_f \rightarrow \mathcal{O}_{X^{\mathrm{an}}}$ which corresponds to multiplication by $f$ (\textit{i.e.} for any open subset $U$ of $X^{\mathrm{an}}$, $D_f(U) \rightarrow \mathcal{O}_{X^{\mathrm{an}}}(U), s \mapsto f \cdot s$). Clearly, $\psi_{D_f}(\varphi_f)=f,$ implying the embedding $\varinjlim_{\mathcal{I}} \text{Hom}_{X^{\mathrm{an}}}(\mathcal{I}, \mathcal{O}_{X^{\mathrm{an}}}) \hookrightarrow \mathscr{M}_{X^{\mathrm{an}}}(X^{\mathrm{an}})$ is surjective, so an isomorphism. 

Let $\mathcal{S}_X$ denote the presheaf on $X$ through which $\mathscr{M}_X$ is defined (see \cite[Section 7.1.1]{liulibri}). Remark that since $A$ is Noetherian (\cite[Proposition 2.1.3]{Ber90}), the scheme $X$ is locally Noetherian. Under this assumption, for any $x \in X,$ $\mathcal{S}_{X,x}$ is the set of all non-zero-divisors of $\mathcal{O}_{X,x}$ (see \cite[7.1.1, Lemma 1.12(c)]{liulibri}). Taking this into account, all the reasoning above does not make use of the fact that $X^{\mathrm{an}}$ is an analytic space, and can be applied \textit{mutatis mutandis} to the scheme $X$ and its sheaf of meromorphic functions $\mathscr{M}_X$. Thus, $\mathscr{M}_X(X) \cong\varinjlim_{\mathcal{J}} \text{Hom}_X(\mathcal{J}, \mathcal{O}_X),$ where the direct limit is taken with respect to coherent ideal sheaves $\mathcal{J}$ on $X,$ for which $\text{depth}_{\mathcal{J}_{X,x}} \mathcal{O}_{X,x}>0$  for all $x \in X.$

Consequently, to show the statement, we need to show that $\varinjlim_{\mathcal{J}} \text{Hom}_X(\mathcal{J}, \mathcal{O}_X)=\varinjlim_{\mathcal{I}} \text{Hom}_{X^{\mathrm{an}}}(\mathcal{I}, \mathcal{O}_{X^{\mathrm{an}}}),$ where the direct limits are taken as above. 
\begin{sloppypar}
By \cite[Annexe A]{poi1} (which was proven in \cite{kopf} in the case of rigid geometry), there is an equivalence of categories between the coherent sheaves on $X$ and those on $X^{\mathrm{an}}$.
Let us show that this induces an equivalence of categories between the coherent ideal sheaves on $X$ and those on $X^{\mathrm{an}}.$ To see this, we only need to show that if $\mathcal{F}$ is a coherent sheaf on $X$ such that $\mathcal{F}^{\mathrm{an}}$ is an ideal sheaf on~$X^{\mathrm{an}},$ then $\mathcal{F}$ is an ideal sheaf on $X.$ By \cite[A.1.3]{poi1}, we have a sheaf isomorphism $\mathscr{H}om(\mathcal{F}, \mathcal{O})^{\mathrm{an}} \cong \mathscr{H}om(\mathcal{F}^{\mathrm{an}}, \mathcal{O}_{X^{\mathrm{an}}}),$ so $\mathscr{H}om(\mathcal{F}, \mathcal{O})^{\mathrm{an}}$ has a non-zero global section $\iota$ corresponding to the injection $\mathcal{F}^{\mathrm{an}} \subseteq \mathcal{O}_{X^{\mathrm{an}}}.$ By \cite[Th\'eor\`eme~A.1(i)]{poi1}, $\mathscr{H}om(\mathcal{F}, \mathcal{O})^{\mathrm{an}}(X^{\mathrm{an}}) \cong \mathscr{H}om(\mathcal{F}, \mathcal{O})(X).$ Let $\iota' \in \mathscr{H}om(\mathcal{F}, \mathcal{O})(X)$ denote the element corresponding to $\iota.$ Then the analytification of $\iota': \mathcal{F} \rightarrow \mathcal{O}_X$ is the morphism $\iota: \mathcal{F}^{\mathrm{an}} \hookrightarrow \mathcal{O}_{X^{\mathrm{an}}}.$ By flatness of $X^{\mathrm{an}} \rightarrow X,$ we obtain that $(\mathrm{ker} \ \iota')^{\mathrm{an}}=\mathrm{ker} \ \iota'^{\mathrm{an}}=\mathrm{ker} \ \iota,$ so $(\mathrm{ker} \ \iota')^{\mathrm{an}}=0,$ implying $\mathrm{ker} \ \iota'=0.$ Consequently, there exists an embedding $\mathcal{F} \hookrightarrow \mathcal{O}_X,$ implying $\mathcal{F}$ is an ideal sheaf on $X.$
\end{sloppypar}
If to a coherent ideal sheaf $\mathcal{J}$ on $X$ we associate the coherent ideal sheaf $\mathcal{J}^{\mathrm{an}}$ on $X^{\mathrm{an}},$ then as seen above: $\text{Hom}_{X}(\mathcal{J}, \mathcal{O}_X)\cong \text{Hom}_{X^{\mathrm{an}}}(\mathcal{J}^{\mathrm{an}}, \mathcal{O}_{X^{\mathrm{an}}})$. 
\begin{sloppypar}
Let us also show that a coherent ideal sheaf $\mathcal{J}$ on $X$ satisfies ${\text{depth}_{\mathcal{J}_{x}} \mathcal{O}_{X,x}>0}$ for all $x \in X$ if and only if $\text{depth}_{\mathcal{J}^{\mathrm{an}}_y} \mathcal{O}_{X^{\mathrm{an}}, y}>0$ for all $y \in X^{\mathrm{an}}.$ To see this, recall that by \cite[Proposition 2.6.2]{ber93}, the morphism $\phi: X^{\mathrm{an}} \rightarrow X$ is surjective and for any $y \in X^{\mathrm{an}},$ the induced morphism of local rings $\mathcal{O}_{X, x} \rightarrow \mathcal{O}_{X^{\mathrm{an}},y}$ is faithfully flat, where $x:=\phi(y).$ By \cite[1.3, Proposition 6]{10bour}, ${\text{depth}_{\mathcal{J}_x} \mathcal{O}_{X,x}=\text{depth}_{\mathcal{J}_x \mathcal{O}_{X^\mathrm{an},y}} \mathcal{O}_{X^{\mathrm{an}},y} \otimes_{\mathcal{O}_{X,x}} \mathcal{O}_{X,x}}.$ At the same time, seeing as the morphism $\mathcal{O}_{X, x} \rightarrow \mathcal{O}_{X^{\mathrm{an}},y}$ 
is flat, ${\mathcal{J}^{\mathrm{an}}_y=\mathcal{J}_x \otimes_{\mathcal{O}_{X,x}} \mathcal{O}_{X^\mathrm{an},y}=\mathcal{J}_x\mathcal{O}_{X^{\mathrm{an}},y}},$ so ${\text{depth}_{\mathcal{J}_x} \mathcal{O}_{X,x}=\text{depth}_{\mathcal{J}_y^{\mathrm{an}}} \mathcal{O}_{X^{\mathrm{an}}, y}}.$   
\end{sloppypar}
From the above, $\varinjlim_{\mathcal{J}} \text{Hom}_X(\mathcal{J}, \mathcal{O}_X)=\varinjlim_{\mathcal{I}} \text{Hom}_{X^{\mathrm{an}}}(\mathcal{I}, \mathcal{O}_{X^{\mathrm{an}}}),$ where the direct limits are taken with respect to coherent ideal sheaves $\mathcal{J}$ on $X$ (resp. $\mathcal{I}$ on $X^{\mathrm{an}}$), for which 
$\text{depth}_{\mathcal{J}_x} \mathcal{O}_{X, x}>0$ for all $x \in X$ (resp. $\text{depth}_{\mathcal{I}_x} \mathcal{O}_{X^{\mathrm{an}}, x}>0$ for all $x \in X^{\mathrm{an}}$). Finally, this implies that $\mathscr{M}_{X}(X)=\mathscr{M}_{X^{\mathrm{an}}}(X^{\mathrm{an}}).$
\end{proof}
 
As an immediate consequence of the theorem above, we obtain that for any integral $k$-affinoid space $Z,$ $\mathscr{M}(\mathbb{P}_Z^{1, \mathrm{an}})=\mathscr{M}(Z)(T).$

\subsection*{Appendix II: Some results on analytic curves}

\begin{lm}\label{62}
Let $C$ be a normal irreducible projective $k$-analytic curve. Let $U$ be a connected affinoid domain of $C$ such that its boundary contains only type 3 points. Then for any $S \subseteq \partial{U},$ $U \backslash S$ is connected.
\end{lm}

\begin{proof}
Suppose that $C$ is generically quasi-smooth. Since $\partial{U}$ contains only type 3 points, all of the points of $S$ are quasi-smooth in $C.$

Let $x,y \in \text{Int} \ U.$ Since $U$ is connected, there exists an arc $[x, y] \subseteq U$ connecting $x$ and~$y$. Let $z \in S.$ We aim to show that $z \not \in [x,y],$ implying $[x,y] \subseteq U \backslash S,$ and thus the connectedness of $U \backslash S.$ 
 
By \cite[Th\'eor\`eme 4.5.4]{Duc}, there exists an affinoid neighborhood $V$ of $z$ in $U$ such that it is a closed virtual annulus, and its Berkovich boundary is $\partial_B(V)=\{z,u\}$ for some $u \in U.$ We may assume that $x,y \not \in V.$ Since $V$ is an affinoid domain in $U,$ by \cite[Proposition ~1.5.5]{ber93}, the topological boundary $\partial_U{V}$ of $V$ in $U$ is a subset of $\partial_B(V)=\{z,u\}.$ Since $V$ is a neighborhood of $z,$ $\partial_U{V}=\{u\}.$

Suppose  $z \in [x,y].$ Then  we could decompose $[x,y]=[x,z] \cup [z,y].$ Since $x,y \not \in V$, and $z \in V,$ the sets $[x,z] \cap \partial_U{V},$ $[z,y] \cap \partial_U{V}$ are non-empty, thus implying $u$ is contained in both $[x,z]$ and $[z,y],$ which contradicts the injectivity of $[x,y].$ Consequently, $U \backslash S$ is connected.
 
Let us get back to the general case. Let $C^{\mathrm{alg}}$ denote the algebraization of $C$ (\textit{i.e.} the normal irreducible projective algebraic curve over $k$ whose analytification is $C$). Since it is normal, there exists a finite surjective morphism $C^{\mathrm{alg}} \rightarrow \mathbb{P}_k^{1}.$ This induces a finite field extension $k(T) \hookrightarrow k(C^{\mathrm{an}})=\mathscr{M}(C)$ of their function fields. Let $F$ denote the separable closure of $k(T)$ in $k(C).$ Then there exists an irreducible normal algebraic curve $X$ over $k$ such that $k(X)=F.$ Seeing as $k(T) \hookrightarrow k(C)$ is separable, the induced morphism $X \rightarrow \mathbb{P}_k^1$ is generically \'etale, so $X$ is generically smooth. 
On the other hand, the finite field extension $k(C)/F$ is purely inseparable, implying the corresponding finite morphism $C^{\mathrm{alg}} \rightarrow X$ is a homeomorphism. 

Finally, the analytification $X^{\mathrm{an}}$ is a normal irreducible projective $k$-analytic curve that is generically quasi-smooth, and there is a finite morphism $f: C \rightarrow X^{\mathrm{an}}$ that is a homeomorphism. By \cite[Proposition 4.2.14]{Duc}, $f(U)$ is a connected proper closed analytic domain of $X^{\mathrm{an}}$. By \cite[Th\'eor\`eme 6.1.3]{Duc}, $f(U)$ is an affinoid domain of $X^{\mathrm{an}}.$ Clearly, $\partial{f(U)}=f(\partial{U}).$ Let $S \subseteq \partial{U},$ and set $S'=f(S).$ As shown above, $f(U) \backslash S'$ is connected. Consequently, $U \backslash S$ is connected.
\end{proof}

\begin{cor} \label{inte}
Let $C$ be a normal irreducible $k$-analytic curve. Let $U$ be an affinoid domain in $C$ containing only type 3 points in its boundary. If $\text{Int}(U) \neq \emptyset$, then $(\text{Int}\ U)^c$ is an affinoid domain in $C$ containing only type 3 points in its boundary. 
\end{cor}

\begin{proof}
Seeing as $U$ is an affinoid domain, it has a finite number of connected components, and by \cite[Corollary 2.2.7(i)]{Ber90}, they  are all affinoid domains in $C.$ Furthermore, each of the connected components of $U$ contains only type 3 points in its boundary. Consequently, by Lemma \ref{62}, $\text{Int}(U)$ has only finitely many connected components. Thus, by \cite[Proposition ~4.2.14]{Duc}, $(\text{Int} \ U)^c$ is a closed proper analytic domain of $C.$ By \cite[Th\'eor\`eme~6.1.3]{Duc},  it is an affinoid domain in $C.$	
\end{proof}

\begin{prop} \label{bababa}
Let $C$ be a compact $k$-analytic curve. For any $x, y \in C,$ there exist only finitely many arcs in $C$ connecting $x$ and $y$.
\end{prop} 
 
\begin{proof}
By \cite[Th\'eor\`eme 3.5.1]{Duc}, $C$ is a real graph. By \cite[1.3.13]{Duc}, for any $z \in C,$ there exists an open neighborhood $U_z$ of $z$ such that: (1) $U_z$ is uniquely arcwise-connected; (2) ~the closure $\overline{U_z}$ of $U_z$ in $C$ is uniquely arcwise-connected; (3) the boundary $\partial{U_z}$ is finite, implying in particular $\partial{U_z}=\partial{\overline{U_z}}.$ Seeing as $C$ is compact, the finite open cover $\{U_z\}_{z \in C}$ admits a finite subcover $\mathcal{U}:=\{U_1, U_2, \dots, U_n\}.$ Set $S:=\bigcup_{i=1}^n \partial{U_i}.$ This is a finite subset of $C.$ 

Let $x,y$ be any two points of $C.$ Let $\gamma: [0,1] \rightarrow C$ be any arc in $C$ connecting $x$ and ~$y.$ Set $S_{\gamma}:=S \cap \gamma([0,1]) \backslash \{x,y\}.$ It is a finite (possibly empty) subset of $C$. For any $\alpha \in S_{\gamma},$ there exists a unique $a \in [0,1]$ such that $\gamma(a)=\alpha.$ This gives rise to an ordering of the points of $S_{\gamma}.$ Set $S_{\gamma}=\{\alpha_1, \alpha_2, \dots, \alpha_m\}$ such that the order of the points is the following: ${\alpha_1< \alpha_2 < \cdots < \alpha_{m}}$ (meaning $\gamma^{-1}(\alpha_1)< \gamma^{-1}(\alpha_2) < \cdots < \gamma^{-1}(\alpha_m)$). To the arc $\gamma$ we associate the finite sequence $\overline{\gamma}:=(\alpha_1, \alpha_2, \dots, \alpha_m)$ of points of $S_{\gamma}.$ Set $\alpha_0=x$ and $\alpha_{m+1}=y.$

For any $i \in \{0,1,\dots, m+1\},$ set $\gamma_i:=\gamma([\gamma^{-1}(\alpha_i), \gamma^{-1}(\alpha_{i+1})])$. This is an arc in $C$ connecting $\alpha_i$ and $\alpha_{i+1}.$ By construction, for any $i,$  $\gamma_i \cap S \subseteq \{\alpha_i, \alpha_{i+1}\}.$  Remark that $\gamma([0,1])=\bigcup_{i=0}^{m+1} \gamma_i.$

Let us show that for any $i \in \{0,1,\dots, m\}$, there exists a unique arc $[\alpha_i, \alpha_{i+1}]_0$ in ~$C$ connecting $\alpha_i$ and $\alpha_{i+1}$ such that $[\alpha_i, \alpha_{i+1}]_0 \cap S \subseteq \{\alpha_i, \alpha_{i+1}\}.$ Let $[\alpha_i, \alpha_{i+1}]$ be any such arc (the existence is guaranteed by the paragraphs above). Let $j \in \{1,2,\dots, n\}$ be such that $[\alpha_i, \alpha_{i+1}] \cap U_j \neq \emptyset.$ Let $z \in [\alpha_i, \alpha_{i+1}] \cap U_j;$ since $[\alpha_i, \alpha_{i+1}] \cap U_j$ is open in $[\alpha_i, \alpha_{i+1}],$ we may choose $z$ such that $z \not \in \{\alpha_i, \alpha_{i+1}\}$. Let us denote by $[\alpha_i,z]$, resp. $[z, \alpha_{i+1}]$ the arc in $C$ induced by $[\alpha_i, \alpha_{i+1}]$ connecting $\alpha_i$ and $z$, resp. $z$ and $\alpha_{i+1}.$ Clearly, $[\alpha_i, \alpha_{i+1}]=[\alpha_i, z] \cup [z, \alpha_{i+1}].$ 

Suppose there exists $u \in [\alpha_i, \alpha_{i+1}] \backslash \overline{U_j}.$ Again, as $[\alpha_i, \alpha_{i+1}] \backslash \overline{U_j}$ is open in $[\alpha_i, \alpha_{i+1}],$ we may assume that $u \not \in \{\alpha_i, \alpha_{i+1}\}.$  Without loss of generality, let us suppose that $u \in [\alpha_i, z]$. Let $[\alpha_i, u]$, resp. $[u, z]$, be the induced arcs connecting $\alpha_i$ and $u$, resp. $u$ and $z.$ Seeing as $z \in U_j$ and $u \not \in U_j,$ $[z,u] \cap \partial{U_j} \neq \emptyset.$ At the same time, $\emptyset \neq [z, u] \cap \partial{U_j} \subseteq [\alpha_i, \alpha_{i+1}] \cap \partial{U_j} \subseteq [\alpha_i, \alpha_{i+1}] \cap S \subseteq \{\alpha_i, \alpha_{i+1}\},$ which contradicts the injectivity of $[\alpha_i, \alpha_{i+1}].$

Consequently, $[\alpha_i, \alpha_{i+1}] \subseteq \overline{U_j}.$ Seeing as $\overline{U_j}$ is uniquely arcwise-connected, we obtain that the arc  $[\alpha_i, \alpha_{i+1}]$ in $C$ connecting $\alpha_i$ and $\alpha_{i+1},$ and satisfying the property $[\alpha_i, \alpha_{i+1}] \cap S \subseteq \{\alpha_i, \alpha_{i+1}\},$ is unique. Thus, $\gamma_i=[\alpha_i, \alpha_{i+1}],$ and the arc $\gamma$ is uniquely determined by its associated ordered sequence $\overline{\gamma}.$

Seeing as $S$ is finite, the set of all finite sequences $(\beta_l)_l$ over $S$ such that $\beta_{l'} \neq \beta_{l''}$ whenever $l' \neq l'',$ is also finite. Consequently, the set of arcs in $C$ connecting $x$ and $y$ is finite.
\end{proof}

\subsection*{Appendix III: Some results on the analytic projective line}
We show here some auxiliary results on the analytic projective line. We recall that there is a classification of points of this analytic curve (see \textit{e.g.} \cite[1.4.4]{Ber90} or \cite[1.1.2.3]{poilibri}). See also \cite[Definition ~2.2, Proposition 2.3]{une} for an exposition on the nature of points of $\mathbb{P}^{1, \mathrm{an}}.$ 

Recall that for a complete ultrametric field $K$, $\mathbb{P}_K^{1, \mathrm{an}}$ is uniquely arcwise-connected. For any $x,y \in \mathbb{P}_K^{1,\mathrm{an}},$ we denote by $[x,y]$ the unique arc in $\mathbb{P}_K^{1, \mathrm{an}}$ connecting $x$ and $y.$ 

Finally, for $a \in K$ and $r \in \mathbb{R}_{>0},$  recall the notation $\eta_{a,r}$ for a point of $\mathbb{P}_K^{1, \mathrm{an}}$ in \mbox{\cite[1.1.2.3]{poilibri}}.

\begin{prop}\label{projectivedom} Let $K$ be a complete ultrametric field. Let $U$ be a connected affinoid domain of $\mathbb{P}_K^{1, \mathrm{an}}$ with only type 3 points in its boundary. Suppose $U$ is not a point. Let us fix a copy of $\mathbb{A}_K^{1, \mathrm{an}}$ and a coordinate $T$ on it.  Let $\partial{U}=\{\eta_{R_i, r_i}:i=1,2,\dots, n\},$ where ${R_i \in K[T]}$ are irreducible polynomials and $r_i \in \mathbb{R}_{>0} \backslash \sqrt{|K^\times|}.$
Then ${U=\bigcap_i \{x : |R_i|_x \bowtie_i r_i\}},$ where $\bowtie_i \in \{\leqslant, \geqslant\}, i=1,2,\dots, n.$  
\end{prop}

\begin{proof}
We need the following two auxiliary results:
\begin{lm} \label{15}
For any $i \in \{1, 2, \dots, n\},$ either $U \subseteq \{x: |R_i|_x \leqslant r_i\}$ or ${U \subseteq \{x: |R_i|_x \geqslant r_i \}}.$ 
\end{lm}
\begin{proof}
To see this, assume that the open subsets $V_1:=U  \cap \{x: |R_i|_x < r_i\}$ and ${V_2:=U \cap \{x: |R_i|_x >r_i\}}$ of $U$ are non-empty. As intersections of two connected subsets of $\mathbb{P}_K^{1, \mathrm{an}}$, both $V_1$ and $V_2$ are connected. Assume $V_j \cap \text{Int}(U) =\emptyset, j=1,2.$ This implies $V_j \subseteq \partial{U},$ and since $V_j$ is connected, it is a single type 3 point $\{\eta_j\}.$ But then, this would be an isolated point of $U,$ which is  in contradiction with the connectedness of $U.$ Consequently, there exist $x_j \in V_j \cap \text{Int}(U), j=1,2.$ By Lemma ~\ref{62}, $\text{Int}(U)$ is a connected set, so there exists a unique arc $[x_1, x_2]$ connecting $x_1, x_2$ that is entirely contained in $\text{Int}(U).$ Since $|R_i|_{x_1}<r_i$, $|R_i|_{x_2}>r_i,$ there exists $x_0 \in [x_1, x_2]$ such that $|R_i|_{x_0}=r_i.$ Since there is a unique point satisfying this condition (\cite[Proposition 2.3(2)]{une}), and it is $\eta_{R_i, r_i},$ we obtain that $\eta_{R_i, r_i} \in [x_1, x_2] \subseteq \text{Int}(U),$ which is in contradiction with the fact that $\eta_{R_i, r_i} \in \partial{U}.$  Thus, there exists $j \in \{1,2\}$ such that $V_j =\emptyset,$ implying the statement.
\end{proof}

\begin{lm} \label{pfff}
For $n \in \mathbb{N},$ let $W_i:=\{x \in \mathbb{P}_K^{1, \mathrm{an}}: |P_i| \bowtie_i r_i\}$, where $P_i \in K[T]$ is irreducible, $r_i \in \mathbb{R}_{>0} \backslash \sqrt{|k^\times|},$ $\bowtie_i \in \{\leqslant, \geqslant\},$ $i \in \{1,2,\dots, n\}.$ Suppose for all $i \neq j,$ $W_i \not \subseteq \mathrm{Int}(W_j).$ Then, for $V:=\bigcap_{i=1}^n W_i,$ $\partial{V}=\bigcup_{i=1}^n \partial{W_i}.$
\end{lm}

\begin{proof}
Since $\text{Int}(V)=\bigcap_{j=1}^n \text{Int}(W_j),$ we obtain that $\partial{V}=\left(\bigcap_{j=1}^n W_j\right)\backslash \left(\bigcap_{i=1}^n \text{Int}(W_i)\right)= \bigcup_{i=1}^n \bigcap_{j=1}^n (W_i \backslash \text{Int}(W_j)).$ Suppose there exist $i, j \in \{1,2,\dots, n\}$ such that $W_i \backslash \text{Int}(W_j)=\emptyset.$ Then $W_i \subseteq \text{Int}(W_j),$ contradicting the hypothesis of the statement. 

Hence, for any $i,j$, $W_i \backslash \text{Int}(W_j) \neq \emptyset.$ In particular, this means that $W_i \cap \text{Int}(W_j)$ is a strict open subset of $W_i,$ so it is contained in $\text{Int}(W_i).$ Consequently, $\{\eta_{P_i, r_i}\}=W_i \backslash \text{Int}(W_i) \subseteq W_i \backslash (W_i \cap \text{Int}(W_j)) \subseteq W_i \backslash \text{Int}(W_j).$ This implies that for any $i$, ${\bigcap_{j=1}^n (W_i \backslash \text{Int}(W_j))=\{\eta_{P_i, r_i}\}}.$ 

Finally, $\partial{V}=\{\eta_{P_i, r_i}: i=1,2,\dots, n\},$ proving the statement.  
\end{proof}

If $ U \subseteq \{x : |R_i|_x \leqslant r_i\}$ (resp. $U \subseteq \{x : |R_i|_x \geqslant r_i\}$), set $U_i=\{x : |R_i|_x \leqslant r_i\} $ (resp. $U_i=\{x : |R_i|_x \geqslant r_i\} $). Remark that for all $i,$ $U_i$ is connected and contains $U.$ Set $V=\bigcap_{i=1}^n U_i.$ Let us show that $\partial{V}=\partial{U}.$  
Assume there exist $i,j$ such that $U_i \subseteq \text{Int}(U_j).$ Then $\eta_{R_j, r_j} \not \in U_i$, so $\eta_{R_j, r_j} \not \in U$, contradiction. Thus, Lemma \ref{pfff} is applicable, and so $\partial{V}=\{\eta_{R_i, r_i}\}_{i=1}^n=\partial{U}.$ 

Remark that $V$ is a connected affinoid domain (as an intersection of connected affinoid domains) of $\mathbb{P}_K^{1, \mathrm{an}}$. Also, $U \subseteq V$ and $\partial{U}=\partial{V}.$ Let us show that $U=V.$ Suppose there exists some $x \in V \backslash U.$ Then $x \in \text{Int}(V).$ Let $y \in \text{Int}(U) \subseteq \text{Int}(V).$ The unique arc $[x,y]$  in $\mathbb{P}_K^{1, \mathrm{an}}$ connecting $x$ and $y$ is contained  in $\text{Int}(V)$ (by connectedness of the latter, see Lemma \ref{62}). At the same time, since $x \not \in U$  and $y \in U,$ the arc $[x,y]$ intersects $\partial{U}=\partial{V},$ contradiction. 
Thus, $U=V=\bigcap_{i=1}^n U_i,$ concluding the proof of Proposition \ref{projectivedom}.
\end{proof}

In particular, the result above implies that every connected affinoid domain of $\mathbb{P}_K^{1, \mathrm{an}}$ with only type 3 points in its boundary is a rational domain.

\begin{lm} \label{16aaa}
Let $K$ be a complete ultrametric field. Let $U, V$ be connected affinoid domains of $\mathbb{P}_K^{1, \mathrm{an}}$ containing only type 3 points in their boundaries, such that 
$U \cap V=\partial{U} \cap \partial{V}$ is a single type 3 point $\{\eta_{R,r}\}$ (\textit{i.e.} $R$ is an irreducible polynomial over $K$ and $r \in \mathbb{R}_{>0} \backslash \sqrt{|K^\times|}$). 
\begin{enumerate}
\item
If $U \subseteq \{x \in \mathbb{P}_K^{1, \mathrm{an}}: |R|_x \leqslant r\}$ (resp. $U \subseteq \{x \in \mathbb{P}_K^{1, \mathrm{an}}: |R|_x \geqslant r\}$), then $V \subseteq \{x \in \mathbb{P}_K^{1, \mathrm{an}}: |R|_x \geqslant r\}$ (resp. $V \subseteq \{x \in \mathbb{P}_K^{1, \mathrm{an}}: |R|_x \leqslant r\}$).
\item Suppose $U \subseteq \{x \in \mathbb{P}_K^{1, \mathrm{an}}: |R|_x \leqslant r\}.$ Set $\partial{U}=\{\eta_{R,r}, \eta_{P_i, r_i}\}_{i=1}^n$ and ${\partial{V}=\{\eta_{R,r}, \eta_{P_j', r_j'}\}_{j=1}^m},$ so that $U=\{x \in \mathbb{P}_K^{1, \mathrm{an}}: |R|_x \leqslant r, |P_i|_x \bowtie_i r_i, i\}$ and $V=\{x \in \mathbb{P}_K^{1, \mathrm{an}}: |R|_x \geqslant r, |P_j'|_x \bowtie'_j r_j', j\}$, where $\bowtie_i, \bowtie_j' \in \{\leqslant, \geqslant\},$ $P_i, P_j' \in K[T]$ are irreducible, and $ r_i, r_j' \in \mathbb{R}_{>0} \backslash \sqrt{|K^{\times}|}$ for all $i,j.$
  
Then $U \cup V=\{x \in \mathbb{P}_K^{1, \mathrm{an}}:|P_i|_x \bowtie_i r_i, |P_j'|_x \bowtie_j' r_j', i=1,\dots, n,j=1,\dots, m\}.$ If $n=m=0,$ this means that $U \cup V=\mathbb{P}_K^{1, \mathrm{an}}.$
\end{enumerate}
\end{lm}
\begin{proof} 
(1) Remark that if $U \subseteq V,$ then $U=\{\eta_{R,r}\},$ so the statement is trivially satisfied. The same is true if $V \subseteq U.$ Let us suppose that neither of $U, V$ is contained in the other.
\begin{sloppypar}
Suppose $U \subseteq \{x \in \mathbb{P}_K^{1, \mathrm{an}}: |R|_x \leqslant r\}$ and $V \subseteq \{x \in \mathbb{P}_K^{1, \mathrm{an}}: |R|_x \leqslant r\}.$ Let ${u \in U \backslash V}$ and $v \in V \backslash U.$ Since $u, v \in \{x:|R|_x < r\}$ - which is a connected set (Lemma \ref{62}), ${[u,v] \subseteq \{x: |R|_x<r\}}.$ At the same time, since $[u, \eta_{R,r}] \subseteq U$ and $[\eta_{R,r}, v] \subseteq V,$ ${[u, \eta_{R,r}] \cap [\eta_{R,r}, v]=\{\eta_{R,r}\}},$ so the arc $[u,v]=[u, \eta_{R,r}] \cup [\eta_{R,r}, v]$ contains the point ~$\eta_{R,r}.$ This is in contradiction with the fact that $[u,v] \subseteq \{x:|R|_x < r\}.$

The case $U, V \subseteq \{x \in \mathbb{P}_K^{1, \mathrm{an}}: |R|_x \geqslant r\}$ is shown to be impossible in the same way.  (Property (1) is in fact true regardless of whether $\partial{U} \backslash \{\eta_{R,r}\}$ and $\partial{V} \backslash \{\eta_{R,r}\}$ contain only type~3 points or not.)
\end{sloppypar}
(2) The statement is clearly true if $m=n=0,$ so we may assume that is not the case.

Remark that $\partial(U \cup V) \subseteq \partial{U} \cup \partial{V}.$ Let $\eta \in \partial{U} \backslash V.$ Let $G$ be any neighborhood of $\eta$ in $\mathbb{P}_K^{1, \mathrm{an}}.$ Since $V$ is closed, there exists a neighborhood $G' \subseteq G$ of $\eta$ such that $G' \cap V=\emptyset.$ Since $\eta \in \partial{U},$ $G'$ contains points of both $U$ and $U^C.$ Consequently, $G'$, and thus $G$, contain points of both $U \cup V$ and $U^C \cap V^C=(U \cup V)^C.$ Seeing as this is true for any neighborhood $G$ of $\eta,$ we obtain that $\eta \in \partial(U \cup V),$ implying $\partial{U} \backslash V \subseteq \partial{(U \cup V)}.$ Similarly, $\partial{V} \backslash U \subseteq \partial(U \cup V).$ It only remains to check for the point $\eta_{R,r}.$ 

Let $x \in \text{Int}(U) \subseteq \text{Int}(U \cup V)$ and $y \in \text{Int}(V) \subseteq \text{Int}(U \cup V).$ Remark that $x \not \in V$ and $y \not \in U.$ Furthermore, $|R|_x<r$ and $|R|_y>r.$ Consequently, $\eta_{R,r} \in [x,y].$ Since $U \cup V$ is a connected affinoid domain containing only type 3 points in its boundary, its interior is connected (see Lemma \ref{62}). Consequently, $[x, y] \subseteq \text{Int}(U \cup V),$ and hence $\eta_{R,r} \in \text{Int}(U\cup V).$ 

We have shown that ${\partial(U \cup V)=\{\eta_{P_i, r_i}, \eta_{P_j', r_j'}: i,j\}}.$ Since ${U \subseteq \{x:|P_i|_x \bowtie_i r_i\}}$ and $V \subseteq \{x: |P_j'|_x \bowtie_j' r_j'\}$ for all $i,j,$ we obtain that $$U \cup V=\{x: |P_i|_{x} \bowtie_i  r_i, |P_j'|_x \bowtie_j' r_j', i,j\}.$$
\end{proof}

The next result describes certain affinoid domains of the analytic projective line. 

\begin{lm}\label{para}
Let $K$ be a complete ultrametric field. Let $R(T)$ be a split unitary polynomial over ~$K.$ Let $r \in \mathbb{R}_{>0}.$ Then for any root $\alpha$ of $R(T)$ there exists a unique positive real number $s_\alpha$ such that $\{y \in \mathbb{P}_K^{1,\mathrm{an}}: |R(T)|_y=r\}=\bigcup_{R(\alpha)=0}\{y \in \mathbb{P}_K^{1,\mathrm{an}}: |T-\alpha|_y=s_{\alpha}\}.$ The point $\eta_{\alpha, s_{\alpha}}$ is the only point $y$ of the arc $[\eta_{\alpha, 0}, \infty]$ in $\mathbb{P}_K^{1, \mathrm{an}}$ for which $|R(T)|_y=r.$ Furthermore, $r=s_{\alpha} \cdot \prod_{R(\beta) = 0, \alpha \neq \beta} \max(s_{\alpha}, |\alpha-\beta|).$
\end{lm}

\begin{proof} 
Remark that if  $y \in \mathbb{P}_K^{1,\mathrm{an}}$ is such that ${|R(T)|_y=0},$  then ${\prod_{R(\alpha)=0} |T-\alpha|_y=0}$, meaning there exists a root $\alpha_0$ of $R(T)$ such that $|T-\alpha_0|_{y}=0$ (notice that we haven't assumed $R(T)$ to be separable, \textit{i.e.} there could be roots with multiplicities).  This determines the unique point $\eta_{{\alpha}_0, 0}$ in $\mathbb{P}_K^{1,\mathrm{an}}.$ Thus, the zeros of $R(T)$ in $\mathbb{P}_K^{1,\mathrm{an}}$ are $\eta_{\alpha, 0}, R(\alpha)=0.$ Remark also that $R$ has only one pole in $\mathbb{P}_K^{1, \mathrm{an}}$ and that is the point $\infty.$ 

By \cite[3.4.23.1]{Duc}, the analytic function $R(T)$ on $\mathbb{P}_K^{1,\mathrm{an}}$ is locally constant everywhere outside of the finite graph $\Gamma:=\bigcup_{R(\alpha)=0} [\eta_{\alpha,0}, \infty].$ Furthermore, its variation is compatible with the canonical retraction $d:\mathbb{P}_K^{1, \mathrm{an}} \rightarrow \Gamma$ in the sense that $|R(T)|_y=|R(T)|_{d(y)}$ for any $y \in \mathbb{P}_K^{1, \mathrm{an}}$ (\textit{cf.} \cite[3.4.23.8]{Duc}). By \cite[3.4.24.3]{Duc}, $R(T)$ is continuously strictly increasing in all the arcs $[\eta_{\alpha, 0}, \infty], R(\alpha)=0,$ where $|R(T)|_{\eta_{\alpha, 0}}=0$ and $|R(T)|_{\infty}=+\infty.$ Consequently, $|R(T)|$ attains the value $r$ exactly one time on each arc $[\eta_{\alpha, 0}, \infty].$ Suppose $s_{\alpha}$ is the unique positive real number for which $|R(T)|_{\eta_{\alpha, s_{\alpha}}}=r.$ Then $\prod_{R(\beta)=0}|T-\beta|_{\eta_{\alpha, s_{\alpha}}}=s_\alpha \cdot \prod_{R(\beta)=0, \alpha\neq \beta} \max(s_{\alpha}, |\alpha-\beta|)=r.$ 

We have shown that there exist positive real numbers $s_{\alpha}$ such that ${\{y \in \Gamma: |R|_y=r\}=}$ $\{\eta_{\alpha, s_{\alpha}}: R(\alpha)=0\}.$ As mentioned before, the variation of $R$ is compatible with the canonical retraction $d$ of $\mathbb{P}_K^{1, \mathrm{an}}$ to $\Gamma.$ Since $d^{-1}(\eta_{\alpha,s_{\alpha}})=\{y \in \mathbb{P}_K^{1, \mathrm{an}}: |T-\alpha|_y=s_{\alpha}\},$ we finally obtain that $\{y \in \mathbb{P}_K^{1, \mathrm{an}}: |R|_y=r\}=\bigcup_{R(\alpha)=0} \{y \in \mathbb{P}_K^{1, \mathrm{an}}: |T-\alpha|_y=s_{\alpha}\}$ with $s_{\alpha}$ as above.  
\end{proof}

\bigskip

{\footnotesize%
 \textsc{Vler\"e Mehmeti}, Laboratoire de math\'ematiques d'Orsay, Universit\'e Paris-Saclay, 91405 ~Orsay Cedex, France \par
  \textit{E-mail address}: \texttt{vlere.mehmeti@u-psud.fr} 
}

\end{document}